\documentclass[11pt,openright,twoside,a4paper,reqno]{amsbook}
\usepackage[a4paper,top=3cm,bottom=3cm,inner=4cm,outer=3cm]{geometry}
\usepackage[english]{babel}
\usepackage{mathrsfs,amsmath,amsfonts,amsthm,amssymb,graphicx,pdfpages}
\usepackage{setspace,fancyhdr,calc,verbatim,bbm,enumerate}
\usepackage{hyperref}
\usepackage[all,2cell]{xy}\UseAllTwocells\SilentMatrices

\usepackage{longtable}
\usepackage{upref}
\usepackage{mathtools}

\newtheorem{thm}{Theorem}[chapter]
\newtheorem{lem}[thm]{Lemma}
\newtheorem{prp}[thm]{Proposition}
\newtheorem{cor}[thm]{Corollary}

\newtheorem{teo}{Theorem}
\newtheorem{pro}[teo]{Proposition}
\newtheorem{crl}[teo]{Corollary}
\newtheorem{con}[teo]{Conjecture}

\theoremstyle{definition}
\newtheorem{dfn}[thm]{Definition}
\newtheorem{exa}[thm]{Example}

\theoremstyle{remark}
\newtheorem{rmk}[thm]{Remark}

\numberwithin{equation}{chapter}

\newcommand{\R}{\mathbb R}
\newcommand{\Z}{\mathbb Z}
\newcommand{\N}{\mathbb N}
\newcommand{\T}{\mathbb T}
\newcommand{\pri}{\mathcal P}
\newcommand{\Pro}{\mathbb P}
\newcommand{\C}{\mathbb C}
\newcommand{\Q}{\mathbb Q}
\newcommand{\Si}{\Sigma}
\newcommand{\op}[1]{\operatorname{#1}}

\makeatletter
\@addtoreset{thm}{chapter}
\makeatother

\onehalfspace

\makeatletter
\renewcommand{\l@chapter}{\@tocline{0}{12pt}{0pt}{}{\bfseries}}
\renewcommand{\l@subsection}{\@tocline{2}{0pt}{2pc}{5pc}{}}
\makeatother

\mainmatter

\begin{document}

\frontmatter

\thispagestyle{empty}

\begin{center}
\mbox{}
{\Large{\bf THE CONTACT PROPERTY\smallskip

FOR MAGNETIC FLOWS ON SURFACES}}
\vskip 0.7in
{\Large{Gabriele Benedetti}}
\vskip 0.4in
{\small Trinity College \\
and \\
Department of Pure Mathematics and Mathematical Statistics \\
University of Cambridge}
\vskip 1.2in
\includegraphics[height=2in]{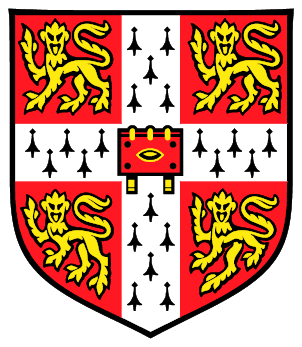}
\vskip 1.2in
This dissertation is submitted for the degree of \\
\emph{Doctor of Philosophy}
\vskip 0.4in
July 2014
\end{center}

\cleardoublepage

\pagestyle{fancy}
\fancyhf{} 
\renewcommand{\headrulewidth}{0pt} 
\pagenumbering{roman}
\begin{flushright}
\null \vspace {\stretch {1}}
\itshape A David e Marisa
\vspace {\stretch {2}}\null
\end{flushright}
\cleardoublepage

\newpage
\mbox{}
\vskip 2in
\begin{quote}
This dissertation is the result of my own work and includes nothing that is the outcome of work done in collaboration except 
where specifically indicated in the text.

This dissertation is not substantially the same as any that I have
submitted for a degree or diploma or any other qualification at any
other university.
\vskip 1in

\hfill Gabriele Benedetti

\hfill \today

\end{quote}

\cleardoublepage
\newpage
\begin{center}
{\large{\bfseries The contact property for magnetic flows on surfaces\\}}
\vskip 0.1in
Gabriele Benedetti
\vskip 0.3in
{\sc Summary}
\vskip 0.1in
\end{center}

\begin{singlespace}
This work investigates the dynamics of \textit{magnetic flows} on closed orientable Riemannian surfaces. These flows are determined by triples $(M,g,\sigma)$, where $M$ is the surface, $g$ is the metric and $\sigma$ is a $2$-form on $M$. Such dynamical systems are described by the Hamiltonian equations of a function $E$ on the tangent bundle $TM$ endowed with a symplectic form $\omega_\sigma$, where $E$ is the kinetic energy. Our main goal is to prove existence results for
\begin{center}
a) periodic orbits\quad\quad and \quad\quad b) Poincar\'e sections
\end{center}
for motions on a fixed energy level $\Sigma_m:=\{E=m^2/2\}\subset TM$.

We tackle this problem by studying the contact geometry of the level set $\Sigma_m$. This will allow us to
\begin{enumerate}[a)]
 \item count periodic orbits using algebraic invariants such as the Symplectic Cohomology
$SH$ of the sublevels $(\{E\leq m^2/2\},\omega_\sigma)$;
 \item find Poincar\'e sections starting from pseudo-holomorphic foliations, using the techniques developed by Hofer, Wysocki and Zehnder in 1998.
\end{enumerate}

In Chapter $3$ we give a proof of the invariance of $SH$ under deformation in an abstract setting, suitable for the applications.

In Chapter $4$ we present some new results on the energy values of contact type. First, we give explicit examples of exact magnetic systems on $\T^2$ which are of contact type at the strict critical value. Then, we analyse the case of non-exact systems on $M\neq\T^2$ and prove that, for large $m$ and for small $m$ with symplectic $\sigma$, $\Sigma_m$ is of contact type. Finally, we compute $SH$ in all cases where $\Sigma_m$ is convex. 

On the other hand, we are also interested in non-exact examples where the contact property fails. While for surfaces of genus at least two, there is always a level not of contact type for topological reasons, this is not true anymore for $S^2$. In Chapter $5$, after developing the theory of magnetic flows on surfaces of revolution, we exhibit the first example on $S^2$ of an energy level not of contact type. We also give a numerical algorithm to check the contact property when the level has positive magnetic curvature.

In Chapter $7$ we restrict the attention to low energy levels on $S^2$ with a symplectic $\sigma$ and we show that these levels are of dynamically convex contact type. Hence, we prove that, in the non-degenerate case, there exists a Poincar\'e section of disc-type and at least an elliptic periodic orbit. In the general case, we show that there are either $2$ or infinitely many periodic orbits on $\Sigma_m$ and that we can divide the periodic orbits in two distinguished classes, \textit{short} and \textit{long}, depending on their period. Then, we look at the case of surfaces of revolution, where we give a sufficient condition for the existence of infinitely many periodic orbits. Finally, we discuss a generalisation of dynamical convexity introduced recently by Abreu and Macarini, which applies also to surfaces with genus at least two. 
\end{singlespace}

\cleardoublepage
\newpage
\mbox{}
\vskip 1in
\begin{center}
{\sc Acknowledgements}
\vskip 0.1in
\end{center}
\begin{flushright}
\vspace{10pt}

{\itshape Sabr\`as que no te amo y que te amo\\
puesto que de dos modos es la vida,\\
la palabra es un ala del silencio,\\
el fuego tiene una mitad de fr\'io.}\\
\smallskip

Pablo Neruda, Soneto XLIV
\end{flushright} 
\vspace{20pt}

\begin{center}
I express my deep gratitude to the following people, things and places. 
\end{center}

To my advisor Gabriel Paternain, for many helpful discussions about the dynamics of the magnetic field and for his invaluable support in preparing this thesis.

To Peter Albers and Ivan Smith, for kindly agreeing to be the examiners of this thesis.

To Alex Ritter, for several conversations that have been fundamental to develop the material contained in Chapter \ref{cha_sh}.

To Jungsoo Kang, for a nice conversation about multiplicity results for Reeb orbits using equivariant Symplectic Cohomology. 

To Leonardo Macarini, for pointing out to me a mistake in Inequality \eqref{gendyn} contained in a previous version of this work.

To Viktor Ginzburg, for his helpful and kind comments on Proposition \ref{progin}.

To Marco Golla, for the simple and nice proof about the parity of the linking number contained in Lemma \ref{lemlin} and for his genuine interest in my work. 

To Joe Waldron, for sharing office and food with me every day for a year. 

To Jan Sbierski, for squash, badminton and basketball and a memorable formal dinner at Magdalene with \textit{Paris Breast} served for dessert.

To Paul Wedrich, for the delicious lunches at Churchill.

To Christian Lund, for coming to my office every day after lunch for push-ups training and for bringing me from a series of 5 to a series of 20 in one month.

To Nick Shepherd-Barron, because his disembodied and charming voice coming from the first floor has been a familiar presence for all the offices on the ground floor.

To the neighbouring Office E0.19, for all the moments of inspiration.

To Carmelo Di Natale, for his mighty hugs and his sincere blessings.

To Enrico Ghiorzi, for popping up in my office to say \textit{``Hello!''} at the end of the working day with the helmet already worn.

To Volker Schlue, for sending from Toronto a correct proof that $0=1$ and very accurate weather forecasts written in Korean.

To Peter Herbrich, for his crucial help during my first days in Cambridge, for bringing me with his red Polo around Scotland on a hiking trip in spite of rain and blisters and for inviting me to the \textit{Polterabend} and the fabulous wedding with Melina (\textit{``Cuscino!''}) in Lunow.

To Gareth, for helping me hiking my way through the PhD, from the green and wet valleys of the Isle of Skye to the windy and rocky peak of the Innerdalst\aa rnet, and for reminding me that \textit{``Power is nothing without control''}.

To John Christian Ottem, for all the good music, for climbing up a mountain only to rescue a plastic bag I thought I had left on top (it was in my backpack, in the end!), for still believing that I can arrive on time, for the beautiful drawing on the blackboard and for all the mathematical jokes and puzzles.

To Anna Gurevich, for the dancing in Lunow and the nice chats about feminism and particles moving on spheres.

To Achilleas Kryftis
, for all the runs we went for together, and for teaching me Greek in the meanwhile. I will miss the three llamas in Coton, the hill and a rocambolesque escape from the shooting range.

To Christina Vasilakopoulou
, for waving at me from the window every day as I entered Pavilion E from the back, for the movie-nights full of \textit{``nibbly things''} she has organised with Mitsos
\ at their place. 

To the inhabitants of CMS, present, past and future. Especially to Li Yan, Guido Franchetti, Micha\l\ Przykucki, Tamara Von Glehn, Guilherme Frederico Lima, Ha-Thu Nguyen, Sebastian Koch, Kirsty Wan, Kyriakos Leptos
, Bhargav Narayanan, Gunnar Peng, Rahul Jha, Ashok and Arun Thillaisundaram, Cangxiong Chen, Shoham Letzter, Dmitry Tonkonog, Claudius Zibrowius, Tom Gillespie, Vittoria Silvestri and Sara Ricc\`{o}.

To all the students I supervised in Cambridge, for their energy and their patience.

To Elena Giusti, because now she knows why the prime numbers are infinite.

To Thomas Bugno, and his sheer passion for knowledge. I am sorry that he got really upset to find out that $0!=1$. 

To Laura Moudarres, for her excellent potato bake and for sharing with me the passion for the salad with apple and cumin.

To Barbara Del Giovane, because she never talked about the weather, but preferred to listen to my doubts and speak about hers with
naturalness and empathy

To Andrew Chen, for making me proofread all the very formal messages he sent to Italian libraries for getting access to ancient books.

To Francesca Pariti, because her smile and good mood has been the perfect substitute for the English sun.

To Flavia Licciardello, for making me feel accepted and independent and an adult.

To Daniele Dorigoni, for his extraordinary generosity counterbalanced by his punting skills.

To Marco Mazzucchelli and Ana Lecuona, for all the wise advice about the future and for making my stay in Marseille the most heart-warming experience.

To Alfonso Sorrentino, for being a wonderful mentor and dear friend in my life and for bringing me to eat the best hummus on Earth.

To Giulio Codogni, because he is not afraid of asking ``Can I don't do the English test?''. He has always been an example to me and the protagonist of the funniest moments in Cambridge. Cheers a lot! 

To Fano \textit{``bello da vicino e pure da lontano''}, for being a really exquisite person and for having organised the best \textit{Mercante in Fiera} of my life.

To Luca Sbordone, because he is the most talented and entertaining storyteller. It doesn't make much difference to me whether what he says really happened or not. 

To Ugo Siciliani De Cumis, who shared with me the darkest hours, because he can cook amazing food (and suffer mine), fix all sort of things (which I had broken, of course) and build a mixer between the cold and hot tap. 

To the thirty-five guests of 24 St.\ Davids House (\textit{Lin}, you will be soon the thirty-sixth, last and long-waited!), for making this place a home to me.

To Mara Barucco and Simone Di Marino, for getting married.

To Pisa, Scuola Normale and Collegio Faedo, for all the magic encounters and because a part of me (as well as my shoes) has remained there in these three years.

To Pietro Vertechi, for his genuine hospitality in Lisbon and for chasing a bar up the stairs of a multi-storey car park.

To Valerio Melani and Maria Beatrice Pozzetti, for staying connected in spite of the distance.

To Anna Bell\'e, for bringing me the most unexpected \textit{schiacciatine} of my entire life when I was ill in bed.

To Claudia Flandoli, for visiting with me the cold and empty rooms of Palazzo Ducale in Urbino with a pair of crutches e \textit{``un sorriso coinvolgente''}.

To Alessandro Di Sciullo, Marco Cucchi, Jacopo Preziosi Standoli, Luca Grossi, Paolo Magnani, Luca Borelli and Ermanno Finotti, for countless \textit{goodbye dinners}. 

To Claudio Baglioni, Alessandro Mannarino, Giorgio Gaber, Samuele Bersani, Tool, Pain of Salvation, Banco del Mutuo Soccorso, Pink Floyd, Rachmaninov for being the soundtrack of these years.

To Il Nido Del Cuculo, for all the silly laughs in front of the screen.

To Psychodynamic therapy, for making me discover a part of me, which I didn't know.

To Cognitive Behavioural Therapy, for reminding me of a part of me that I know too well.

To Margherita, Eleonora, Andrea, Federica, Marisa, David, Agnese, Delvina, Placido and Riccardo for all their love: it has been shaped in the form of a thesis.
\cleardoublepage
\null \vspace {\stretch {1}}
\begin{flushright}
 
\begin{minipage}{10cm}
{\itshape Malo mortuum impendere quam vivum occidere.}

\medskip

\begin{flushright}
Petronius, Satyricon
\end{flushright} 

\vspace{2cm}

{\itshape Then there is an intellectual or a student or two, very few of them though, here and there, with ideas in their heads that are often vague or twisted. Their \textquotedblleft country\textquotedblright\ consists of words, or at the most of some books. But as they fight they find that those words of theirs no longer have any meaning, and they make new discoveries about men's struggles, and they just fight on without asking themselves questions, until they find new words and rediscover the old ones, changed now, with unsuspected meanings.}

\medskip

\begin{flushright}
Italo Calvino, The Path to the Spiders' Nests
\end{flushright}
\vspace{2cm}

{\itshape Every two or three days, at the moment of the check, he told me: "I've finished that book. Have you another one to lend me? But not in Russian: you know that I have difficulty with Russian." Not that he was a polyglot: in fact, he was practically illiterate. But he still \textquoteleft read\textquoteright\ every book, from the first line to the last, identifying the individual letters with satisfaction, pronouncing them with his lips and laboriously reconstructing the words without bothering about their meaning. That was enough for him as, on different levels, others take pleasure in solving crossword puzzles, or integrating differential equations or calculating the orbits of the asteroids.}

\medskip

\begin{flushright}
Primo Levi, The Truce
\end{flushright}

\end{minipage}
\end{flushright}

\vspace {\stretch {4}}\null
\cleardoublepage
\tableofcontents
\cleardoublepage

\pagestyle{headings}
\pagenumbering{arabic}

\chapter{Introduction}\label{cha_int}
\section{Overview of the problem}
Let the triple $(M,g,\sigma)$ represent a magnetic system, where $M$ is a closed manifold, $g$ is a Riemannian metric and $\sigma$ is a closed $2$-form on $M$. We say that a magnetic system is \textit{exact}, \textit{non-exact} or \textit{symplectic} if $\sigma$ is such. Such data give rise to a Hamiltonian vector field $X^{\sigma}_E$ on the symplectic manifold $(TM,\omega_\sigma:=d\lambda-\pi^*\sigma)$, where $\pi$ is the footpoint projection from $TM$ to $M$ and $\lambda$ is the pull-back of the Liouville $1$-form on the cotangent bundle via the duality isomorphism given by $g$. The kinetic energy $E(x,v)=\frac{1}{2}g_x(v,v)$ is the Hamiltonian function associated to $X^{\sigma}_E$. These systems model the motion of a charged particle on the manifold $M$ under the influence of a magnetic force, which enters the evolution equations via the term $-\pi^*\sigma$ in the symplectic form. In the absence of a magnetic force, namely when $\sigma=0$, the particle is free to move and we get back the geodesic flow of 
the Riemannian manifold $(M,g)$. The zero section is the set of rest points for the flow and all the smooth hypersurfaces $\Sigma_m:=\{E=\frac{m^2}{2}\}$, with $m>0$, are invariant sets. Hence, we can analyse the vector fields $X^\sigma_E\big|_{\Sigma_m}$ separately for every $m$ and compare the dynamics for different value of such parameter. In the geodesic case, no interesting phenomenon arises from this point of view. Even if the dynamics of the geodesic flow has a very rich structure and its understanding requires a deep study, the flows on $\Sigma_m$ and $\Sigma_{m'}$ are conjugated up to a constant time factor and, therefore, the parameter $m$ does not play any role in the analysis.

On the other hand, the situation for a general magnetic term $\sigma$ is quite different. As the parameter $m$ varies and passes through some distinguished values, the dynamics on $\Sigma_m$ can undergo significant transformations, in the same way as the topology of the levels of a Morse function changes when we cross a critical value \cite{cmp}. How do we detect and measure these differences?  Since the flows $\Phi^{X^\sigma_E}\big|_{\Sigma_m}$ can exhibit a very complicated behaviour, we have to focus our analysis on some simple property of these systems. In the present work we have chosen to look at the \textit{periodic orbits}. In particular, we examine \begin{enumerate}[a)]
 \item existence and multiplicity of periodic orbits on a given free homotopy class;
 \item knottedness of a single periodic orbit and linking of pairs, through the existence of global Poincar\'e sections and periodic points of the associated return map;
 \item local properties of the flow at a periodic orbit (e.g.\ stability), that are related to the linearisation of the flow (elliptic/hyperbolic orbits).
\end{enumerate}
Periodic orbits of magnetic flows received much attention in the last thirty years and here we seize the occasion to mention some of the relevant literature on the subject. 

When $M$ is a surface, the two classical approaches that have been pursued are Morse-Novikov theory and symplectic topology (see Ta{\u\i}manov's \cite{tai4} and Ginzburg's \cite{gin2} surveys for details and further references). Refinements of these old techniques and completely new strategies have been developed recently. Some authors work with (weakly) exact magnetic forms \cite{bahtai,pol1,mac1,cmp,osu,con,pat2,frasch2,mer1,tai5,amp1,ammp,ab}. Others seek periodic solutions with low kinetic energy \cite{schl}, the majority of them assuming further that $\sigma$ is symplectic \cite{ker,ginker1,ginker2,mac2,cgk,gg1,ker3,lu,gg2,ush}. Schneider's approach \cite{sch1,sch2,sch3} for orientable surfaces and symplectic $\sigma$ uses a suitable index theory for vector fields on a space of loops and shows in a very transparent way how the Riemannian geometry of $g$ influences the problem. Finally, we point out \cite{koh} where heat flow techniques are employed and \cite{mer2,fmp1,fmp2} which construct a Floer 
theory for particular classes of magnetic fields.

This thesis would like to give its contribution to the understanding of magnetic fields on \textit{closed orientable surfaces} by studying the contact geometry of the level sets $\Sigma_m$. In reference to point a), b), and c) presented above, this will allow us to
\begin{enumerate}[a)]
 \item count periodic orbits in a given free homotopy class $\nu$ using algebraic invariants such as the Symplectic Cohomology $SH_\nu$ of the sublevels $(\{E\leq \frac{m^2}{2}\},\omega_\sigma)$;
 \item find Poincar\'e sections starting from pseudo-holomorphic foliations \cite{hwz1,hls};
 \item determine whether a periodic orbit is elliptic or hyperbolic by looking at its Conley-Zehnder index \cite{dadoe,am}. 
\end{enumerate}

\section{Contact hypersurfaces and Symplectic Cohomology}
We say that a closed hypersurface $\Sigma$ in a symplectic manifold $(W,\omega)$ is of \textit{contact type} if there exists a primitive of $\omega$ on $\Sigma$ which is a contact form. Hypersurfaces of contact type have been intensively studied in relation to the problem of the existence of closed orbits. Indeed, in this case the Hamiltonian dynamics on $\Sigma$ is the dynamics of the Reeb flow associated to the contact form, up to a time reparametrisation. After some positive results in particular cases \cite{wei1,rab1,rab2}, in 1978 Alan Weinstein conjectured that every closed hypersurface of contact type (under some additional homological condition now thought to be unnecessary) carries a periodic orbit \cite{wei2}. The conjecture is still open in its full generality, but for magnetic systems on orientable surfaces is a consequence of the solution to the conjecture for every closed $3$-manifolds by Taubes \cite{tau1} (see also \cite{hut} for an expository account). Such proof uses Embedded Contact 
Homology (another kind of algebraic invariant) to count periodic orbits. Recently, Cristofaro-Gardiner and Hutchings have refined Taubes' approach and raised the lower bound on the number of periodic orbits in dimension $3$ to two \cite{crihut}. The case of irrational ellipsoids in $\C^2$ shows that their estimate is sharp, at least for lens spaces \cite{huttau1}.

When $\Sigma$ bounds a compact region $W^\Sigma\subset W$, we can also compare the orientation induced by the ambient manifold and the one induced by the contact form. We say that the hypersurface is of \textit{positive} contact type if these two orientations agree and it is of \textit{negative} contact type otherwise. The importance of such distinction relies on the fact that in the positive case, and under some additional assumption on the Chern class of $\omega$, we can define the Symplectic cohomology $SH^*$ of $(W^\Sigma,\omega)$. Its defining complex is generated by the cohomology of $W^\Sigma$ and the periodic orbits on $\Sigma$. Since the cohomology of the interior is known, if we can compute $SH^*$, we gain information about the periodic orbits on the boundary. As is typical in Floer theory, such computation is divided into two steps:
\begin{enumerate}
 \item finding explicitly $SH^*$ in simple model cases;
 \item proving that $SH^*$ is invariant under a particular class of deformations, that bring the case of interest to one of the models.
\end{enumerate}

We prove such invariance in Chapter \ref{cha_sh} in a setting that will be useful for the applications to magnetic flows on surfaces.  This result has been inspired to us by reading \cite{ar}, where a similar invariance is proven in the setting of ALE spaces (see \cite[Theorem 33 and Lemma 50]{ar}. We warmly thank Alexander Ritter for several useful discussions on this topic. Let us now give the precise statement. 
\begin{teo}\label{thminv_int}
Let $W$ be an open manifold and let $\omega_s$ be a family of symplectic forms on $W$, with $s\in[0,1]$. Suppose $W_s\subset W$ is a family of zero-codimensional embedded compact submanifold in $W$, which are all diffeomorphic to a model $W'$. Let $(W_s,\omega_s\big|_{W_s},j_s)$ be a convex deformation. Fix a free homotopy class of loops $\nu$ in $W'\simeq W_s$ such that the contact forms $\alpha_0$ and $\alpha_1$ are both $\nu$-non-degenerate. If $c_1(\omega_s)$ is $\nu$-atoroidal for every $s\in[0,1]$ and $s\mapsto\omega_s$ is projectively constant on $\nu$-tori, then
\begin{equation*}
SH^*_\nu(W_0,\omega_0,j_0)\simeq SH^*_\nu(W_1,\omega_1,j_1). 
\end{equation*}
\end{teo}
Before commenting on this result, we clarify the terminology used. We refer to Chapter \ref{cha_sh} for a more thorough discussion. The map $j_s$ denotes a collar of the boundary. The normal vector field associated to $j_s$ induces the contact form $\alpha_s$ on $\partial W_s$. By \textit{convex} deformation, we mean a deformation for which the boundary stays of positive contact type. A contact form is \textit{$\nu$-non-degenerate}, if all its periodic orbits in the class $\nu$ are transversally non-degenerate. A two-form $\rho$ on $W'$ is said to be \textit{$\nu$-atoroidal} if the cohomology class of its transgression $[\tau(\rho)]\in H^1(\mathscr L_\nu W')$ is zero. A family of two-forms $\{\omega'_s\}$ on $W'$ is said to be \textit{projectively constant on $\nu$-tori}, if the class $[\tau(\omega'_s)]$ is independent of $s$ up to a positive factor. 

Symplectic Cohomology for manifolds with boundary of positive contact type was introduced by Viterbo in \cite{vit}. In \cite[Theorem 1.7]{vit} (see also \cite[Theorem 2.2]{oan}) he proves the invariance of $SH^*$ for contractible loops when all the $\omega_s$ are $0$-atoroidal (in other words, they are aspherical). However, we do not understand his proof. In particular, when he deals with the ``Generalized maximum principle'' (Lemma 1.8), the coordinate $z$ is treated in the differentiation as if it did not depend explicitly on the variable $s$ in contrast to what happens in the general case.

In \cite{bf}, Bae and Frauenfelder prove a similar result for the Symplectic Cohomology of contractible loops on a closed manifolds and for Rabinowitz Floer Homology of twisted cotangent bundles. In their setting all the $\omega_s$ are aspherical and the primitives of $\omega_s$ to the universal cover of $W$ grow at most linearly at infinity. It would interesting to find out whether the deformation invariance that we get in Theorem \ref{thminv_int} and the one Bae and Frauenfelder get in \cite{bf} fit into the long exact sequence between Symplectic Homology, Symplectic Cohomology and Rabinowitz Floer Homology \cite{cfo}. 

As far as the hypotheses on the symplectic forms are concerned we observe two things. First, the fact that $c_1(\omega_s)$ is $\nu$-atoroidal has two consequences: Symplectic Cohomology is well-defined since $c_1(\omega_s)$ is also aspherical and, hence, no holomorphic sphere can bubble off \cite{hofsal}; Symplectic Cohomology is $\Z$-graded. Second, the fact that $\omega_s$ is projectively constant on $\nu$-tori implies that the local system of coefficients associated to $\tau(\omega_s)$, which appears as a weight in the definition of the Floer differential, is independent of $s$ up to isomorphism.

We construct the isomorphism mentioned in Theorem \ref{thminv_int} in two steps. First, in virtue of Gray's Theorem we find an auxiliary family of symplectic manifolds $(W_s',\omega_s')$ such that $(W_s',\omega_s')$ is isomorphic to $(W_s,\omega_s)$ and the support of $\omega_{s}'-\omega_0'$ is disjoint from the boundary. Then, we define a class of admissible paths of pairs $\{(H_s,J_s)\}$ such that $H_s$ is uniformly small and $J_s$ is uniformly bounded on the support of $\omega_s'-\omega_0'$. In this way the $1$-periodic orbits of $X_{H_s}$ does not depend on $s$. The moduli spaces of Floer cylinders corresponding to $\{(H_s,J_s)\}$ satisfy uniform bounds on the energy, since by a Palais-Smale Lemma, the time they spend on the support of $\omega_s'-\omega_0'$ is uniformly bounded. This implies that the moduli spaces are compact and, therefore, we can use them to define continuation homomorphisms between the chain complexes $SC^*_\nu(W_s,\omega_s',j_s',H_s,J_s)$. Such homomorphisms will be weighted using the local system of coefficients associated to $\tau(\omega_s)$, so that they will commute with the Floer differentials, yielding maps in cohomology. Two homotopies of homotopies arguments show that these maps are isomorphisms and that they commute with the direct limit.
\medskip

In Chapter \ref{cha_exa} we apply the general theory of contact hypersurfaces to magnetic systems $(M,g,\sigma)$, in order to prove the existence of periodic orbits. We denote by $\op{Con}^+(g,\sigma)$, respectively $\op{Con}^-(g,\sigma)$, the set of all $m$ such that $\Sigma_m$ is of positive, respectively negative, contact type. 

Historically, the first examples that have been studied are exact systems, since these can be equivalently described using tools from Lagrangian mechanics. In this case we can define $m_0(g,\sigma)$ and $m(g,\sigma)$ the \textit{Ma\~n\'e critical values} of the abelian and universal cover, respectively (after the reparametrisation $m\mapsto \frac{m^2}{2}$). As far as the contact property is concerned, it is known that
\begin{itemize}
 \item for $m>m_0(g,\sigma)$, $m\in\op{Con}^+(g,\sigma)$ and, up to time reparametrisation, the dynamics is given by the geodesic flow of a Finsler metric \cite{cipp};
 \item for $m\leq m_0(g,\sigma)$, $\Sigma_m$ is not of restricted contact type, namely the contact form cannot be extended to the interior, since there exists $M_m\subset M$ a compact manifold with non-empty boundary such that $\partial M_m$ is the union of supports of periodic orbits whose total action is negative \cite{tai2,cmp}. Each of these orbits is a \textit{waist}: namely a local minimiser for the action functional. If $M\neq\T^2$, this implies that $\op{Con}^+(g,\sigma)=(m_0(g,\sigma),+\infty)$, while, if $M=\T^2$, there are examples such that $m_0(g,\sigma)\in\op{Con}^+(g,\sigma)$ \cite{cmp}.
\end{itemize}

As far as the periodic orbits are concerned, it is known that
\begin{itemize}
 \item for $m>m(g,\sigma)$ there exists at least a periodic orbit in every non-trivial free homotopy class;
 \item for almost every $m\leq m(g,\sigma)$ there is a contractible periodic orbit \cite{con} and infinitely many periodic orbits in the same homotopy class of a waist \cite{amp1,ammp}. Each of these orbits is a mountain pass, namely is obtained by a minimax argument on a $1$-dimensional family of loops.
\end{itemize}

We give a survey of the results about the contact property in Section \ref{sec_exa}. The only original element we add is an explicit construction of the contact structures at $m_0(g,\sigma)$ for the two-torus. Indeed, the argument in the original paper \cite{cmp} was not constructive. As in the supercritical case, the Symplectic Cohomology of the filling is isomorphic to the singular homology of the loop space. In particular, $SH^*_0\neq 0$. For restricted contact type hypersurfaces, Ritter proved in \cite[Theorem 13.3]{ritqft} that the non-vanishing of Symplectic Cohomology implies the non-vanishing of Rabinowitz Floer Homology (see \cite{cfo} for the relation between these two homology theories). By Theorem 1.2 in \cite{cf}, the non-vanishing of $RFH$ implies the non-displaceability of the boundary. Hence, if Ritter's theorem could be generalised to the filling of arbitrary hypersurfaces of contact type, we would have proven that the critical energy level is non-displaceable (see \cite{cfp} for a 
discussion of displaceability of hypersurfaces in twisted tangent bundles).    

We proceed to the non-exact case in Section \ref{sec_ne}. The first condition we need is for the $2$-form $\omega_\sigma$ to be exact on $\Sigma_m$. This happens if and only if $M\neq\T^2$. Then, we look for a contact form within the class of primitives defined in \eqref{fun_nexa}. In Corollary \ref{cor_negen} and Corollary \ref{cor_s2}, we find that $m\in\op{Con}^+(g,\sigma)$ for $m$ large enough and we compute the Symplectic Cohomology of the filling in terms of the Symplectic Cohomology of the model cases. 
\begin{pro}
Let $(M,g,\sigma)$ be a magnetic system on a surface different from the two-torus. If $m$ belongs to the unbounded component of $\op{Con}^+(g,\sigma)$, then 
\begin{equation*}
SH^*\left(\left\{E\leq \frac{m^2}{2}\right\},\omega_\sigma\right)=
\begin{cases}
H_{-*}(\mathscr L M,\Z)&\mbox{ if }M \mbox{ has positive genus},\\
0&\mbox{ if }M=S^2.
\end{cases}
\end{equation*}
\end{pro}

If the magnetic form is symplectic we also show in Section \ref{sub_s2} and Section \ref{sub_lg} that low energy levels are
\begin{itemize}
 \item of positive contact type, if $M=S^2$;
 \item of negative contact type, if $M$ is a surface of higher genus. 
\end{itemize}
In the first case, we compute $SH$ and we again find that it vanishes. In the second case, we cannot define $SH$ of the filling, since the Liouville vector field points inwards at the boundary. However, it would be interesting to find a compact symplectic manifold $(W',\omega')$ with boundary of positive contact type that can be glued to $\Sigma_m$ in such a way that $(\{E\leq \frac{m^2}{2}\},\omega_\sigma)\sqcup_{\Sigma_m}(W',\omega')$ is a closed symplectic manifold. 

Finally, in Section \ref{sec_lb}, we use the computation of Symplectic Cohomology to reprove some known lower bounds on the number of periodic orbits in the exact and non-exact case. We collect them in the following proposition.
\begin{pro}
Let $(M,g,\sigma)$ be a magnetic system and suppose that $m$ belongs either to the unbounded component of $\op{Con}^+(g,\sigma)$, or, if $M=S^2$ and $\sigma$ is symplectic, to the component of $\op{Con}^+(g,\sigma)$ containing $0$. Then,
\begin{itemize}
 \item if $M=S^2$, there exists a periodic orbit on $\Sigma_m$. If all the iterates of this periodic orbit are non-degenerate, then there exists another geometrically distinct periodic orbit;
 \item if $M=\T^2$, there exists a periodic orbit on $\Sigma_m$ in every non-trivial free homotopy class. If such orbit is non-degenerate, then there exists another geometrically distinct periodic orbit in the same class;
 \item if $M$ is a surface of higher genus, there exists one periodic orbits on $\Sigma_m$ in every non-trivial free homotopy class. 
\end{itemize}
\end{pro}

\section{The contact property on surfaces of revolution}

In Chapter \ref{cha_rev} we look at the contact property on surfaces of revolution, in order to test the general results contained in Chapter \ref{cha_exa} in a concrete case. We construct the surface $S^2_\gamma\subset\R^3$ by rotating a profile curve $(\gamma,\delta)\subset\R^2$ parametrised by arc-length. We study the magnetic system $(S^2_\gamma,g_\gamma,\mu_\gamma)$ where $g_\gamma$ is the restriction of the Euclidean metric on $\R^3$ to $S^2_\gamma$ and $\mu_\gamma$ is the area form associated to this metric. Up to a homothety, we also assume that the area of the surface is $4\pi$. From the previous discussion, we know that there exists two positive values $m_{-,\gamma}<m_{+,\gamma}$ such that $[0,m_{-,\gamma})\cup(m_{+,\gamma},+\infty)\subset\op{Con}^+(g_\gamma,\mu_\gamma)=:\op{Con}_\gamma$. These can be taken to be the two roots of the quadratic equation $m^2-m_\gamma m+1=0$ (where $m_\gamma\geq0$ is defined below) if such roots are real. In this case $m_\gamma\geq2$ and the length of the gap $m_{+,\
gamma}-m_{-,\gamma}$ is $\sqrt{m_\gamma^2-4}$ and, hence, it increases with $m_\gamma$. If the roots are not real, or, in other words, $m_\gamma<2$, then we simply have $\op{Con}_\gamma=[0,+\infty)$. The number $m_\gamma$ depends on the Riemannian geometry of $S^2_\gamma$. It is defined by
\begin{equation*}
m_\gamma:=\inf_{\beta\in\pri^{(1-K)\mu_\gamma}}\Vert\beta\Vert,
\end{equation*}
where $K$ is the Gaussian curvature and $\pri^{(1-K)\mu_\gamma}$ is the set of primitives of $(1-K)\mu_\gamma$. We give an explicit formula for $m_\gamma$ in terms of the function $\gamma$ in Proposition \ref{estm}. It can be used to get the following estimate on the contact property.

\begin{pro}
If $S^2_\gamma$ is symmetric with respect to the equator and the curvature increases from the poles to the equator, then $m_\gamma\leq1$. Therefore, $\op{Con}_\gamma=[0,+\infty)$.

On the other hand, for every $C>0$, there exists a convex $S^2_\gamma$ such that $m_\gamma>C$.
\end{pro}
The second part of the proposition above relies on the fact that $m_\gamma$ can be arbitrarily big provided the curvature is sufficiently concentrated around at least one of the poles. Thus, for these surfaces, the contact property can be proved only for a small set of parameters.
 
Hence, we are led to ask, in general, how good is the set $[0,m_{-,\gamma})\cup(m_{+,\gamma},+\infty)$ in approximating $\op{Con}_\gamma$, the actual set of energies where the contact property holds. For this purpose we employ McDuff's criterion \cite{mcd}, which says that $\Sigma_m$ is of contact type provided all the invariant measures supported on this hypersurface have positive action. Finding the actions of an invariant measure is usually a difficult task. However, for surfaces of revolution there are always some latitudes that are the supports of periodic orbits. We compute the action of such latitudes in Proposition \ref{paract}. If the magnetic curvature $K_m:\Sigma_m\rightarrow\R$, defined as $K_m:=m^2K+1$ is positive, we only have two periodic orbits which are latitudes (see Proposition \ref{reddyn}). By Proposition \ref{paract}, their action is positive. Therefore, they do not represent an obstruction to the contact property. In addition, under the same curvature assumption, we have a simple 
description of the dynamics of the system after reduction by the rotational symmetry. In particular, this allows us to devise a numerical strategy to compute the action of all the ergodic invariant measures as we explain in Section \ref{sub_act}. The data we have collected suggest that all such actions are positive, hinting, therefore, at the following conjecture.
\begin{con}\label{conj}
Let $(S^2,g,\sigma)$ be a symplectic magnetic system and suppose that for some $m>0$, the magnetic curvature is positive. Then, $\Sigma_m$ is of contact type.
\end{con}
The numerical computations, and possibly an affirmative answer to the conjecture, would then indicate that the system $(S^2_\gamma,g_\gamma,\mu_\gamma)$ associated to a \textit{convex} surface would be of contact type at every energy level. This shows that the inclusion $[0,m_{-,\gamma})\cup(m_{+,\gamma},+\infty)\subset\op{Con}_\gamma$ will be strict, in general. Establishing the conjecture would also yield another proof of Corollary 1.3 in \cite{sch2} about the existence of two closed orbits on every energy level, when $K\geq0$ and $f>0$.

To complete the picture, we see in Proposition \ref{nonnec} that positive magnetic curvature is not necessary for having the contact property. Moreover, using again Proposition \ref{paract}, we give the first known example of an energy level on a non-exact magnetic system on $S^2$ which is not of contact type (Proposition \ref{noncon}). This shows that the inclusion $\op{Con}_\gamma\subset[0,+\infty)$ can be strict, as well.

\section{Low energy levels of contact type on the two-sphere}

In the last chapter of the thesis we focus on low energy levels of general symplectic magnetic systems on $S^2$. We can write $\sigma=f\mu$, where $\mu$ is the Riemannian area and $f:S^2\rightarrow\R$ is a function. Without loss of generality we assume that $\int_{S^2}\sigma=4\pi$ and that $f$ is positive. By the previous discussion, we know that these levels are of contact type. Hence, there exists a family of primitives $m\mapsto\lambda_m$ of $\omega_\sigma|_{\Sigma_m}$ made of contact forms. In Lemma \ref{lemconb} we see that, using Gray Stability Theorem, one can find diffeomorphisms $F_m:\Sigma_1\rightarrow\Sigma_m$ such that
\begin{equation*}
\check{\lambda}_m:=F_m^*\lambda_m=\frac{\check{\lambda}_0}{\rho_m}+o(m^2),\quad\quad\mbox{where }\ \rho_m:=1-\frac{m^2}{2f}    
\end{equation*}
The form $\check{\lambda}_0$ is an $S^1$-connection on $SS^2:=\Sigma_1$ with curvature $\sigma$. In Lemma \ref{lem_intexp} we show that the Reeb vector field of $\check{\lambda}_0/\rho_m$ can be written as the composition of the $\sigma$-Hamiltonian flow with Hamiltonian $\rho_m$ on the base $S^2$ and a rotation in the fibres with angular speed $\rho_m$. This allows us to expand Ginzburg's action function $S_m:SS^2\rightarrow\R$ for the form $\check{\lambda}_m$, defined in Section \ref{sec_gin}, in the parameter $m$ around zero. This function was introduced for the first time in \cite{gin1} using local Poincar\'e sections. Its critical points are those periodic orbits of the Reeb flow of $\check{\lambda}_m$ which are close to a curve that winds once around a fibre of $SS^2\rightarrow S^2$.
\begin{pro}\label{progin}
The following expansion holds
\begin{equation}
S_m(x,v)=2\pi+\frac{\pi}{f(x)}m^2+o(m^2).
\end{equation}
As a consequence, if $x\in S^2$ is a non-degenerate critical point of $f$, then there exists a family of loops $m\mapsto\gamma_m$, such that $\gamma_0$ winds uniformly once around $S_xS^2$ in the positive sense and the support of $F_m(\gamma_m)$ is a periodic orbit on $\Sigma_m$. 
\end{pro}

The existence of periodic orbits close to non-degenerate critical points is stated without proof in \cite{gin2}.

Now we move to analyse the Conley-Zehnder indices of the Reeb flow of $\lambda_m$ and present some dynamical corollaries of this analysis. As $M=S^2$, $\Sigma_m$ is diffeomorphic to the lens space $L(2,1)\simeq \R\Pro^3$. On lens spaces we can identify the distinguished class of \textit{dynamically convex} contact forms. We say that a contact form is dynamically convex if the Conley-Zehnder index $\mu_{\op{CZ}}$ of every contractible periodic orbit of the Reeb flow is at least $3$. Such forms were introduced by Hofer, Wysocki and Zehnder \cite{hwz1} on $S^3$ as a contact-invariant generalisation of convex hypersurfaces in $\C^2$. If $\tau\in\Omega^1(L(p,q))$ is a dynamically convex contact form, we have the following two implications on the dynamics of the associated Reeb flow $R^\tau$.
\begin{enumerate}[(i)]
 \item\label{pun1} There exists a global Poincar\'e section of disc-type for $R^\tau$, under the condition that, when $p>1$ (namely $L(p,q)\neq S^3$), all periodic orbits are non-degenerate. 
 \item\label{pun2} When $p>1$ and all periodic orbits are non-degenerate, there exists an elliptic periodic orbit for $R^\tau$.
\end{enumerate}
Point (\ref{pun1}) was proven in \cite{hwz1} for $p=1$. The proof for $p>1$ is contained in \cite{hls}. Probably the non-degeneracy assumption can be removed by running the same lengthy approximation argument contained in \cite{hwz1}.
Point (\ref{pun2}) was proven for convex hypersurfaces in $\C^2$ symmetric with respect to the origin in \cite{dadoe}. A proof of the general case stated above was recently announced in \cite{am}.

Dynamical convexity was proven in the context of the \textit{standard} tangent bundle by Harris and G. Paternain \cite{pathar}. They showed that, if $(S^2,g)$ is a Riemannian two-sphere with $1/4$-pinched curvature, the geodesic flow is dynamically convex on every energy level. In this thesis we prove dynamical convexity for \textit{twisted} tangent bundles.
\begin{teo}\label{teo_dc}
Let $(S^2,g,\sigma)$ be a symplectic magnetic system. If $m$ is low enough, there exists a primitive $\lambda_m$ of $\omega_\sigma$ on $\Sigma_m$ which is a dynamically convex contact form. 
\end{teo}

We give two different arguments to show this result.
\begin{itemize}
 \item  In Section \ref{conve}, we construct a hypersurface $\hat{\Sigma}_m\subset\C^2$ and a double cover $p_m:\hat{\Sigma}_m\rightarrow \Sigma_m$ such that $p_m^*\lambda_m=-\lambda_{\op{st}}|_{\hat{\Sigma}_m}$. The hypersurface $\hat{\Sigma}_m$ is convex for small $m$, since it tends to the sphere of radius $2$ when $m$ goes to zero.
 \item In Section \ref{sec_ies}, we give a direct estimate of the Conley-Zehnder index of contractible periodic orbits of $R^{\lambda_m}$.
\end{itemize}
The latter argument follows closely the strategy of proof of \cite{pathar}: the fact that the magnetic form is symplectic and the energy is low, plays the same role as the pinching condition on the curvature. Moreover, this second proof can be adapted to surfaces of higher genus to show that, if $\gamma$ is a periodic solution of $R^{\lambda_m}$ which is free homotopic to $|e_M|$ times a vertical fibre, then $\mu_{\op{CZ}}(\gamma)\leq 2e_M+1$ (here $e_M$ denotes the Euler characteristic). This inequality fits into a notion of \textit{generalised} dynamical convexity which is currently being developed by Abreu and Macarini \cite{am}. They claim that they can prove the existence of an elliptic periodic orbit in this wider setting (see Point (\ref{pun2})). However, it is not known so far, if it is possible to extend results on the existence of global Poincar\'e sections to this case (see Point (\ref{pun1})).

Theorem \ref{teo_dc} can be used to obtain the following information on the dynamics.

\begin{crl}
Let $(S^2,g,\sigma)$ be a symplectic magnetic system and let $m$ be low enough. On $\Sigma_m$ there are either two periodic orbits homotopic to a vertical fibre or infinitely many periodic orbits. If $f$ has three distinct critical points $x_{\op{min}}$, $x_{\op{Max}}$ and $x_{\op{nondeg}}$ such that $x_{\op{min}}$ is an absolute minimiser, $x_{\op{Max}}$ is an absolute maximiser and $x_{\op{nondeg}}$ is non-degenerate, then the second alternative holds.  Moreover, if $\Sigma_m$ is non-degenerate
\begin{itemize}
 \item there exists a Poincar\'e section of disc-type for the magnetic flow on $\Sigma_m$;
 \item there exists an elliptic periodic orbit $\gamma$ on $\Sigma_m$ and, therefore, generically, there exists a flow-invariant fundamental system of neighbourhoods for $\gamma$. Hence, the dynamical system is not ergodic with respect to the Liouville measure on $\Sigma_m$.
\end{itemize}
\end{crl}

We do not have any example where there are exactly two non-contractible periodic orbits on low energy levels. However, by \cite[Assertion 3]{gin1} we know that these two orbits are \textit{short}, namely they are close to curves that wind around a vertical fibre once. In the next theorem, we show that we can make rigorous a dichotomy between short and long orbits.
\begin{teo}\label{alabanc}
Let $(S^2,g,\sigma=f\mu)$ be a symplectic magnetic system. Given $\varepsilon>0$ and a positive integer $n$, there exists $m_{\varepsilon,n}>0$ such that for every $m<m_{\varepsilon,n}$ the projection $\pi(\gamma)$ of a periodic prime solution $\gamma$ on $\Sigma_m$ either is a simple curve on $S^2$ with length in $\big(\frac{2\pi-\varepsilon}{\max f}m,\frac{2\pi+\varepsilon}{\min f}m\big)$ or has at least $n$ self-intersections and length larger than $\frac{m}{\varepsilon}$. 
\end{teo}

This result is an adaptation to the magnetic settings of \cite[Theorem 1.6]{hrysal1} for Reeb flows of convex hypersurfaces close to $S^3\subset \C^2$. That paper brings in the contact category classical results for pinched Riemannian metrics on $S^2$ \cite{bal} and on spheres of any dimension \cite{ban}. Theorem 1.6 in \cite{hrysal1} also contain a lower bound on the linking number between short and long orbits. This statement is substituted in Theorem \ref{alabanc} by a lower bound on the number of self-intersections for long orbits. It is likely that the estimates on the linking number obtained by Hryniewicz and Salom\~ao can be used as a black box to get our estimates for the self-intersections, but we did not pursue this strategy explicitly. Instead, our proof is based on an application of the Gauss-Bonnet formula for surfaces.

Even if we do not know if it is possible to have an energy level with exactly two periodic orbits, it has been proven by Schneider in \cite[Theorem 1.3]{sch1} that there are examples with exactly two \textit{short} orbits. In Section \ref{sec_twi}, we have a closer look at this problem for surfaces of revolution $(S^2_\gamma,g_\gamma,f\mu_\gamma)$, where $f$ is a rotationally invariant function. We show the existence of a smooth family of embedded tori $C_m:[-\ell,+\ell]/_\sim\times\T_{2\pi}\hookrightarrow \Sigma_m$, where $\ell$ is the length of $\gamma$ and $\sim$ is the equivalence relation which identifies $-\ell$ and $+\ell$. They are defined by the formula
\begin{equation*}
C_m(u,\psi):=
\begin{cases}
(-u,-\pi/2,\psi)&\mbox{if }u<0,\\
(u,\pi/2,\psi+\pi) &\mbox{if }u>0,
\end{cases}
\end{equation*}
where we have put on $\Sigma_m$ the triple of coordinates $(t,\varphi,\theta)$. The first and third one are the latitude and longitude, respectively. The second one is the angle in the vertical fibre counted starting from the longitudinal direction. Each $C_m$ is obtained by gluing two global Poincar\'e sections of cylinder-type, $C_m^-$ and $C^+_m$, along their common boundary. This boundary is made by the two unique periodic orbits that project to latitudes on $S^2_\gamma$. The two \textit{smooth} return maps defined on the interior of $C^-_m$ and $C^+_m$ extend to a global \textit{continuous} map $F^2_m:C_m\rightarrow C_m$ (Proposition \ref{sm_prp}). Denote by $\Omega_f:S^2_\gamma\rightarrow\R$ the extension of the function $-\frac{\dot{f}}{\gamma f^3}$ to the poles and set $\Omega_{f}^-:=\inf|\Omega_f|$. We have the following result for the return map. 
\begin{pro}\label{pro_twi}
The family of maps $F^2_m:[-\ell,+\ell]/_\sim\times\T_{2\pi}\rightarrow [-\ell,+\ell]/_\sim\times\T_{2\pi}$ admits the expansion
\begin{equation*}
F^2_m(u,\psi)=\left(u,\psi+\pi\Omega_fm^2+o(m^2)\right).
\end{equation*}
Hence, if $\Omega_f$ is not constant, there are infinitely many periodic orbits on every low energy level. Such condition is satisfied if, for example, $\frac{\ddot{f}}{f^3}(0)\neq-\frac{\ddot{f}}{f^3}(\ell)$. If $\Omega_f^->0$ (namely $\dot{f}=0$ only at the poles and $\ddot{f}\neq0$ there), the period $T$ of a Reeb orbit of $R^{\lambda_m}$, different from a latitude, satisfies
\begin{equation*}
T\geq \frac{2}{\Omega_f^-m^2}+O\left(\frac{1}{m^3}\right).
\end{equation*}
In particular, there are only two short orbits and their projection to $S^2_\gamma$ is supported on two latitudes.
\end{pro}

Finally, we observe that this proposition also gives potential candidates for magnetic systems with only two closed orbits. Namely, those for which $\Omega_f$ is constant. 
\chapter{Preliminaries}\label{cha_pre}
In this chapter we set the notation and recall the prerequisites needed in the subsequent discussion. The first section describes the conventions and symbols used in the paper. The second section is devoted to the basic properties of the tangent bundle of an oriented Riemannian surface. The third section introduces magnetic fields. The fourth section deals with Hamiltonian structures on three-manifolds. Magnetic flows on a positive energy level are particular instances of this general class of dynamical systems. 
\section{General notation}\label{sec_gn}
All objects are supposed to be smooth unless otherwise specified.

If $T>0$ is a real number, we set $\T_T:=\R/T\Z$.

If $p,q$ are coprime positive integers, we write $L(p,q)$ for the associated lens space. It is the quotient of $S^3\subset\C^2$ by the $\Z/p\Z$-action given by
\begin{equation*}
[k]\cdot(z_1,z_2)=(e^{2\pi i\frac{k}{p}}z_1,e^{2\pi i\frac{kq}{p}}z_2).  
\end{equation*}

If $\gamma$ and $\gamma'$ are two knots in $S^3$ we denote by $\op{lk}(\gamma,\gamma')$ their linking number.

If $M$ is a manifold we write $\Omega^k(M)$ for the space of $k$-differential forms on $M$ and $\Gamma(M)$ for the space of vector fields on $M$. The interior product between $Z\in\Gamma(M)$ and $\omega\in\Omega^k(M)$ will be written as $\imath_Z\omega$. If $\omega\in\Omega^k(M)$, we denote by $\mathcal P^{\omega}$ the set of its primitives. Namely, $\mathcal P^{\omega}=\{\tau\in\Omega^{k-1}(M)\ |\ \omega=d\tau\}$.

If $Z\in\Gamma(M)$, we denote by $\mathcal L_Z$ the associated Lie derivative and by $\Phi^Z$ the flow of $Z$ defined on some subset of $\mathbb R\times M$. We write its time $t$ flow map as $\Phi^Z_t$. A periodic orbit for $Z$ is a loop $\gamma:\T_T\rightarrow M$, such that $\dot{\gamma}=Z_\gamma$. When we want to make the period of the orbit explicit we use the notation $(\gamma,T)$. We call $\Omega^k_{Z}(M)$ the space of $Z$-invariant $k$-forms: $\tau$ belongs to $\Omega^k_Z(M)$ if and only if $\mathcal L_Z\tau=0$.

If $\nu$ is a free homotopy class of loops in $M$, we denote by $\mathscr L_\nu M$ the space of loops in $\nu$.

Let $\pi:E\rightarrow M$ be an $S^1$-bundle over a surface $M$ with orientation $\mathfrak o_M$. We denote by $V$ the generator of the $S^1$-action and we endow $E$ with the orientation $\mathfrak O_E:=\mathfrak o_M\oplus -V$. An $S^1$-connection form on $E$ is a $\tau\in\Omega^1(E)$ such that $\tau(V)=1$ and $d\tau=-\pi^*\sigma$ for some $\sigma\in\Omega^2(M)$ called the curvature form.

If $(M,\omega)$ is a symplectic manifold and $H:M\rightarrow \R$ is a real function, the Hamiltonian vector field $X_H$ is defined by
\begin{equation}\label{hameq}
\imath_{X_H}\omega=-dH.
\end{equation}

We define now some objects on $\C^n\simeq \R^{2n}$. Denote by $J_{\op{st}}$ the \textit{standard complex structure} and by $g_{\op{st}}$ the \textit{Euclidean inner product}. Define the \textit{standard Liouville form} $\lambda_{\op{st}}\in\Omega^1(\C^n)$ as $(\lambda_{\op{st}})_z(W):=\frac{1}{2}g_{\op{st}}(J_{\op{st}}(z),W)$, for $z\in\C^n$ and $W\in T_z\C^n\simeq \C^n$. Finally, the \textit{standard symplectic form} is defined as $\omega_{\op{st}}:=d\lambda_{\op{st}}$, or in standard real coordinates $\omega_{\op{st}}=\sum_idx^i\wedge dy^i$.

\section{The geometry of an oriented Riemannian surface}\label{sec_geo}
Let $M$ be a closed orientable surface and let $e_M\in\Z$ be its Euler characteristic. Let $\pi:TM\rightarrow M$ be the tangent bundle of $M$, $L^{\mathcal V}_{(x,v)}:T_xM\rightarrow T_{(x,v)}TM$ the associated vertical lift and write $TM_0$ for the complement of the zero section in $TM$. Suppose that we also have fixed an orientation $\mathfrak o$ on $M$. If $\sigma\in\Omega^2(M)$ we use the shorthand $[\sigma]:=\int_M\sigma$ (where the integral is with respect to $\mathfrak o$).

Let $g$ be a Riemannian metric on $M$. It yields an isomorphism $\flat:TM\rightarrow T^*M$ which we use to push forward the metric for tangent vectors to a dual metric $g$ for $1$-forms. Call $\sharp:T^*M\rightarrow TM$ the inverse isomorphism. We write $\vert\cdot\vert$ for the induced norms on each $T_xM$ and $T^*_xM$. From the duality construction we have the identity $\vert l\vert=\sup_{\vert v\vert=1}|l(v)|$, $\forall\,l\in T_x^*M$. We collect this family of norms together to get a supremum norm $\Vert\cdot\Vert$ for sections:
\begin{equation}\label{norms}
\forall Z\in \Gamma(M),\ \ \Vert Z\Vert:=\sup_{x\in M}\vert Z_x\vert;\quad\quad \forall \beta\in \Omega^1(M),\ \ \Vert \beta\Vert:=\sup_{x\in M}\vert \beta_x\vert.
\end{equation}
The Riemannian metric induces a \textit{kinetic energy function} $E:TM\rightarrow \R$ defined by $E\big((x,v)\big):=\frac{1}{2}g_x(v,v)$. The level sets $\Sigma_m:=\{E=\frac{1}{2}m^2\}\subset TM$ are such that
\begin{itemize}
 \item the zero level $\Sigma_0$ is the zero section $\{(x,0)\,|\,x\in M\}$;
 \item for $m>0$, $\pi_|:\Sigma_m\rightarrow M$ is an $S^1$-bundle.
\end{itemize}
Let $\nabla$ be the Levi-Civita connection of $g$ and $\frac{\nabla^\gamma}{dt}$ be the associated covariant derivative along a curve $\gamma$ on $M$. For every $(x,v)\in TM$, $\nabla$ gives rise to a horizontal lift $L^{\mathcal H}_{(x,v)}:T_xM\rightarrow T_{(x,v)}TM$. It has the property that $d_{(x,v)}\pi\circ L^{\mathcal H}_{(x,v)}=\op{Id}_{T_xM}$. Finally, denote by $\mu\in\Omega^2(M)$ the positive Riemannian area form and by $K$ the Gaussian curvature. We combine them to obtain the curvature form $\sigma_g\in\Omega^2(M)$ associated to $g$. It is defined by $\sigma_g:=K\mu$.
\medskip

We introduce a frame of $T(TM_0)$ and a coframe of $T^*(TM_0)$ depending on $(g,\mathfrak o)$.

The first element of the frame is $Y$, which is defined as $Y_{(x,v)}:=L^{\mathcal V}_{(x,v)}(v)$.

The geodesic equation
\begin{equation}\label{geoeq}
\frac{\nabla^\gamma}{dt}\dot{\gamma}=0 
\end{equation}
for curves $\gamma$ in $M$ gives rise to a flow on $TM$. The second element of the frame is the generator $X$ of such flow. It is called the \textit{geodesic vector field} and can be equivalently defined as $X_{(x,v)}=L^{\mathcal H}_{(x,v)}(v)$.

Consider the $2\pi$-periodic flow $\Phi^V_\varphi:TM\rightarrow TM$, which rotates every fibre of $\pi$ by an angle $\varphi$. We choose as the third element of our frame, $V$ the generator of this flow. If we denote by $\jmath:TM\rightarrow TM$ the rotation of $\pi/2$ in every fibre, then $V_{(x,v)}=L^{\mathcal V}_{(x,v)}(\jmath_xv)$ and $V\big|_{\Sigma_m}\in\Gamma(\Sigma_m)$ is the generator of the $S^1$-action on $\Sigma_m$.

Finally, the last element of the frame is $H$, defined by $H_{(x,v)}:=L_{(x,v)}^{\mathcal H}(\jmath_x v)$. 

The first element of the coframe is $dE$, the differential of the kinetic energy.

The second element is $\lambda$, the pull-back under $\flat$ of the standard Liouville form on $T^*M$. It acts by $\lambda_{(x,v)}(\xi):=g_x(v,d_{(x,v)}\pi\,\xi)$ on $\xi\in T_{(x,v)}TM$. The exterior differential $\omega:=d\lambda$ yields a symplectic form on $TM$. We denote by $\mathfrak O$ the orientation associated to $\omega\wedge\omega$. Every $\Sigma_m$ inherits an orientation $\mathfrak O_{\Sigma_m}$ which is obtained from $\mathfrak O$ following the convention of putting the outward normal to $\Sigma_m$ (namely $Y$) first. Observe that such orientation matches with the orientation of an $S^1$-bundle over an oriented surface as defined in Section \ref{sec_gn}.

The third element is $\tau$, the angular component of the Levi-Civita connection. On $\xi\in T_{(x,v)}TM$, it acts by $\tau_{(x,v)}(\xi)=g_x(\jmath_xv,\frac{\hat{x}}{dt}\hat{v})$, where $(\hat{x},\hat{v})$ is any path passing through $(x,v)$ with tangent vector $\xi$.

The fourth element is given by $\eta:=\jmath^*\lambda$.

We now recall the main properties of the frame and the coframe, which are a particular case of the results contained in \cite{guka}.

\begin{prp}\label{prphv} The frame $(X,Y,H,V)$ is $\mathfrak O$-positive with dual coframe
\begin{equation*}
\left(\frac{\lambda}{2E},\frac{dE}{2E},\frac{\eta}{2E},\frac{\tau}{2E}\right).
\end{equation*}
We have the following bracket relations for the frame
\begin{equation}
\left\{\begin{array}{lll}
        \mbox{}\displaystyle[Y,X]\ =\ X&\ \ [Y,H]\ =\ H&\ \ [Y,V]\ =\ 0\\\vspace{5pt}
        \mbox{}\displaystyle[V,X]\ =\ H&\ \ [V,H]\ =\ -X&\ \ [X,H]\, =\ 2EKV. 
       \end{array}\right.
\end{equation}
Accordingly, we get the following structural equations for the coframe
\begin{equation}\label{relstr}
 \left\{\begin{array}{rcl}
\displaystyle d\Big(\frac{\lambda}{2E}\Big)&=&\displaystyle \frac{\lambda}{2E}\wedge\frac{dE}{2E}-\frac{\eta}{2E}\wedge\frac{\tau}{2E},\vspace{.2cm}\\
\displaystyle d\Big(\frac{dE}{2E}\Big)&=&\displaystyle 0,\vspace{.2cm}\\
\displaystyle d\Big(\frac{\eta}{2E}\Big)&=&\displaystyle \frac{\lambda}{2E}\wedge\frac{\tau}{2E}-\frac{dE}{2E}\wedge\frac{\tau}{2E},\vspace{.2cm}\\
\displaystyle d\Big(\frac{\tau}{2E}\Big)&=&\displaystyle -2EK\frac{\lambda}{2E}\wedge\frac{\eta}{2E}.
\end{array}\right.
\end{equation}
The last equation in \eqref{relstr} can be rewritten as
\begin{equation}\label{equifour}
d\Big(\frac{\tau}{2E}\Big)=-\pi^*\big(K\mu\big).
\end{equation}
It implies that $\frac{\tau}{2E}$ is an $S^1$-connection form on every $\Sigma_m$ with curvature $\sigma_g$.
\end{prp}
For the following discussion it will be convenient to have also a corresponding statement for the restriction to $SM:=\Sigma_1$ of the frame and coframe defined above. 
\begin{cor}
The triple $(X,V,H)$ is an $\mathfrak O_{SM}$-positive frame of $T(SM)$ with dual coframe $(\lambda,\tau,\eta)$. We have the following relations 
\begin{equation}\label{relsm}
\left\{\begin{array}{ccc}
\mbox{}[V,X]\ =\ H &\ \ [H,V]\ =\ X &\ \ [X,H]\ =\ KV;\\
d\lambda\ =\ \tau\wedge\eta& \ \ d\tau\ =\ K\eta\wedge\lambda\ =\ -\pi^*\sigma_g&\ \ d\eta\ =\ \lambda\wedge\tau.
\end{array}\right.
\end{equation}
\end{cor}
\begin{dfn}\label{dfn_lio}
The volume form $\chi\in\Omega^3(SM)$ defined by
\begin{equation}
\chi:=\lambda\wedge\tau\wedge\eta=-\tau\wedge\pi^*\mu,
\end{equation}
associated to the dual coframe, is called the \textit{Liouville volume form}.
\end{dfn}

We end up the general discussion about the geometry of $M$ by proving some properties related to local sections of $SM$.

Consider an open set $U\subset M$ and denote by $SU$ the unit sphere bundle over $U$. Let $Z:U\rightarrow SU$ be a section and associate to $Z$ an angular function $\varphi_Z:SU\rightarrow \T_{2\pi}$. The value of $\varphi_Z$ at the point $(x,v)$ is the angle between $v$ and $Z_x$. In the next lemma we compute the differential of the angular function. First, observe that, if $v\in T_xU$, we have $g_x\big((\nabla_{v}Z)_x,Z_x\big)=0$. Thus, there exists $\kappa^Z_x\in T^*_xM$ such that
\begin{equation}\label{dfn_kap}
(\nabla_{v}Z)_x=\kappa^Z_x(v)\jmath_xZ_x, 
\end{equation}

\begin{lem}\label{lemfi}
If $Z$ is a section of $SU$, we have
\begin{equation}
\left\{
\begin{aligned}
 d\varphi_Z(X)(x,v)&=-\kappa^Z_x(v),\label{equ_kap}\\
 d\varphi_Z(V)(x,v)&=1,\\
 d\varphi_Z(H)(x,v)&=-\kappa^Z_x(\jmath_xv).
\end{aligned}\right. 
\end{equation}
\end{lem}
\begin{proof}
We can write the angle $\varphi_Z$ as an element of $S^1\subset\C$:
\begin{equation}
(\cos\varphi_Z(x,v),\sin\varphi_Z(x,v))=\big(g_x(v,Z_x),g_x(v,\jmath_x Z_x)\big).
\end{equation}
Differentiating along $X$ this map, we get
\begin{equation*}
\big(d_{(x,v)}(\cos\varphi_Z)(X),d_{(x,v)}(\sin\varphi_Z)(X)\big)=\big(g_x(\left(\nabla_vZ\right)_x,v),-g_x(\left(\nabla_vZ\right)_x,\jmath_x v)\big).
\end{equation*}
On the other hand, we also have
\begin{equation*}
\big(d_{(x,v)}(\cos\varphi_Z)(X),d_{(x,v)}(\sin\varphi_Z)(X)\big)=d\varphi_Z(X)(x,v)\cdot \big(-\sin\varphi_Z(x,v),\cos\varphi_Z(x,v)\big). 
\end{equation*}
Thus, we can express $d\varphi_Z(X)(x,v)$ as the standard inner product in $\mathbb R^2$ between $\big(g_x(\left(\nabla_vZ\right)_x,v),-g_x(\left(\nabla_vZ\right)_x,\jmath v)\big)$ and $\big(-\sin\varphi_Z,\cos\varphi_Z\big)$.
\begin{align*}
d\varphi_Z(X)(x,v)&=\ -\sin\varphi_Z(x,v) g_x\big(\left(\nabla_vZ\right)_x,v\big)+\cos\varphi_Z(x,v) g_x\big(\left(\nabla_vZ\right)_x,\jmath_x v\big)\\
&=\ -\sin\varphi_Z(x,v)\kappa^Z_x(v)g_x(\jmath_xZ,v)-\cos\varphi_Z(x,v)\kappa^Z_x(v)g_x\big(\jmath_xZ,\jmath_x v\big)\\
&=\ -\kappa^Z_x(v).
\end{align*} 
For the second statement we differentiate in $t$ the identity $\varphi_Z\circ\Phi_t^V=\varphi_Z+t$.

The third statement follows by using the identity $H=d_\jmath \jmath^{-1}(X_\jmath)$:
\begin{align*}
d_{(x,v)}\varphi_Z(H_{(x,v)})&=\ d_{(x,v)}\varphi_Z\Big(d_{(x,\jmath_xv)} \jmath^{-1}\big(X_{(x,\jmath_xv)}\big)\Big)\\
&=\ d_{(x,\jmath_xv)}(\varphi_Z\circ\jmath^{-1})\big(X_{(x,\jmath_xv)}\big)\\
&=\ d_{(x,\jmath_xv)}\varphi_Z\big(X_{(x,\jmath_xv)}\big)\\
&=\ -\kappa^Z_x(\jmath_xv)\qedhere
\end{align*}
\end{proof}

We give an explicit formula for $\kappa^Z_x$ in the next lemma. First, define the \textit{geodesic curvature} of $Z$ as the function
\begin{equation}\label{geocur}
\begin{aligned}
k_Z:U&\longrightarrow\ \mathbb R\\
x&\longmapsto\ g_x\big((\nabla_{Z}Z)_x,\jmath_x Z_x\big).
\end{aligned} 
\end{equation}
\begin{lem}
At every $x\in M$, we have $\sharp\,\kappa^Z_x=k_Z(x)Z_x+k_{\jmath Z}(x)\jmath_xZ_x$.
\end{lem}
\begin{proof}
Let us evaluate $\kappa^Z_x$ on the basis $(Z_x,\jmath_xZ_x)$:\vspace{.2cm}
\[\begin{array}{c|c}
\begin{aligned}
\kappa^Z_x(Z_x)&=g_x\big((\nabla_{Z}Z)_x,\jmath_xZ_x\big)\\
&=k_Z(x);\\
&\\
&
\end{aligned}
&
\begin{aligned}
\kappa^Z_x(\jmath_xZ_x)&=g_x\big((\nabla_{\jmath Z}Z)_x,\jmath_xZ_x\big)\\
&=g_x\big(\jmath_x(\nabla_{\jmath Z}Z)_x,\jmath_x(\jmath_xZ_x)\big)\\
&=g_x\big((\nabla_{\jmath Z}\jmath Z)_x,\jmath_x(\jmath_xZ_x)\big)\\
&=k_{\jmath Z}(x).
\end{aligned}
\end{array}\]
\end{proof}

\section{Magnetic fields}\label{sec_mag}
In the previous section we introduced various geometrical objects associated to $(M,g)$. Historically, the geodesics equation \eqref{geoeq} has been used as the main tool to understand the geometry of Riemannian manifolds \cite{kli,pat}. It is a classical result that the geodesic flow can be studied in the framework of symplectic geometry, since $X$ is the $d\lambda$-Hamiltonian vector field associated to the kinetic energy function (see \cite[Chapter 3]{kli}). In other words,
\begin{equation}
\imath_Xd\lambda=-dE.
\end{equation}
From this equation we see that $\Sigma_0$ is the set of rest points for $X$ and for every positive $m$, $\Sigma_m$ is invariant under the geodesic flow. Hence, we can study $X\big|_{\Sigma_m}$ for a fixed value of $m>0$. By the homogeneity of \eqref{geoeq} (which is translated in the bracket relation $[Y,X]=X$), $X\big|_{SM}$ and $X\big|_{\Sigma_{m}}$ have the same dynamics up to reparametrisation by a constant factor. The map conjugating the two flows is given by the flow of $Y$:
\begin{equation}\label{inty1}
d\Phi^Y_{-\log m}\cdot X\big|_{\Sigma_{m}}=mX\big|_{\Sigma_{SM}}
\end{equation}
One of the guiding principles of this thesis is to see how the introduction of a magnetic perturbation in the geodesic equation breaks the homogeneity and gives rise to a family of systems whose dynamics changes with $m$. 

For any $\sigma\in\Omega^2(M)$ consider the symplectic form $\omega_\sigma:=d\lambda-\pi^*\sigma\in\Omega^2(TM)$. When $\sigma$ is not zero, we refer to the symplectic manifold $(TM,\omega_\sigma)$ as a \textit{twisted} tangent bundle to distinguish it from the \textit{standard} tangent bundle $(TM,\omega_0=d\lambda)$.

We will see later some remarkable differences (in terms of Symplectic Cohomology and displaceability, for example) between the geometry of the standard and twisted tangent bundle. Such phenomena show, in this setting, the contrast between symplectic and volume-preserving geometry. Indeed, we claim that
\begin{equation}\label{vol}
\omega_{\sigma}\wedge\omega_{\sigma}=\omega_0\wedge \omega_0.
\end{equation}
Hence, the twisted and the standard tangent bundle have the same volume form, although they are different in the symplectic category. To prove the claim notice that $\pi^*\sigma\wedge\pi^*\sigma=0$ and that $d\lambda\wedge\pi^*\sigma=d(\lambda\wedge\pi^*\sigma)=0$, since $\lambda\wedge\pi^*\sigma=0$ by Proposition \ref{prphv} above.

In the next proposition we prove that also the first Chern class of the underlying class of compatible almost complex structures is unaffected by the twist.
\begin{prp}\label{che}
If $\sigma$ is a $2$-form on $M$, then $c_1(TM,\omega_\sigma)=0$.
\end{prp}
\begin{proof}
Consider the vertical distribution $\mathcal V\rightarrow TM$. It is a Lagrangian sub-bundle of the symplectic bundle $(T(TM),\omega_\sigma)\rightarrow TM$. If $J_\sigma$ is an almost complex structure compatible with $\omega_\sigma$, then $\mathcal V$ is totally real with respect to $J_\sigma$. As a consequence, we have the isomorphism of complex bundles
\begin{align*}
(\mathcal V_{\C},i)&\longrightarrow\ (T(TM),J_\sigma)\\
(u+iv)&\longmapsto\ u+J_\sigma v,
\end{align*}
where $(\mathcal V_{\C},i)$ is the complexification of $\mathcal V$. Therefore, $c_1(TM,J_\sigma)=c_1(\mathcal V_{\C},i)$. On the other hand, $c_1(\mathcal V_{\C},i)=0$ since $(\mathcal V_{\C},i)\simeq(\mathcal V_{\C},-i)$, $c_1(\mathcal V_{\C},-i)=-c_1(\mathcal V_{\C},i)$ and $H^2(TM,\Z)$ is torsion-free (see \cite[Lemma 14.9]{milsta}. Alternatively, notice that $\mathcal V\rightarrow TM$ is orientable and therefore $\op{det}(\mathcal V)\rightarrow TM$ is trivial. Since taking the determinant bundle and taking the complexification do commute, we get
\begin{equation*}
c_1(\mathcal V_{\C},i)=c_1(\op{det}(\mathcal V_{\C}),\op{det}(i))=c_1((\op{det}\mathcal V)_{\C},i)=0.\qedhere
\end{equation*}
\end{proof}

\begin{dfn}\label{dfn_magne}
We call $\sigma$ a \textit{magnetic form} and the triple $(M,g,\sigma)$ a \textit{magnetic system}. If we fix $M$, we write $\op{Mag}(M)$ for the space of all pairs $(g,\sigma)$ such that $(M,g,\sigma)$ is a magnetic system. For every $(g,\sigma)\in\op{Mag}(M)$, there exists a unique function $f:M\rightarrow \R$ called the \textit{magnetic strength} such that $\sigma=f\mu$. We say that the magnetic system is \textit{exact}, \textit{non-exact}, or \textit{symplectic}, if such is the magnetic form $\sigma$ in $\Omega^2(M)$. We denote by $\op{Mag}_e(M)\subset\op{Mag}(M)$ the subset of exact magnetic systems.

Finally, we define the \textit{magnetic vector field} $X^{\sigma}_{E}\in\Gamma(TM)$ as the $\omega_\sigma$-Hamiltonian vector field associated to $E$ and we refer to $\Phi^{X^{\sigma}_{E}}$ as the \textit{magnetic flow}. 
\end{dfn}

As the geodesic flow, $X^{\sigma}_{E}$ comes from a second order ODE for curves in $M$: 
\begin{equation}\label{lorsig}
\frac{\nabla^\gamma}{dt}\dot\gamma=F_\gamma\big(\dot{\gamma}\big).
\end{equation}
Here $F:TM\rightarrow TM$ is the \textit{Lorentz force}. It is a bundle map given by \begin{equation}\label{mag_for}
g(F_x(v),w):=\sigma_x(v,w)
\end{equation}
and can be expressed using the magnetic strength as $F_x(v)=f(x)\jmath_xv$. Using the relation between the Levi-Civita connection and the horizontal lifts, one finds that
\begin{equation}
X^{\sigma}_{E}=X+fV. 
\end{equation}

The kinetic energy $E$ is still an integral of motion for the magnetic flow: $\Sigma_0$ is the set of rest points for $X^{\sigma}_E$ and, for $m>0$, $X^{\sigma}_E$ restricts to a nowhere vanishing vector field on $\Sigma_m$.

However, the bracket relations now imply
\begin{equation}\label{inty2}
d\Phi^Y_{-\log m}\cdot X^{\sigma}_E\big|_{\Sigma_{m}}=(mX+fV)\big|_{SM}=mX^{\frac{\sigma}{m}}_E\big|_{SM}.
\end{equation}
From this equality we infer two things. First, that, unlike the geodesic flow, the magnetic flow is not homogeneous. Second, that studying the dynamics of the magnetic system $(M,g,\sigma)$ on $\Sigma_m$ is the same as studying the rescaled magnetic system $(M,g,\frac{\sigma}{m})$ on $SM$. The advantage of the latter formulation is that it describes the dynamics of the magnetic flow by using a $1$-parameter family of flows on a \textit{fixed} closed three-manifold. We also introduce the real parameter $s$, which is related to $m$ by the equation $sm=1$. In the following discussion, we will use both parameters, choosing every time the one that most simplify the notation. As a heuristic rule, the parameter $m$ will be more convenient for low values of $E$ ($m$ close to $0$), while $s$ will be more convenient for high values of $E$ ($s$ close to $0$).
When there is no risk of confusion, we adopt the shorthand $X^s:=X^{s\sigma}_E\big|_{SM}$ and $\omega_s:=\omega_{s\sigma}$, when we use $s$, and $X^m:=mX^{\frac{\sigma}{m}}_E\big|_{SM}=mX+fV$ and $\omega_m:=m\omega_{\frac{\sigma}{m}}=md\lambda-\pi^*\sigma$, when we use $m$. Notice that we have the limits
\begin{equation*}
\begin{array}{ccccccc}
\displaystyle\lim_{s\rightarrow 0}X^s&=& X,&\quad&\displaystyle\lim_{m\rightarrow 0}X^m&=&fV,\vspace{2pt}\\
\displaystyle\lim_{s\rightarrow 0}\omega_s&=& d\lambda;&\quad&\displaystyle\lim_{m\rightarrow 0}\omega_m&=&-\pi^*\sigma.
\end{array}
\end{equation*}

In the standard tangent bundle, the properties of the linearisation of the geodesic flow on $SM$ are influenced by the curvature $K$ (see also the bracket relations above). Analogously, when there is a magnetic term, for every $m>0$, we can define the \textit{magnetic curvature} function $K_m:SM\rightarrow\R$ as $K_m:=m^2K-m(df\circ\jmath)+f$, which is related to the linearisation of $X^m$ \cite{ada1} (see also Chapter \ref{cha_rev}).

\begin{rmk}\label{rmk_tancot}
The duality isomorphism $\flat:TM\rightarrow T^*M$ yields a symplectomorphism between the twisted tangent bundle $(TM,\omega_\sigma)$ and $(T^*M,\omega^*_\sigma:=d\lambda^*-\pi^*\sigma)$, where $\lambda^*$ is the standard Liouville form on the cotangent bundle and we denote with $\pi$ also the footpoint projection $T^*M\rightarrow M$. The $\omega_\sigma$-Hamiltonian system $X^\sigma_E$ is sent through $\flat$ to the $\omega^*_\sigma$-Hamiltonian system associated to $E^*:=E\circ \flat^{-1}$. Observe that $E^*$ is the kinetic energy function on $T^*M$ associated to the dual metric, namely $E^*(x,l)=\frac{1}{2}g_x(l,l)$. In particular $\Sigma^*_m:=\flat(\Sigma_m)=\{E^*=\frac{1}{2}m^2\}$. Unlike $\omega_\sigma$, $\omega^*_\sigma$ depends only on $\sigma$ and not on $g$, which enters in the picture only through the Hamiltonian $E^*$. 

The tangent bundle and the cotangent bundle formulation are equivalent but have different advantages and disadvantages as far as the exposition of the material is concerned. In the tangent bundle we have an intuitive understanding of the flow given by the second order ODE. On the other hand, in the cotangent bundle we have seen that $g$ and $\sigma$ contributes separately to determine the magnetic flow: the former through the Hamiltonian and the latter through the symplectic form.
\end{rmk}

We have seen that looking at different energy levels of a fixed magnetic system $(M,g,\sigma)$ corresponds to a suitable rescaling of the magnetic field: $X^{\sigma}_E\big|_{\Sigma_m}\sim X^{\frac{\sigma}{m}}_E\big|_{\Sigma_1}$. What happens when we rescale the metric $g\mapsto m^2g$? In the cotangent bundle formulation, we have that $E^*_{m^2g}=\frac{E^*_g}{m^2}$. Thus, $X^\sigma_{E^*_{m^2g}}\big|_{\Sigma_1^*}\sim X^\sigma_{E^*_g}\big|_{\Sigma^*_m}$ and, consequently, 
\begin{equation}\label{eq_resc}
X^\sigma_{E_{m^2g}}\big|_{\Sigma_1}\sim X^\sigma_{E_g}\big|_{\Sigma_m}. 
\end{equation}

\subsection{The Maupertuis' principle}\label{sec_mau}
In this subsection we discuss how a magnetic system $(M,g,\sigma)$ is affected by the introduction of a potential $U:M\rightarrow\R$. In this case we speak of a \textit{mechanical system} $(M,g,\sigma,U)$. The Hamiltonian has an additional term $E(x,v)=\frac{1}{2}g_x(v,v)+U(x)$ while the symplectic form remains the same $\omega_\sigma=d\lambda-\pi^*\sigma$. In this case the topology of the energy levels depends on $U$. Consider the map $\pi|:\Sigma_m\rightarrow M$, where $\Sigma_m:=\{E=m^2/2\}$. Its image is the sublevel $M_m:=\{ U\leq m^2/2\}$ and the preimage of a point $x\in M_m$ are the vectors in $T_xM$ with norm $\sqrt{m^2-2U(x)}$. This is either a single point, namely the zero vector of $T_xM$, if $U(x)=m^2/2$, or a circle, if $U(x)<m^2/2$. In particular, $\pi|:\Sigma_m\rightarrow M$ is an $S^1$-bundle if and only if $m^2/2>\max U$. We claim that in this case the dynamics on $\Sigma_m$ is equivalent, up to reparametrisation, to the dynamics of the magnetic flow associated to $(M,g_{m,U}:=(m^2-2U)g,\
sigma)$ on $SM$ (where the sphere bundle is taken with respect to the metric $g_{m,U}$). To prove the claim, we work in the cotangent bundle formulation. We start by observing that $\{E^*=m^2/2\}=\{\frac{1}{2}\frac{g_x(l,l)}{m^2-2U(x)}=\frac{1}{2}\}$. Therefore, the $\omega^*_\sigma$-Hamiltonian flow of $E^*$ and the one of $\widetilde{E}^*(x,l):=\frac{1}{2}\frac{g_x(l,l)}{m^2-2U(x)}$ are the same up to reparametrisation on this common energy level. Finally, we observe that $\widetilde{E}$ is the kinetic energy function for the metric $\frac{g}{m^2-2U}$ which is the dual of $g_{m,U}$.

From the discussion above, we see that the results contained in this thesis can also be applied to magnetic systems with potential, provided we know $U$ and $m$ well enough to understand the properties of $(M,g_{m,U},\sigma)$. For example, fix $m>0$ and take a potential $U$ whose $C^0$-norm is small compared to $m^2$. We claim that the mechanical system $(M,g,\sigma,U)$ and the magnetic system $(M,g,\sigma)$ will be close to each other on $\Sigma_m$. By Equation \eqref{inty2} and \eqref{eq_resc}, we know that the magnetic system associated to $(M,g_{m,U},\sigma)$ is equivalent, up to reparametrisation, to the magnetic system associated to $(M,\frac{g_{m,U}}{m^2},\frac{\sigma}{m})$. Moreover,
\begin{equation}
\left\Vert \frac{g_{m,U}}{m^2}-g\right\Vert\leq \frac{2\Vert U\Vert}{m^2}\Vert g\Vert. 
\end{equation}
Thus, if $\Vert U\Vert$ is small compared to $m^2$, the magnetic flow of $(M,g_{m,U},\sigma)$ on the unit tangent bundle of $g_{m,U}$ is close, up to reparametrisation, to the magnetic flow of $(M,g,\frac{\sigma}{m})$ on the unit tangent bundle of $g$. 

\subsection{Physical interpretation: constrained particles and rigid bodies with symmetry}\label{sub_phy}
The adjective 'magnetic' for the systems we have introduced in the present section is due to the fact that they describe the following simple phenomenon in the theory of classical electromagnetism. Consider a particle $q$ of unit mass constrained to move on a frictionless surface $M$ in the Euclidean space $(\R^3,g_{\op{st}})$. The metric induces by pull-back a metric $g$ on $M$ and a corresponding Levi-Civita connection $\nabla$. With this assumption the motion of $q$ will obey to the geodesic equation \eqref{geoeq}. Suppose now that $q$ has unit charge and that $M$ is immersed in a stationary magnetic field $\overrightarrow{B}$. In this case the particle is subject to a Lorentz force $\overrightarrow{F}=\overrightarrow{v}\times\overrightarrow{B}$. The force $\overrightarrow{F}$ is obtained from the $2$-form $\sigma:=\imath_{\overrightarrow{B}}\op{vol}_{g_{\op{st}}}$ as prescribed by Equation \eqref{mag_for}. Here $\op{vol}_{g_{\op{st}}}$ denotes the Euclidean volume form in $\R^3$. The motion of $q$ will 
obey to the Equation \eqref{lorsig}.

Making the natural assumption that $\overrightarrow{B}$ is defined on the whole $\R^3$, we see that $\sigma$ is the pull-back of a closed $2$-form on $\R^3$ and, hence, it is \textit{exact}. With this physical interpretation, non-exact magnetic systems on $M$ correspond to fields $\overrightarrow{B}$ generated by a Dirac monopole located in a region outside the surface.
\medskip

Even if a magnetic monopole has not been observed so far, non-exact magnetic flows on $S^2$ have, nonetheless, a concrete importance since they are symplectic reductions of certain physical systems with phase space $TSO(3)$, possessing an $S^1$-symmetry. The study of the periodic orbits of these systems was initiated in the early Eighties by Novikov in a series of papers \cite{nov1,nov2,nov3,nov4}. His interest was motivated by the fact that in these cases the Lagrangian action functional is multivalued and, hence, one needs a generalisation of the standard Morse theory (nowadays known as Novikov theory) for proving the existence of critical points.

Among the many examples considered by Novikov, here we only look at the motion of a rigid body with rotational symmetry and a fixed point. We refer the reader to \cite{kha} for a careful explanation of the setting and for the proofs of the statement we make.

A rigid body with a fixed point is described by a positive orthonormal basis $\mathbf e=(e_1,e_2,e_3)$ in $\R^3$. Given a fixed positive orthonormal basis $\mathbf n=(n_1,n_2,n_3)$, to every $\mathbf e$ we can associate a unique element in $SO(3)$ which is the isometry of $\R^3$ sending $\mathbf n$ to $\mathbf e$. Hence, the configuration space of a rigid body with a fixed point can be identified with the group $SO(3)$.

The kinetic energy of the body is obtained from a left-invariant Riemannian metric $g$ on $SO(3)$, which, in the standard basis of $\mathfrak{so}(3)$, is represented by the matrix
\begin{equation}
I=\left(\begin{array}{ccc}
        I_1&0&0\\
	0&I_2&0\\
	0&0&I_3
        \end{array}
\right)
\end{equation}
for some positive numbers $I_1$, $I_2$ and $I_3$. Let us consider the $\T_{2\pi}$-action on the configuration space that rotates the rigid body around the fixed axis $n_3$. In $SO(3)$ this corresponds to left multiplication by the subgroup $G$ of rotations with axis $n_3$. Call $Z\in\Gamma(SO(3))$ the infinitesimal generator of $G$. 

The quotient map for this action is the $S^1$-bundle map
\begin{align*}
p:SO(3)&\longrightarrow\ S^2\subset \R^3\\
(e_1,e_2,e_3)&\longmapsto\ \left(\begin{array}{c}
                                 \alpha_1\\ 
                                 \alpha_2\\
                                 \alpha_3
\end{array}
\right):=\left(\begin{array}{c}
                                 g_{\op{st}}(e_1,n_3)\\ 
                                 g_{\op{st}}(e_2,n_3)\\
                                 g_{\op{st}}(e_3,n_3)
\end{array}
\right),
\end{align*}
which yields the coordinates of the vector $n_3$ in the frame $\mathbf n$. Denote by $\hat{g}$ the metric induced on $S^2$ by $p$. It is obtained pushing forward the restriction of $g$ to the orthogonal of $\ker dp$. Notice that the function $\mathbf e\mapsto \sqrt{g_{\mathbf e}(Z,Z)}$ is invariant under $G$, so that it descends to a function $\nu:S^2\rightarrow(0,+\infty)$ such that $\sqrt{g_{\mathbf e}(Z,Z)}=\nu(p(\mathbf e))$.

Finally, suppose that the rigid body is immersed in a conservative force field, which is invariant under rotations around $n_3$. Thus, the force is described by a potential $U:SO(3)\rightarrow \R$ which is invariant under the $\T_{2\pi}$-action. Namely, there exists $\hat{U}:S^2\rightarrow\R$ such that $U=\hat{U}\circ p$.

The dynamics of the body is given by the Hamiltonian vector field on $TSO(3)$ associated to the mechanical system $(SO(3),g,0,U)$. The $\T_{2\pi}$-symmetry arises from a momentum map $J:TSO(3)\rightarrow \mathfrak g\simeq \R$, which is an integral of motion. For every $k\in \R$, we look at the dynamics on the invariant set $\{J=k\}$ projected to $TS^2$ via the quotient map $p$. It is the dynamics of the Hamiltonian vector field associated to $(S^2,\hat{g},k\sigma_{\hat{g}},\hat{U}_k)$, where $\hat{U}_k=\hat{U}+\frac{k^2}{2\nu^2}$. A computation shows that the Gaussian curvature of $\hat{g}$ is \textit{always positive} and, hence, the magnetic form is \textit{symplectic}.

Applying Maupertuis' principle, we know that, on $\Sigma_m$ such that $m^2/2>\max \hat{U}_k$, the reduced system above is equivalent to the magnetic system $(S^2,\hat{g}_{m,\hat{U}_k},k\sigma_{\hat{g}})$. Moreover, if $\Vert \hat{U}_k\Vert$ is small compared to $m^2$, up to reparametrisation, the dynamics is close to that of $(S^2,\hat{g},\frac{k}{m}\sigma_{\hat{g}})$ on the unit tangent bundle of $\hat{g}$.

\subsection{From dynamics to geometry}
We come back to the general discussion on magnetic flows associated to $(M,g,\sigma)$. We saw that the dynamics of $X^{\sigma}_E$ on $TM$ can be studied by looking at the dynamics of $X^s:=X^{s\sigma}_E\big|_{SM}$, for positive values of the parameter $s$. We are going to look at this new problem from a geometric point of view, by investigating what geometric structure the symplectic manifold $TM$ induces on the three-manifold $SM$.

Consider the restriction $\omega'_\sigma:=\omega_\sigma\big|_{SM}$. It is a closed, nowhere vanishing $2$-form on the unit tangent bundle. Hence, it has an orientable $1$-dimensional kernel $\ker\omega'_\sigma$ at every point. We fix for it an orientation $\mathfrak O_{\ker\omega'_\sigma}$ satisfying the relation
\begin{equation}\label{or_lin}
\mathfrak O_{SM}=\mathfrak O_{\ker \omega'_\sigma}\oplus \mathfrak O^{\omega_\sigma'}, 
\end{equation}
where $\mathfrak O^{\omega_\sigma'}$ is the orientation on $\frac{T(SM)}{\ker\omega'_\sigma}$ induced by $\omega_\sigma'$. Since $\chi=\lambda\wedge d\lambda=\lambda\wedge\omega'_\sigma$, we readily see that
\begin{equation}\label{con_mag}
\imath_{X^\sigma_E\big|_{SM}}\chi=\lambda(X^\sigma_E\big|_{SM})\omega'_\sigma-\lambda\wedge\imath_{X^\sigma_E\big|_{SM}}\omega'_\sigma=\omega'_\sigma.
\end{equation}
Therefore, Equation \eqref{or_lin} implies that the magnetic vector field $X^\sigma_E\big|_{SM}$ is a \textit{positive} nowhere vanishing section of $\ker\omega'_\sigma$.

From the previous discussion, we argue that the real object of interest for understanding the geometric properties of the orbits, and not their actual parametrisation, is $\ker\omega'_\sigma$ rather than $X^\sigma_E\big|_{SM}$. The $2$-form $\omega_\sigma'$ is a particular instance of what is called a \textit{Hamiltonian Structure} (or HS for brevity). We introduce such objects in an abstract setting in the next section.

\section{Hamiltonian structures}
We develop below the notion of HS and we identify the special subclass of HS of contact type. Proving that a HS is of contact type has striking consequences for the dynamics, since in this case we can describe the dynamical system associated to the HS by the means of a Reeb flow. This yields the following two advantages.
\begin{enumerate}[\itshape a)]
 \item We can count periodic orbits of Reeb flows using algebraic invariants, such as Contact Homology \cite{bou} and Embedded Contact Homology \cite{tau1}, since closed orbits are generator for the corresponding chain complexes. When the contact manifold is the \textit{positive} boundary of a symplectic manifold, like in the case of $SM$, one can define a further invariant called Symplectic Homology, which takes into account the interplay between the contact structure on the boundary and the symplectic structure of the filling. In the next chapter we are going to define Symplectic Homology \cite{vit,sei} (or, more precisely, its cohomological version) in detail and prove an abstract invariance result which we will use in Chapter \ref{cha_exa} to compute it in the case of magnetic systems. This will yield a lower bound on the number of periodic orbits for such systems.
 \item The theory of pseudo-holomorphic foliations developed by Hofer, Wysocki and Zehnder \cite{hwz1} allows to find Poincar\'e sections for \textit{dynamically convex} Reeb flows. This allows for a description of the flow in terms of an area-preserving discrete dynamical system on the two-dimensional disc. Information about periodic points of such systems are then obtained applying results by Brouwer \cite{bro} and Franks \cite{fra2}. The abstract setting will be presented in Chapter \ref{cha_dc} and applied to magnetic systems in Chapter \ref{cha_le}
\end{enumerate}

In the rest of this section let $(N,\mathfrak O)$ be a three-manifold $N$ endowed with an orientation $\mathfrak O$.
\begin{dfn}
A closed and nowhere vanishing $2$-form $\omega$ on $N$ is called a \textit{Hamiltonian Structure}. We say that the Hamiltonian Structure is \textit{exact}, if $\omega$ is such. The $1$-dimensional distribution $\ker\omega$ is called the \textit{characteristic distribution} associated to $\omega$. It has an orientation $\mathfrak O_{\ker\omega}$ satisfying the relation $\mathfrak O=\mathfrak O_{\ker \omega}\oplus \mathfrak O^{\omega}$, where $\mathfrak O^{\omega}$ is the orientation on $\frac{TN}{\ker\omega}$ induced by $\omega$. Alternatively, if $\chi$ is any positive volume form, $Z^\chi\in\Gamma(N)$ defined by $\imath_{Z^\chi}\chi=\omega$ is a positive section of $\ker\omega$.
\end{dfn}
\begin{dfn}\label{def_non}
Let $\omega$ be a HS and let $Z\in\Gamma(N)$ be a positive section of the characteristic line bundle associated with $\omega$. We say that a periodic orbit $(\gamma,T)$ for $\Phi^Z$ is \textit{non-degenerate} if the multiplicity of $1$ in the spectrum of $d\Phi^Z_T$ is exactly one. Observe that this property does not depend on the choice of $Z$.
If $(\gamma,T)$ is transversally non-degenerate, call \textit{transverse spectrum} the set made of the two eigenvalues of $d\Phi^Z_T$ different from $1$. We say that $\gamma$ is
\begin{enumerate}
 \item \textit{elliptic}, if the transverse spectrum lies on the unit circle in $\C$;
 \item \textit{hyperbolic}, if the transverse spectrum lies on the real line.
\end{enumerate}

Finally, we say that $\omega$ is \textit{$\nu$-non-degenerate} if all the periodic orbits in the class $\nu$ are non-degenerate. 
\end{dfn}

\begin{rmk}\label{rmk_ham}
If $(W,\omega)$ is a symplectic four-manifold and $N\hookrightarrow W$ is an orientable embedded hypersurface, then $\omega\big|_N$ is a Hamiltonian Structure on $W$. If we take $N$ to be the regular level of a Hamiltonian function $H:W\rightarrow\R$, we can endow $N$ with the orientation induced from $\omega\wedge\omega$ by putting the gradient of $H$ first. Using the Hamilton equation \eqref{hameq}, we see that $X_H\big|_N$ is a positive section of $\ker \omega\big|_N$. Moreover, $\omega\big|_N$ is $\nu$-non-degenerate if and only if all the periodic orbits of $H$ on $N$ with free homotopy class $\nu$ are \textit{transversally non-degenerate} (see next chapter).
\end{rmk}

This observation shows that the dynamics up to reparametrisation of an autonomous Hamiltonian system on a four-dimensional symplectic manifold $(W,\omega)$ at a regular level $N$ can be read off the geometry of the oriented one-dimensional distribution $\ker\omega\big|_N$ associated to the Hamiltonian structure $\omega\big|_N$.

\subsection{Hamiltonian structures of contact type}
We now introduce Hamiltonian structures of contact type for which the characteristic distribution has a section with special properties.

\begin{dfn}
We say that $\omega$ is \textit{of contact type} if it is exact and there exists a contact form $\tau\in\mathcal P^\omega$. Denote by $R^\tau$ the \textit{Reeb vector field} of $\tau$, defined by $\imath_{R^\tau}d\tau=0$ and $\tau(R^\tau)=1$. One among $R^\tau$ and $-R^\tau$ is a positive section of $\ker\omega$. We say that $\omega$ is of \textit{positive} or \textit{negative} contact type accordingly. 
\end{dfn}

\begin{rmk}\label{rmk_con}
If we fix a positive section $Z$ of $\ker\omega$, being of positive (respectively negative) contact type is equivalent to finding $\tau\in\mathcal P^\omega$ such that $\tau(Z):N\rightarrow\R$ is a positive (respectively negative) function. In this case $R^{\tau}=\frac{Z}{\tau(Z)}$.
\end{rmk}

\begin{exa}\label{exa_consm}
If $(M,g)$ is a Riemannian surface as above, then $\omega'_0=d\lambda\big|_{SM}$ is a HS of positive contact type. Indeed, $X$ is a positive section of $\ker \omega'_0$ and $\lambda_{(x,v)}(X)=g_x(v,d_{(x,v)}\pi X)=1$. This implies that $\lambda\big|_{SM}$ is a positive contact form and $X$ is its Reeb vector field.
\end{exa}

\begin{exa}\label{exa_concon}
Suppose that $\pi:E\rightarrow M$ is an $S^1$-bundle over an oriented surface $(M,\mathfrak o_M)$. If $\sigma\in\Omega^2(M)$ is a positive symplectic form, then $\pi^*\sigma$ is a Hamiltonian structure. We claim that $\pi^*\sigma$ is exact if and only if $E$ is non-trivial. To prove necessity, we observe that if $E$ is trivial, the integral of $\pi^*\sigma$ over a section of $\pi$ is non-zero. To prove sufficiency, we define $e\neq0$ the Euler number of the $S^1$-bundle. Then, we use the classical result \cite{kob} saying that a $2$-form on $M$, whose integral is $2\pi e$, is a curvature form for an $S^1$-connection on $E$. Thus, it is exact. This implies that there exists $\tau\in\Omega^1(E)$ such that
\begin{equation}
\bullet\ \ \tau(V)\ =\ 1\,,\quad\quad\quad\bullet\ \ d\tau\ =\ -\frac{2\pi e}{[\sigma]}\pi^*\sigma\,.
\end{equation}
In this case $-\frac{[\sigma]}{2\pi e}\tau$ is a primitive of $\pi^*\sigma$ and a contact form with Reeb vector field $-\frac{2\pi e}{[\sigma]}V$. Therefore, $\pi^*\sigma$ is a Hamiltonian structure of contact type. It is positive if and only if $e$ is positive.
\end{exa}

\begin{exa}
For examples of structures of contact type in the context of the circular planar restricted three-body problem, we refer to \cite{affkp}. 
\end{exa}

The most direct way to detect the contact property is to use Remark \ref{rmk_con}. However, this method could be difficult to apply, especially if we want to prove that a HS is not of contact type, since we should check that every function $\tau(Z)$ vanishes at some point. This problem is overcome by the following necessary and sufficient criterion contained in McDuff \cite{mcd}.
\begin{prp}\label{mcdcri}
Let $\omega'$ be an exact HS and $Z$ be a positive section of $\ker\omega'$. Then, $\omega'$ is of positive (respectively negative) contact type if and only if the action of every null-homologous $Z$-invariant measure is positive (respectively negative).
\end{prp}
We end this subsection by recalling the basic notions about invariant measures needed in the statement of the proposition above.

If $Z$ is a vector field on a closed manifold $N$, a \textit{$Z$-invariant measure} $\zeta$ is a Borel probability measure on $N$, such that
\begin{equation}
\forall\, h:N\rightarrow\R,\quad\quad \int_Ndh(Z)\zeta=0.
\end{equation}
We denote the set of $Z$-invariant measure by $\mathfrak M(Z)$. We associate to every $\zeta\in\mathfrak M(Z)$ an element $\rho(\zeta)$ in $H^1(N,\R)^*=H_1(N,\R)$ defined as
\begin{equation}
\forall\, [\beta]\in H^1(N,\R),\quad\quad <\rho(\zeta),[\beta]>\ :=\int_N\beta(Z)\zeta.
\end{equation}
Suppose $Z$ is a positive section of $\ker\omega$, with $\omega$ an exact HS on an oriented three-manifold $N$, $\zeta\in\mathfrak M(Z)$ and $\tau\in\mathcal P^{\omega'}$. The \textit{action} of $\zeta$ is defined as
\begin{align}
\mathcal A^\tau_Z:\mathfrak M(Z)&\longrightarrow\ \R\nonumber\\
\zeta&\longmapsto \int_N\tau(Z)\zeta.
\end{align}
If $\zeta$ is \textit{null-homologous} (namely $\rho(\zeta)=0$), then $\mathcal A^\tau_Z(\zeta)$ is independent of $\tau$ and, therefore, in this case we write $\mathcal A^\omega_Z(\zeta):=\mathcal A^\tau_Z(\zeta)$.

\subsection{The Conley-Zehnder index}
We finish this section by associating an integer, called the \textit{Conley-Zehnder index}, to a periodic orbit of a Hamiltonian structure. It encodes information about the linearisation of the system along the orbit and will play a crucial role in defining the notion of dynamical convexity for HS of contact type. We refer to \cite{hofkri} for proofs and further details.

We start with the definition of the \textit{Maslov index} for a path with values in $\op{Sp}(1)$, the group of $2\times 2$-symplectic matrices. For any $T>0$, we set 
\begin{equation*}
\op{Sp}_T(1):=\{\Psi:[0,T]\rightarrow \operatorname{Sp}(1)\ |\ \Psi(0)=\operatorname{Id}\}. 
\end{equation*}
We call $\Psi\in \op{Sp}_T(1)$ \textit{non-degenerate} if $\Psi(T)$ does not have $1$ as eigenvalue.

Given $\Psi\in\op{Sp}_T(1)$, we associate to every $u\in\mathbb R^2\setminus \{0\}$ a \textit{winding number} $\Delta\theta(\Psi,u)$ as follows. Identify $\R^2$ with $\C$ and let
\begin{equation}\label{teta}
\frac{\Psi(t)u}{|\Psi(t)u|}=e^{i\theta^\Psi_u(t)},
\end{equation}
for some function $\theta^\Psi_u:[0,T]\rightarrow\R$. We define $\Delta\theta(\Psi,u):=\theta^\Psi_u(T)-\theta^\Psi_u(0)$. Let
\begin{equation}
I(\Psi):=\left\{\frac{1}{2\pi}\Delta\theta(\Psi,u)\ \Big|\ u\in\R^2\setminus \{0\}\right\}.
\end{equation}
The interval $I(\Psi)$ is closed and its length is strictly less than $1/2$. We notice that the set $e^{2\pi i I(\Psi)}\subset S^1$ is completely determined by the endpoint $\Psi(T)$. In particular, we see that $\Psi$ is non-degenerate if and only if $\mathbb Z\cap \partial I(\Psi)=\emptyset$. We define the Maslov index for a non-degenerate path as
\begin{equation}
\mu(\Psi):=\left\{\begin{array}{cl}
 2k,&\mbox{if }k\in I(\Psi),\mbox{ for some }k\in\Z;\\
2k+1,&\mbox{if }I(\Psi)\subset(k,k+1),\mbox{ for some }k\in\Z.
\end{array}\right.
\end{equation}
We extend the definition to the degenerate case either by taking the maximal lower semi-continuous extension $\mu^l$ or the minimal upper semi-continuous extension $\mu^u$. This amounts to using the same recipe as in the non-degenerate case, but in the definition of $\mu^l$, respectively of $\mu^u$, we shift the interval $I(\Psi)$ to the left, respectively to the right, by an arbitrarily small amount. For any $k\in\Z$, there hold  
\begin{align}\label{rmkdyn}
\mu^l(\Psi)\geq 2k+1\ \ \Longrightarrow\ \ I(\Psi)\subset(k,+\infty)\ \ \Longrightarrow \ \ \mu^*(\Psi)\geq 2k+1,\\
\mu^u(\Psi)\leq 2k-1\ \ \Longrightarrow\ \ I(\Psi)\subset(-\infty,k)\ \ \Longrightarrow \ \ \mu^*(\Psi)\leq 2k-1,
\end{align}
where we have used $\mu^*$ to denote both $\mu^l$ and $\mu^u$.

We move now to describe the Conley-Zehnder index for a Hamiltonian structure $\omega$ on $(N,\mathfrak O)$. Consider the $2$-dimensional distribution $\xi_\omega:=\frac{TN}{\ker\omega}$ and observe that $(\xi_\omega,\overline{\omega})\rightarrow N$ is a symplectic vector bundle, where $\overline{\omega}$ is the symplectic form on the fibres induced by $\omega$. Let $\nu$ be a free-homotopy class of loops in $N$ and choose $\gamma_\nu$ a reference loop in $\nu$ together with a symplectic trivialisation $\Upsilon_\nu$ of $\xi_\omega$ along $\gamma_\nu$. If $\nu$ is the trivial class, just choose a constant loop with the constant trivialisation. Suppose that $c_1(\xi_\omega)$ is $\nu$-atoroidal. This means that the integral of $c_1(\xi_\omega)$ over a cylinder $i:C:=[0,1]\times\T_1\rightarrow N$ such that $i(0,\cdot)=\gamma_\nu(\cdot)$ depends only on $i(1,\cdot)$.

Let $Z$ be a positive section of $\ker\omega$ and observe that $d\Phi^Z_t$ acts on $(\xi_\omega,\overline{\omega})$ as a symplectic bundle map. Let $(\gamma,T)$ be a periodic orbit of $Z$ in the class $\nu$. Choose a cylinder $i:C\rightarrow N$ connecting $\gamma_\nu$ with $\gamma$ and let $\Upsilon:(i^*\xi_\omega,i^*\overline{\omega})\rightarrow(\epsilon^2_C,\omega_{\op{st}})$ be a $\overline{\omega}$-symplectic trivialisation of $i^*\xi_\omega$ on $C$ extending $\Upsilon_\nu$. Here $\epsilon^2_C$ is the trivial rank $2$-vector bundle over $C$. We form the path of symplectic matrices $\Psi^{C,\Upsilon}_\gamma\in\op{Sp}_T(1)$
\begin{equation}\label{psigamma}
\Psi^{C,\Upsilon}_\gamma(t):=\Upsilon_{\gamma(t)}\circ d_{\gamma(0)}\Phi^{Z}_t\circ \Upsilon^{-1}_{\gamma(0)}\in \op{Sp}(1).
\end{equation}
\begin{dfn}
The \textit{Conley-Zehnder index} of $\gamma$ is $\mu^*_{\op{CZ}}(\gamma,\Upsilon_\nu):=\mu^*(\Psi^{C,\Upsilon}_\gamma)$. This number does not depend on the choice of $Z$ and the hypothesis on the Chern class ensures that it is also independent of the pair $(C,\Upsilon)$. If we choose a different $\gamma_\nu$ with a different trivialisation $\Upsilon_\nu$ the index gets shifted by an integer which is the same for every closed orbit $\gamma$ in $\nu$.
\end{dfn}
\begin{rmk}
If $\omega$ is of contact type, the vector bundle $(\ker\tau,\omega\big|_{\ker\tau})$ is symplectic and isomorphic to $(\xi_\omega,\overline{\omega})$. Hence, we can use the former bundle for the computation of the index.
\end{rmk}

\chapter{Symplectic Cohomology and deformations}\label{cha_sh}
The material we present in this chapter build up from the work contained in \cite{ar}. On the one hand, it grew out of several fruitful discussions with Alexander Ritter and it is likely to appear in a forthcoming paper \cite{alga}. On the other hand, the statement and proof of the crucial Proposition \ref{isoprop}, which enables us to prove the main result of this section (Theorem \ref{invthm}), are due exclusively to the author of the present thesis.

The focus is Symplectic Cohomology for general convex manifolds as first developed by Viterbo in \cite{vit}. In the next chapter, we will apply the abstract results proved here to magnetic systems. 

In Section \ref{csm} we recall the notion of convex symplectic manifold $(W,\omega,j)$ in the open and compact case and we give a notion of isomorphism in such categories. In Section \ref{sc}, we introduce the Symplectic Cohomology of such manifolds, a set of algebraic invariants $SH^*_\nu$ counting $1$-periodic orbits of Hamiltonian flows in the free homotopy class $\nu$. We will work under the hypothesis that $c_1(\omega)$ is $\nu$-atoroidal.

In Section \ref{pert_aut} we see how to define $SH^*_\nu$ perturbing autonomous Hamiltonian functions close to the non-constant periodic orbits.

In Section \ref{sec:fil}, we define $SH^*_\nu$ using a subclass of autonomous Hamiltonians, whose non-constant periodic orbits are reparametrised Reeb orbits for the contact form at infinity. For such Hamiltonians we have filtrations on the complex yielding $SH^*_\nu$ given by the action (when $(W,\omega)$ a Liouville domain) and, in general, by the period of Reeb orbits. These filtrations will be used for the applications in the next chapter.

In Section \ref{sec:inv0} we quote the results contained in \cite{ar}, asserting that $SH^*_\nu$ does not depend on the isomorphism class of a convex manifold.

In Section \ref{inv}, we prove the invariance of $SH^*_\nu$ for compact manifolds under convex deformations $(W_s,\omega_s,j_s)$ which are projectively constant on $\nu$-tori.

\section{Convex symplectic manifolds}\label{csm}

Let us start by treating the non-compact case. First, we define the manifolds we are interested in.
\begin{dfn}
Let $W$ be an open manifold (non-compact and without boundary) of dimension $2m$ and let $\Si$ be a closed manifold of dimension $2m-1$. We say that $W$ is \textit{an open manifold with cylindrical end (modeled on $\Si$)} if there exists a diffeomorphism $j=(r,p):U\stackrel{~}{\rightarrow}(a,+\infty)\times \Si$. Here $U$ is an open subset of $W$, such that $W\setminus j^{-1}((b,+\infty)\times \Si)$ is a compact subset of $W$ with boundary $j^{-1}(\{b\}\times \Si)$, for every (or, equivalently, some) $b>a$. We write $i:\Sigma\rightarrow W$ for the embedding $i:=j^{-1}\big|_{\{b\}\times\Sigma}$.

We call $j$ (or less precisely $U$) a \textit{cylindrical end of $W$} and we denote by $\partial_r\in TU$ the vector field generated by the coordinate $r$.
\end{dfn}

We now fix some additional notation for open manifolds with cylindrical end.
\begin{dfn}\label{dfn_gra}
If $f:\Sigma\rightarrow (a,+\infty)$, define by $\Gamma_f:\Sigma\rightarrow (a,+\infty)\times\Sigma$ the map $\Gamma_f(x)=(f(x),x)$ and by $\Sigma_f$ the image of $\Gamma_f$ inside $W$. Denote by $W^{f}$ and $W_{f}$ the compact and non-compact submanifolds of $W$ with boundary $\Si_{f}$. If we have two functions $f_0,f_1:\Si\rightarrow(a,+\infty)$, with $f_0<f_1$, let $W^{f_1}_{f_0}$ be the compact submanifold between $\Sigma_{f_0}$ and $\Sigma_{f_1}$. 
\end{dfn}

We proceed to consider the subclass of open manifolds for which the symplectic structure at infinity is compatible with some contact structure on the model $\Si$.

\begin{dfn}
Let $(W,\omega)$ be a symplectic manifold having a cylindrical end $j:U\rightarrow (a,+\infty)\times\Sigma$. Suppose that $(W,\omega)$ is \textit{exact at infinity}, meaning that $\omega$ is exact on $U$. We say that the triple $(W,\omega,j)$ is a \textit{convex symplectic manifold} if $\theta:=\imath_{\partial_r}\omega\in\Omega^1(U)$ is a primitive for $\omega$ on $U$, or, equivalently, if $(j^{-1})^*\theta=e^rp^*\alpha$ for some (unique) contact form $\alpha$ on $\Si$. We denote by $R^\alpha\in\Gamma(\Sigma)$ the Reeb vector field of $\alpha$ and we define the auxiliary function $\rho:=e^r$. We call $\partial_r$ the \textit{Liouville vector field} and $\theta$ the \textit{Liouville form}.

We say that a convex symplectic manifold is a \textit{Liouville domain}, if $\theta$ extends to $W$ in such a way that $\omega=d\theta$ on the whole manifold.
\end{dfn}

\begin{rmk}
Since $\omega$ is a symplectic form, an $\alpha$ satisfying the equality above is automatically a contact form on $\Si$ inducing the orientation obtained from $\omega^n\big|_U$ putting the vector $\partial_r$ first. Using the language introduced in the previous chapter we can say that every $\Sigma_f$ is a hypersurface of positive contact type in $(W,\omega)$.
\end{rmk}

We now define isomorphisms of convex symplectic manifolds. 
\begin{dfn}\label{isocon}
Let $(W_i,\omega_i,j_i)$, for $i=0,1$, be two convex symplectic manifolds. An \textit{isomorphism} between $(W_0,\omega_0,j_0)$ and $(W_1,\omega_1,j_1)$ is a diffeomorphism $F:W_0\rightarrow W_1$ with the following properties:
\begin{enumerate}
 \item $F$ is a symplectomorphism, i.e.\ $F^*\omega_1=\omega_0$,
 \item there exist open cylindrical ends $V_i\subset U_i$ such that $F(V_0)\subset V_1$,
 \item $j_1\circ F\circ j^{-1}_0\big|_{V_0}(r_0,p_0)=(r_0-f(p_0),\psi(p_0))$ with $\psi^*\alpha_1=e^f\alpha_0$.
\end{enumerate}
The third condition is equivalent to $(F\big|_{V_0})^*\theta_1=\theta_0$. Furthermore, a simple argument shows that $\psi:\Si_0\rightarrow \Si_1$ is indeed a diffeomorphism.                                                                                                                                                                                                                                                                                                                                                                                                                                                                          
\end{dfn}
\begin{rmk}
Let $(W,\omega)$ be a symplectic manifold with cylindrical end and suppose we are given two convex ends $j_0$ and $j_1$. If $\theta_0=\theta_1$ (or, equivalently, $\partial_{r_0}=\partial_{r_1}$) on some common cylindrical end, then $\op{Id}:(W,\omega,j_0)\stackrel{~}{\longrightarrow}(W,\omega,j_1)$ is an isomorphism. For this reason, we will also use the notation $(M,\omega,\theta)$ to designate any convex symplectic manifold $(M,\omega,j)$ such that $\theta:=\imath_{\partial_r}\omega$.
\end{rmk}

For compact symplectic manifolds we can give an analogous notion of convexity.
\begin{dfn}
Let $(W,\omega)$ be a compact symplectic manifold with boundary $\Sigma$ and let $j=(r,p):U\stackrel{~}{\rightarrow}(-\varepsilon,0]\times\Si$ be a collar of the boundary. We say that $(W,\omega,j)$ is a \textit{(compact) convex symplectic manifold} if $\theta:=i_{\partial_r}\omega\in\Omega^1(U)$ is a primitive for $\omega$ on $U$, or equivalently, if $(j^{-1})^*\theta=e^rp^*\alpha$ for some (unique) contact form $\alpha$ on $\Si$.
\end{dfn}

A standard construction called \textit{completion} yields an open convex symplectic manifold $(\hat{W},\hat{\omega},\hat{j})$ starting from a compact convex symplectic manifold $(W,\omega,j)$:
\begin{equation}
(\hat{W},\hat{\omega},\hat{j}):=(W,\omega,j)\bigsqcup_j\Big((-\varepsilon,+\infty)\times\Si,d(e^r\alpha),\op{Id}\Big),
\end{equation}
The new object is obtained by gluing along $j$ the original convex symplectic manifold with a positive cylindrical end of the symplectisation of $(\Si,\alpha)$.
\begin{dfn}
We say that two convex compact symplectic manifolds are isomorphic if their completions are isomorphic according to Definition \ref{isocon}.
\end{dfn}
\begin{rmk}\label{rmk_com}
Let $(W,\omega)$ be a symplectic manifold with convex cylindrical end $j:U\stackrel{~}{\rightarrow}(a,+\infty)\times\Sigma$. Let $f:\Sigma\rightarrow(a,+\infty)$ be any function and consider $W^f$ (see Definition \ref{dfn_gra}). Then, $(W^f,\omega\big|_{W^f},j\big|_{W^f})$ is a compact convex manifold. We do not lose any information by passing from $W$ to $W^f$. Indeed, we observe that we can get back the whole open manifold by taking the completion as explained above. In other words,
\begin{equation}\label{equ_com}
(\hat{W}^f,\hat{\omega}\big|_{W^f},\hat{j}\big|_{W^f})\simeq (W,\omega,j).
\end{equation}
We can rephrase \eqref{equ_com} by saying that every open convex symplectic manifold $(W,\omega)$ is the completion of a compact convex submanifold $(W',\omega\big|_{W'})\subset (W,\omega)$. This has a natural generalisation for a family of open convex symplectic manifolds $(W,\omega_s,j_s)$. Namely, there exists a corresponding family of zero-codimensional compact submanifolds $W_s\subset W$ such that $\partial W_s$ is $j_s$-convex and
\begin{equation}
(\hat{W}_s,\hat{\omega}_s\big|_{W_s},\hat{j}_s\big|_{W_s})\simeq (W,\omega_s,j_s).
\end{equation}
\end{rmk}

\section{Symplectic Cohomology of convex symplectic manifolds}\label{sc}
Assume for the rest of this section that $(W,\omega,j)$ is an open convex symplectic manifold and $\nu$ is a free homotopy class of loops in $W$.
\subsection{Preliminary conditions}
We now define the Symplectic Cohomology of $(W,\omega,j)$ in the class $\nu$ under the assumption that
\begin{enumerate}
 \item[\bfseries (S1)] all the Reeb orbits of the associated contact form $\alpha$ in the classes $i^{-1}_*(\nu)$ are non-degenerate (in this case we say that $\alpha$ is \textit{$\nu$-non-degenerate});
 \item[\bfseries (S2)] $c_1(\omega)$ is $\nu$-atoroidal.
\end{enumerate}
Denote by $\op{Spec}(\alpha,\nu)$ the set of periods and by $T(\alpha,\nu)>0$ the minimal period of a Reeb orbit of $\alpha$ in the class $\nu$. Assumption \textbf{(S1)} guarantees that $\op{Spec}(\alpha,\nu)$ is a discrete subset of $[T(\alpha,\nu),+\infty)$.

\subsection{Admissible Hamiltonians}

Symplectic Cohomology counts the number of $1$-periodic orbits in the class $\nu$ for a particular kind of $1$-periodic Hamiltonians. To introduce them, we need first the following two general definitions.
\begin{dfn}\label{dfn_nondeg}
Let $(W,\omega)$ be a symplectic manifold and consider a function $H:\T_1\times W \rightarrow \R$. We say that a $1$-periodic orbit $x$ for $X_{H}$ is \textit{non-degenerate} if $1$ does not belong to the spectrum of $d_{x(0)}\Phi^{H}_1$. We say that $H$ is $\nu$-non-degenerate, if all the $1$-periodic orbits in the class $\nu$ are non-degenerate.                                                                                                                                                                                                                                                                                                                                                                              
\end{dfn}
\begin{dfn}
Let $(W,\omega,j)$ be a symplectic manifold with cylindrical end. A function $H:\T_1\times W\rightarrow\R$ is said to have \textit{constant slope at infinity}, if there exist constants $T_{H}\in(0,+\infty)$ and $a_{H}\in\R$ such that $H\circ j^{-1}=T_{H}e^{r}+a_{H}$ on some cylindrical end $V\subset U$. The number $T_{H}$ is called the \textit{slope} of $H$.                                                                                                                                                                                                                                                                                                                                                                                           
\end{dfn}

We are now ready to introduce the class of Hamiltonians that we are going to use on $(W,\omega,j)$.
\begin{dfn}\label{dfn:adm}
A function $H:\T_1\times W\rightarrow \R$ is called \textit{$\nu$-admissible} if it satisfies the following two properties:
\begin{enumerate}
 \item[\bfseries (H1)] it is $\nu$-non-degenerate,
 \item[\bfseries (H2)] it has constant slope at infinity and $T_H$ does not belong to $\op{Spec}(\alpha,\nu)$.
\end{enumerate}
We denote the set of all $\nu$-admissible Hamiltonians by $\mathcal H_\nu$.
\end{dfn}

\subsection{The action 1-form, the grading and the moduli spaces}
Consider $H\in\mathcal H_\nu$. Its $1$-periodic orbits are the zeros of a closed $1$-form $d\mathcal A^{\omega}_H$ on $\mathscr L_\nu W$. If $x\in \mathscr L_\nu W$ and $\xi\in \Gamma(x^*TW)$ is an element in $T_x(\mathscr L_\nu W)$, the $1$-form is defined by
\begin{equation}
d_x\mathcal A^{\omega}_H\cdot\xi=\int_{\T_1} \left(\imath_{\dot{x}(t)}\omega_{x(t)}+d_{x(t)}H(t,x(t))\right)\cdot\xi(t)\,dt.
\end{equation}

The action $1$-form yields a function $\mathcal A^{\omega}_H$ on the set of $\nu$-cylinders by integration. If $u:[a,b]\rightarrow \mathscr L_\nu W$ is a $\nu$-cylinder, we set
\begin{align*}
\mathcal A^{\omega}_H(u):&=\int_a^bd_{u(s)}\mathcal A^{\omega}_H\cdot\frac{du}{ds}(s)\,ds\\
&=-\omega(u)+\int_{\T_1}H(t,u(b,t))\,dt-\int_{\T_1}H(t,u(a,t))\,dt,
\end{align*}
where $\omega(u)$ is the integral of $\omega$ over $u$. Observe that $\mathcal A^{\omega}_H(u)$ does not change if we homotope the cylinder $u$ keeping the end points fixed.

We can associate a degree $|x|$ to every $1$-periodic orbit $x$ in $\nu$ as follows. Choose a reference loop $x_\nu$ in the class $\nu$ and fix a symplectic trivialisation $\Upsilon_{x_\nu}$ of $x_\nu^*(TW)$. Take a connecting cylinder $C_x$ from $x_\nu$ to $x$ and extend the trivialisation over $C_x$. This will induce a symplectic trivialisation $\Upsilon_x$ of $x^*TW$, whose homotopy class depends only on $\Upsilon_{x_\nu}$ since $c_1(\omega)$ is $\nu$-atoroidal. Writing the linearisation of the Hamiltonian flow $d\Phi^{X_H}_t$ using $\Upsilon_x$ yields a path of symplectic matrices along $x$. We call \textit{Conley-Zehnder index} the Maslov index of this path and we denote it by $\mu_{\op{CZ}}(x)\in\Z$. Finally, we set $|x|:=\frac{\op{dim}(W)}{2}-\mu_{\op{CZ}}(x)$.

We now define a moduli space of cylinders connecting two $1$-periodic orbits using a particular kind of almost complex structures compatible with convexity.
\begin{dfn}
Let $(W,\omega,j)$ be a convex manifold. An $\omega$-compatible $1$-periodic almost complex structure $J$ is \textit{convex} if for big $r$, $J$ is independent of time and $(d\rho)\circ J=-\theta$.
\end{dfn}
Take a $1$-periodic $\omega$-compatible convex almost complex structure $J$ and consider Floer's equation for cylinders $u:\R\times\T_1\rightarrow W$
\begin{equation}\label{fleq}
\partial_s u+J(\partial_tu-X_H)=0\,.
\end{equation}
If $u$ is a Floer trajectory we denote its energy by
\begin{equation}\label{ene}
E(u):=\int_\R\Vert\partial_su\Vert^2_J\,ds=\int_{\R\times\T_1}|\partial_su|^2_J\,ds\,dt
\end{equation}
Using Floer's equation \eqref{fleq}, one gets the identity
\begin{equation}\label{eneid}
E(u)=-\mathcal A^{\omega}_H(u).
\end{equation}

Call $\mathcal M'(H,J,x_-,x_+)$ the space of Floer trajectories that converge uniformly to the $1$-periodic orbits $x\pm$ for $s\rightarrow\pm\infty$. Suppose that $J$ is \textit{$H$-regular}, namely that the operator $u\mapsto \partial_s u+J(\partial_tu-X_H)$ on the space of all cylinders is regular at $\mathcal M'(H,J,x_-,x_+)$. This implies that all the moduli spaces are smooth manifolds and, if non-empty, $\dim\mathcal M'(H,J,x_-,x_+)=\mu_{\op{CZ}}(x_+)-\mu_{\op{CZ}}(x_-)=|x_-|-|x_+|$.

Let $\mathcal M(H,J,x_-,x_+)$ be the quotient of $\mathcal M'(H,J,x_-,x_+)$ under the $\R$-action obtained by shifting the variable $s$. If $A$ is a homotopy class of cylinders relative ends, we denote by $\mathcal M^A(H,J,x_-,x_+)$ the subset of the moduli space whose elements belong to $A$.

\subsection{The cochain complex and the differential}
We build a complex $SC^*_\nu(W,\omega,j,H)$ as the free $\Lambda$-module generated by the $1$-periodic orbits $x$ of $H$. Here $\Lambda$ is the Novikov ring defined as
\begin{equation*}
\Lambda:=\left\{\sum_{i=0}^{+\infty}n_it^{a_i}\ \big|\ n_i\in\Z,\ a_i\in \R,\ \lim_{i\rightarrow+\infty}a_i=+\infty\right\}.
\end{equation*}
The differential $\delta_J:SC^*_\nu(W,\omega,j,H)\rightarrow SC^{*+1}_\nu(W,\omega,j,H)$ of such complex is defined on the generators as
\begin{equation*}
\delta_J x:=\sum_{\substack{ u\in\mathcal M(H,J,y,x)\\ |y|-|x|=1}}\!\!\epsilon(u)\,t^{-\mathcal A^\omega_H\!(u)}y,
\end{equation*}
where $\epsilon(u)$ is an orientation sign defined in \cite[Appendix 2]{ritqft}. We extend it on the whole $SC^*_\nu(W,\omega,j,H)$ by $\Lambda$-linearity. If $x$ and $y$ are two periodic orbits, we denote by $<\delta_Jx,y>$, the component of $\delta_J x$ along the subspace generated by $y$.

By a standard argument in Floer theory, in order to prove that $\delta_J$ is a well-defined map of $\Lambda$-modules and that $\delta_J\circ\delta_J=0$, we need to prove that the elements of $\mathcal M^A(H,J,x_-,x_+)$ have:
\begin{enumerate}[\itshape a)]
 \item Uniform $C^0$-bounds; these stem from a maximum principle.
 \item Uniform $C^1$-bounds; these follow from Lemma \ref{lem:bub} below, where we show that $c_1(\omega)$ is aspherical (see \cite{hofsal}).
 \item Uniform energy bounds; these follow readily from \eqref{eneid}.
\end{enumerate}

\begin{lem}\label{lem:bub}
If $\sigma\in\Omega^2(W)$ is $\nu$-atoroidal, then $\sigma$ is aspherical. 
\end{lem}
\begin{proof}
We have an action $\sharp:\pi_2(W)\times H_2(W,\Z)\rightarrow H_2(W,\Z)$ by connected sum. Moreover, if $([S^2],[M])\in \pi_2(W)\times H_2(W,\Z)$, then
\begin{equation*} 
[\sigma]([S^2\sharp M])=[\sigma]([S^2])+[\sigma]([M]).
\end{equation*}
Therefore, the lemma is proven if we show that the action preserves classes represented by $\nu$-tori. Indeed, in this case the above identity would become
\begin{equation*}
0=[\sigma]([S^2])+0. 
\end{equation*}
Suppose $M$ is a $\nu$-torus parametrised by $(s,t)\mapsto \Gamma^M(s,t)=\gamma^M_s(t)$, with $\gamma_s\in\nu$ and $(s,t)\in\T_1\times\T_1$. The connected sum between $M$ and some $S^2$ can be represented by a map $\Gamma^{S^2\sharp M}:\T_1\times\T_1\rightarrow W$, which coincides with $\Gamma^M$ outside a small square $[s_0,s_1]\times[t_0,t_1]$ and inside the square it coincides with the map
\begin{equation*}
[s_0,s_1]\times[t_0,t_1]\stackrel{\sim}{\longrightarrow} S^2\setminus D^2\longrightarrow W,  
\end{equation*}
where $D^2$ is the small disc that we remove from $S^2$ in order to perform the connected sum. Hence, $S^2\sharp M$ is still a torus and the curves $t\mapsto \gamma^{S^2\sharp M}_s(t):=\Gamma^{S^2\sharp M}(s,t)$ are still in the class $\nu$ because $\gamma^{S^2\sharp M}_s=\gamma^M_s$ if $s\notin[s_0,s_1]$. 
\end{proof}

Thanks to the lemma, we have well-defined cohomology groups obtained from $\delta_J$. We denote them by $SH^*_\nu(W,\omega,j,H,J)$.
We put a partial order on pairs $(H,J)$, where $H$ and $J$ are as above, by saying that $(H^+,J^+)\preceq (H^-,J^-)$ if and only if $T_{H^+}\leq T_{H^-}$ (we use the reverse notation for the signs, since in the definition of maps on a generator $x$, we take the moduli space of cylinders arriving at $x$).

When $(H^+,J^+)\preceq(H^-,J^-)$ we can construct continuation maps
\begin{equation*}
\varphi_{H^-,J^-}^{H^+,J^+}:SH^*_\nu(W,\omega,j,H^+,J^+)\longrightarrow SH^*_\nu(W,\omega,j,H^-,J^-) 
\end{equation*}
such that,
\begin{equation*}
(H^+,J^+)\preceq(H^0,J^0)\preceq(H^-,J^-)\quad \Longrightarrow\quad\varphi_{H^-,J^-}^{H^0,J^0}\circ\varphi_{H^0,J^0}^{H^+,J^+}=\varphi_{H^-,J^-}^{H^+,J^+}.
\end{equation*}
Such maps are defined as follows. We consider a homotopy $(H^s,J^s)$ with $s$ ranging in $\R$, which is equal to $(H^-,J^-)$, for $s$ very negative, and to $(H^+,J^+)$, for $s$ very positive. Such homotopy yields an $s$-dependent Floer equation
\begin{equation}\label{fleq2}
\partial_s u+J^s(\partial_tu-X_{H^s})=0\,.
\end{equation}
As before, we denote by $\mathcal M^A(\{H^s\},\{J_s\},x_-,x_+)$ the moduli spaces of solutions in the class $A$, connecting a $1$-periodic orbit of $X_{H^-}$ with a $1$-periodic orbit of $X_{H^+}$. They are smooth manifolds and, when they are non-empty, we have that $\dim\mathcal M^A(\{H^s\},\{J_s\},x_-,x_+)=\mu_{\op{CZ}}(x_+)-\mu_{\op{CZ}}(x_-)=|x_-|-|x_+|$. For these moduli spaces we have to show properties {\itshape a), b), c)} as before. Property \textit{a)} stems again from the maximum principle, which holds provided $\frac{d}{ds}T_{H^s}\leq0$.
Property \textit{b)} is once more a consequence of Lemma \ref{lem:bub}. Property \textit{c)}, stems from the modified energy-action identity
\begin{equation}\label{enact2}
E(u)=-\mathcal A^{\omega}_{H^+}(u)+\int_{\T_1}\!(H^--H^+)(t,x_-(t))dt+\int_{\R\times\T_1}\!\!(\partial_s H^s)(t,u(s,t))dsdt\,,
\end{equation}
which replaces \eqref{eneid} for $s$-dependent cylinders. Notice that the rightmost term of \eqref{enact2} is bounded thanks to the uniform $C^0$-bounds and the fact that $\partial_s H^s\neq0$ only on a finite interval. If we define $\eta(x_-):=\int_{\T_1}(H^+-H^-)(t,x_-(t))dt$, then at the chain level
\begin{equation}
\varphi_{H^-,J^-}^{H^+,J^+}x:=\sum_{\substack{u\in\mathcal M(\{H^s\},\{J_s\},y,x)\\ |y|-|x|=0}}\!\!\epsilon(u)\,t^{\eta(y)-\mathcal A^{\omega}_{H^+}(u)}y\,.
\end{equation}
An argument similar to the one at the end of Section \ref{sub:chain}, shows that such a map intertwines $\delta_{J^-}$ and $\delta_{J^+}$, and, therefore, yields the desired map in cohomology.

The \textit{Symplectic Cohomology (in the class $\nu$)} is defined to be the direct limit of the direct system $\big(SH^*_\nu(W,\omega,j,H,J),\varphi_{H^-,J^-}^{H^+,J^+}\big)$:
\begin{equation*}
SH^*_\nu(W,\omega,j):=\varinjlim_{(H,J)} SH^*_\nu(W,\omega,j,H,J).
\end{equation*}

We define the Symplectic Cohomology of a compact convex symplectic manifold as the Symplectic Cohomology of its completion: $SH^*_\nu(W,\omega,j):=SH^*_\nu(\hat{W},\hat{\omega},\hat{j})$.

\section{Perturbing a non-degenerate autonomous Hamiltonian}\label{pert_aut}
We saw that the Symplectic Cohomology of $(W,\omega,j)$ is defined starting from a Hamiltonian function whose $1$-periodic orbits are non-degenerate. However, if $H:W\rightarrow\R$ is an autonomous Hamiltonian and $x$ is a non-constant $1$-periodic orbit for the flow of $X_H$, $x$ belongs to an $S^1$-worth of periodic orbits. Hence, $x$ is degenerate.
In this setting, we have to look to a weaker notion of non-degeneracy.

\begin{dfn}
Let $x$ be a $1$-periodic orbit of $X_H$, where $H$ is an autonomous Hamiltonian and denote by $S_x$ the connected component of set of $1$-periodic orbits of $X_H$ to which $x$ belongs. We say that $x$ is \textit{transversally non-degenerate}
\begin{itemize}
 \item either if $x$ is a constant orbit and it is non-degenerate according to Definition \ref{dfn_nondeg} (namely, $d_{x(0)}\Phi^{X_H}_1$ does not have $1$ in the spectrum);
 \item or if $x$ is non-constant and $1$ has algebraic multiplicity $2$ in the spectrum of $d_{x(0)}\Phi^{X_H}_1$.
\end{itemize}
We say that $H$ is \textit{transversally $\nu$-non-degenerate} if all $1$-periodic orbits in the class $\nu$ are transversally non-degenerate.
\end{dfn}
\begin{rmk}
When $\dim W=4$ the notion of transversally non-degenerate coincides with the notion given for Hamiltonian structures in Definition \ref{def_non}.
\end{rmk}
\begin{rmk}
If $x$ is transversally non-degenerate, then $S_x=\{x(\cdot+t')\}_{t'\in\T_1}$. Therefore, $S_x=\{x\}$, if $x$ is constant and $S_x\simeq\T_1$, if $x$ is not constant.
\end{rmk}

If $H$ is an autonomous Hamiltonian such that all the $1$-periodic orbits of $X_H$ in the class $\nu$ are \textit{transversally} non-degenerate, we can construct a small perturbation $H_\varepsilon:W\times S^1\rightarrow \R$ such that
\begin{itemize}
 \item $H_\varepsilon$ differs from $H$ only in a small neighbourhood $U_\varepsilon$ of the non-constant periodic orbits of $X_H$ and $\Vert H-H_\varepsilon\Vert$ is small;
 \item the constant periodic orbits of $X_{H_\varepsilon}$ are the same as the constant periodic orbits of $X_H$;
 \item for every non-constant $1$-periodic orbit $x$ of $X_{H}$ there are two periodic orbits $x_{\op{min}}$ and $x_{\op{Max}}$ of $X_{H_\varepsilon}$ supported in $U_\varepsilon$. All the non-constant periodic orbits of $X_{H_\varepsilon}$ arise in this way.
\end{itemize}
Consider a function $H:W\rightarrow \R$ that satisfies the following two properties:
\begin{enumerate}
 \item[\bfseries (H'1)] it is transversally $\nu$-non-degenerate,
 \item[\bfseries (H'2)] it has constant slope at infinity and $T_H$ does not belong to $\op{Spec}(\alpha,\nu)$\,.
\end{enumerate}
Denote the set of such functions by $\mathcal H_\nu'$. We readily see that if we carry out the above perturbation to $H\in\mathcal H_\nu'$, the resulting function $H_\varepsilon$ belongs to $\mathcal H_\nu$ (see Definition \ref{dfn:adm}) and we can use it to compute Symplectic Cohomology.

The degree of the new non-constant orbits is given by
\begin{equation}
\bullet\ |x_{\op{min}}|=\frac{\op{dim}(W)}{2}-\mu_{\op{CZ}}(x),\quad\quad\bullet\ |x_{\op{Max}}|=\frac{\op{dim}(W)}{2}-\mu_{\op{CZ}}(x)-1,
\end{equation}
where $\mu_{\op{CZ}}(x)$ is the transverse Conley-Zehnder index.
It will be interesting for the applications to compute $<\delta_J x_{\op{Max}},x_{\op{min}}>$. Proposition 3.9(ii) in \cite{bo} tells us that
\begin{equation}
 <\delta_J x_{\op{Max}},x_{\op{min}}>\ =
 \begin{cases}
 0&\mbox{if }x \mbox{ is good},\\
 \pm 2t^{a_x}&\mbox{if }x \mbox{ is bad},
 \end{cases}
 \end{equation}
where $a_x$ is a small positive number. Recall that an orbit is \textit{bad} if it is an even iteration of a hyperbolic orbit with odd index and it is \textit{good} otherwise.

\section{Reeb orbits and two filtrations of the Floer Complex}\label{sec:fil}
In this section we restrict the admissible Hamiltonians to a subclass $\hat{\mathcal H}_\nu\subset\mathcal H_\nu$, whose non-constant periodic orbits are in strict relation with Reeb orbits of $\alpha$.

\begin{dfn}\label{dfn:adm2}
Fix some $b$ belonging to the image of the function $r:U\rightarrow\R$ and denote $\rho_b:=e^b$. Fix also a $C^2$-small Morse function $H_b:W^b\rightarrow \R$, such that close to the boundary $H_b=h_b(e^r)$ for some strictly increasing convex function $h_b$. Let $H:W\rightarrow\R$ be an element of $\mathcal H_\nu'$ with the following additional properties
\begin{enumerate}
 \item[\bfseries (\^{H}1)] on $W^b$, $H=H_b$,
 \item[\bfseries (\^{H}2)] on $W_b$, $H=h(e^r)$ for a function $h:[\rho_b,+\infty)\rightarrow$ which is strictly increasing and strictly convex on some interval $[\rho_b,\rho_H)$ and satisfies $h(\rho)=T_H\rho+a_H$, for some $T_H\notin\op{Spec}(\alpha,\nu)$ and $a_H\in\R$, on $[\rho_H,+\infty)$.
\end{enumerate}
We denote by $\hat{\mathcal H}_\nu'$ the subset of all the Hamiltonians in $\mathcal H_\nu'$ satisfying these two properties and by $\hat{\mathcal H}_\nu\subset\mathcal H_\nu$ the corresponding set of perturbations inside $\mathcal H_\nu$. 
\end{dfn}

The $1$-periodic orbits of $H\in\hat{\mathcal H}_\nu'$ fall into two classes: constant periodic orbits inside $W^b$ and non-constant periodic orbits inside $W_b$. Observe that in the region $W_b$, $X_H=h'R^\alpha$ where $h'$ is the derivative of $h$. Hence, $x(t)=(r(t),p(t))$ is a $1$-periodic orbit for $X_H$ if and only if $r(t)$ is constant and equal to some $\log\rho_x$ and there exists a $h'(\rho_x)$-periodic Reeb orbit $\gamma:\T_T\rightarrow\Sigma$ such that $p(t)=\gamma(h'(\rho_x)t)$. Since $h'$ is increasing and $h'(\rho_b)<T(\alpha,\nu)$, the non-constant periodic orbits of $X_H$ are in bijection with all the Reeb orbits with period smaller than $T_H$. 

Consider now the Floer Complex $(SC^*_\nu(W,\omega,j,H_\varepsilon),\delta_{J})$ with $H_\varepsilon\in\hat{\mathcal H}_\nu$ and some $J$ which is $H_\varepsilon$-admissible. In the next subsections we show that we have two ways of filtrating this complex.

The first one is a filtration by the action and it works when $(W,\omega)$ is a Liouville domain. In this case we can define the action functional $\mathcal A^\omega_{H_\varepsilon}:\mathscr L_\nu W\rightarrow\R$ on the space of loops in the class $\nu$.

The second one is a filtration by the period of Reeb orbits. It works for general convex manifolds but we have to restrict the class of admissible complex structures to some $\mathcal J(H)$ (see Definition \ref{dfn:ac}), so that we can apply the maximum principle (see \cite[page 654]{bo2} and Lemma \ref{max-pri}).

When $\nu=0$, both filtrations imply that the singular cohomology complex $(C^{*+n}(W,\Lambda),\delta_0)$ shifted by $n$ is a subcomplex of $(SC^*_\nu(W,\omega,j,H_\varepsilon),\delta_{J})$. 

\subsection{Filtration by the action for Liouville domains}
Suppose $(W,\omega)$ is a Liouville domain. This means that the Liouville form $\theta$ extends to a global primitive for $\omega$ on the whole $W$. If $H:\T_1\times W\rightarrow\R$ is any Hamiltonian function, we define the action functional $\mathcal A^\omega_H:\mathscr L_\nu W\rightarrow\R$ by
\begin{equation*}
\mathcal A^\omega_H(x)=-\int_{\T_1}x^*\theta+\int_{\T_1}H(t,x(t))dt\,,
\end{equation*}

If $u$ is a Floer cylinder connecting $x_-$ and $x_+$, \eqref{eneid} yields
\begin{equation*}
0\leq E(u)=-\mathcal A^\omega_H(u)=\mathcal A^\omega_H(x_-)-\mathcal A^\omega_H(x_+)\,.
\end{equation*}
This implies that $\mathcal A^\omega_H(x_-)\geq \mathcal A^\omega_H(x_+)$. Therefore, if $H$ is $\nu$-non-degenerate and $J$ is $H$-admissible, then $\delta_J$ preserves the superlevels of $\mathcal A^\omega_H$ and we get the following corollary.
\begin{cor}\label{corfil}
If $x$ and $y$ are two different $1$-periodic orbits for $H$ and $\mathcal A^\omega_H(y)\leq\mathcal A^\omega_H(x)$, then
\begin{equation}\label{matvan}
<\delta_Jx,y>=0\,.
\end{equation}
\end{cor}

We denote by $SC^{*,>a}_\nu(W,\omega,j,H)$ the subspace of $SC^{*}_\nu(W,\omega,j,H)$ generated by all the orbits with action bigger than some $a\in\R$. Corollary \ref{corfil} shows that $\delta_J$ restricts to a linear operator on this subspace, and the inclusion induces a map in cohomology:
\begin{equation}\label{vitmap}
SH^{*,> a}_\nu(W,\omega,j,H,J)\longrightarrow SH^{*}_\nu(W,\omega,j,H,J)\,.
\end{equation}

Let $(H^+,J^+)\preceq(H^-,J^-)$ and let $(H^s,J^s)$ be a homotopy between the two pairs satisfying $\frac{d}{ds}T_{H^s}\leq0$. If $u$ is an $s$-dependent Floer cylinder, relation \eqref{enact2} can be rewritten as
\begin{equation*}
\mathcal A^\omega_{H^+}(x_+)-\mathcal A^\omega_{H^-}(x_-)=-E(u)+\int_{\R\times\T_1}(\partial_s H^s)(t,u(s,t))dsdt\,.
\end{equation*}
From this equation we see that the continuation map preserves the filtration, provided that $\partial_sH^s\leq0$. Under this additional hypothesis, $\varphi_{H^-,J^-}^{H^+,J^+}$ induces a map
\begin{equation}\label{vitmapcont}
SH^{*,>a}_\nu(W,\omega,j,H^+,J^+)\longrightarrow SH^{*,>a}_\nu(W,\omega,j,H^-,J^-)\,.
\end{equation}

Let us now specialise to Hamiltonians $H\in\hat{\mathcal H}_\nu'$. Call $a_h:[\rho_b,+\infty)\rightarrow\R$, the function $a_h(\rho):=h(\rho)-\rho h'(\rho)$. The absolute value of $a_h(\rho_b)$ is small and
\begin{equation}\label{ah}
a_h'(\rho)=-\rho h''(\rho)\,. 
\end{equation}
Hence, $a_h$ is strictly decreasing on $[\rho_b,\rho_H]$ and $a_h\equiv a_H$ on $[\rho_H,+\infty)$.

If $x$ is a non-constant periodic orbit in the class $\nu$, then $x\subset W_b$ and we can write $x(t)=(\log\rho_x,\gamma_x(h'(\rho_x)t))$, for some Reeb orbit $\gamma_x$. We compute
\begin{equation}
\mathcal A^\omega_H(x)\ =\ -\int_{0}^{h'(\rho_x)}\!\!\rho_x \gamma_x^*\alpha\ +\ h(\rho_x)\ =\ a_h(\rho_x)\,.
\end{equation}

Fix $a\in\R$ and take $H^+$ and $H^-$ in $\hat{\mathcal H}_0'$, such that $T_{H^+}\leq T_{H^-}$ and $H^-=H^+$ on some $W^c$, where $a=a_{h^+}(e^c)=a_{h^-}(e^c)$. In particular, $a> \max\{a_{H^+},a_{H^-}\}$. Take perturbations $H^+_\varepsilon$ and $H^-_\varepsilon$ and let $(H^s_\varepsilon,J^s)$ be a homotopy between them such that $\partial_s H^s_\varepsilon\leq 0$ and which is constant on $W^c$. We readily see that
\begin{equation*}
SC^{*,>a}_\nu(W,\omega,j,H^+_\varepsilon)\ =\ SC^{*,>a}_\nu(W,\omega,j,H^-_\varepsilon)\,.  
\end{equation*}
Consider continuation cylinders between the generators of these two subcomplexes. The maximum principle shows that they are contained inside $W^c$. However, in this region $\partial_s H^s=0$ and the transversality of the moduli spaces shows that we only have the constant cylinders. We have proved the following corollary.
\begin{cor}\label{corfil2}
Under the hypotheses of the paragraph above, the filtered continuation map $SC^{*,>a}_\nu(W,\omega,j,H^+_\varepsilon)\rightarrow SC^{*,> a}_\nu(W,\omega,j,H^-_\varepsilon)$ is the identity.
\end{cor}

Take now $\nu=0$ and assume that $H\in\hat{\mathcal H}_0'$. We want to use the map \eqref{vitmap} to single out the constant orbits of $H$ among all contractible orbits. If $\overline{x}$ is a constant orbit, we readily compute
\begin{equation}
\mathcal A^\omega_H(\overline{x})=H(\overline{x})\geq \min_{W^b} H\,.
\end{equation}
If $x(t)=(\log\rho_x,\gamma_x(h'(\rho_x)t))$ is non-constant, we know that $\mathcal A^\omega_H(x)=a_h(\rho_x)$. Moreover, \eqref{ah} implies that $a_h'\leq-\rho_bh''$ and, therefore,
\begin{equation*}
a_h(\rho_x)\ \leq\  a_h(\rho_b)-\rho_b(h'(\rho_x)-h'(\rho_b))\,. 
\end{equation*}
Since $h'(\rho_x)\geq T(\alpha,\nu)>0$, we see that there exists $a_0<0$ such that
\begin{equation*}
\mathcal A^\omega_H(x)<a_0<\mathcal A^\omega_H(\overline{x})\,.
\end{equation*}
Taking a small perturbation $H_\varepsilon\in\hat{\mathcal H}_0$, this inequality still holds and we find that $SC^{*,> a_0}_0(W,\omega,j,H_\varepsilon)$ is generated by the constant periodic orbits only. A further analysis analogous to the case of closed symplectic manifolds shows that $\delta_J$ restricted to this subcomplex is the differential for the Morse Cohomology of the function $H_\varepsilon|_{W^b}=H_b$ (see \cite[Section 1.2]{vit}), where we identify constant periodic orbits of $X_{H_\varepsilon}$ with critical points of $H_{\varepsilon}$ up to a shifting of degrees by $n$. Thus, we have $SH^{*,> a_0}_0(W,\omega,j,H_\varepsilon,J)\simeq H^{*+n}(W,\Lambda)$. If $(H^+_\varepsilon,J^+)\preceq (H^-_\varepsilon,J^-)$ and $H^+$, $H^-$ belong to $\hat{\mathcal H}_0'$, then $SH^{*,> a_0}_0(W,\omega,j,H^+_\varepsilon,J^+)\simeq SH^{*,> a_0}_0(W,\omega,j,H^-_\varepsilon,J^-)$, by Corollary \ref{corfil2}. Since $\hat{\mathcal H}_0$ is cofinal in $\mathcal H_0$, we can take direct limits and arrive to the following corollary.

\begin{cor}\label{cor:vit1}
If $(W,\omega,j)$ is a Liouville domain, then, for every $k\in\Z$, there exists a map
\begin{equation}
H^{k+n}(W,\Lambda)\longrightarrow SH^k_0(W,\omega,j)\,.  
\end{equation}
If $\alpha$ does not have any Reeb periodic orbit $\gamma$ contractible in $W$ such that
\begin{equation*}
\mu_{\op{CZ}}(\gamma)\in\{\dim W/2-k,\dim W/2-k+1\}\,,  
\end{equation*}
then such map is an isomorphism.
\end{cor}

\subsection{Filtration by the period}\label{sec:per}

Let $(W,\omega)$ be a general open convex manifold. We restrict the class of admissible almost complex structures as follows.
\begin{dfn}\label{dfn:ac}
Fix an $\omega$-compatible almost complex structure $J_0$ such that $(d\rho)\circ J_0=-\theta$ on $W_b$. For every $H\in\hat{\mathcal H}_\nu'$ let $\mathcal J(H)$ be the set of the $\omega$-compatible almost complex structures $J_\varepsilon$ which are small compact perturbations of $J_0$ and which are $H_\epsilon$-admissible for some $H_\varepsilon\in\hat{\mathcal H}_\nu$.
\end{dfn}
The next lemma, which is taken from \cite[page 654]{bo2}, gives us some information on the behaviour of Floer cylinders for $(H,J_0)$ which are asymptotic to a reparametrisation of a Reeb orbit for $s\rightarrow-\infty$. 
\begin{lem}\label{max-pri}
Take $H\in \hat{\mathcal H}_\nu'$ and suppose that $u:(-\infty,s_*]\times\T_1\rightarrow W_b$ is a non-constant Floer half-cylinder for the pair $(H,J_0)$. If $u$ is asymptotic for $s\rightarrow-\infty$ to $x(t)=(r_x,\gamma_x(h'(e^{r_x})t))$, a non-constant $1$-periodic orbit for $X_H$, then $u$ is not contained in $W^{r_x}$.
\end{lem}
\begin{proof}
Using the cylindrical end $j=(r,p)$, we write $u(s,t)=(r(s,t),p(s,t))$. If we employ the splitting $TW_b\simeq\,<\partial_r>\oplus<R^\alpha>\oplus\,\ker\alpha$, the Floer equation for $u$ takes the following form:
\begin{equation}\label{flcp}
\left\{\begin{aligned}
\partial_sr-\alpha(\partial_t p)+ h'(e^r)&=\ 0\\
\alpha(\partial_sp)+\partial_tr&=\ 0\\
\pi_{\ker\alpha}(\partial_s p)+J_0\pi_{\ker\alpha}(\partial_tp)&=\ 0\,,
       \end{aligned}\right.
\end{equation}
where $\pi_{\ker\alpha}$ is the projection on the third factor of the splitting.

Now define the real functions $\overline{r}$ and $\overline{a}$ on the interval $(-\infty,s_*]$ by integrating in the $t$-direction:
\begin{equation*}
\overline{r}(s):=\int_{\T_1}r(s,t)dt\,,\quad\quad \overline{a}(s):=\int_{\T_1}p(s,\cdot)^*\alpha\,.
\end{equation*}
By Stokes's Theorem we have that
\begin{equation}\label{steq}
\overline{a}(s_1)-\overline{a}(s_0)=\int_{[s_0,s_1]\times\T_1}p^*d\alpha\,.
\end{equation}

Since $J_0$ leaves $\ker\alpha$ invariant and it is $\omega$-compatible, the third equation yields that the rightmost term in \eqref{steq} is non-negative and, therefore, $\overline{a}$ is a non-decreasing function.

We integrate the first equation of \eqref{flcp} in $t$ and use this monotonicity property:
\begin{equation}\label{inerbar}
\frac{d}{ds}\overline{r}(s)\geq -\int_{\T_1}h'(e^{r(s,t)})dt+\overline{a}(-\infty)=h'(e^{r_x})-\int_{\T_1}h'(e^{r(s,t)})dt\,.
\end{equation}
Suppose now, by contradiction that $r(s,t)\leq r_x$ for every $(s,t)\in(-\infty,s_*]\times\T_1$. Since $\rho\mapsto h'(\rho)$ is increasing, we find that the rightmost term in \eqref{inerbar} is non-negative. Hence, $\frac{d}{ds}\overline{r}(s)\geq 0$. We have two possibilities. Either there exists a point $s_{**}<s_*$ such that $\frac{d}{ds}\overline{r}(s_{**})>0$, or $\frac{d}{ds}\overline{r}\equiv0$ on the whole cylinder. In the first case, we must have $\overline{r}(s_{**})> r_x$, which is a contradiction. In the second case, $\overline{r}\equiv r_x$ and, since $r(s,t)\leq r_x$, we see that $r(s,t)\equiv r_x$ on the whole cylinder. By the first equation in \eqref{flcp} $\overline{a}$ is also constant and by \eqref{steq}, we have that
\begin{equation*}
\int_{(-\infty,s_*]\times\T_1}p^*d\alpha=0\,.
\end{equation*}
Putting this together with the fact that the function $r(s,t)$ is constant, we get that $E(u)=0$, which contradicts the fact that $u$ is not constant. Since both cases led to a contradiction, we see that $u$ is not contained in $W^{r_x}$ and the lemma is proven.
\end{proof}
We have the following immediate corollary.
\begin{cor}\label{cor:per}
Suppose $u:\R\times\T_1\rightarrow W$ is a non-constant Floer cylinder for the pair $(H,J_0)$, with $H\in\hat{\mathcal H}'_\nu$, which connects two $1$-periodic orbits for $X_H$. If $x_-=(r_{x_-},\gamma_-(T_-t))$ for some $T_-$-periodic Reeb orbit $\gamma_-$, with $T_-=h'(e^{r_{x_-}})$, then $x_+=(r_{x_+},\gamma_+(T_+t))$ for some $T_+$-periodic Reeb orbit $\gamma_+$, with $T_+=h'(e^{r_{x_+}})$, such that \begin{equation}
T_+\ >\ T_-\,. 
\end{equation}
\end{cor}
\begin{proof}
Consider the function $r(s,t):=r(u(s,t))$. By the lemma we just proved, $\max r> r_{x_-}$. By the maximum principle, $x_+$ is not contained in $W^{r_{x_-}}$. Hence, $x_+=(r_{x_+},\gamma_+(T_+t))$ is a non-constant periodic orbit and $r_{x_+}>r_{x_-}$. The thesis follows since $\rho\mapsto h'(\rho)$ is strictly increasing in the interval $[\rho_b,\rho_H]$.
\end{proof}
The previous corollary gives important information on $SH_\nu$, once we take a perturbation in order to achieve transversality. 
\begin{cor}\label{matvan2}
Given $H\in\hat{\mathcal H}'_\nu$, there exist $H_\varepsilon\in\hat{\mathcal H}_\nu$ and $J_\varepsilon\in\mathcal J(H)$ such that the following statement is true. Let $x$ and $y$ be $1$-periodic orbits of $X_H$ with $S_x\neq S_y$ such that $x$ is associated with a Reeb orbit of period $T_x$ and $y$ is either constant or it is associated with a Reeb orbit of period $T_y$, with $T_y\leq T_x$. If $x_\varepsilon$ and $y_\varepsilon$ are $1$-periodic orbits of $X_{H_\varepsilon}$ close to $x$ and $y$, then 
\begin{equation}
<\delta_{J_\varepsilon}x_\varepsilon,y_\varepsilon>\ =\ 0\,.
\end{equation}
\end{cor}
\begin{proof}
We argue by contradiction and suppose that there exists $\varepsilon_k\rightarrow0$ and a sequence of cylinders $u_k\in\mathcal M'(H_{\varepsilon_k},J_{\varepsilon_k},x_{\varepsilon_k},y_{\varepsilon_k})$.

We claim that, up to taking a subsequence, there is an $m\in\N$ such that for every $i=1,\ldots,m$ there exists a shift $s_k^i\in \R$ and for every $i=0,\ldots,m$ there exists a $1$-periodic orbit $x_i$ for $X_H$ with the following properties:
\begin{itemize}
 \item $x_0=x$ and $x_m=y$,
 \item for every $i=1,\ldots,m-1$, $s_k^i<s_k^{i+1}$,
 \item for every $i=1,\ldots,m$, $u_k(\cdot+s_k^i)\rightarrow u_\infty^i$, in the $C^{\infty}_{\op{loc}}$-topology, for some cylinders $u^i_\infty$ satisfying the Floer Equation \eqref{fleq} for the pair $(H,J_0)$ and with asymptotic conditions
\begin{equation}
u_\infty^i(-\infty)\in S_{x_{i-1}}\,,\quad\quad u_\infty^i(+\infty)\in S_{x_i}\,.
\end{equation}
\end{itemize}
The claim follows from Proposition 4.7 in \cite{bo} with the only difference that in our case also the complex structure is allowed to vary with $k$. This however does not represent a problem, since such proposition relies on Floer's compactness theorem \cite[Proposition 3(c)]{flo}, in which the almost complex structure is allowed to vary as well.

Using inductively Corollary \ref{cor:per} and the fact that the positive end of $u_\infty^{i-1}$ coincides with the negative end of $u_\infty^i$, we see that $x_i$ is a non-constant orbit with $x_i(t)=(r_i,\gamma_i(T_it))$ and that $T_i\leq T_{i+1}$. Since $S_x\neq S_y$, there exists a non-constant $u_\infty^{i_*}$. This implies that $T_{i_*-1}<T_{i_*}$ and, therefore, that $T_y>T_x$. This contradiction proves the corollary.
\end{proof}
As happened for the action filtration, Corollary \ref{matvan2} shows that for any $T>0$, $\delta_{J_\epsilon}$ leaves invariant the subspace $SC^*_{\nu,<T} (W,\omega,j,H_\varepsilon)$ generated by the periodic orbits in $\nu$ associate with Reeb orbits with period smaller than $T$ and by the constant periodic orbits (when $\nu=0$). We call the associated cohomology group $SH^*_{\nu,<T}(W,\omega,j,H_\varepsilon,J_\varepsilon)$.

Fix $T>0$ and take $H^+$ and $H^-$ in $\hat{\mathcal H}_0'$, such that $T_{H^+}\leq T_{H^-}$ and $H^-=H^+$ on some $W^c$, where $T=(h^+)'(e^c)=(h^-)'(e^c)$. In particular, $T< \min\{T_{H^+},T_{H^-}\}$. Take perturbations $H^+_\varepsilon$ and $H^-_\varepsilon$ and let $(H^s_\varepsilon,J^s_\varepsilon)$ be a homotopy such that $\partial_s H^s_\varepsilon\leq 0$ and which is constant on $W^c$. We readily see that
\begin{equation*}
SC^*_{\nu,<T} (W,\omega,j,H^+_\varepsilon)\ =\ SC^*_{\nu,<T} (W,\omega,j,H^-_\varepsilon)\,. 
\end{equation*}
Consider continuation cylinders between the generators of these two subcomplexes. The maximum principle shows that they are contained inside $W^c$. Since in this region $\partial_s H^s_\varepsilon=0$, the transversality of the the moduli spaces of such cylinders shows that we only have constant cylinders. We have proved the following corollary.
\begin{cor}\label{corfil2per}
Under the hypotheses of the paragraph above, the filtered continuation map $SC^{*}_{\nu,<T}(W,\omega,j,H^+_\varepsilon)\rightarrow SH^{*}_{\nu,<T}(W,\omega,j,H^-_\varepsilon)$ is the identity.
\end{cor}

Take now $\nu=0$ and assume that $H\in\hat{\mathcal H}_0'$. If $T_0:=T(\alpha,0)$, we find that $SC^{*}_{0,<T_0}(W,\omega,j,H_\varepsilon)$ is generated by the constant periodic orbits only. As before, $SH^{*}_{0,<T_0}(W,\omega,j,H_\varepsilon,J_\varepsilon)\simeq H^{*+n}(W,\Lambda)$. If $H^+$, $H^-$ belong to $\hat{\mathcal H}_0'$ and $(H^+_\varepsilon,J^+_\varepsilon)\preceq (H^-_\varepsilon,J^-_\varepsilon)$, then Corollary \ref{corfil2per} implies that 
\begin{equation*}
SH^{*}_{0,<T_0}(W,\omega,j,H^+_\varepsilon,J^+_\varepsilon)\ =\ SH^{*}_{0,<T_0}(W,\omega,j,H^-_\varepsilon,J^-_\varepsilon)\,. 
\end{equation*}
Taking direct limits we arrive at a generalisation of Corollary \ref{cor:vit1} for general convex manifolds.
\begin{cor}\label{cor:vit2}
If $(W,\omega,j)$ is a convex symplectic manifold, then, for every $k\in\Z$, there exists a map
\begin{equation}
H^{k+n}(W,\Lambda)\longrightarrow SH^k_0(W,\omega,j)\,.  
\end{equation}
If the contact structure at infinity does not have any Reeb periodic orbit $\gamma$ contractible in $W$ such that
\begin{equation*}
\mu_{\op{CZ}}(\gamma)\in\{\dim W/2-k,\dim W/2-k+1\}\,,  
\end{equation*}
then such map is an isomorphism.
\end{cor}

\section{Invariance under isomorphism and rescaling}\label{sec:inv0}
Alexander Ritter proved in \cite{ar} that $SH$ is an invariant of convex symplectic manifolds up to isomorphism.
\begin{thm}\label{mainr}
If $F:(W_0,\omega_0,j_0)\rightarrow (W_1,\omega_1,j_1)$ is an isomorphism of convex symplectic manifolds , then
\begin{equation*}
SH^*_\nu(W_0,\omega_0,j_0)\simeq SH^*_{F(\nu)}(W_1,\omega_1,j_1). 
\end{equation*}
\end{thm}
Moreover, $SH$ is also invariant under multiplication of the symplectic form $\omega$ by a positive constant $c$.
\begin{prp}\label{res}
If $(W,\omega,j)$ is a convex symplectic manifold and $c$ is a positive number, then
\begin{equation*}
SH^*_\nu(W,\omega,j)\simeq SH^*_{\nu}(W,c\omega,j).
\end{equation*}
\end{prp}
The isomorphism is constructed from a rescaling $\rho_{c}:\Lambda\rightarrow\Lambda$ of the Novikov ring, defined by $\rho_c\left(\sum_in_it^{a_i}\right)=\sum_in_it^{ca_i}$. The rescaling yields the chain isomorphisms
\begin{align*}
SC^*_\nu(W,\omega,j,H)\ &\longrightarrow\ SC^{*}_\nu(W,c\omega,j,cH)\\
\sum_k\lambda_kx_k\ &\longmapsto\ \sum_k\rho_c(\lambda_k)x_k.
\end{align*}

\section{Invariance under deformations projectively constant on tori}\label{inv}
In this section we address the problem of the invariance of $SH$ under convex deformations $(W_s,\omega_s,j_s)$. We are going to analyse the case of compact manifolds $W_s$ contained in a bigger open manifold $W$. Thanks to Remark \ref{rmk_com}, this is also enough for studying deformations of an open convex symplectic manifold $(W,\omega_s,j_s)$.

From the case of ALE spaces, considered in \cite{ar} and from the discussion in the next section, we know that we cannot expect $SH$ to be invariant under general convex deformations. Our goal will be to prove the invariance of $SH_\nu$ under the additional assumption that the deformation is \textit{projectively constant on $\nu$-tori}. The relevance of this assumption is due to the fact that it ensures that the systems of local coefficients $\underline{\Lambda}_{\tau(\omega_s)}$ on $\mathscr L_\nu W$ are all isomorphic. Let us start by giving the relevant definitions and clarifying the statement we have just made. 

The \textit{transgression} $\tau:\Omega^2(W)\rightarrow \Omega^1(\mathscr L_\nu W)$ is defined by integration on paths $u:[0,1]\rightarrow \mathscr L_\nu W$ as 
\begin{equation*}
\tau(\sigma)(u)=\int_0^1ds\left(\int_{\T_1} \sigma_{u}(\partial_s u,\partial_t u)\,dt\right)=\int_u\sigma=\sigma(u) 
\end{equation*}
where we identify the path $u$ with the underlying cylinder, which is a map from $[0,1]\times\T_1$ to $W$. The transgression yields a map $[\tau]:H^2(W,\R)\rightarrow H^1(\mathscr L_\nu W,\R)$ between the de Rham cohomologies. We call a closed $2$-form on $W$ \textit{$\nu$-atoroidal} if it belongs to $\ker [\tau]$ or, equivalently, if there exists $\eta:\mathscr L_\nu W\rightarrow\R$ such that $\tau(\sigma)=d\eta$. One such primitive can be obtained by setting $\eta(x)=\sigma(u_x)$, where $u_x$ is any cylinder connecting a fixed reference loop $x_\nu$ to $x$.

A family of $2$-forms $\omega_s$ is called \textit{projectively constant on $\nu$-tori} if there exists a family of positive real numbers $c_s$ such that $c_0=1$ and $c_s^{-1}[\tau(\omega_s)]$ is constant in $H^1(\mathscr L_\nu W)$. This is equivalent to the existence of $\nu$-atoroidal forms $\sigma_s$, with $\sigma_0=0$, such that
\begin{equation}\label{ato}\begin{cases}
\tau(\sigma_s)=d\eta_s\\
\omega_s=c_s(\omega_0+\sigma_s).  
\end{cases}
\end{equation}

Define $\underline{\Lambda}_{\tau(\omega_s)}$ as the trivial $\Lambda$-bundle over $\mathscr L_\nu W$ together with the multiplication isomorphism $t^{-\omega_s(u)}:\Lambda_{x}\rightarrow \Lambda_{x'}$ for every cylinder $u$ connecting $x$ to $x'$. Since $\omega_s$ is projectively constant on $\nu$-tori, we use \eqref{ato} to construct the multiplication map $t^{\eta_s}:\underline{\Lambda}_{\tau(\omega_0)}\rightarrow\underline{\Lambda}_{\tau(\omega_0+\sigma_s)}$ and the rescaling map $\rho_{c_s}:\underline{\Lambda}_{\tau(c_s^{-1}\omega_s)}\rightarrow\underline{\Lambda}_{\tau(\omega_s)}$. The fact that these two maps are isomorphisms of local systems follows from the commutativity of the two diagrams below for every cylinder $u$ connecting $x$ to $x'$.
\begin{align*}\xymatrixcolsep{5pc}
\xymatrix{
 \Lambda_x \ar[r]^{t^{-\omega_0(u)}} \ar[d]_{t^{-\eta_s(x)}} & \Lambda_{x'} \ar[d]^{t^{-\eta_s(x')}} \\
 \Lambda_x \ar[r]^{t^{-(\omega_0+\sigma_s)(u)}} & \Lambda_{x'} }
& \quad\quad\xymatrixcolsep{5pc}\xymatrix{
 \Lambda_x \ar[r]^{t^{-c_s^{-1}\omega_s(u)}} \ar[d]_{\rho_{c_s}} & \Lambda_{x'} \ar[d]^{\rho_{c_s}} \\
 \Lambda_x \ar[r]^{t^{-\omega_s(u)}} & \Lambda_{x'} }
\end{align*}

Therefore, we conclude that $\underline{\Lambda}_{\tau(\omega_s)}$ and $\underline{\Lambda}_{\tau(\omega_0)}$ are isomorphic if $\omega_s$ is projectively constant on $\nu$-tori.

One can define a version of Symplectic Cohomology with twisted coefficients denoted by $SH^*_\nu(W,\omega,j;\underline{\Lambda}_\alpha)$ using any local system $\underline{\Lambda}_\alpha$ with structure ring $\Lambda$. The main result in \cite{ar} says that $SH^*_\nu(W,d\lambda+\sigma,j)\simeq SH^*_\nu(W,d\lambda,j;\underline{\Lambda}_{\tau(\sigma)})$, where $\sigma$ is a sufficiently small $2$-form with compact support. In view of this isomorphism, it is natural to study the invariance of Symplectic Cohomology for convex perturbations $(W_s,\omega_s,j_s)$ which preserve the isomorphism class of $\underline{\Lambda}_{\tau(\omega_s)}$. 

After this introductory discussion let us state the theorem we aim to prove. 

\begin{thm}\label{invthm}
Let $W$ be an open manifold and let $\omega_s$ be a family of symplectic forms on $W$. Suppose $W_s\subset W$ is a family of zero-codimensional embedded compact submanifold in $W$, which are all diffeomorphic to a model $W'$. Let $(W_s,\omega_s\big|_{W_s},j_s)$ be a convex deformation. Fix $\nu$ a free homotopy class of loops in $W'\simeq W_s$ such that the contact forms $\alpha_0$ and $\alpha_1$ are both $\nu$-non-degenerate. If $c_1(\omega_s)$ is $\nu$-atoroidal for every $s\in[0,1]$ and $s\mapsto\omega_s$ is projectively constant on $\nu$-tori, then
\begin{equation*}
SH^*_\nu(W_0,\omega_0,j_0)\simeq SH^*_\nu(W_1,\omega_1,j_1). 
\end{equation*}
\end{thm}
\begin{cor}
Let $(W,\omega_s,j_s)$ be a family of open convex symplectic manifolds. Fix $\nu$ a free homotopy class of loops in $W$ such that the contact forms $\alpha_0$ and $\alpha_1$, associated to $j_0$ and $j_1$, are both $\nu$-non-degenerate. If $c_1(\omega_s)$ is $\nu$-atoroidal for every $s\in[0,1]$ and $s\mapsto\omega_s$ is projectively constant on $\nu$-tori, then
\begin{equation*}
SH^*_\nu(W,\omega_0,j_0)\simeq SH^*_\nu(W,\omega_1,j_1). 
\end{equation*}
\end{cor}
\begin{proof}
Use Remark \ref{rmk_com} to find a family of compact manifolds $W_s\subset W$ satisfying the hypothesis of Theorem \ref{invthm}.
\end{proof}
The theorem can also be used to prove that in the compact case, Symplectic Cohomology does not depend on the choice of a convex collar.
\begin{cor}\label{cor_indcor}
Let $(W,\omega)$ be a compact symplectic manifold and suppose that $c_1(\omega)$ is $\nu$-atoroidal for some class $\nu$. If $j_0$ and $j_1$ are two $\nu$-non-degenerate convex collars of the boundary, then
\begin{equation}
SH^*_\nu(W,\omega,j_0)\simeq SH^*_\nu(W,\omega,j_1).
\end{equation}
\end{cor}
\begin{proof}
Let $\alpha_0$ and $\alpha_1$ be the two contact forms induced on the boundary. Define the family $\alpha_s:=(1-s)\alpha_0+s\alpha_1$ and observe that $d\alpha_s=i^*\omega$. Hence, $\alpha_s\wedge d\alpha_s=(1-s)\alpha_0\wedge\omega+s\alpha_1\wedge\omega$ is a positive contact form. Now apply Theorem \ref{invthm} to the deformation $(W,\omega,j_s)$, where $j_s$ is any family of convex collars inducing $\alpha_s$ on $\partial W$ and connecting $j_0$ to $j_1$. 
\end{proof}
\begin{rmk}
Thanks to the corollary we just proved, if $(W,\omega,j)$ is a convex symplectic manifold with $\nu$-non-degenerate boundary, its Symplectic Cohomology depends only on $\omega$ and not on $j$. Thus, we will use the notation
\begin{equation}
SH^*_\nu(W,\omega):=SH^*_\nu(W,\omega,j).
\end{equation}
We point out that we were not able to prove a statement analogous to Corollary \ref{cor_indcor} for open convex symplectic manifolds, since we ignore if the space of convex cylindrical ends is path connected.
\end{rmk}

We now proceed to the proof of the theorem, by making three preliminary reductions. First, thanks to Proposition \ref{res}, we can assume that $c_s=1$.

Second, it is enough to prove invariance of symplectic cohomology for all the \textit{non-degenerate} parameters belonging to a small interval around any fixed  $s_*\in[0,1]$ (a parameter $s$ is non-degenerate if the corresponding $\alpha_s$ is non-degenerate, so that $SH$ is well defined). Then, the theorem follows by observing that the compactness of $[0,1]$ implies that there is a finite number of non-degenerate $s_i\in[0,1]$ with $0=s_0<s_1<\dots<s_k=1$ such that
\begin{equation}\label{equ_sub}
SH^*_\nu(W_{s_i},\omega_{s_i},j_{s_i})\simeq SH^*_\nu(W_{s_{i+1}},\omega_{s_{i+1}},j_{s_{i+1}}).
\end{equation}
Hence, the isomorphism claimed in Theorem \ref{invthm} is obtained by taking the composition of all the intermediate isomorphisms.

Third, by the second reduction, we can make the following assumptions.
\begin{itemize}
 \item All the $W_s$ are contained in a fixed symplectic manifold $W$ with cylindrical end $j=(r,p):U\stackrel{\sim}{\longrightarrow}\R\times \Sigma$ such that the boundary of $W_s$ is contained in $U$ and transverse to $\partial_r$. Thus, there exists a family of functions $f_s:\Sigma\rightarrow\R$ such that $W_s=W^{f_s}$. To ease the notation we also set $\Gamma_s:=\Gamma_{f_s}$, $\Sigma_s:=\Sigma_{f_s}$ (see Definition \ref{dfn_gra}).
 \item The convex collars $j_s$ of $\partial W_s=\Sigma_s$ extend to a family of convex neighbourhoods $j_s=(r_s,p_s):U_s\stackrel{~}{\rightarrow}(-\varepsilon_s,\varepsilon_s)\times\Sigma$ of $\Sigma_s$ such that $r_s(\Sigma_s)=0$ and $\theta_s:=\imath_{\partial_{r_s}}\omega_s$ is a Liouville form for $\omega_s$. Observe that with this choice, we have $j_s^{-1}(r_s,p_s)=\Phi^{\partial_{r_s}}_{r_s}(\Gamma_s(p_s))$.
\end{itemize}
Thanks to these two assumptions, we can define the family of positive contact forms $\alpha_s=\Gamma_s^*\theta_s\in\Omega^1(\Sigma)$.

To prove the invariance when the variation of the parameter is small and $c_s=1$, we first bring the convex manifolds in a special form through a suitable isomorphism. In the following discussion let $r_0,r_1$ be two real numbers such that $r_0<r_1<f_s(x)$, for every $s$ in $[0,1]$ and $x$ in $\Sigma$.
\begin{prp}\label{prpinv1}
For every parameter $s_*\in[0,1]$, there exists $\delta_{s_*}$ such that for $|s-s_*|<\delta_{s_*}$, there exists a family of symplectic forms $\omega'_s$ on $W_{s_*}$  with the following properties:
\begin{enumerate}[a)]
 \item $\omega'_s=\omega_{s_*}$ on a neighbourhood of $\Sigma_{s_*}$;
 \item there exists a family of isomorphisms $F_s:(\hat{W}_{s_*},\hat{\omega}'_s,\hat{j}_{s_*})\rightarrow(\hat{W}_s,\hat{\omega}_s,\hat{j}_s)$, restricting to the identity on $W^{r_0}$.
\end{enumerate}
\end{prp}
\begin{proof}
Gray Stability Theorem applied to the family of contact forms $\alpha_s$ on $\Sigma$ yields a family of functions $u_s:\Sigma\rightarrow\R$ and of diffeomorphisms $\psi_s:\Sigma\rightarrow\Sigma$, such that 
\begin{equation}\label{gray}
\psi_s^*\alpha_s=e^{-u_s}\alpha_{s_*},\quad \quad\quad u_{s_*}=0,\ \ \psi_{s_*}=\op{Id}_{\Sigma}.
\end{equation}
Consider the associated map between the symplectisations
\begin{align*}
\Psi_s:(\R\times\Sigma,e^{r_{s_*}}\alpha_{s_*})&\ \ \longrightarrow\ \ (\R\times\Sigma,e^{r_s}\alpha_s)\\
(r_{s_*},p_{s_*})&\ \ \longmapsto\ \ (r_{s_*}+u_s(p_{s_*}),\psi_s(p_{s*})). 
\end{align*}
Equation \eqref{gray} is equivalent to the fact that $\Psi_0=\op{Id}_{\R\times\Sigma}$ and $\Psi_s$ is an exact symplectomorphism: $\Psi^*_s(e^{r_s}\alpha_s)=e^{r_{s_*}}\alpha_{s_*}$. 

As $\Sigma_s$ is the graph of the function $f_s$, the vector fields $\partial_{r_s}$ and $\partial_r$ are both positively transverse to $\Sigma_s$. Hence, we can define a vector field $Y_s$ on a neighbourhood $V_s$ of $W_{r_1}^{f_s}\subset W$ in such a way that
\begin{itemize}
 \item it is equal to $\partial_{r_s}$ on a neighbourhood $U'_s$ of $\Sigma_{s}$ compactly contained in $U_s$;
 \item it is equal to $\partial_r$ close to $\Sigma_{r_1}$;
 \item every flow line hits $\Sigma_{r_1}$ in the past.
\end{itemize}
Such vector field is obtained by considering a convex combination 
\begin{equation*}
Y_s=a_s(r,p)\partial_{r_s}+(1-a_s(r,p))\partial_r,  
\end{equation*}
where $a_s:U_s\rightarrow[0,1]$ is a function with compact support, which is $1$ close $\Sigma_s$. We are now ready to construct the map $F_s$ dividing $\hat{W}_{s_*}$ and $\hat{W}_s$ in four pieces.
\begin{enumerate}[\itshape i)]
 \item Take $\delta_{s_*}$ in such a way that, for $|s-s_*|<\delta_{s_*}$, $\Psi_s((-\frac{\varepsilon_{s_*}}{2},\frac{\varepsilon_{s_*}}{2})\times\Sigma)\subset U'_s$ and set
\begin{equation*}
F_s:=\ {\Psi_s}_|:\{(r_{s_*},p_{s_*})\, |\ r_{s_*}\geq0\}\longrightarrow\{(r_s,p_s)\,|\ r_s\geq u_s(p_s)\}. 
\end{equation*}
\item Define $A_s:\Sigma\rightarrow W$ by $A_s(p_{s_*}):=j^{-1}_s\circ\Psi_s(0,p_{s_*})$. In particular, $A_{s_*}=\Gamma_{f_{s_*}}$. Then, $\partial_{r_s}$ is transverse to $\Sigma'_s:=A_s(\Sigma)$ since it is close to $\Sigma'_{s_*}=\Sigma_{s_*}$.

For every $p_{s_*}$ in $\Sigma$, consider the first negative time $\tau_s(p_{s_*})<0$ such that $\Phi^{Y_s}_{\tau_s(p_{s_*})}(A_s(p_{s_*}))\in\Sigma_{r_1}$. Let $Q_s:\Sigma\!\rightarrow\!\Sigma$ be defined as $Q_s:=p\circ \Phi^{Y_s}_{\tau_s(p_{s_*})}(A_s(p_{s_*}))$. Furthermore, take a diffeomorphism $T_s^{p_{s_*}}:[-1,0]\rightarrow[\tau_s(p_{s_*}),0]$ which is the identity near $0$. We have diffeomorphisms 
\begin{align*}
B_s:[-1,0]\times\Sigma&\ \ \longrightarrow\ \ W_{r_1}^{\Sigma'_s}\\
(t,p_{s_*})&\ \ \longmapsto\ \ \Phi^{Y_s}_{T_s^{p_{s_*}}(t)}(A_s(p_{s_*})).  
\end{align*}
We set $F_s:=B_s\circ B_{s_*}^{-1}:W_{r_1}^{f_{s_*}}\rightarrow W_{r_1}^{\Sigma'_{s}}$. By definition
\begin{equation*}
j_s\circ B_s(t,p_{s_*})=j_s(\Phi^{Y_s}_t(A_s(p_{s_*})))=(t+u_s(p_{s_*}),\psi_s(p_{s_*}))=\Psi_s(t,p_{s_*})
\end{equation*}
for $t$ close to zero. Hence, the two definitions of $F_s$ match along $\Sigma_{s_*}$.
 \item Close to $\Sigma_{r_1}$ the map $B_s$ takes the form $B_s(t,p_{s_*})=(r_1-1-t,Q_s(p_{s_*}))$. Therefore, in that region we have that $F_s=B_s\circ B_{s_*}^{-1}(r,p)=(r,Q_s\circ Q^{-1}_{s_*}(p))$. Now, define
\begin{equation*}
F_s(r,p):=(r,Q_{b(r)s+(1-b(r))s_*}\circ Q^{-1}_{s_*}(p)):W^{r_1}_{r_0}\rightarrow W^{r_1}_{r_0}
\end{equation*}
where $b:[r_0,r_1]\rightarrow [0,1]$ is equal to $0$ near $r_0$ and equal to $1$ near $r_1$.
 \item Close to $\Sigma_{r_0}$ the previous definition yields $F_s(r,p)=(r,Q_{s_*}\circ Q^{-1}_{s_*}(p))=(r,p)$. Thus, we simply let $F_s:=\op{Id}_{W^{r_0}}:W^{r_0}\rightarrow W^{r_0}$.
\end{enumerate}
\end{proof}
By the proposition we just proved, for every non-degenerate $s$ in $B_{s_*}(\delta_{s_*})$, $SH^*_\nu(W_s,\omega_s,j_s)\simeq SH^*_\nu(W_{s_*},\omega'_s,j_{s_*})$ and $\omega'_{s_*}=\omega_{s_*}$. Notice that $[\omega_s']=[\omega_s]$ so that $\omega_s'$ is still constant on $\nu$-tori. Thus, in order to prove the invariance under a small variation of the parameter $s$ around $s_*$, it is enough to show that $SH^*_\nu(W_{s_*},\omega'_s,j_{s_*})\simeq SH^*_\nu(W_{s_*},\omega_{s_*},j_{s_*})$.

The argument is contained in the next proposition.
\begin{prp}\label{isoprop}
Let $(W,\omega_0,j)$ be a compact convex symplectic manifold and $\nu$ a free homotopy class of loops, such that $\alpha_0$ is non-degenerate on $\nu$. Suppose we have a family of $\sigma_s\in\Omega^2(W)$ for $s\in[0,1]$, supported in $W\setminus U$, such that
\begin{enumerate}[a)]
 \item $\sigma_0=0$;
 \item $\sigma_s$ is $\nu$-atoroidal;
 \item $\omega_s:=\omega_0+\sigma_s$ is a path of symplectic forms.
\end{enumerate}
Then, the following isomorphism holds
\begin{equation*}
SH^*_\nu(W,\omega_0,j)\simeq SH^*_\nu(W,\omega_1,j).
\end{equation*}
\end{prp}
The proof of this proposition requires several steps and will occupy the rest of the present section.

\subsection{Admissible pairs compatible with a symplectic deformation}\label{sub_adm}
Consider $(\hat{W},\hat{\omega}_s,j)$ the completion of the compact symplectic manifolds contained in the statement of the proposition. Observe that $\hat{W}$ does not depend on $s$, since $\op{supp}\sigma_s\cap U$ is empty. In particular, $\alpha_s=\alpha_0$ and hence, $SH_\nu$ is defined for every $s$. Define the total support of the $1$-parameter family $\{\sigma_s\}$ as
\begin{equation*}
\op{supp}\{\sigma_s\}:=\overline{\bigcup_s\op{supp}\sigma_s}. 
\end{equation*}
It is compact and contained in $W\setminus U$.

We define a subclass of admissible Hamiltonians $\mathcal H^{\{\omega_s\}}_\nu\subset\mathcal H_\nu$ on $\hat{W}$, depending on the deformation $\{\omega_s\}$.
\begin{dfn}
Fix once and for all a function $K:W\rightarrow \R$ without critical points on $\op{supp}\{\sigma_s\}$. Then, take a sufficiently small $\varepsilon_K$ in such a way that, for every $s\in[0,1]$, $X^{\omega_s}_{\varepsilon_K K}$ does not have $1$-periodic orbits intersecting $\op{supp}\{\sigma_s\}$.

We say that a function $H\in\mathcal H_\nu$ belongs to $\mathcal H^{\{\omega_s\}}_\nu$ (in the following discussion we drop the subscript $\nu$, since the free homotopy class is fixed), if $H=\varepsilon_K K$ on a neighbourhood of $\op{supp}\{\sigma_s\}$.

Define the sets $\mathcal P_s$ made of pairs $(H,J_s)$ such that
\begin{itemize}
 \item $H\in\mathcal H^{\{\omega_s\}}$;
 \item $J_s$ is an $\omega_s$-compatible convex almost complex structure which is $H$-regular. Namely, it is regular for the solutions $u(s',t)$ of the Floer equation \ref{fleq}, which in this case reads
 \begin{equation*}
\partial_{s'} u+J_s(\partial_tu-X^{\omega_s}_H)=0\,.
 \end{equation*}
\end{itemize}
\end{dfn}
We also assume that the following three properties can be achieved.
\begin{enumerate}[{\itshape i)}]
 \item The projection $\mathcal P_s\rightarrow \mathcal H^{\{\omega_s\}}$ is surjective.
 \item Denote by $\mathcal J_s$ the image of the projection $(H,J_s)\mapsto J_s$ from $\mathcal P_s$ and define $\mathcal J^{\{\omega_s\}}:=\cup_{s\in[0,1]}\mathcal J_s$. Then, $\mathcal J^{\{\omega_s\}}$ is bounded on $\op{supp}\{\sigma_s\}$. This can be rephrased by saying that there exists $B\geq1$ such that, if $J_s$ and $J_{s'}$ belong to $\mathcal J^{\{\omega_s\}}$, then 
\begin{equation}
\frac{1}{B}|\cdot|_{J_{s'}}\leq|\cdot|_{J_{s}}\leq B |\cdot|_{J_{s'}}\quad\mbox{on }\op{supp}\{\sigma_s\}.
\end{equation}
This property will be used below to get a uniform bound on the norm of the Hamiltonian vector field on $\op{supp}\{\sigma_s\}$.
 \item Given $H\in\mathcal H^{\{\omega_s\}}$, denote by $\mathcal J_H^{\{\omega_s\}}$ the set of all $J_s$ such that $(H,J_s)\in\mathcal P_s$ for some $s$. There exists a cylindrical end $U_H$ such that every $J_s\in\mathcal J_H^{\{\omega_s\}}$ is convex on $U_H$, $H$ has constant slope on $U_H$ and $J_s$ is uniformly bounded on $U_H^c$. Namely, there exists $B_H\geq1$ such that, for every $J_s$ and $J_{s'}$ in $\mathcal J_H^{\{\omega_s\}}$,
\begin{equation}\label{uniine}
\frac{1}{B_H}|\cdot|_{J_{s'}}\leq|\cdot|_{J_{s}}\leq B_H |\cdot|_{J_{s'}}\quad\mbox{on }U_H^c.
\end{equation}
This property implies that the distance induced by $J_s$ on $U_H^c$ are all Lipschitz equivalent. This will be used to get a constant $\varepsilon_H$, which is uniform in $J_s$ in Lemma \ref{lemmaPS} below.

To ease the notation, set $V_H:=U_H^c\cap\op{supp}\{\sigma_s\}^c$. It is a relatively compact neighbourhood of the periodic orbits of $X_H$ disjoint from $\op{supp}\{\sigma_s\}$.

\end{enumerate}
We readily see from the construction that each $\mathcal P_s$ is a cofinal set in the set of all admissible pairs for $(\hat{W},\hat{\omega}_s,\hat{j})$. Thus, it is enough to use continuation maps between elements of $\mathcal P_s$ in order to define $SH^*_\nu(\hat{W},\hat{\omega}_s,\hat{j})$.

\subsection{A Palais-Smale lemma}

Now we show that, given $H\in\mathcal H^{\{\omega_s\}}$, a loop $x$ in $U_H^c$ which is almost a critical point for $d\mathcal A^{\omega_s}_H$ cannot  intersect the support of the perturbation. The following lemma is taken from \cite[Exercise 1.22]{sal}. The only difference here is that we require that $\varepsilon_H$ does not depend on $J_s$.

\begin{lem}\label{lemmaPS}
There exists $\varepsilon_H$ such that, if $x\subset U_H^c$, then for every $J_s\in\mathcal J_H^{\{\omega_s\}}$,
\begin{equation}
|d_x\mathcal A_H^{\omega_s}|_{J_s}<\varepsilon_H\quad\Longrightarrow\quad x\subset V_H.
\end{equation}
\end{lem}
\begin{proof}
First, we claim that, for every $C>0$, the set of smooth loops
\begin{equation*}
L=\{x\subset U_H^c\ |\ |d_x\mathcal A_H^{\omega_s}|_{J_s}\leq C, \mbox{ for some }s\mbox{ and }J_s\in\mathcal J_H^{\{\omega_s\}}\}
\end{equation*}
is relatively compact for the $C^0$-topology. By the Arzel\`a-Ascoli Theorem, the claim follows by showing that the elements of $L$ are equicontinuous with respect to the distance induced by some $J_{s_*}$.
We compute
\begin{align*}
C\geq |d_x\mathcal A_H^{\omega_s}|_{J_s}&=\left(\int_{\T_1} |\dot{x}-X_H(x)|^2_{J_s}dt\right)^{\frac{1}{2}}\\
&\geq\left(\int_{\T_1} |\dot{x}|_{J_s}^2dt\right)^{\frac{1}{2}}-\Vert X_H\Vert_{J_s}.
\end{align*}
By \eqref{uniine}, $\Vert X_H\Vert_{J_s}\leq C'$ uniformly in $J_s$. Thus, the $L^2$-norm of $\dot{x}$ with respect to $J_s$ is uniformly bounded and, hence, $x$ is uniformly $1/2$-H\"older continuous for the distance induced by $J_s$. Since all these distances are Lipschitz equivalent thanks to \eqref{uniine} again, the equicontinuity follows.

Now, let us prove the lemma. Arguing by contradiction, we can find a sequence of loops $(x_n)$, such that 
\begin{equation}\label{lim0}
\lim_{n\rightarrow+\infty}|d_{x_n}\mathcal A_H^{\omega_{s_n}}|_{J_{s_n}}=0,\quad\quad x_n\not\subset U_H^c\cap\op{supp}\{\sigma_s\}^c=V_H.
\end{equation}
By the claim, passing to a subsequence, which we denote in the same way, $x_{n}$ converges in the $C^0$-norm to $x_\infty\not\subset V_H$. By compactness of the interval, we can also suppose that $s_{n}\rightarrow s_\infty$. We claim that $x_\infty$ is differentiable and it is a critical point of $d\mathcal A_H^{\omega_{s_\infty}}$. Indeed, $X^{\omega_{s_{n}}}_H(x_{n})\rightarrow X^{\omega_{s_\infty}}_H(x_\infty)$ in the $L^2$-topology. By \eqref{lim0}, $\dot{x}_{n}\rightarrow X^{\omega_{s_\infty}}_H(x_\infty)$ in the $L^2$-topology. Thus, $x_\infty$ is in $W^{1,2}$ and its weak derivative is $X^{\omega_{s_\infty}}_H(x_\infty)$. Since $X^{\omega_{s_\infty}}_H(x_\infty)$ is continuous, we conclude that $\dot{x}_\infty(t)$ exists at every $t$ and $\dot{x}_\infty(t)=X^{\omega_{s_\infty}}_H(x_\infty(t))$. Therefore, $x_\infty$ is a periodic orbit of $H$, contradicting the fact that $x_\infty\not\subset V_H$.
\end{proof}
 
\subsection{Subdividing the deformation in substeps}
We decompose the deformation $\omega_s$ in shorter perturbations by a suitable smooth change of parameters $s^{s_0}_{s_1}:\R\rightarrow [s_0,s_1]$, such that
\[s^{s_0}_{s_1}(S)=\begin{dcases}s_0& \mbox{if }S\leq0\\
            s_1& \mbox{if }S\geq1.
           \end{dcases}
\] 
We set $\omega^{s_0,s_1}_S:=\omega_{s^{s_0}_{s_1}(S)}$. Notice that there exists $C>0$ such that, if $J_s\in\mathcal J^{\{\omega_s\}}$, then
\begin{equation}
\left\Vert\dot{\omega}^{s_0,s_1}_S\right\Vert_{J_{s}}\leq C|s_1-s_0|.
\end{equation} 

To ease the notation, we drop the superscript and we write $\omega_S$ for $\omega^{s_0,s_1}_S$. Moreover, we denote $s_0$ by $-$ and $s_1$ by $+$.

\subsection{Two kind of homotopies and two kind of homotopies of homotopies}
Let $H\in\mathcal H^{\{\omega_s\}}$ and let $(H,J_-)\in \mathcal P_-$ and $(H,J_+)\in\mathcal P_+$. A \textit{symplectic homotopy} or \textit{s-homotopy} $S\mapsto(\omega_S,H,J_S)$ is a path connecting $(H,J_-)$ to $(H,J_+)$ and constant at infinity.

Given $S\in\R$ and $(H^+,J^+_S)\preceq(H^-,J^-_S)$ in $\mathcal P_S$, a \textit{monotone homotopy} or \textit{m-homotopy} is a path $r\mapsto(H^r,J^r_S)$, with $r\in\R$, connecting $(H^-,J^-_S)$ to $(H^+,J^+_S)$ and constant at infinity, such that
\begin{itemize}
 \item each $H^r$ has constant slope at infinity;
 \item $\partial_rT_{H^r}\leq0$;
 \item $(J^r_S)$ is a family of $\omega_S$-compatible convex almost complex structures satisfying bounds analogous to the ones described in condition \textit{ii)} and \textit{iii)} of Section \ref{sub_adm}.
\end{itemize}

Let $H^-$ and $H^+$ be in $\mathcal H^{\{\omega_s\}}$ and let $J^-_-, J^-_+, J^+_-, J^+_+$ be almost complex structures such that $(H^+,J^+_-)\preceq (H^-,J^-_-)$ in $\mathcal P_-$ and $(H^+,J^+_+)\preceq (H^-,J^-_+)$ in $\mathcal P_+$. A \textit{symplectic-monotone homotopy} or \textit{sm-homotopy} is a map $(S,r)\mapsto(\omega_S,H^r,J^r_S)$ with $(S,r)\in\R^2$ such that
\begin{itemize}
 \item restricting it to $r=\pm\infty$, we get two s-homotopies;
 \item restricting it to $S=\pm\infty$, we get two m-homotopies.
\end{itemize}

Consider a path $r\mapsto \sigma^r_S$ of $2$-forms constant at infinity and such that $\sigma^-_S=0$ and $\sigma^+_S=\sigma^+$. Suppose that all the elements of the $2$-parameters family $\omega^r_S:=\omega+\sigma^r_S$ satisfy the hypotheses of Proposition \ref{isoprop}. We can define a class $\mathcal H^{\{\omega^r_S\}}$ of admissible Hamiltonians associated to $\{\omega^r_S\}$ in the same way we defined $\mathcal H^{\{\omega_s\}}$ for a $1$-parameter deformation. As before, this class yields a set of admissible pairs $\mathcal P^r_S$, which enjoys properties similar to \textit{i), ii), iii)} outlined before for $\mathcal P^s$. Notice that when $|S|$ is sufficiently big, $\mathcal P^r_S=\mathcal P_\pm$ does not depend on $r$ (and on $S$). Given $(H,J_-)\in\mathcal P_-$ and $(H,J_+)\in\mathcal P_+$, a \textit{symplectic-symplectic homotopy} or \textit{ss-homotopy} is the $2$-parameter family $(\omega^r_S,H,J^r_S)$ such that $J^r_S=J_\pm$ for big $|S|$ and such that restricting it to $r=\pm\infty$, we get two s-homotopies.

We remark that the Palais-Smale lemma still holds for sm- and ss-homotopies.

\subsection{Moduli spaces for s-homotopies and chain maps}\label{sub:chain}
We know that we can associate chain maps called \textit{continuation maps} to monotone homotopies by considering suitable moduli spaces. We now carry out a similar construction to associate chain maps to symplectic homotopies $(\omega_S,H,J_S)$ as well. For this we will need first to have $|s_1-s_0|<\delta_H$, for some $\delta_H>0$. The general definition of the chain map will then be as the composition of intermediate maps as we did before in \eqref{equ_sub}. More precisely, we choose $\delta_H$ in such a way that
\begin{equation}\label{muh}
\mu_H:=\sup_{(S,J_s)\in\R\times\mathcal J^{\{\omega_s\}}} \left\Vert\dot{\omega}^{s_1,s_0}_S\right\Vert_{J_s}\leq C\delta_H\leq\frac{1}{\frac{\varepsilon_K}{\varepsilon_H}\Vert X_K\Vert+1},
\end{equation}
where we have defined
\begin{equation*}
\Vert X_K\Vert:= \sup_{J_s\in\mathcal J^{\{\omega_s\}}}\sup_{z\in\op{supp}\{\sigma_s\}} |X_K(z)|_{J_s}.
\end{equation*}
Observe that $\Vert X_K\Vert$ is finite thanks to condition \textit{ii)} in Section \ref{sub_adm}.

Thus, let $(\omega_S,H,J_S)$ be a symplectic homotopy satisfying \eqref{muh} and define the moduli spaces $\mathcal M(H,\{J_S\},x_-,x_+)$ of $S$-dependent Floer cylinders relative to such homotopy, connecting two $1$-periodic orbits $x_-$ and $x_+$ of $H$. Choosing the family $\{J_S\}$ in a generic way, we can suppose that $\mathcal M(H,\{J_S\},x_-,x_+)$ is a smooth manifold. If this space is non-empty, we have
\begin{equation*}
\dim \mathcal M(H,\{J_S\},x_-,x_+)\ =\ \mu_{\op{CZ}}(x_+)-\mu_{\op{CZ}}(x_-)\ =\ |x_-|-|x_+|\,.  
\end{equation*}
If $A$ is a homotopy class of cylinders relative ends, let $\mathcal M^A(H,\{J_S\},x_-,x_+)$ be the subset of the moduli space whose elements belong to $A$.

We claim that $\mathcal M^A(H,\{J_S\},x_-,x_+)$ is relatively compact in the $C^\infty_{\op{loc}}$-topology. As before, this is achieved in a standard way once we have the three ingredients below.
\begin{enumerate}[\itshape a)]
 \item Uniform $C^0$-bounds. They stem from the fact that the perturbation $\sigma_S$ is compactly supported and, hence, the maximum principle still applies.
 \item Uniform $C^1$-bounds. As before, they stem from the fact that bubbling off of holomorphic spheres generically does not occur by Lemma \ref{lem:bub}.
 \item Uniform bounds on the energy $E$ as defined in \eqref{ene}. These require a longer argument, which we describe below.
\end{enumerate}

First, observe that Lemma \ref{lemmaPS} allows to bound the amount of time an $S$-dependent Floer cylinder $u$ in $\mathcal M^A(H,\{J_S\},x_-,x_+)$ spends on $\op{supp}\{\sigma_s\}$. Define
\begin{equation*}
S(u,H):=\{S\in\R\, |\, u(S)\cap \op{supp}\{\sigma_s\}\neq\emptyset\},
\end{equation*}
where we have denoted by $u(S)$ the loop $t\mapsto u(S,t)$.
\begin{lem}
The following estimate holds
\begin{equation*}
|S(u,H)|\leq \frac{E(u)}{\varepsilon_H^2},
\end{equation*}
where $|\cdot|$ denotes the Lebesgue measure on $\R$. 
\end{lem}
\begin{proof}
Observe that $S(u,H)\times\T_1\subset\{(S,t)\ |\ |\partial_Su|_{J_{S}}\geq\varepsilon_H\}$ and then use Chebyshev's inequality.
\end{proof}

We claim that for an $S$-dependent Floer cylinder $u$, identity \eqref{eneid} must be replaced with the more general
\begin{equation}\label{enerel}
E(u)=-\mathcal A^{\omega_+}_H(u)+\int_\R{\dot{\omega}_S}(u|_{Z^S})\,dS,
\end{equation}
where $Z^S:=(-\infty,S]\times\T_1$. Indeed,
\begin{align*}
E(u)=-\int_{\R}d_{u(S)}\mathcal A^{\omega_S}_H\cdot \frac{du}{dS}(S)\,dS&=-\int_{\R}\left[\frac{d}{dS}\left(\mathcal A^{\omega_S}_H(u|_{Z^S})\right)-\left(\frac{\partial \mathcal A^{\omega_S}_H}{\partial S}\right)(u|_{Z^S})\right]dS\\
&=-\mathcal A^{\omega_+}_H(u)+\int_{\R}\left(\frac{\partial \mathcal A^{\omega_S}_H}{\partial S}\right)(u|_{Z^S})\,dS\,.
\end{align*}
The claim follows by observing that $\left(\frac{\partial \mathcal A^{\omega_S}_H}{\partial S}\right)(u|_{Z^S})=\dot{\omega}_S(u|_{Z^S})$.

The first summand on the right of \eqref{enerel} depends only on the homotopy class of $u$ relative ends, so we need to bound only the second one: 
\begin{equation*}
\int_\R\dot{\omega}_S(u|_{Z^S})\,dS=\int_0^1\dot{\omega}_S(u|_{Z^S})dS\leq\sup_{S\in[0,1]}\dot{\omega}_S\left(u|_{Z^S}\right),
\end{equation*}
where we have used the fact that $\dot{\omega}_S=0$ for $S\notin[0,1]$. We now estimate the last term, remembering that, for every $S\in[0,1]$, $\dot{\omega}_S$ vanishes outside $\op{supp}\{\sigma_s\}$:
\begin{align*}
\dot{\omega}_S\left(u|_{Z^S}\right)&\leq\int_{Z^S}\big|(\dot{\omega}_S)_{u}(\partial_{S'}u,\partial_tu)\big|\,dS'dt\\
&\leq\int_{S(u,H)\times\T_1}\big|(\dot{\omega}_S)_{u}(\partial_{S'}u,\partial_tu)\big|\,dS'dt\\
&\leq \int_{S(u,H)\times\T_1}|\dot{\omega}_S|_{S'}|\partial_{S'}u|_{S'}|\partial_tu|_{S'}\,dS'dt\\
&\leq \left(\sup_{J_s\in\mathcal J}\Vert\dot{\omega}_S\Vert_{J_s}\right)\int_{S(u,H)\times\T_1}|\partial_{S'}u|_{S'}|\partial_tu|_{S'}\,dS'dt.
\end{align*}
Using Floer's equation, the quantity under the integral sign can be bounded by
\begin{equation*}
|\partial_{S'}u|_{S'}|\partial_tu|_{S'}\leq|\partial_{S'}u|_{S'}\Big(|\partial_{S'}u|_{S'}+|X_H(u)|_{S'}\Big)\leq |\partial_{S'}u|_{S'}^2+ \varepsilon_K\Vert X_K \Vert|\partial_{S'}u|_{S'}.
\end{equation*}
Therefore,
\begin{align*}
\int_{S(u,H)\times\T_1}\hspace{-16pt}|\partial_{S'}u|_{S'}|\partial_tu|_{S'}\,dS'dt&\leq \int_{S(u,H)\times\T_1}\hspace{-16pt}|\partial_{S'}u|^2_{S'}\,dS'dt+\varepsilon_K\Vert X_K\Vert\int_{S(u,H)\times\T_1}\hspace{-16pt}|\partial_{S'}u|_{S'}\,dS'dt\\
&\leq E(u)+\varepsilon_K\Vert X_K\Vert\left(\int_{S(u,H)\times\T_1}\hspace{-16pt}|\partial_{S'}u|^2_{S'}\,dS'dt\right)^{\frac{1}{2}}\!\!\!|S(u,H)|^{\frac{1}{2}}\\
&\leq E(u)+\varepsilon_K\Vert X_K\Vert E(u)^{\frac{1}{2}}\frac{E(u)^{\frac{1}{2}}}{\varepsilon_H}\\
&=\left(1+\frac{\varepsilon_K}{\varepsilon_H}\Vert X_K\Vert\right)E(u).
\end{align*}
Putting all the estimates together, we find
\begin{equation*}
\int_\R{\dot{\omega}_S}(u|_{Z^S})\,dS\leq \mu_H\left(1+\frac{\varepsilon_K}{\varepsilon_H}\Vert X_K\Vert\right)E(u).
\end{equation*}
Plugging this inequality into \eqref{enerel}, we finally get the desired bound for the energy
\begin{equation}\label{ineqen}
E(u)\leq \frac{-\mathcal A_H^{\omega_+}(u)}{1-\mu_H\left(1+\frac{\varepsilon_K}{\varepsilon_H}\Vert X_K\Vert\right)}.
\end{equation}

After this preparation, we define $\eta$ by $d\eta=\tau(\omega_+)-\tau(\omega_-)$ and claim that
\begin{align*}
\varphi^{(\{\omega_S\},H,\{J_S\})}:SC^*_\nu(\hat{W},\hat{\omega}_-,\hat{j},H)\ &\longrightarrow\ SC^*_\nu(\hat{W},\hat{\omega}_+,\hat{j},H)\\
x\ &\longmapsto \sum_{\substack {v\in\mathcal M(H,\{J_S\},x',x)\\ |x'|-|x|=0}}\hspace{-16pt}\epsilon(v)t^{\eta(x)-\mathcal A_H^{\omega_+}(v)}x',
\end{align*}
is well-defined and a chain map. The good definition stems from the fact that, thanks to \ref{ineqen}, for every $C\in\R$, there are only finitely many connecting cylinders $v\in\mathcal M(H,\{J_S\},x',x)$ such that $\eta(x)-\mathcal A_H^{\omega_+}(v)\leq C$. Moreover, the properties {\itshape a), b), c)} above implies that the boundary of a component of $M(H,\{J_S\},y',x)$ with $|y'|-|x|=1$ is either empty or consists of two elements $v^x_{x'}\ast u^{x'}_{y'}$ and $u^x_y\ast v^y_{y'}$ where
\[
\left\{\begin{aligned}
        &u^x_y\ \in \mathcal M(H,J_-,y,x),\\
        &u^{x'}_{y'}\in \mathcal M(H,J_+,y',x'),\\
        &v^x_{x'}\in\mathcal M(H,\{J_S\},x',x),\\
        &v^y_{y'}\in\mathcal M(H,\{J_S\},y',y).
       \end{aligned}\right.
\]
The gluing construction in Floer theory shows also that every $v^x_{x'}\ast u^{x'}_{y'}$ with the properties above arises as a boundary element of a component of some $M(H,\{J_S\},y',x)$. Therefore, to see that $\varphi^{(\{\omega_S\},H,\{J_S\})}$ is a chain map, we only have to prove that
\begin{equation}\label{chain}
-\mathcal A_H^{\omega_+}(u^{x'}_{y'})+\big(\eta(x)-\mathcal A_H^{\omega_+}(v^x_{x'})\big)=\big(\eta(y)-\mathcal A_H^{\omega_+}(v^y_{y'})\big)-\mathcal A_H^{\omega_-}(u^x_y).
\end{equation}
First, since $\omega_+-\omega_-$ is $\nu$-atoroidal, we have
\begin{equation*}
\mathcal A_H^{\omega_-}(u^x_y)=\mathcal A_H^{\omega_+}(u^x_y)+(\omega_+-\omega_-)(u^x_y)=\mathcal A_H^{\omega_+}(u^x_y)+\eta(y)-\eta(x). 
\end{equation*}
Plugging this relation into \eqref{chain}, we get
\begin{equation*}
-\mathcal A_H^{\omega_+}(u^{x'}_{y'})-\mathcal A_H^{\omega_+}(v^x_{x'})=-\mathcal A_H^{\omega_+}(v^y_{y'})-\mathcal A_H^{\omega_+}(u^x_y),
\end{equation*}
which is true since $v^x_{x'}\ast u^{x'}_{y'}$ and $u^x_y\ast v^y_{y'}$ are homotopic relative ends and, as we observed before, $\mathcal A_H^{\omega_+}$ is invariant under such homotopies. Thus, we have proved that $\varphi^{(\{\omega_S\},H,\{J_S\})}$ is a chain map. We denote the induced map in cohomology by $\left[\varphi^{(\{\omega_S\},H,\{J_S\})}\right]$.

\subsection{Moduli spaces for sm-homotopies and commutative diagrams}
In the previous subsection, we showed the existence of chain maps associated to s-homotopies. Hence, an sm-homotopy $(S,r)\mapsto(\omega_S,H^r,J^r_S)$ yields a diagram
\begin{equation}
\xymatrixcolsep{5pc}\xymatrix{
 SC^*_\nu(W,\omega_-,j,H^+) \ar[r]^{\varphi^{(\{\omega_S\},H^+,\{J^+_S\})}} \ar[d]_{\varphi^{(H^+,J^+_-)}_{(H^-,J^-_-)}} & SC^*_\nu(W,\omega_+,j,H^+) \ar[d]^{\varphi^{(H^+,J^+_+)}_{(H^-,J^-_+)}} \\
 SC^*_\nu(W,\omega_-,j,H^-)  \ar[r]^{\varphi^{(\{\omega_S\},H^-,\{J^-_S\})}} & SC^*_\nu(W,\omega_+,j,H^-)  }\label{comdia}
\end{equation}
We claim that such diagram is commutative.

The proof follows an argument similar to the one we used to define $\varphi^{(\{\omega_S\},H,\{J_S\})}$. We subdivide the sm-homotopy in rectangles $[S_i,S_{i+1}]\times\R$, so that the corresponding constants $\mu^{S_i,S_{i+1}}_{\{H^r\}}$, defined in the same way as $\mu_H$, are sufficiently small. From the commutativity of the intermediate diagrams, we infer the commutativity of the original one.

\subsection{Moduli spaces for ss-homotopies and bijectivity of cohomology maps}
Since diagram \eqref{comdia} is commutative, the system of maps $\left[\varphi^{(\{\omega_S\},H,\{J_S\})}\right]$ yields a map $[\varphi]:SH^*_\nu(W,\omega_-,j)\rightarrow SH^*_\nu(W,\omega_+,j)$. The purpose of this subsection is to show that $[\varphi]$ is actually an isomorphism, hence proving Proposition \ref{isoprop}.

From the properties of the direct limit of groups, is enough to show that each $\left[\varphi^{(\{\omega_S\},H,\{J_S\})}\right]$ is an isomorphism. We claim that its inverse is the map which is obtained by inverting the s-homotopy. Namely, $\left[\varphi^{(\{\omega_S\},H,\{J_S\})}\right]^{-1}=\left[\varphi^{(\{\omega_{-S}\},H,\{J_{-S}\})}\right]$.

First, we notice that the composition of maps in cohomology intertwines with the concatenation of s-homotopies: \begin{equation*}    
\left[\varphi^{(\{\omega_{-S}\},H,\{J_{-S}\})}\right]\circ\left[\varphi^{(\{\omega_{-S}\},H,\{J_{-S}\})}\right]=\left[\varphi^{(\{\omega_S\}\ast\{\omega_{-S}\},H,\{J_S\}\ast\{J_{-S}\})}\right].
\end{equation*}
We omit the details of this fact. Thus, it is enough to prove that \begin{equation}\label{inveq}
\left[\varphi^{(\{\omega_S\}\ast\{\omega_{-S}\},H,\{J_S\}\ast\{J_{-S}\})}\right]=\op{Id}.
\end{equation}
Consider the standard ss-homotopy that connects the concatenated s-homotopy $(\{\omega_S\},H,\{J_S\})\ast (\{\omega_{-S}\},H,\{J_{-S}\})$ to the constant s-homotopy $(\omega_-,H,J_-)$ and notice that the chain map associated to the latter s-homotopy is the identity. Therefore, to show \eqref{inveq}, we need to prove that if two s-homotopies are connected by an ss-homotopy, they induce the same map in cohomology. This is done by adjusting the argument above for s- and sm-homotopies to ss-homotopies.

The proof of Proposition \ref{isoprop} and, hence of Theorem \ref{invthm}, is completed.
\chapter{Energy levels of contact type}\label{cha_exa}
In this chapter we give sufficient conditions for an energy level of a magnetic system to be of contact type and we apply the results of the previous chapter to establish lower bounds on the number of periodic orbits, when the contact structure is positive.

First, we observe that because of Proposition \ref{che}, condition \textbf{(S2)} on the first Chern class is satisfied for every $\nu$. Thus, if $DM:=\{E\leq \frac{1}{2}\}$, it is possible to define the groups $SH^*_\nu(DM,\omega_s)$, whenever $\omega'_s:=\omega_s\big|_{SM}$ is a $\nu$-non-degenerate HS of positive contact type (see Definition \ref{def_non}).

\begin{dfn}
We say that a parameter $s$ is \textit{$\nu$-non-degenerate} if $\omega_s'$ is transversally $\nu$-non-degenerate. 
\end{dfn}

We saw that the dynamics of the magnetic flow associated to $(M,g,\sigma)$ can be studied by looking at the family of HS $\omega'_s$, where $s$ is a real parameter. We now want to understand when such structures are of contact type so that we can apply techniques from Reeb dynamics and, in particular, use the abstract results proved in the previous chapter for the computation of Symplectic Cohomology. The first step is to determine when $\omega'_{s}$ is exact. The answer is given by the following proposition.
\begin{prp}\label{prp_exa}
The $2$-form $\omega'_{s}$ is exact if and only if one of the two alternatives holds
\begin{enumerate}
 \item[\upshape\bfseries (E)] either $\sigma$ is exact,
 \item[\upshape\bfseries (NE)] or $\sigma$ is not exact and $M$ is not the $2$-torus.
\end{enumerate}
\end{prp}
\begin{proof}
We are going to prove the proposition by defining suitable classes of primitives for $\omega'_{s}$.
If $\sigma$ is exact on $M$, we just consider the injection
\begin{align}\label{inj_priex}
\mathcal P^{\sigma}&\longrightarrow\ \mathcal P^{\omega'_{s}}\nonumber\\
\beta&\longmapsto\ \lambda_{s,\beta}:=\lambda-s\pi^*\beta.
\end{align}
When $\sigma$ is not exact and $M\neq\T^2$, then $\frac{2\pi e_M}{[\sigma]}\sigma$ is a curvature form for the $S^1$-bundle $\pi:SM\rightarrow M$, since $e_M$ is the Euler number of the bundle \cite{kob}. By the discussion contained in Example \ref{exa_concon} we know that $\pi^*\sigma$, and hence $\omega'_{s}$, is exact. More explicitly, since $\tau$ is also a connection form on $SM$ with curvature $\sigma_g$, we have the injection
\begin{align}\label{inj_pri}
\mathcal P^{\sigma-\frac{[\sigma]}{2\pi e_M}\sigma_g}&\longrightarrow\ \mathcal P^{\omega'_{s}}\nonumber\\
\beta&\longmapsto\ \lambda^g_{s,\beta}:=\lambda+s\left(\frac{[\sigma]}{2\pi e_M}\tau-\pi^*\beta\right)
\end{align}
(notice, that if $M\neq\T^2$ and $\sigma$ is also exact, this formula reduces to \eqref{inj_priex} above).

It only remains to show that if $\sigma$ is not exact and $M=\T^2$, then $\omega'_{s}$ is not exact. This follows from the fact that $\pi:S\T^2\rightarrow \T^2$ is a trivial bundle, as we mentioned in Example \ref{exa_concon}. More precisely, take $Z:\T^2\rightarrow S\T^2$ a section of this bundle. Then, $Z$ is a surface embedded in $S\T^2$ and
\begin{equation*}
\int_{\T^2}Z^*\omega'_{s}=\int_{\T^2}d(Z^*\lambda)-\int_{\T^2}Z^*(s\pi^*\sigma)=-s\int_{\T^2}\sigma\neq0.\qedhere
\end{equation*}
\end{proof}

From now on we restrict our discussion only to case \textbf{(E)} and \textbf{(NE)} contained in the statement of Proposition \ref{prp_exa}. We refer to them as the \textit{exact} and \textit{non-exact} case. Non-exact magnetic form on $\T^2$ will not play any role in the rest of this thesis.

\begin{dfn}
Let $(M,g,\sigma)$ be a magnetic system of type \textbf{(E)} or \textbf{(NE)} and define
\begin{align}
\bullet\ \op{Con}^+(g,\sigma):=&\left\{s\in[0,+\infty)\ \Big|\ \omega_{s}' \mbox{ is of positive contact type}\right\},\\
\bullet\ \op{Con}^-(g,\sigma):=&\left\{s\in[0,+\infty)\ \Big|\ \omega_{s}' \mbox{ is of negative contact type}\right\}.
\end{align}
If we fix the manifold $M$ and we vary the pairs $(g,\sigma)$ defined on $M$, we get the space of all magnetic systems of positive, respectively negative, contact type on $M$:
\begin{align}
\bullet\ \op{Con}^+(M):=&\bigcup_{(g,\sigma)}\left\{(g,\sigma,s)\,|\, s\in \op{Con}^+(g,\sigma)\right\}\subset \op{Mag}(M)\times[0,+\infty)\\
\bullet\ \op{Con}^-(M):=&\bigcup_{(g,\sigma)}\left\{(g,\sigma,s)\,|\, s\in \op{Con}^-(g,\sigma)\right\}\subset \op{Mag}(M)\times[0,+\infty)
\end{align}
\end{dfn}
We now prove the following proposition about the sets we have just introduced.
\begin{prp}\label{prp_def}
The sets $\op{Con}^+(M)$ and $\op{Con}^-(M)$ are open in the space $\op{Mag}(M)\times[0,+\infty)$. As a consequence, the sets $\op{Con}^+(g,\sigma)$ and $\op{Con}^-(g,\sigma)$ are open in $[0,+\infty)$.

If $s\in\op{Con}^+(g,\sigma)$ is $\nu$-nondegenerate, then $SH^*_\nu(DM,\omega_{s})$ is well defined. The isomorphism class of $SH^*_\nu$ is constant on the connected components of
\begin{enumerate}[\itshape a)]
 \item $\op{Con}^+(M)\cap\Big(\op{Mag}_e(M)\times[0,+\infty)\Big)$;
 \item $\op{Con}^+(M)$, if $M$ has positive genus;
 \item $\op{Con}^+(S^2)\setminus \Big(\op{Mag}_e(S^2)\times[0,+\infty)\cup\op{Mag}(S^2)\times\{0\}\Big)$.
\end{enumerate}
\end{prp}
\begin{proof}
Since a small perturbation of a contact form is still a contact form, the first statement readily follows. 

Then, observe that $(DM,\omega_{s})$ is convex if and only if $s\in\op{Con}^+(g,\sigma)$.  Since $c_1(\omega_{s})=0$ by Proposition \ref{che}, its Symplectic Cohomology is well defined as soon as $s$ is non-degenerate. In order to prove the statement about its isomorphism class, we apply Theorem \ref{invthm}. We only need to check that the cohomology class $[\tau(\omega_{s})]\in H^1(\mathscr L_\nu TM,\R)$ is projectively constant in each of the three cases presented above. The base $M$ is a deformation retraction of $TM$ and under this deformation the cohomology class $[\omega_{s}]\in H^2(TM,\R)$ is sent to $[s\sigma]\in H^2(M,\R)$. Thus, we only need to check the analogous statement for $[\tau(\sigma)]\in H^2(\mathcal L_\nu M,\R)$.

In case \textit{a)}, $\sigma$ is exact on $M$ and, therefore, $[\tau(\sigma)]=0$.

In case \textit{b)}, we consider any $\nu$-torus $u:\T^2\rightarrow M$. Lemma 2.2 in \cite{mer1} implies that $u^*\sigma$ is an exact form. Therefore,  $[\tau(\sigma)]=0$.

In case \textit{c)}, we only have the class of contractible loops. There are isomorphisms $\pi_1(\mathscr LS^2)\stackrel{\sim}{\rightarrow}\pi_2(S^2)\stackrel{\sim}{\rightarrow}H_2(S^2,\Z)\stackrel{\sim}{\rightarrow}\Z$. Hence, $H_1(\mathscr LS^2,Z)\simeq\Z$ and consequently $[\tau]:H^2(S^2,\R)\rightarrow H^1(\mathscr LS^2,\R)$ is an isomorphism. Consider two triples $(g,\sigma,s)$ and $(g',\sigma',s')$ such that $s$,$s'$ are positive and $\sigma$,$\sigma'$ are not exact. Then, $[\tau(\omega_{s\sigma})]$ and $[\tau(\omega_{s'\sigma'})]$ are the same up to a positive factor if and only if the same is true for $[\sigma]$ and $[\sigma']$. This happens if and only if the two magnetic systems lie in the same connected component of $\op{Mag}(S^2)\setminus \Big(\op{Mag}_e(S^2)\times[0,+\infty)\cup\op{Mag}(S^2)\times\{0\}\Big)$.
\end{proof}

In the next two subsections we are going to use the injections \eqref{inj_priex} and \eqref{inj_pri} in order to find \textit{sufficient} conditions for an energy level to be of contact type. Thanks to Remark \ref{rmk_ham} and Remark \ref{rmk_con}, we just need to study the sign of the functions  
\begin{align}\label{fun_exa}
\lambda_{s,\beta}(X^{s})_{(x,v)}&=(\lambda-s\pi^*\beta)(X+sfV)_{(x,v)}\nonumber\\
&=1-\beta_x(v)s
\end{align}
for case \textbf{(E)} and the functions
\begin{align}\label{fun_nexa}
\lambda^g_{s,\beta}(X^{s})_{(x,v)}&=\left(\lambda+\frac{s[\sigma]}{2\pi e_M}\tau-s\pi^*\beta\right)(X+sfV)_{(x,v)}\nonumber\\
&=1-\beta_x(v)s+\frac{[\sigma]f(x)}{2\pi e_M}s^2
\end{align}
for case \textbf{(NE)}.
 
Before performing this task, we single out some \textit{necessary} conditions for $\omega'_{s}$ to be of contact type. Indeed, the next proposition shows that the normalised Liouville measure is a null-homologous invariant measure for $X^{s}$. Combining the computation of its action (see \cite{pat1}) and McDuff's criterion, contained in Proposition \ref{mcdcri}), we find the mentioned necessary conditions in Corollary \ref{cor_act}. 
\begin{prp}
The normalised Liouville measure $\widetilde{\chi}:=\frac{\chi}{2\pi[\mu]}$ belongs to the set $\mathfrak M(X^{s})$. Its rotation vector is the Poincar\'e dual of $\frac{[\omega_{s}']}{2\pi[\mu]}$. Therefore $\rho(\widetilde{\chi})=0$ if and only if $\omega_{s}'$ is exact.
In case \textbf{(E)}, its action is given by
\begin{equation}
\mathcal A^{\omega'_{s}}(\widetilde{\chi})=1.
\end{equation}

In case \textbf{(NE)}, its action is given by
\begin{equation}
\mathcal A^{\omega'_{s}}(\widetilde{\chi})=1+\frac{[\sigma]^2}{2\pi e_M[\mu]}s^2.
\end{equation}
\end{prp}
\begin{proof}
If $\zeta$ is any $1$-form on $SM$, then by \eqref{con_mag}
\begin{equation}
\big(\imath_{X^{s}}\zeta\big)\widetilde{\chi}=\zeta\wedge\left(\imath_{X^{s}}\widetilde{\chi}\right)=\zeta\wedge\frac{\omega_{s}'}{2\pi[\mu]}.
\end{equation}
Integrating this equality over $SM$, we see that the $\rho(\widetilde{\chi})$ is the Poincar\'e dual of $\frac{[\omega_{s}']}{2\pi[\mu]}$.

We now proceed to compute the action of $\widetilde{\chi}$. First, define $I:SM\rightarrow SM$ the flip $I(x,v)=(x,-v)$ and observe that $I^*\chi=\chi$. We readily see that
\begin{equation*}
\int_{SM} \beta_x(v)\chi=\int_{SM}\beta_x(-v)I^*\chi=-\int_{SM}\beta_x(v)\chi.
\end{equation*}
Therefore, $\int_{SM}\beta_x(v)\chi=0$.

Consider first case \textbf{(E)} and let $\lambda_{s,\beta}$ be a primitive of $\omega'_{s}$. Then,
\begin{equation*}
\mathcal A^{\omega_{s}'}_{X^{s}}(\widetilde{\chi})=\int_{SM}\lambda_{s,\beta}(X^{s})\,\widetilde{\chi}=\int_{SM}(1-\beta_x(v)s)\,\widetilde{\chi}=1.
\end{equation*}

Consider now case \textbf{(NE)} and let $\lambda_{s,\beta}^g$ be a primitive of $\omega'_{s}$. Then,
\begin{align*}
\mathcal A^{\omega_{s}'}_{X^{s}}(\widetilde{\chi})=\int_{SM}\lambda^g_{s,\beta}(X^{s})\,\widetilde{\chi}&=\int_{SM}\left(1-\beta_x(v)s+\frac{[\sigma]f(x)}{2\pi e_M}s^2\right)\widetilde{\chi}\\
&=1+\frac{[\sigma]}{2\pi e_M}s^2\int_{SM}f(x)\,\widetilde{\chi}\\
&=1+\frac{[\sigma]}{2\pi e_M}s^2\frac{1}{{2\pi[\mu]}}2\pi\int_M f(x)\,\mu\\
&=1+\frac{[\sigma]^2}{2\pi e_M[\mu]}s^2.\qedhere
\end{align*}
\end{proof}

The corollary below relates the action of $\widetilde{\chi}$ and the contact property. For the case \textbf{(NE)} on surfaces of genus at least two, we first need the following definition \cite{mof,arkh,pat1}.
\begin{dfn}
Let $(M,g,\sigma)$ be a non-exact magnetic system on a surface of genus at least $2$. Define its \textbf{helicity} $s_h(g,\sigma)\in(0,+\infty)$ as
\begin{equation}
s_h(g,\sigma):=\sqrt{\frac{2\pi |e_M|[\mu]}{[\sigma]^2}}
\end{equation}
\end{dfn}
\begin{cor}\label{cor_act}
For exact magnetic systems or for non-exact magnetic systems on $S^2$, $\omega'_{s}$ cannot be of negative contact type. Namely, $\op{Con}^-(g,\sigma)=\emptyset$.

For non-exact magnetic systems on a surface of genus at least $2$, we have the inclusions
\begin{align}
\bullet\ \op{Con}^+(g,\sigma)&\subset(0,s_h(g,\sigma))\\
\bullet\ \op{Con}^-(g,\sigma)&\subset(s_h(g,\sigma),+\infty).
\end{align}
In particular, for $s=s_h(g,\sigma)$, $\omega_{s}'$ is not of contact type.
\end{cor}

\section{Contact property for case \textbf{(E)}}\label{sec_exa}
The results of this section are classical and well-known by the experts, except possibly the computation of the subcritical Symplectic Cohomology on $\T^2$. Anyway, we decided to include them here in order to put case \textbf{(NE)} in a wider context. 

From Corollary \ref{cor_act}, we know that we only need to check when $\lambda_{s,\beta}$ is a positive contact form. Identity \eqref{fun_exa} implies that
\begin{equation}
\lambda_{s,\beta}(X^{s})_{(x,v)}=1-\beta_x(v)s\geq1-|\beta_x|s\geq1-\Vert\beta\Vert s.
\end{equation}
Therefore, $\lambda_{s,\beta}(X^{s})$ is positive provided $s<\Vert\beta\Vert^{-1}$. Hence, there exists a $\beta\in\pri^\sigma$ such that $(DM,\omega_{s},\lambda_{s,\beta})$ is convex provided 
\begin{equation*}
s<s_0(g,\sigma):=\left(\inf_{\beta\in\pri^\sigma}\Vert\beta\Vert\right)^{-1}.
\end{equation*}
\begin{rmk}
Observe that combining \cite[Theorem 1.1]{pats} with \cite[Theorem A]{cipp}, one finds that $s_0(g,\sigma)$ is the \textit{Ma\~n\'e critical value of the Abelian cover} of the Lagrangian function $L_{\beta}(x,v):=E(x,v)-\beta_x(v)$ after the reparametrisation $s\mapsto c(s)=\frac{1}{2s^2}$. 
\end{rmk}
If we call $[0,s_1(g,\sigma))$ the connected component of $\op{Con}^+(g,\sigma)$ containing $0$. The above computation shows that $s_0(g,\sigma)<s_1(g,\sigma)$.
By Proposition \ref{prp_def}, for all $\nu$-non-degenerate $s<s_1(g,\sigma)$, $SH^*_\nu(DM,\omega_{s})\simeq SH^*_\nu(DM,\omega_0,j_0)$. Since the latter cohomology is known by the results of Viterbo \cite{vit2}, Salamon-Weber \cite{sawe} and Abbondandolo-Schwarz \cite{absc}, we get the following statement.
\begin{prp}\label{prp_exacon}
If $s<s_1(g,\sigma)$ is $\nu$-non-degenerate, then $(DM,\omega_{s})$ is convex and
\begin{equation}
SH^*_\nu(DM,\omega_{s})\simeq H_{-*}(\mathscr L_\nu M,\Z).
\end{equation}
\end{prp}
The value of $s_1(g,\sigma)$ can be exactly estimated when $M\neq\T^2$ as the following proposition due to Contreras, Macarini and G. Paternain shows (see \cite{cmp}). 
\begin{prp}
If $s\geq s_0(g,\sigma)$, there exists an invariant measure $\zeta_s$ such that
\begin{equation}
\bullet\ \pi_*\rho(\zeta_s)=0\in H_1(M,\mathbb R),\quad\quad\bullet\ \forall \beta\in\mathcal P^\sigma,\ \ \mathcal A^{\lambda_{s,\beta}}_{X^{s}}(\zeta_s)\leq 0. 
\end{equation}
Therefore, if $M\neq\T^2$, $\op{Con}^+(g,\sigma)=[0,s_0(g,\sigma))$ and $s_1(g,\sigma)=s_0(g,\sigma)$.

If $M=\T^2$, $\omega_{s}'$ is not of restricted contact type for $s\geq s_0(g,\sigma)$. However there are examples for which $s_0(g,\sigma)<s_1(g,\sigma)$.
\end{prp}
\begin{rmk}
It is an open question to determine whether on $\T^2$ $\op{Con}^+(g,\sigma)$ is always connected, namely to see whether $[0,s_1(g,\sigma))=\op{Con}^+(g,\sigma)$, or not.
\end{rmk}

In \cite{cmp}, the existence of a magnetic system which is of contact type at $s_0(g,\sigma)$ was proven using McDuff's sufficient criterion, which we stated in Proposition \ref{mcdcri}. The drawback of this method is that the criterion is not constructive, since it finally relies on an application of the Hahn-Banach Theorem. In the subsection below, we outline an explicit construction of the contact form for the kind of systems considered in \cite{cmp}. 

\subsection{Ma\~n\'e Critical values of contact type on $\T^2$}\label{ssec_critor}

The key observation is that when $M=\T^2$, $\pi^*:H^1(M,\R)\rightarrow H^1(SM,\R)$ is not surjective. To pick a class which is not in the image, just consider $Z:\T^2\rightarrow S\T^2$ a section of the $S^1$-bundle $\pi:S\T^2\rightarrow \T^2$. This yields the angular coordinate $\varphi_{Z}:T\T^2_0\rightarrow \T_{2\pi}$. Therefore, $d\varphi_{Z}\in\Omega^1(T\T^2_0)$ is a closed form which is not in the image of $\pi^*$. If $r\in\R$, we consider the family of primitives
\begin{align}\label{inj_pritor}
\mathcal P^{\sigma}&\longrightarrow\ \mathcal P^{\omega'_{s}}\nonumber\\
\beta&\longmapsto\ \lambda_{s,\beta}^{r,Z}:=\lambda-s\pi^*\beta+rd\varphi_Z,
\end{align}
which reduces to the class introduced in \eqref{inj_priex} for $r=0$. Namely, $\lambda_{s,\beta}^{0,Z}=\lambda_{s,\beta}$.

Recall the definition of $\kappa^Z\in\Omega^1(\T^2)$ from Equation \eqref{dfn_kap}.
Exploiting Formula \eqref{equ_kap}, we get on $S\T^2$
\begin{align}
\lambda_{s,\beta}^{r,Z}(X^{s})_{(x,v)}&=1-\beta_x(v)s+r(f(x)s-\kappa^Z_x(v))\nonumber\\
&\geq 1-|\beta_x|s+r(f(x)s-|\kappa^Z_x|)\,.\label{ineq}
\end{align}

Now we are going to define a distinguished class of magnetic forms $\sigma$ for which the right-hand side of \eqref{ineq} is positive for $s=s_0(g,\sigma)$.

Fix $t\mapsto \gamma(t)$ an embedded contractible closed curve on $\T^2$ parametrised by arc length. Suppose that its period is $T_\gamma$ and that its geodesic curvature $k_\gamma$ satisfies
\begin{equation}
k_\gamma(t)-|\kappa^Z_{\gamma(t)}|\geq\varepsilon,
\end{equation}
for some $\varepsilon>0$. Fix a $B\in\Gamma(\T^2)$ such that
\begin{enumerate}
 \item $\gamma$ is an integral curve for $B$;
 \item if $M_B:=\{x\in\T^2\ |\ |B_x|=\Vert B\Vert\}$, then $M_B=\op{supp}\gamma$ and $\Vert B\Vert=1$.
\end{enumerate}
Finally, let $\beta=\flat B$ and set $\sigma:=d\beta$.

We claim that $f(\gamma(t))=k_\gamma(t)$. Indeed,
\begin{align*}
f\vert B\vert^2=f\mu(B,\jmath B)=d(\flat B)(B,\jmath B)&\stackrel{(*)}{=}\,B\left(\flat B(\jmath B)\right)-\jmath B\left(\flat B(B)\right)-(\flat B)([B,\jmath B])\\
&\hspace{-3pt}\stackrel{(**)}{=}0-\jmath B(|B|^2)-g(B,\nabla_B\jmath B-\nabla_{\jmath B}B)\\
&=\;-\jmath B(|B|^2)+g(\nabla_BB,\jmath B)-\frac{1}{2}\jmath B(|B|^2)\\
&=\;|B|^3k_B-\frac{3}{2}\jmath B(|B|^2),
\end{align*}
where we used Cartan formula for the exterior derivative in $(*)$ and the symmetry of the Levi-Civita connection in $(**)$. Remember that $k_B$ is the geodesic curvature of $B$ defined in \eqref{geocur}. The claim follows noticing that on $\op{supp}\gamma$
\begin{itemize}
 \item $|B|=\Vert B\Vert=1$;
 \item $k_B=k_\gamma$ since $\gamma$ is an integral line for $B$;
 \item $\jmath B(|B|^2)=0$, because the function $|B|^2$ attains its global maximum there.
\end{itemize}
As a by-product, we obtain that $t\mapsto(\gamma(t),\dot{\gamma}(t)=B_{\gamma(t)})\in S\T^2$ is a periodic orbit for $X^1$ because $\gamma$ satisfies the magnetic equation \eqref{lorsig}.

We claim that $s_0(g,\sigma)=\Vert B\Vert=1$. If we take any $\beta'\in\pri^\sigma$,
\begin{equation}
\int_\gamma\beta'=\int_\gamma\beta=\int_0^{T_\gamma}\beta_{\gamma(t)}(\dot{\gamma}(t))dt=T_\gamma.
\end{equation}
On the other hand, 
\begin{equation}
\int_\gamma\beta'=\int_0^{T_\gamma}\beta'_{\gamma(t)}(\dot{\gamma}(t))dt\leq T_\gamma\Vert\beta'\Vert,
\end{equation}
therefore, $\Vert \beta'\Vert\geq\Vert \beta\Vert=1$. So, the infimum in the definition of $s_0(g,\sigma)$ is equal to $1$, it is a minimum and it is attained at $\beta$. Hence, the claim is proven.

We can now estimate from below \eqref{ineq}. When $s=1$,
\begin{equation}
 \lambda_{1,\beta}^{r,Z}(X^{\sigma})\geq (1-|\beta_x|)+r(f(x)-|\kappa^Z_x|).
\end{equation}
The right-hand side is the sum of two pieces:
\begin{enumerate}[\itshape i)]
\item $r(f(x)-|\kappa^Z_x|)$. When $r>0$, this quantity is bigger than $r\cdot\frac{\varepsilon}{2}$ on a neighbourhood $U$ of $\op{supp}\gamma$;
\item $1-|\beta_x|$. This quantity is strictly positive outside $\op{supp}\gamma$ and it vanishes on it. In particular, there exists $\delta>0$ such that $1-|\beta_x|\geq\delta$ on $\T^2\setminus U$. 
\end{enumerate}
This means that the right-hand side is positive on $U$ and on $\T^2\setminus U$ as soon as
\begin{equation}
0<r<\frac{\delta}{\max\{0,\sup_{x\notin U}|\kappa^Z_x|-f(x)\}}.
\end{equation}
Therefore, for $r$ in this range, we see that $(D\T^2,\omega_{\sigma},\lambda_{1,\beta}^{r,Z})$ is convex. We conclude that $1\in \op{Con}^+(g,\sigma)$ and $s_1(g,\sigma)>s_0(g,\sigma)=1$. 


\section{Contact property for case \textbf{(NE)}}\label{sec_ne}
To ease the notation, in this section we suppose that the magnetic form $\sigma$ is rescaled by a constant factor, in such a way that $[\sigma]=2\pi e_M$. This operation will only induce a corresponding rescaling of the parameter $s$ and, hence, will not affect our study. We begin with two easy examples and then we move to the general discussion subdivided in the subsections below.

\begin{exa}\label{exa_mod}
Denote by $g_0$ the metric of curvature $|K_{g_0}|=1$ on $M$ and consider the magnetic systems $(M,g_0,\mu_{g_0})$. We see that $\mu_{g_0}-\frac{[\mu_{g_0}]}{2\pi e_M}\sigma_{g_0}=0$. Hence, we can choose $\beta=0$ and the corresponding family of primitives $\lambda^{g_0}_{s,0}$.

If $M=S^2$, then \eqref{fun_nexa} yields $\lambda^{g_0}_{s,0}(X^{s})=1+s^2$. Thus, $\op{Con}^+(g_0,\mu_{g_0})=[0,+\infty)$.

If $M$ is a surface of genus at least $2$, \eqref{fun_nexa} reduces to $\lambda^{g_0}_{s,0}(X^{s})=1-s^2$. Thus, $\op{Con}^+(g_0,\mu_{g_0})=[0,1)$ and $\op{Con}^-(g_0,\mu_{g_0})=(1,+\infty)$. Notice that in this case, $s_h(g_0,\mu_{g_0})=1$ is the only value of the parameter which is not of contact type.
\end{exa}
\begin{exa}\label{exa_conv}
Consider a convex two-sphere. Namely, we endow $S^2$ with a metric $g$ of positive curvature. Take the magnetic system $(S^2,g,\sigma_g=K\mu)$. As in the example above, we can choose $\beta=0$ and the family of primitives $\lambda^{g}_{s,0}$. In this case, \eqref{fun_nexa} reduces to $\lambda^{g}_{s,0}(X^{s})=1+Ks^2>0$. Thus, $\op{Con}^+(g,\sigma_g)=[0,+\infty)$. This is one of the hints to Conjecture \ref{conj}.
\end{exa}

\subsection{High energy levels} 
Let us start by checking when $s\in\op{Con}^+(g,\sigma)$. Identity \eqref{fun_nexa} implies that
\begin{align}\label{f_nexa}
\lambda^g_{s,\beta}(X^{s})_{(x,v)}&=1-\beta_x(v)s+f(x)s^2\geq1-\Vert\beta\Vert s+\inf fs^2
\end{align}
Consider the quadratic equation $1-\Vert\beta\Vert s+\inf fs^2=0$. If $\inf f\leq0$, we have only one positive root. If $\inf f>0$ we either have two positive roots or two non-real roots. In any case, call $s_-(g,\sigma,\beta)$ the smallest positive root and set $s_-(g,\sigma,\beta)=+\infty$ if the equation does not have real roots. Observe that $1-\Vert\beta\Vert s+\inf fs^2>0$ for $s\in[0,s_-(g,\sigma,\beta))$, no matter the sign of $\inf f$. We write
\begin{equation}
s_-(g,\sigma):=\sup_{\beta\in\pri^{\sigma-\sigma_g}}s_-(g,\sigma,\beta).
\end{equation}
Notice that $s_-(g,\sigma)$ is the smallest positive root (with the same convention as before if there are no real roots) of the quadratic equation $1-m(g,\sigma)s+\inf fs^2=0$, where
\begin{equation}
m(g,\sigma):=\inf_{\beta\in\pri^{\sigma-\sigma_g}}\Vert \beta\Vert.
\end{equation}

Thus, we conclude that there exists $\overline{s}_-(g,\sigma)\geq s_-(g,\sigma)$ such that $[0,\overline{s}_-(g,\sigma))$ is the connected component of $\op{Con}^+(g,\sigma)$ containing $0$.

Proposition \ref{prp_def} implies the following corollary for surfaces of high genus. 

\begin{cor}\label{cor_negen}
If $M$ is a surface of genus at least $2$ and $s<\overline{s}_-(g,\sigma)$ is $\nu$-non-degenerate, then $(DM,\omega_{s})$ is convex and
\begin{equation}
SH^*_\nu(DM,\omega_{s})\simeq SH^*_\nu(DM,\omega_0)\simeq H_{-*}(\mathscr L_\nu M,\Z).
\end{equation}
\end{cor}

We now deal with the case of $S^2$. If we fix $s<s_-(g,\sigma)$, there exists $\beta\in\pri^{\sigma-\sigma_g}$ such that $s<s_-(g,\sigma,\beta)$. Consider the Riemannian metric $g_0$ with $K^{g_0}=1$ and let $g_r:=rg+(1-r)g_0$ be the linear interpolation between $g$ and $g_0$. Define a corresponding family of $\beta_r\in\pri^{\sigma-\sigma_{g_r}}$ such that $\beta_1=\beta$. By compactness of the interval,
\begin{equation*}
\hat{s}_-(g,\sigma,\beta):=\inf_{r\in[0,1]}s_-(g_r,\sigma,\beta_r) 
\end{equation*}
is positive.

Take $s_*\leq s$ such that $0<s_*<\hat{s}_-(g,\sigma,\beta)$ and consider the deformation
\begin{equation*}
s'\in[s_*,s],\quad s'\mapsto(DS^2,\omega_{s'\sigma},\lambda^{g}_{s',\beta}).
\end{equation*}
Since $s'\leq s$, the boundary stays convex for every $s'$. Moreover, the deformation is projectively constant because $s'>0$. Then, define a second deformation
\begin{equation*}
r\in[0,1],\quad r\mapsto(D^{g_r}S^2,\omega^{g_r}_{s_*\sigma},\lambda^{g_r}_{s_*,\beta_r}),
\end{equation*}
where we have made explicit the dependence of the disc bundle and the symplectic form on the metric using the superscript $g_r$.
Since $s_*<\hat{s}_-(g,\sigma,\beta)$, the boundary remains convex also in this case. Moreover, the cohomology class of $\omega^{g_r}_{s_*\sigma}$ does not change.

Finally, consider a last deformation with a parameter $u\in[0,1]$. Define the $1$-parameter family of $2$-forms $\sigma_u=u\sigma+(1-u)\sigma_{g_0}=\sigma_{g_0}+ud\beta_0$ and take the corresponding family
\begin{equation*}
u\mapsto (D^{g_0}S^2,\omega^{g_0}_{s_*\sigma_u},\lambda^{g_0}_{s_*,u\beta_0}).
\end{equation*}
Let $f^{g_0}_u$ be the function such that $\sigma_u=f^{g_0}_u\mu_{g_0}$. Then, $f^{g_0}_u=uf^{g_0}_1+(1-u)K^{g_0}$ and $\inf f^{g_0}_u=u\inf f_{g_0}+(1-u)K^{g_0}\geq\inf f^{g_0}_1$ (observe that $f^{g_0}_u$ and $K^{g_0}$ have the same integral over $M$). Therefore, $s_*<\hat{s}_-(g,\sigma,\beta)\leq s_-(g_0,\sigma,\beta_0)\leq s_-(g_0,\sigma_u,u\beta_0)$, for every $u\in[0,1]$. Hence, the boundary stays convex also during this last deformation and the cohomology class of $\omega^{g_0}_{s_*\sigma_u}$ does not change since $[s_*\sigma_u]$ is fixed.

Summing up, we have shown that for $s<s_-(g,\sigma)$, $(g,\sigma,s)$ and $(g_0,\sigma_{g_0},s_*)$ lie in the same connected component of $\op{Con}^+(S^2)\setminus \Big(\op{Mag}_e(S^2)\times[0,+\infty)\cup\op{Mag}(S^2)\times\{0\}\Big)$, for any $s_*>0$ (remember that in Example \ref{exa_mod}, we proved that $s_-(g_0,\sigma_{g_0})=+\infty$). We can now apply Proposition \ref{prp_def}.
\begin{cor}\label{cor_s2}
If $M=S^2$ and $0<s<\overline{s}_-(g,\sigma)$ is non-degenerate, then $(DS^2,\omega_{s})$ is convex and
\begin{equation}\label{sh_sph}
SH^*(D^gS^2,\omega^g_{s\sigma})\simeq SH^*(D^{g_0}S^2,\omega^{g_0}_{s_*\sigma_{g_0}}),
\end{equation}
where $g_0$ is the metric with constant curvature $1$ and $s_*$ is any positive number.
\end{cor}
In the next section, we are going to compute the right-hand side of \eqref{sh_sph}. In the remainder of this section we are going to discuss other two cases where we can use Identity \eqref{fun_nexa} to find energy levels of contact type. The original idea is contained in \cite[Remark 2.2]{pat1}.

\subsection{Low energy levels on $S^2$}\label{sub_s2}
Let us go back to Inequality \eqref{f_nexa} and see what happens when the polynomial $1-s\Vert\beta\Vert+s^2\inf f$ has a second positive root $s^+_+(g,\sigma,\beta)$. We observed that this happens if and only if $\inf f>0$. Because of the normalisation we made at the beginning of this section, we see that this can happen only if $\sigma$ is a \textit{symplectic} form on $S^2$. Under these hypotheses, we have that the right-hand side of \eqref{f_nexa} is positive for $s\in s^+_+(g,\sigma,\beta)$. As before we denote
\begin{equation}
s^+_+(g,\sigma)=\inf_{\beta\in\pri^{\sigma-\sigma_g}}s^+_+(g,\sigma,\beta),
\end{equation}
which can also be defined as the biggest positive root of $1-m(g,\sigma)s+\inf f s^2=0$.

We conclude that there exists $\overline{s}^+_+(g,\sigma)\leq s^+_+(g,\sigma)$ such that $(\overline{s}^+_+(g,\sigma),+\infty)$ is the unbounded connected component of $\op{Con}^+(g,\sigma)$.
\begin{cor}
Let $M=S^2$ and $\sigma$ be a symplectic form. If $s>\overline{s}^+_+(g,\sigma)$ is non-degenerate, then $(DS^2,\omega_{s})$ is convex and
\begin{equation}\label{sh_sphtwo}
SH^*(D^gS^2,\omega^g_{s\sigma})\simeq SH^*(D^{g_0}S^2,\omega^{g_0}_{s_*\sigma_{g_0}}),
\end{equation}
where $g_0$ is the metric with constant curvature $1$ and $s_*$ is any positive number.
\end{cor}
Again, we refer to the next section for the computation of the right-hand side.

\subsection{Low energy levels on a surface of high genus}\label{sub_lg}
We now want to investigate levels of negative contact type. Thanks to Corollary \ref{cor_act}, they can arise only on a surface of genus at least $2$. Let us bound $\lambda_{s,\beta}^g(X^{s})$ from above using \eqref{fun_nexa}:
\begin{equation*}
\lambda^g_{s,\beta}(X^{s})_{(x,v)}=1-\beta_x(v)s+f(x)s^2\leq1+\Vert\beta\Vert s+\sup f s^2. 
\end{equation*}
If $\sup f\geq0$ the right-hand side cannot be less than zero. If $\inf f<0$, there exists a unique positive root of the associated quadratic equation $s^-_+(g,\sigma,\beta)$ such that the right-hand side is less than zero, for every $s>s^-_+(g,\sigma,\beta)$. Denote
\begin{equation}
s^-_+(g,\sigma):=\inf_{\beta\in\pri^{\sigma-\sigma_g}}s^-_+(g,\sigma,\beta),
\end{equation}
which can also be defined as the unique positive root of $1+m(g,\sigma)s-\inf f s^2=0$.
Therefore, we conclude that there exists $\overline{s}^-_+(g,\sigma)\leq s^-_+(g,\sigma)$ such that $(\overline{s}^-_+(g,\sigma),+\infty)$ is the unbounded connected component of $\op{Con}^-(g,\sigma)$.

We cannot compute Symplectic Cohomology for compact concave symplectic manifolds (namely, for symplectic manifolds whose boundary is of negative contact type). A possible solution to this problem would be to consider only invariants of the boundary such as (embedded) contact homology. Another possibility would be to look for a compact convex symplectic manifold $(W,\omega)$ that could be used to cap off $(DM,\omega_{s})$ from outside in order to form a closed symplectic manifold $(DM\sqcup_{SM}W,\omega_{s}\sqcup_{SM}\omega)$. We plan to deal with this second approach in a future research project.

\begin{rmk}
Thanks to Corollary \ref{cor_act}, we have the chain of inequalities
\begin{equation}
0<\overline{s}_-(g,\sigma)\leq s_h(g,\sigma)\leq \overline{s}^-_+(g,\sigma). 
\end{equation}
Moreover, G. Paternain proved in \cite{pat1} that 
\begin{equation}
0<s_c(g,\sigma)\leq s_h(g,\sigma),
\end{equation}
where
\begin{equation*}
s_c(g,\sigma)=\inf_{\widetilde{\theta}\in\pri^{\widetilde{\sigma}}}\Vert\widetilde{\theta}\Vert_{\widetilde{g}} 
\end{equation*}
is the critical value of the universal cover after the reparametrisation $c(s)=\frac{1}{2s^2}$. It is an open problem to study the relation between $\overline{s}_-(g,\sigma)$ and $s_c(g,\sigma)$.
\end{rmk}

\section{Symplectic Cohomology of a round sphere}\label{secsph}
In this subsection we compute the Symplectic Cohomology of $(DS^2,\omega_{s\sigma_{g_0}})$, for $s>0$, when $g_0$ is the metric of constant curvature $1$.

First, we look for a primitive $\hat{\lambda}^{g_0}_{s,0}$ of $\omega_{s\sigma_{g_0}}$ on the whole $TS^2_0$ which extends $\lambda^{g_0}_{s,0}=\lambda+s\tau\in\Omega^1(SS^2)$. Using Identity \eqref{equifour}, we readily find $\hat{\lambda}^{g_0}_{s,0}=\lambda+s\frac{\tau}{2E}$. Integrating the Liouville vector field associated to $\hat{\lambda}^{g_0}_{s,0}$ starting from $SS^2$ yields a convex neighbourhood $j^{g_0}_{s,0}=(r_s,p_s):TS^2_0\hookrightarrow \R\times SS^2$ such that $SS^2=\{r_s=0\}$. Let us differentiate $E$ with respect to this coordinate:
\begin{equation*}
\frac{dE}{dr_s}=dE(\partial_{r_s})=\omega_{s\sigma_{g_0}}(\partial_{r_s},X^{s\sigma_{g_0}}_E)=\hat{\lambda}^{g_0}_{s,0}(X^{s\sigma_{g_0}}_E)=2E+s^2.
\end{equation*}
Dividing both sides by $2E+s^2$ and integrating between $0$ and $r_s$, we get
\begin{equation}
r_s=\log\sqrt{\frac{2E+s^2}{1+s^2}},
\end{equation}
or, using the auxiliary variable, $R_s=e^{r_s}$,
\begin{equation*}
R_s=\sqrt{\frac{2E+s^2}{1+s^2}}.
\end{equation*}
From this equation we see two things. First, that $R_s$ (or equivalently $r_s$) is also smooth at the zero section and, hence, it can be extended to a smooth function on the whole tangent bundle. Second, that the image of $j^{g_0}_{s,0}$ is $\left(\log\left(\frac{s}{\sqrt{1+s^2}}\right),+\infty\right)\times SS^2$. In particular, the flow of $\partial_{r_s}\in\Gamma(TS^2_0)$ is positively complete and, therefore, $j^{g_0}_{s,0}$ is actually a cylindrical end. Thus, $(TS^2,\omega_{s\sigma_{g_0}},\hat{\lambda}^{g_0}_{s,0})$ is the symplectic completion of $(DS^2,\omega_{s\sigma_{g_0}},\lambda^{g_0}_{s,0})$.

Let us look now at the dynamics on $SS^2$. We claim that all orbits are periodic and the prime orbits have all the same minimal period $T_s=\frac{2\pi}{\sqrt{1+s^2}}$. The claim can be proven either by finding explicitly all the curves with constant geodesic curvature (these are the boundaries of geodesic balls) or by using Gray Stability Theorem. The latter strategy yields an isotopy $F_{s'}:SS^2\rightarrow SS^2$, with $s'\in[0,s]$, such that 
\begin{equation*}
F_{s'}^*(\lambda+s'\tau)=\frac{1}{\sqrt{1+(s')^2}}\lambda. 
\end{equation*}
Since the Reeb vector field corresponding to $\lambda$ is $X$, whose orbits are all periodic and with minimal period $2\pi$, the claim follows.

We are ready to compute the Symplectic Cohomology. Take as a family of increasing Hamiltonian, linear at infinity, the functions $H^k_s:=k\sqrt{1+s^2}R_s$, with $k\in\N$. The associated Hamiltonian flow generates an $S^1$-action on $TS^2$ with period $2\pi k$. Hence, the only $1$-periodic orbits of the flow are the constant orbits, which lie on the zero section. Let us compute their Conley-Zehnder index (for the computation in the Lagrangian setting we refer to \cite[Lemma 5.4]{fmp2}).

The linearisation of the flow $d\Phi^{H^k_s}_t:T_{(x,0)}TS^2\rightarrow T_{(x,0)}TS^2$ can be described as follows. First, we choose standard coordinates on $TS^2$, close to the point $(x,0)$. Then, we compute the differential of the vector field $X_{H^k_s}$ at $(x,0)$ using these coordinates. We find the matrix
\begin{equation*}
d_{(x,0)}X_{H^k_s}=k\left( \begin{array}{cc}
0&\op{Id}\\
0&s\jmath_x\end{array} \right)
\end{equation*}
Since $d_{(x,0)}\Phi^{H^k_s}_t=\exp(t\cdot d_{x_0}X_{H^k_s})$ (here $\exp$ denote the exponential of a linear endomorphism), the eigenvalues of the linearisation are $1$ and $e^{ikst}$, both with algebraic multiplicity $2$. Observe that since $1$ is an eigenvalue, $(x,0)$ is degenerate (actually, the zero section form a Morse-Bott component of critical points since the eigenvalue in the normal directions is $e^{ikst}$, which is different from $1$ when $t$ is positive). For this reason, we take the lower semi-continuous extension of the index and find
\begin{equation*}
\mu_{\op{CZ}}^l(x,0)=2\left\lfloor\frac{ks}{2\pi}\right\rfloor+1+({-1})=2\left\lfloor\frac{ks}{2\pi}\right\rfloor,
\end{equation*}
where we used the additivity of the Conley-Zehnder index under direct products.

Consider now a time-dependent compact perturbation $H^k_{s,\delta_k}$ of $H^k_s$ such that all the $1$-periodic orbits of the new system are non-degenerate. Here $\delta_k>0$ is a perturbative parameter that we take arbitrarily small depending on $k$. We claim that the direct limit for $k\rightarrow +\infty$ of the symplectic cohomology groups with Hamiltonian $H^k_{s,\delta_k}$ is zero. Indeed, let $\gamma$ be a $1$-periodic orbit of the Hamiltonian system associated to $H^k_{s,\delta_k}$. Since $\delta_k$ is small, $\gamma$ is close to a constant solution on the zero section. By the lower semicontinuity of the index, $|[\gamma]|=2-\mu_{\op{CZ}}(\gamma)\leq2-2\left\lfloor\frac{ks}{2\pi}\right\rfloor$. Therefore, the symplectic cohomology with Hamiltonian $H^k_{s,\delta_k}$ is zero in degree bigger than $2-2\left\lfloor\frac{ks}{2\pi}\right\rfloor$. Since $2-2\left\lfloor\frac{ks}{2\pi}\right\rfloor\rightarrow-\infty$ as $k\rightarrow+\infty$, the direct limit is zero in every degree. Thus, we have proven 
the following proposition.
\begin{prp}\label{prp_vanst}
For every $s>0$, there holds $\displaystyle SH^*(DS^2,\omega_{s\sigma_{g_0}})=0$.
\end{prp}
\begin{cor}\label{cor_van}
If $(S^2,g,\sigma)$ is a non-exact magnetic system and $s$ is a non-degenerate parameter such that $s<\hat{s}_-(g,\sigma)$ or $s>\hat{s}^+_+(g,\sigma)$, then 
\begin{equation}
SH^*(DS^2,\omega_{s})= 0.
\end{equation}
\end{cor}

\section{Lower bound on the number of periodic orbits}\label{sec_lb}
The computation of the symplectic cohomology we performed in the previous sections can be used to prove the existence of periodic orbits on $SM$.
\begin{prp}\label{per_exasp}
Let $(S^2,g,\sigma)$ be an exact magnetic system. If $s<s_0(g,\sigma)$, then $X^{s}\in\Gamma(SM)$ has at least one periodic orbit. If the iterates of such orbit are transversally non-degenerate, then there is another geometrically distinct periodic orbit on $SM$.
\end{prp}
\begin{proof}
From \cite{zil2}, we know the singular homology with $\Z$-coefficients of the free loop space of $S^2$. It is zero in negative degree and, for $k\in\N$, we have  
\begin{equation}
H_k(\mathscr LS^2,\Z)=\begin{cases}
                       \Z& \mbox{if }k=0\ \mbox{or }k\mbox{ is odd},\\
                       \Z\oplus\Z/2\Z& \mbox{if }k \mbox{ is even and positive}.
                      \end{cases}
\end{equation}
Since $H_{-*}(\mathscr LS^2,\Z)\neq H^{*+2}(S^2,\Z)$, there exists at least one prime periodic orbit $\gamma$, by Corollary \ref{cor:vit1}. Call $T$ its period.

Suppose that $\gamma$ and all its iterates are transversally non-degenerate. Construct inductively a sequence of $k\mapsto H^k\in\hat{\mathcal H}_0'$ (see Definition \ref{dfn:adm2} on the completion $\widehat{DS^2}$ of the disc bundle, so that
\begin{itemize}
 \item $T_{H^k}<T_{H^{k+1}}$;
 \item $kT<T_{H^k}$;
 \item $H^k=H^{k+1}$ on some ${\widehat{DS^2}}^{c_k}$, where $kT=(h^k)'(e^{c_k})=(h^{k+1})'(e^{c_k})$.
\end{itemize}
Take small perturbations $H^k_\varepsilon \in\hat{\mathcal H}_0$ and let $J^k$ be a $H^k_\varepsilon$-admissible almost complex structure. Define $SC^*(k):=SC^*(DS^2,\omega_s,H^k_\varepsilon)$, let $\delta_k$ the Floer differentials $\delta_k:SC^*(k)\rightarrow SC^*(k)$, let $SH^*(k)$ the associated cohomology groups and let $\varphi_k:SC^*(k)\rightarrow SC^*(k)$ the continuation maps. Denote by $SH^*(\infty)$ the direct limit of these groups. For every $\tilde k=1,\ldots,k$, we have two non-constant generators of $SC^*(k)$: $\gamma^{\tilde k}_{\op{min}}$ and $\gamma^{\tilde k}_{\op{Max}}$. They have degree $2-\mu_{\op{CZ}}(\gamma^{\tilde k})$ and $2-\mu_{\op{CZ}}(\gamma^{\tilde k})-1$, respectively. By Corollary \ref{corfil2} (or Corollary \ref{corfil2per}), we know that $\varphi_k$ is the inclusion map. 

Since $SH^*(\infty)$ is non-zero in arbitrarily low degree, $\mu_{\op{CZ}}(\gamma^k)>0$ for every $k\geq1$, by the properties of the index. We distinguish three cases.
\begin{enumerate}[\itshape a)]
 \item $\mu_{\op{CZ}}(\gamma)\geq 2$. If this happens, then for every $\tilde{k}$, $\mu_{\op{CZ}}(\gamma^{\tilde{k}+1})-\mu_{\op{CZ}}(\gamma^{\tilde{k}})>1$. Hence, for $\tilde{k}\leq -2$, there can be at most one generator for $SC^{\tilde k}(k)$. This contradicts the fact that $SH^{-4}(\infty)\simeq \Z\oplus\Z/2$.
 \item $\mu_{\op{CZ}}(\gamma)=1$ and $\gamma$ is hyperbolic. If $k>2$, then $\gamma_{\op{min}}$ is the only generator in degree $1$ and there are three generators in degree $0$: $\gamma_{\op{Max}}$, $\gamma^2_{\op{min}}$ and $p$ which is the image under the map $C^2(S^2,\Z)\rightarrow SC^2(k)$ of a generating cochain in $C^{2}(S^2,\Z)$. We claim that $\delta_k$ is zero in degree $1$ and $0$. This would imply that the class of $\gamma_{\op{min}}$ is non-trivial in $SH^1(k)$. Together with the fact that $\varphi_k$ is the inclusion, we get a contradiction to the fact that $SH^1(\infty)=0$.
 
 Since there are no generators in degree $2$, it is clear that $\delta_k$ is zero in degree $1$. Since $\gamma$ is a good orbit, $\delta_k\gamma_{\op{Max}}=0$ and since $p$ is the image of a cochain under a chain map, we also have $\delta_k p=0$. Finally, we compute
$0=\delta_k^2\gamma^2_{\op{Max}}=\delta_k(2t^a\gamma^2_{\op{min}})$,
which implies that $\delta_k\gamma^2_{\op{min}}=0$ as well.
\item $\mu_{\op{CZ}}(\gamma)=1$ and $\gamma$ is elliptic. Under these hypotheses, there exists a maximum $k_0\geq1$ such that $\mu_{\op{CZ}}(\gamma^{k_0})=1$. If $k\geq k_0$, $SC^1(k)$ is generated by the elements $\gamma^{\tilde{k}}_{\op{min}}$, with $\tilde{k}\in[1,k_0]$ and $SC^0(k)$ is generated by $p$ and the elements $\gamma^{\tilde{k}}_{\op{Max}}$, with $\tilde{k}\in[1,k_0]$. Consider the action $A^{\omega_s}_{H^k_{\varepsilon}}$ of the periodic orbits. Since the action is decreasing with the period, we know that $\gamma^{k_0}_{\op{min}}$ has action smaller than $\gamma^{\tilde{k}}_{\op{Max}}$, for $\tilde{k}<k_0$. By Corollary \ref{corfil}, $<\delta_k\gamma^{\tilde k}_{\op{Max}},\gamma^{k_0}_{\op{min}}>=0$. Moreover, $\delta_k p=0$ and, since $\gamma^{k_0}$ is a good orbit, also $<\delta_k\gamma^{k_0}_{\op{Max}},\gamma^{k_0}_{\op{min}}>=0$. We conclude that $\gamma^{k_0}_{\op{min}}$ is a cochain, which is not a coboundary in $SC^1(k)$. Together with the fact that $\varphi_k$ is the inclusion, we get a contradiction to the fact that $SH^1(\infty)=0$.
\end{enumerate}
Since each case leads to a contradiction, we have proved that there exists a periodic orbit geometrically distinct from $\gamma$.  
\end{proof}

\begin{prp}
Let $(\T^2,g,\sigma)$ be an exact magnetic system. If $s<s_1(g,\sigma)$, then $X^{s}\in\Gamma(SM)$ has at least one periodic orbit in every non-trivial free homotopy class. If such periodic orbit is transversally non-degenerate, then there is another periodic orbit in the same class.
\end{prp}
\begin{proof}
If there are no periodic orbits $SH_\nu^*(D\T^2,\omega_{s})=0$. If we only have one periodic orbit and such orbit is non degenerate, then $\op{rk}SH_\nu^*(D\T^2,\omega_{s})\in\{0,1\}$. In any case $\op{rk}SH_\nu^*(D\T^2,\omega_{s})\leq 1$. On the other hand, we claim that $\mathscr L_\nu \T^2$ is homotopy equivalent to $\T^2$. If $\gamma_\nu$ is a reference loop in the class $\nu$, the two homotopy equivalences are given by
\begin{align*}
\begin{aligned}
\T^2&\longrightarrow\ \mathscr L_\nu\T^2\\
x&\longmapsto\ \gamma_{\nu,x}(t):=\gamma_\nu(t)+x;
\end{aligned}
& \quad\quad\quad
\begin{aligned}
\mathscr L_\nu\T^2&\longrightarrow\ \T^2\\
\gamma&\longmapsto\ \gamma(0).
\end{aligned}
\end{align*}
Thanks to Proposition \ref{prp_exacon}, this leads to the contradiction
\begin{equation*}
2\geq\op{rk}SH_\nu^*(D\T^2,\omega_{s})=\op{rk}H_{-*}(\mathscr L_\nu \T^2,\Z)=\op{rk}H_{-*}(\T^2,\Z)=4.\qedhere
\end{equation*}
\end{proof}
\begin{prp}
Let $(M,g,\sigma)$ be a magnetic system on a surface of genus at least two. If $\sigma$ is exact and $s<s_0(g,\sigma)$ or $\sigma$ is not exact and $s<\hat{s}_-(g,\sigma)$, then $X^{s}\in\Gamma(SM)$ has at least one periodic orbit in every non-trivial free homotopy class.
\end{prp}
\begin{proof}
If there are no orbits in the class $\nu\neq0$, $SH_\nu^*(DM,\omega_{s})=0$. However, $\mathscr L_\nu M$ is homotopy equivalent to $S^1$, if $\nu$ is non-trivial. Thanks to Proposition \ref{prp_exacon} or Corollary \ref{cor_negen}, this leads to the contradiction 
\begin{equation*}
0=SH_\nu^*(DM,\omega_{s})=H_{-*}(\mathscr L_\nu M,\Z)=H_{-*}(S^1,\Z)\neq0. \qedhere
\end{equation*}
\end{proof}
\begin{prp}\label{prp_bounds2}
Let $(S^2,g,\sigma)$ be a non-exact magnetic system. If the parameter $s$ belongs to $(0,\overline{s}_-(g,\sigma))\cup(\overline{s}^+_+(g,\sigma),+\infty)$, then $X^{s}\in\Gamma(SM)$ has at least one periodic orbit. If the iterates of such periodic orbit are transversally non-degenerate, then there is another geometrically distinct periodic orbit on $SM$.
\end{prp}
\begin{proof}
The proof is similar to the one of Proposition \ref{per_exasp} and, therefore, we only discuss the parts where we have to argue differently from there.

Since $SH^*(DS^2,\omega_{s})=0\neq H^{*+2}(S^2,\Z)$, there exists at least one prime periodic orbit $\gamma$, by Corollary \ref{cor:vit2}. Call $T$ its period.

Suppose that $\gamma$ and all its iterates are transversally non-degenerate. Construct the sequences $H^k$ and $H^k_\varepsilon$ as in Proposition \ref{per_exasp} and let $J^k_\varepsilon$ be a sequence in $\mathcal J(H^k)$ (see Definition \ref{dfn:ac}).
Define $SC^*(k)$, $\delta_k$, $SH^*(k)$, $\varphi_k$ and $SH^*(\infty)$ as before.

For every $\tilde k=1,\ldots,k$, we have two non-constant generators of $SC^*(k)$: $\gamma^{\tilde k}_{\op{min}}$ and $\gamma^{\tilde k}_{\op{Max}}$. They have degree $2-\mu_{\op{CZ}}(\gamma^{\tilde k})$ and $2-\mu_{\op{CZ}}(\gamma^{\tilde k})-1$, respectively. By Corollary \ref{corfil2per}, we know that $\varphi_k$ is the inclusion map. 

If $\mu_{\op{CZ}}(\gamma^k)\leq0$ for every $k\geq1$, then the degrees of the non-constant periodic orbits are all positive. This means that $SH^{-2}(\infty)\simeq H^0(S^2,\Z)$, which is a contradiction. Therefore, we assume that $\mu_{\op{CZ}}(\gamma^k)>0$, for every $k\geq1$. We distinguish four cases.
\begin{enumerate}[\itshape a)]
 \item $\mu_{\op{CZ}}(\gamma)\geq3$. We claim that there exists $k_0$ such that $\gamma^{k_0}$ is good and $\mu_{\op{CZ}}(\gamma^{k_0+1})-\mu_{\op{CZ}}(\gamma^{k_0})>2$. This implies that $\gamma^{k_0}_{\op{Max}}$ represents a non-zero class on $SC^*(k)$, for $k\geq k_0$. cocycle but cannot be a coboundary. Taking the direct limit, we conclude that $SH^*(\infty)\neq0$, contradicting our assumption. We only need to prove the claim. 

We use the iteration formula for the index (see \cite{lon}):
\begin{equation*}
\mu_{\op{CZ}}(\gamma^k)=\begin{cases}
                         2\lfloor k\vartheta\rfloor+1,\ \vartheta\in(0,+\infty)\setminus\Q,&\ \mbox{if $\gamma$ is elliptic};\\
                         k\mu_{\op{CZ}}(\gamma),&\ \mbox{if $\gamma$ is hyperbolic}.
                        \end{cases}
\end{equation*}
If $\gamma$ is hyperbolic, then $\mu_{\op{CZ}}(\gamma^{k+1})-\mu_{\op{CZ}}(\gamma^k)=\mu_{\op{CZ}}(\gamma)>2$ and the claim is satisfied by any good iterate of $\gamma$. If $\gamma$ is elliptic, then $\vartheta>1$ and 
\begin{equation*}
\mu_{\op{CZ}}(\gamma^{k+1})-\mu_{\op{CZ}}(\gamma^k)=2+2\left(\lfloor (k+1)(\vartheta-1)\rfloor-\lfloor k(\vartheta-1)\rfloor\right).
\end{equation*}
Therefore, there exists $k_0$ such that $\lfloor (k_0+1)(\vartheta-1)\rfloor-\lfloor k_0(\vartheta-1)\rfloor\geq1$. For such $k_0$ we have $\mu_{\op{CZ}}(\gamma^{k_0+1})-\mu_{\op{CZ}}(\gamma^{k_0})>2$. Since all the iterates of an elliptic orbit are good, the claim is proven also in this case.
 \item $\mu_{\op{CZ}}(\gamma)=2$. In this case $\gamma_{\op{min}}$ has degree zero and $SC^{\tilde{k}}(k)=0$ for positive $\tilde{k}$. Therefore, $\delta_k\gamma_{\op{min}}=0$. Moreover, $SC^{-1}(k)$ is generated only $\gamma_{\op{Max}}$ and since $\gamma$ is a good orbit, we know that $<\delta_k \gamma_{\op{Max}},\gamma_{\op{min}}>=0$. This implies that $\gamma_{\op{min}}$ represents a non-zero class in $SH^0(k)$. Taking the direct limit, we conclude that $SH^0(\infty)\neq0$, contradicting our assumption.
\item $\mu_{\op{CZ}}(\gamma)=1$ and $\gamma$ is hyperbolic. This case leads to a contradiction following the same argument given in Proposition \ref{per_exasp}.
 \item $\mu_{\op{CZ}}(\gamma)=1$ and $\gamma$ is elliptic. Also this case leads to a contradiction following the same argument given in Proposition \ref{per_exasp}, but we have to use the filtration with the period (see Section \ref{sec:per}), instead of the filtration with the action.
\end{enumerate}
Since each case leads to a contradiction, we have proved that there exists a periodic orbit geometrically distinct from $\gamma$. 
\end{proof}
\begin{rmk}
For the case of $S^2$, we do not know whether the orbits we find are homotopic to a fibre or not.
\end{rmk}
\chapter{\texorpdfstring{Rotationally symmetric magnetic systems on $S^2$}{Rotationally symmetric magnetic systems on S2}}\label{cha_rev}
In this chapter we aim at studying the contact property when $(S^2_\gamma,g_\gamma)$ is a surface of revolution with profile function $\gamma$ and $\sigma=\mu_\gamma$ is the area form. We are going to estimate the interval $[s_-(g_\gamma,\mu_\gamma),s_+^+(g_\gamma,\mu_\gamma)]$ in terms of geometric properties of $\gamma$ and try to understand, at least with numerical methods, the gaps $[s_-(g_\gamma,\mu_\gamma),\overline{s}_-(g_\gamma,\mu_\gamma)]$ and $[\overline{s}_+^+(g_\gamma,\mu_\gamma),s_+^+(g_\gamma,\mu_\gamma)]$. To ease the notation we are going to use the parameter $m$ instead of $s$ in this chapter. Hence, we define
\begin{align*}
X^m_\gamma:=mX^{m\mu_\gamma}_{E_\gamma}\big|_{SS^2_\gamma}\,,\quad\quad&\quad\quad \omega_{m,\gamma}:=m\omega_{\frac{\mu_\gamma}{m}}\,,\\
m_{+,\gamma}:=\frac{1}{s_-(g_\gamma,\mu_\gamma)}\,,\quad\quad&\quad\quad m_{-,\gamma}:=\frac{1}{s_+^+(g_\gamma,\mu_\gamma)}\,,\\
\overline{m}_{+,\gamma}:=\frac{1}{\overline{s}_-(g_\gamma,\mu_\gamma)}\,,\quad\quad&\quad\quad \overline{m}_{-,\gamma}:=\frac{1}{\overline{s}_+^+(g_\gamma,\mu_\gamma)}\,,\\
\op{Con}_\gamma:=\Big\{\,m\in(0,+\infty)\ \Big|\ &\frac{1}{m}\in\op{Con}^+(g_\gamma,\mu_\gamma)\,\Big\}\cup\{0\}\,.
\end{align*}

\section{The geometry of a surface of revolution}

To construct a surface of revolution, take a function $\gamma:[0,\ell_\gamma]\rightarrow \R$ and consider the conditions:
\begin{enumerate}[1)]
 \item $\gamma(t)=0$ if $t=0$ or $t=\ell_\gamma$ and is positive otherwise,
 \item $\dot{\gamma}(0)=1$, $\dot{\gamma}(\ell_\gamma)=-1$ and $|\dot{\gamma}(t)|<1$ for $t\in(0,\ell_\gamma)$,
 \item all even derivatives of $\gamma$ vanish for $t\in\{0,\ell_\gamma\}$,
 \item the following equality is satisfied
\begin{equation}\label{norminto}
\int_0^{\ell_\gamma}\!\gamma(t)\,dt=2.
\end{equation}
\end{enumerate}

\begin{dfn}
A function $\gamma$ satisfying the first three hypotheses of the list is called a \textit{profile function}. If also the fourth one holds, we say that $\gamma$ is \textit{normalised}. 
\end{dfn}

Let $S^2_\gamma$ be the quotient of $[0,\ell_\gamma]\times \T_{2\pi}$ with respect to the equivalence relation that collapses the set $\{0\}\times\T_{2\pi}$ to a point and the set $\{\ell_\gamma\}\times\T_{2\pi}$ to another point. We call these points the \textit{south} and the \textit{north pole}. Outside the poles the smooth structure is given by the coordinates $(t,\theta)\in(0,\ell_\gamma)\times \T_{2\pi}$, which also determine a well-defined orientation on $S^2_\gamma$.

On $S^2_\gamma$ we put the Riemannian metric $g_\gamma$, defined in the $(t,\theta)$ coordinates by the formula $g_\gamma=dt^2+\gamma(t)^2d\theta^2$. This metric extends smoothly to the poles because of conditions $2)$ and $3)$ listed before. Moreover, condition $4)$ yields the normalisation $\op{vol}_{g_\gamma}(S^2_\gamma)=4\pi$.

Let us denote by $(t,\theta,v_t,v_\theta)$ the associated coordinates on the tangent bundle. If $\varphi:=\varphi_{\partial_t}:S\big(S^2_\gamma\setminus\{\mbox{south pole, north pole}\}\big)\rightarrow\T_{2\pi}$ is the angular function associated to the section $\partial_t$, we have the relations 
\begin{equation}
v_t=\cos\varphi,\quad\quad v_\theta=\frac{\sin\varphi}{\gamma(t)}.
\end{equation}
As a consequence, $(t,\varphi,\theta)$ are coordinates on $SS^2_{\gamma}$, which are compatible with the orientation $\mathfrak O_{SS^2}$ defined in Section \ref{sec_geo}. Using the coframe $(\lambda,\tau,\eta)$, we can express the coordinate frame $(\widetilde{\partial}_t,\partial_\varphi,\widetilde{\partial}_\theta)$ in terms of the frame $(X,V,H)$ and vice versa:
\begin{equation}\label{relaz2}
\left\{\begin{array}{rcl}
\widetilde{\partial}_t&=&\displaystyle\cos\varphi X-\sin\varphi H\\
\partial_\varphi&=&V\\
\widetilde{\partial}_\theta&=&\displaystyle\gamma\sin\varphi X+\gamma\cos\varphi H+\dot{\gamma}V.
\end{array}\right.
\end{equation}
\begin{equation}\label{relaz1}
\left\{\begin{array}{rcl}
X&=&\displaystyle\cos\varphi \widetilde{\partial}_t-\frac{\dot{\gamma}\sin\varphi}{\gamma}\partial_\varphi+\frac{\sin\varphi}{\gamma}\widetilde{\partial}_\theta\\
V&=&\displaystyle\partial_\varphi\\
H&=&\displaystyle-\sin\varphi \widetilde{\partial}_t-\frac{\dot{\gamma}\cos\varphi}{\gamma}\partial_\varphi+\frac{\cos\varphi}{\gamma}\widetilde{\partial}_\theta,
\end{array}\right. 
\end{equation}
We have put a tilde on $\widetilde{\partial}_t$ and $\widetilde{\partial}_\theta$ to distinguish them from the coordinate vectors $\partial_t$ and $\partial_\theta$ associated to the coordinates $(t,\theta)$ on $S^2_\gamma$.

As anticipated at the beginning of this chapter, we consider as the magnetic form on the surface $SS^2_\gamma$, the Riemannian area form $\mu_\gamma$. This is a symplectic form which satisfies the normalisation $[\mu_\gamma]=4\pi$ introduced in Section \ref{sec_ne}. In coordinates $(t,\theta)$ we have $\mu_\gamma=\gamma dt\wedge d\theta$. With this choice $f\equiv1$ and $\mu_\gamma-\sigma_{g}=(1-K)\mu_\gamma$. We recall that the Gaussian curvature for surfaces of revolution is $K=-\frac{\ddot{\gamma}}{\gamma}$.

If we define
\begin{equation*}
m_\gamma:=m(g_\gamma,\mu_\gamma)=\inf_{\beta\in\pri^{(1-K)\mu_\gamma}}\Vert\beta\Vert,
\end{equation*}
we have the formula
\begin{equation*}
m_{\pm,\gamma}=\frac{m_\gamma\pm\sqrt{m_\gamma^2-4}}{2}. 
\end{equation*}
In particular, the length of the interval $[m_{-,\gamma},m_{^+,\gamma}]$ is equal to $\sqrt{m_\gamma^2-4}$, a monotone increasing function of $m_\gamma$. This observation readily gives a sufficient condition for having the contact property at every energy level.
\begin{cor}\label{cor_allcon}
If $m_\gamma<2$, then $\op{Con}_\gamma=[0,+\infty)$.
\end{cor}
In the next section we compute $m_\gamma$, showing that the infimum in its definition is actually a minimum. Namely, there exists $\beta^\gamma\in\mathcal P^{(1-K)\mu_\gamma}$ such that $m_\gamma=\Vert\beta^\gamma\Vert$.

\section{Estimating the set of energy levels of contact type}\label{sub_inv}
Consider an arbitrary closed Riemannian manifold $(M,g)$. Let $Z\in\Gamma(M)$ be a vector field that generates a $2\pi$-periodic flow of isometries on $M$. The projection operator on the space of $Z$-invariant $k$-forms $\Pi^k_{Z}:\Omega^k(M)\rightarrow \Omega^k_{Z}(M)$ is defined as
\begin{equation}\label{invdef}
\forall \tau\in\Omega^k(M),\quad\Pi^k_Z(\tau):=\frac{1}{2\pi}\int_{0}^{2\pi}(\Phi^Z_{\theta'})^*\tau\, d{\theta'}.
\end{equation}
\begin{prp}\label{inva}
The operators $\Pi^k_Z$ commute with exterior differentiation:
\begin{equation}
d\circ \Pi^k_Z=\Pi^{k+1}_Z\circ d.
\end{equation}
The projection $\Pi^1_Z$ does not increase the norm $\Vert\cdot\Vert$ defined in \eqref{norms}:
\begin{equation}
\forall \beta\in\Omega^1(M),\quad \Vert \Pi^1_Z(\beta)\Vert\leq\Vert\beta\Vert. 
\end{equation}
\end{prp}
\begin{proof}
For the proof of the first part we refer to \cite[Section: The De Rham groups of Lie groups]{boo}. For the latter statement we use that $\Phi^Z$ is a flow of isometries. We take some $(x,v)\in TM$ and compute
\begin{align*}
\Big|\Pi^1_Z(\beta)_x(v)\Big|&=\ \left|\frac{1}{2\pi}\int_{0}^{2\pi}\beta_{\Phi^Z_{\theta'}(x)}\big(d_x\Phi^Z_{\theta'}v\big) d{\theta'}\right|\\
&\leq\ \frac{1}{2\pi}\int_{0}^{2\pi}\Vert \beta\Vert\big\vert d_x\Phi^Z_{\theta'}v\big\vert\, d{\theta'}\\
&=\ \frac{1}{2\pi}\int_{0}^{2\pi}\Vert \beta\Vert\vert v\vert\, d{\theta'}\\
&=\ \Vert \beta\Vert\vert v\vert.
\end{align*}
Thus, $\vert \Pi^1_Z(\beta)_x\vert\leq\Vert \beta\Vert$. The proposition follows by taking the supremum over $x$.
\end{proof}
Let us apply this general result to $S^2_\gamma$. Consider the coordinate vector field $\partial_\theta$. This extends to a smooth vector field also at the poles and $\Phi^{\partial_\theta}$ is a $2\pi$-periodic flow of isometries on the surface. Applying Proposition \ref{inva} to this case, we get the following corollary.
\begin{cor}
The $2$-form $(1-K)\mu_\gamma$ is $\partial_\theta$-invariant. Hence, $\Pi^1_{\partial_\theta}$ sends $\mathcal P^{(1-K)\mu_\gamma}$ into itself.
\end{cor}
\begin{proof}
The first statement is true because $\Phi^{\partial_\theta}$ is a flow of isometries and thus $\mu_\gamma$ and $K$ are invariant under the flow. To prove the second statement, we observe that, if $\beta\in\mathcal P^{(1-K)\mu_\gamma}$, the previous proposition yields
\begin{equation*}
d\big(\Pi^1_{\partial_\theta}(\beta)\big)=\Pi^2_{\partial_\theta}(d\beta)=\Pi^2_{\partial_\theta}((1-K)\mu_\gamma)=(1-K)\mu_\gamma. 
\end{equation*}
Hence, $\Pi^1_{\partial_\theta}(\beta)\in\mathcal P^{(1-K)\mu_\gamma}$.
\end{proof}
Suppose now that $\beta$ lies in $\mathcal P^{(1-K)\mu_\gamma}\cap\Omega^1_{\partial_\theta}(S^2_\gamma)$. Thus, $\beta=\beta_t(t)dt+\beta_\theta(t)d\theta$. The function $\beta_\theta=\beta(\partial_\theta)$ is uniquely defined by $d\beta=(1-K)\mu_\gamma$ and the fact that it vanishes on the boundary of $[0,\ell_\gamma]$. Indeed,
\begin{equation}
d\beta=d\Big(\beta_tdt+\beta_\theta d\theta\Big)=\dot{\beta}_\theta dt\wedge d\theta=\frac{\dot{\beta}_\theta}{\gamma}\mu_\gamma
\end{equation}
Recalling the formula for $K$, we have $\dot{\beta}_\theta=\gamma+\ddot{\gamma}$. Hence, $\beta_\theta=\Gamma+\dot{\gamma}$, where $\Gamma:[0,\ell_\gamma]\rightarrow[-1,1]$ is the only primitive of $\gamma$ such that $\Gamma(0)=-1$. Notice that
\begin{enumerate}[1)]
 \item $\Gamma$ is increasing,
 \item $\Gamma(\ell_\gamma)=-1+\int_0^{\ell_\gamma}\gamma(t)\,dt=-1+2=1$,
 \item the odd derivatives of $\Gamma$ vanish at $\{0,\ell_\gamma\}$.
\end{enumerate}
Since $\beta_\theta$ and its derivatives of odd orders are zero for $t=0$ and $t=\ell_\gamma$, the $1$-form $\beta^\gamma:=\beta_\theta d\theta$ is well defined also at the poles and belongs to $\mathcal P^{(1-K)\mu_\gamma}\cap\Omega^1_{\partial_\theta}(S^2_\gamma)$. Finally, the norm of this new primitive is less than or equal to the norm of $\beta$:
\begin{equation}
\big|\beta_{(t,\theta)}\big|=\sqrt{\beta_t^2+\frac{\beta^2_\theta}{\gamma^2}}\geq \left|\frac{\beta_\theta}{\gamma}\right|=\big|\beta^\gamma_{(t,\theta)}\big|.
\end{equation}
Summing up, we have proven the following proposition.
\begin{prp}\label{estm}
There exists a unique $\beta^\gamma\in\mathcal P^{(1-K)\mu_\gamma}$ that can be written in the form $\beta^\gamma=\beta^\gamma_\theta(t)d\theta$. It satisfies $m_\gamma=\Vert \beta^\gamma\Vert$. Moreover $\beta^\gamma_\theta=\Gamma+\dot{\gamma}$ and
\begin{equation}
\Vert \beta^\gamma\Vert=\sup_{t\in[0,\ell_\gamma]}\left|\frac{\Gamma(t)+\dot{\gamma}(t)}{\gamma(t)}\right|.
\end{equation}
\end{prp}

Thus, we see that we can compute $m_\gamma$ directly from the function $\gamma$. As an application, we now produce a simple case where $m_\gamma$ can be bounded from above.
\begin{prp}\label{curvin}
Suppose $\gamma:[0,\ell_\gamma]\rightarrow\mathbb R$ is a normalised profile function such that $\gamma(t)=\gamma(\ell_\gamma-t)$. If $K$ is increasing in the variable $t$, for $t\in[0,\ell_\gamma/2]$, then $m_\gamma\leq1$. Therefore, $\op{Con}_\gamma=[0,+\infty)$.
\end{prp}
\begin{proof}
We observe that the functions $\Gamma$, $\dot{\gamma}$ and, hence, $\beta^\gamma_\theta$ are odd with respect to  $t=\ell_\gamma/2$. This means, for example, that $\beta^\gamma_\theta(t)=-\beta^\gamma_\theta(\ell_\gamma-t)$. Therefore, in order to compute $m_\gamma$, we can restrict the attention to the interval $[0,\ell_\gamma/2]$. We know that $\beta^\gamma_\theta(0)=\beta^\gamma_\theta(\ell_\gamma/2)=0$ and we claim that, if $K$ is increasing, $\beta^\gamma_\theta$ is positive in the interior. Indeed, observe that $\{t\in[0,\ell_\gamma/2]\,|\, \dot{\beta}^\gamma_\theta=0\}$ is an interval and that $\dot{\beta}^\gamma_\theta$ is positive before this interval and it is negative after. Therefore, either $\beta^\gamma_\theta$ is constantly zero or it does not have a local minimum in the interior. Thus, it cannot assume negative values.

Let us estimate $\beta^\gamma_\theta/\gamma$ at an interior absolute maximiser $t_0$. The condition $\frac{d}{dt}\big|_{t=t_0}\frac{\beta^\gamma_\theta}{\gamma}=0$ is equivalent to
\begin{equation}\label{valcrit}
\frac{\big(\Gamma(t_0)+\dot{\gamma}(t_0)\big)\dot{\gamma}(t_0)}{\gamma^2(t_0)}=1-K(t_0).
\end{equation}
Since $\Gamma(t_0)+\dot{\gamma}(t_0)=\beta^\gamma_\theta(t_0)\geq0$ and $\dot{\gamma}(t_0)>0$, we see that $1-K(t_0)\geq0$. Moreover, using that $\Gamma(t_0)<0$, we get
\begin{equation}
\left(\frac{\dot{\gamma}(t_0)}{\gamma(t_0)}\right)^2>1-K(t_0).
\end{equation}
Finally, exploiting Equation \eqref{valcrit} again, we find
\begin{equation}
m_\gamma=\frac{\beta^\gamma_\theta(t_0)}{\gamma(t_0)}=\big(1-K(t_0)\big)\frac{\gamma(t_0)}{\dot{\gamma}(t_0)}\leq \sqrt{1-K(t_0)}\leq 1.
\end{equation}
The fact that $\op{Con}_\gamma=[0,+\infty)$ now follows from Corollary \ref{cor_allcon}.
\end{proof}
To complement the previous proposition we show that if, on the contrary, we assume that the curvature of $S^2_\gamma$ is sufficiently concentrated at one of the poles, $m_\gamma$ can be arbitrarily large. We are going to prove this behaviour in the class of convex surfaces since in this case we also have a strategy, explained in Section \ref{sub_act}, to compute numerically the action of invariant measures. Using McDuff's criterion, this will enable us to estimate the gaps $[m_{-,\gamma},\overline{m}_{-,\gamma}]$ and $[\overline{m}_{+,\gamma},m_{+,\gamma}]$.

Before we need a preliminary lemma. Recall that $S^2_\gamma$ is convex, i.e.\ $K\geq0$, if and only if $\ddot{\gamma}\leq0$.
\begin{lem}\label{bigmlem}
For every $0<\delta<\frac{\pi}{2}$ and for every $\varepsilon>0$, there exists a normalised profile function $\gamma_{\delta,\varepsilon}$ such that $S^2_{\gamma_{\delta,\varepsilon}}$ is convex and
\begin{equation}\label{inelem}
\dot{\gamma}_{\delta,\varepsilon}(\delta)<\varepsilon.
\end{equation}
\end{lem}
\begin{proof}
Given $\delta$ and $\varepsilon$, we find $a\in(\frac{2\delta}{\pi},1)$ such that $0<\cos\left(\frac{\delta}{a}\right)<\varepsilon$. This is achieved by taking $\frac{\delta}{a}$ very close to $\pi/2$ from below. Consider the profile function $\gamma_{a}:[-\frac{\pi a}{2},\frac{\pi a}{2}]\rightarrow[0,a]$ of a round sphere of radius $a$, where the domain is taken to be symmetric to zero to ease the following notation. It is defined by $\gamma_a(t):=a\cos(\frac{t}{a})$. Then, $\dot{\gamma}_a(-\frac{\pi a}{2}+\delta)=\cos\left(\frac{\delta}{a}\right)<\varepsilon$ and so, up to shifting the domain again, Inequality \eqref{inelem} is satisfied. However, $\gamma_a$ is not normalised since
\begin{equation*}
\int_{-\frac{\pi a}{2}}^{\frac{\pi a}{2}}\gamma_a(t)dt=2a^2<2. 
\end{equation*}
In order to get the normalisation in such a way that Inequality \eqref{inelem} is not spoiled, we are going to stretch the sphere in the interval $\left(-(\frac{\pi a}{2}-\delta),\frac{\pi a}{2}-\delta\right)$.

We claim that, for every $C>0$ there exists a diffeomorphism $F_C:\mathbb R\rightarrow\mathbb R$ with the property that
\begin{itemize}
 \item it is odd: $\forall t\in\mathbb R,\ F_C(t)=-F_C(-t)$;
 \item for $t\geq \frac{\pi a}{2}-\delta$, $F_C(t)=t+C$ and for $t\leq -(\frac{\pi a}{2}-\delta)$, $F_C(t)=t-C$;
 \item $\dot{F}_C\geq1$;
 \item for $t\leq0$, $\ddot{F}_C(t)\geq0$ and for $t\geq0$, $\ddot{F}_C(t)\leq0$. 
\end{itemize}
Such a map can be constructed as a time $C$ flow map $\Phi^\psi_C$, where $\psi:\mathbb R\rightarrow\mathbb R$ is an odd increasing function such that, for $t\geq \frac{\pi a}{2}-\delta$, $\psi(t)=1$.

Consider the function $\gamma^C_a:[-C-\frac{\pi a}{2},C+\frac{\pi a}{2}]\rightarrow\mathbb R$, where $\gamma^C_a(s):=\gamma_a(F^{-1}_C(s))$. One readily check that $\gamma^C_a$ (up to a shift of the domain) is a profile function satisfying the convexity conditions and for which \eqref{inelem} holds. To finish the proof it is enough to find a positive real number $C_2$ such that $\int_\mathbb R\gamma^{C_2}_a=2$. Since we know that $\gamma^{0}_a=\gamma_a$ and $\int_\mathbb R\gamma_a<2$, it suffices to show that
\begin{equation*}
\lim_{C\rightarrow+\infty}\int_\mathbb R\gamma^{C}_a(s)ds=+\infty.
\end{equation*}
Observe that $b:=\gamma_a^C(-C-\frac{\pi a}{2}+\delta)=\gamma_a(-\frac{\pi a}{2}+\delta)=a\sin(\frac{\delta}{a})>0$. Then, for $s\in[-C-\frac{\pi a}{2}+\delta,C+\frac{\pi a}{2}-\delta]$, $\gamma_a^C(s)\geq b$ and we have the lower bound
\begin{equation*}
\int_\mathbb R\gamma^{C}_a(s)ds\geq \int_{-C-\frac{\pi a}{2}+\delta}^{C+\frac{\pi a}{2}-\delta}\gamma^{C}_a(s)ds\geq\int_{-C-\frac{\pi a}{2}+\delta}^{C+\frac{\pi a}{2}-\delta}\!b\, ds=2(C+\frac{\pi a}{2}-\delta)b.
\end{equation*}
The last quantity tends to infinity as $C$ tends to infinity.
\end{proof}
\begin{prp}\label{bigm}
For every $C>0$, there exists a convex surface with total area $4\pi$ such that $m_\gamma>C$.
\end{prp}
\begin{proof}
Fix an $\varepsilon_0<1$. Take any $\delta<\sqrt{1-\varepsilon_0}$ and consider the normalised profile function $\gamma_{\delta,\varepsilon_0}$ given by the lemma. We know that
\begin{equation*}
\gamma_{\delta,\varepsilon_0}(\delta)=\int_0^\delta\dot{\gamma}_{\delta,\varepsilon_0}(t)dt\leq\int_0^\delta1\, dt=\delta. 
\end{equation*}
In the same way we find $\Gamma_{\delta,\varepsilon_0}(\delta)\leq -1+\delta^2$. From this last inequality we get
\begin{equation*}
\Gamma_{\delta,\varepsilon_0}(\delta)+\dot{\gamma}_{\delta,\varepsilon_0}(\delta)<-1+\delta^2+\varepsilon_0<0.
\end{equation*}
This yields the following lower bound for $m_{\gamma_{\delta,\varepsilon_0}}$:
\begin{equation*}
m_{\gamma_{\delta,\varepsilon_0}}\geq\left|\frac{\Gamma_{\delta,\varepsilon_0}(\delta)+\dot{\gamma}_{\delta,\varepsilon_0}(\delta)}{\gamma_{\delta,\varepsilon_0}(\delta)}\right|\geq \frac{1-\varepsilon_0}{\delta}-\delta.
\end{equation*}
The proposition is proven taking $\delta$ small enough.
\end{proof}
To sum up, we saw that the rotational symmetry gives us a good understanding of the set $[0,m_{-,\gamma})\cup(m_{+,\gamma},+\infty)$. Understanding the set $[0,\overline{m}_{-,\gamma})\cup(\overline{m}_{-,\gamma},+\infty)$, or even better $\op{Con}_\gamma$, is more subtle. In Section \ref{sub_act} we perform this task only numerically and when the magnetic curvature $K_m=m^2K+1$ is positive.

As a first step, in the next section we will briefly study the symplectic reduction associated to the symmetry and the associated reduced dynamics (for the general theory of symplectic reduction we refer to \cite{abm}). Proposition \ref{paract} and the numerical computation outlined in Section \ref{sub_act} suggest that, if $K_m>0$, the contact property holds. In particular, if $K\geq0$, every energy level should be of contact type. To complete the picture, we show in Proposition \ref{nonnec} that the assumption on the magnetic curvature is not necessary and in Proposition \ref{noncon} that there are cases where the magnetic curvature is not positive and that are not of contact type.

\section{The symplectic reduction}
Observe that the flow $\Phi^{\partial_\theta}$ on $S^2_\gamma$ lifts to a flow $d\Phi^{\partial_\theta}$ on $SS^2_\gamma$. Since $\Phi^{\partial_\theta}$ is a flow of isometries, $d\Phi^{\partial_\theta}_{\theta'}$ in coordinates is simply translation in the variable $\theta$:
$d\Phi^{\partial_\theta}_{\theta'}(t,\varphi,\theta)=(t,\varphi,\theta+\theta')$. Hence, $d\Phi^{\partial_\theta}$ is generated by $\widetilde{\partial}_\theta$. As the flow $\Phi^{\widetilde{\partial}_\theta}=d\Phi^{\partial_\theta}$ is $2\pi$-periodic and acts freely on $SS^2_\gamma$, we can take its quotient $\widehat{SS^2_\gamma}$ with respect to this $\T_{2\pi}$-action. Furthermore, the quotient map $\widehat{\pi}:SS^2_\gamma\rightarrow \widehat{SS^2_\gamma}$ is a submersion. The variables $t$ and $\varphi$ descend to coordinates defined on $\widehat{SS^2_\gamma}$ minus two points, which are the fibres of the unit tangent bundle over the south and north pole. In these coordinates we simply have $\widehat{\pi}(t,\varphi,\theta)=(t,\varphi)$. In particular, $\widehat{SS^2_\gamma}$ is diffeomorphic to a $2$-sphere.

Any $\tau\in\Omega^k_{\widetilde{\partial}_\theta}(SS^2_\gamma)$ such that $\imath_{\widetilde{\partial}_\theta}\tau=0$ passes to the quotient and yields a well-defined form on $\widehat{SS^2_\gamma}$. The $2$-form $\imath_{\widetilde{\partial}_\theta}\nu_\gamma$ falls into this class and, hence, there exists $\Theta_\gamma\in\Omega^2(\widehat{SS^2_\gamma})$ such that $\imath_{\widetilde{\partial}_\theta}\nu_\gamma=\widehat{\pi}^*\Theta_\gamma$. Moreover, the form $\Theta_\gamma$ is symplectic on $\widehat{SS^2_\gamma}$. On the other hand, $X^m_\gamma$ is also $\widetilde{\partial}_\theta$-invariant thanks to Equation \eqref{relaz1}. So there exists $\widehat{X}^m_\gamma\in\Gamma(\widehat{SS^2_\gamma})$ such that $d\widehat{\pi}(X^m_\gamma)=\widehat{X}^m_\gamma$. We claim that this new vector field is $\Theta_\gamma$-Hamiltonian. We start by noticing that if $\beta^\gamma$ is as defined in Proposition \ref{estm} and $\lambda^\gamma_{m}:=m\lambda-\pi^*\beta^\gamma+\tau$, then $\lambda^\gamma_m\in\Omega^1_{\widetilde{\partial}_\theta}(SS^2_\gamma)$. Then, Cartan identity implies
\begin{align}\label{chain2}
\imath_{X^m_\gamma}(\imath_{\widetilde{\partial}_\theta}\nu_\gamma)=-\imath_{\widetilde{\partial}_\theta}(\imath_{X^m_\gamma}\nu_\gamma)=-\imath_{\widetilde{\partial}_\theta}\omega_{m,\gamma}=-\imath_{\widetilde{\partial}_\theta}(d\lambda^\gamma_m)\nonumber&=-\mathcal L_{\widetilde{\partial}_\theta}\lambda^\gamma_m+d\big(\imath_{\widetilde{\partial}_\theta}\lambda^\gamma_m\big)\nonumber\\
&=d\big(\lambda^\gamma_m(\widetilde{\partial}_\theta)\big).
\end{align}
Define $I_{m,\gamma}:=\lambda^\gamma_m(\widetilde{\partial}_\theta)$. Since $I_{m,\gamma}$ is $\widetilde{\partial}_\theta$-invariant, there exists $\widehat{I}_{m,\gamma}:\widehat{SS^2_\gamma}\rightarrow\mathbb R$ such that $I_{m,\gamma}=\widehat{\pi}^*\widehat{I}_{m,\gamma}$. 
Thus, reducing equality \eqref{chain2} to $\widehat{SS^2_\gamma}$, we find that $\widehat{X}^m_\gamma$ is the $\Theta_\gamma$-Hamiltonian vector field generated by $-\widehat{I}_{m,\gamma}$. Using Equation \eqref{relaz2}, we find the coordinate expression
\begin{equation}
\widehat{I}_{m,\gamma}(t,\varphi)=m\gamma(t)\sin\varphi-\Gamma(t).
\end{equation}
As a by-product, we also observe that $I_{m,\gamma}$ is an integral of motion for $X^m_\gamma$.

Let us consider now the two auxiliary functions $\widehat{I}^\pm_{m,\gamma}:[0,\ell_\gamma]\rightarrow \R$ defined by $\widehat{I}^\pm_{m,\gamma}(t):=\widehat{I}_{m,\gamma}(t,\pm \pi/2)=\pm m\gamma(t)-\Gamma(t)$. We know that
\begin{equation}
\widehat{I}^-_{m,\gamma}(t)\leq \widehat{I}_{m,\gamma}(t,\varphi)\leq \widehat{I}^+_{m,\gamma}(t),
\end{equation}
with equalities if and only if $\varphi=\pm\pi/2$. On the one hand, we have $\widehat{I}^+_{m,\gamma}\geq-1$ and $\widehat{I}^+_{m,\gamma}(t)=-1$ if and only if $t=\ell_\gamma$. On the other hand, $\widehat{I}^+_{m,\gamma}$ attains its maximum in the interior. Indeed,
\begin{equation}\label{derI}
\frac{d}{dt}\widehat{I}^+_{m,\gamma}=m\dot{\gamma}-\gamma.
\end{equation}
and so $\frac{d}{dt}\widehat{I}^+_{m,\gamma}(0)=m>0$. Since $\widehat{I}^+_{m,\gamma}(0)=1$, the maximum is also strictly bigger than $1$. A similar argument tells us that the maximum of $\widehat{I}^-_{m,\gamma}$ is $1$ and it is attained at $0$ and the minimum of $\widehat{I}^-_{m,\gamma}$ is strictly less than $-1$ and it is attained in $(0,\ell_\gamma)$. 
As a consequence, $\max \widehat{I}_{m,\gamma}=\max \widehat{I}^+_{m,\gamma}>1$ and $\min \widehat{I}_{m,\gamma}=\min \widehat{I}^-_{m,\gamma}<-1$.
\medskip

In the next proposition we deal with the critical points $\widehat{\op{Crit}}_{m,\gamma}$ of $\widehat{I}_{m,\gamma}$ and study their relations with the critical points $\widehat{\op{Crit}}^\pm_{m,\gamma}$ of $\widehat{I}^\pm_{m,\gamma}$. We are going to show that, if $K_m>0$, the only elements in $\widehat{\op{Crit}}_{m,\gamma}$ are the (unique) maximiser and the (unique) minimiser. In this case the dynamics of $\widehat{X}^m_\gamma$ is very simple: besides the two rest points, all the other orbits are periodic and wind once in the complement of these two points.
\begin{prp}\label{reddyn}
There holds
\begin{equation*}
\widehat{\op{Crit}}_{m,\gamma}=\widehat{\op{Crit}}^-_{m,\gamma}\times\left\{-\frac{\pi}{2}\right\}\bigcup\, \widehat{\op{Crit}}^+_{m,\gamma}\times\left\{+\frac{\pi}{2}\right\}. 
\end{equation*}

Moreover, $t_0\in \widehat{\op{Crit}}^\pm_{m,\gamma}$ if and only if $\pm m\dot{\gamma}(t_0)=\gamma(t_0)$ and $\{(t_0,\pm\pi/2,\theta)\}$ is the support of a closed orbit for $X^m_\gamma$. All the periodic orbits, whose projection to $S^2_\gamma$ is a latitude, arise in this way. Every regular level of $\widehat{I}_{m,\gamma}$ is the support of a closed orbit of $\widehat{X}^m_\gamma$ and its preimage in $SS^2_\gamma$ is an $X^m_\gamma$-invariant torus. 

Finally, if $K_m>0$, $\widehat{\op{Crit}}^\pm_{m,\gamma}$ contains only the absolute maximiser (respectively minimiser) of $\widehat{I}^\pm_{m,\gamma}$. We denote this unique element by $t^\pm_{m,\gamma}$. 
\end{prp}

\begin{proof}
The first statement follows from the fact that $\partial_\varphi \widehat{I}_{m,\gamma}=0$ if and only if $\varphi=\pm\pi/2$. Recalling \eqref{derI} we see that $\pm m\dot{\gamma}(t)=\gamma(t)$ is exactly the equation for the critical points of $\widehat{I}^\pm_{m,\gamma}$. The statement about the relation between closed orbits of $X^m_\gamma$ and latitudes follows from the fact that at a critical point of $\widehat{I}_{m,\gamma}$, $\widehat{X}^m_\gamma=0$. Hence, on its preimage $X^m_\gamma$ is parallel to $\widetilde{\partial}_\theta$. By the implicit function theorem the regular level sets of $\widehat{I}_{m,\gamma}$ and $I_{m,\gamma}$ are closed submanifolds of codimension $1$. In the latter case they are tori since $X^m_\gamma$ is tangent to them and nowhere vanishing. 

We prove now uniqueness under the hypothesis on the curvature. We carry out the computations for $\widehat{I}^+_{m,\gamma}$ only. To prove that the absolute maximiser is the only critical point, we show that if $t_0$ is critical, the function is concave at $t_0$. Indeed,
\begin{align*}
\frac{d^2}{dt^2}\widehat{I}^+_{m,\gamma}(t_0)\ =\ m\ddot{\gamma}(t_0)-\dot{\gamma}(t_0)&=\ m\gamma(t_0)\left(\frac{\ddot{\gamma}(t_0)}{\gamma(t_0)}-\frac{\dot{\gamma}(t_0)}{m\gamma(t_0)}\right)\\
&=\ -m\gamma(t_0)\left(K(t_0)+\frac{1}{m^2}\right)<0.\qedhere
\end{align*}
\end{proof}
The picture above shows qualitatively $\widehat{I}^-_{m,\gamma}$ and $\widehat{I}^+_{m,\gamma}$ when $K_m$ is positive.
\begin{figure}
\includegraphics[width=4in]{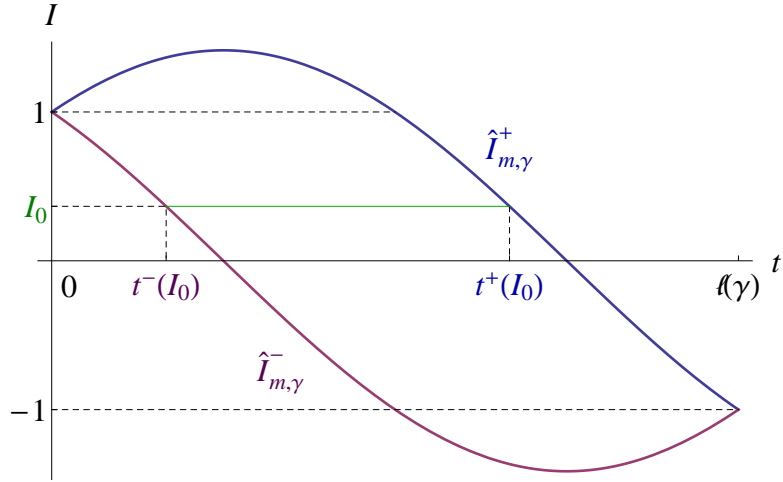}
\caption{Graphs of the functions $\widehat{I}^-_{m,\gamma}$ and $\widehat{I}^+_{m,\gamma}$}\label{pic}
\end{figure} 
\medskip

In order to decide whether $\omega_{m,\gamma}'$ is of contact type or not, the first thing to do is to compute the action of latitudes.
\begin{prp}\label{paract}
Take $t_0\in(0,\ell_\gamma)$ such that $\dot{\gamma}(t_0)\neq0$ and let $m_{t_0}:=\left|\frac{\gamma(t_0)}{\dot{\gamma}(t_0)}\right|$. The lift of the latitude curve $\{t=t_0\}$ parametrised by arc length and oriented by $\big(\op{sign}\dot{\gamma}(t_0)\big)\partial_\theta$ is the support of a periodic orbit for $X^{m_{t_0}}_\gamma$. We call $\zeta_{t_0}$ the associated invariant probability measure. Its action is given by
\begin{equation}\label{acpar}
\mathcal A^{\omega_{m_{t_0},\gamma}'}_{X^{m_{t_0}}_\gamma}(\zeta_{t_0})=\frac{\gamma(t_0)^2-\dot{\gamma}(t_0)\Gamma(t_0)}{\dot{\gamma}(t_0)^2}
\end{equation}
and $I_{m_{t_0},\gamma}\big|_{\op{supp}\zeta_{t_0}}=\dot{\gamma}(t_0)\mathcal A(\zeta_{t_0})$. As a result, if $K_{m_{t_0}}>0$, then $\mathcal A(\zeta_{t_0})>0$.
\end{prp}
\begin{proof}
The curve $u\mapsto\left(t_0,\op{sign}\dot{\gamma}(t_0)\frac{\pi}{2},\frac{m\op{sign}\dot{\gamma}(t_0)}{\gamma(t_0)}u\right)$ is a periodic orbit by the previous proposition. On the support of this orbit we have
\begin{equation}
m_{t_0}v_\theta=\left|\frac{\gamma(t_0)}{\dot{\gamma}(t_0)}\right|\frac{\op{sign}\dot{\gamma}(t_0)}{\gamma(t_0)}=\frac{1}{\dot{\gamma}(t_0)} 
\end{equation}
and, as a consequence,
\begin{equation}
\lambda^\gamma_{m_{t_0}}(X^{m_{t_0}}_\gamma)=\frac{\gamma(t_0)^2}{\dot{\gamma}(t_0)^2}-\beta^\gamma_\theta(t_0)\frac{1}{\dot{\gamma}(t_0)}+1=\frac{\gamma(t_0)^2-\dot{\gamma}(t_0)\Gamma(t_0)}{\dot{\gamma}(t_0)^2}.
\end{equation}
Since this is a constant, we get Identity \eqref{acpar} for the action.

The second identity is proved using the definition of $I_{m_{t_0},\gamma}$:
\begin{equation*}
I_{m_{t_0},\gamma}\big|_{\op{supp}\zeta_{t_0}}=\left|\frac{\gamma(t_0)}{\dot{\gamma}(t_0)}\right|\gamma(t_0)\op{sign}\dot{\gamma}(t_0)-\Gamma(t_0)=\dot{\gamma}(t_0)\frac{\gamma(t_0)^2-\dot{\gamma}(t_0)\Gamma(t_0)}{\dot{\gamma}(t_0)^2}.
\end{equation*}
Under the curvature assumption, $I_{m_{t_0},\gamma}$ is maximised or minimised at $\op{supp}\zeta_{t_0}$ according to the sign of $\dot{\gamma}(t_0)$. In both cases $I_{m_{t_0},\gamma}\big|_{\op{supp}\zeta_{t_0}}$ and $\dot{\gamma}(t_0)$ have the same sign. Hence, also the third statement is proved. 
\end{proof}
This proposition shows that, when $K_m>0$, the action of the periodic orbits that project to latitudes is not an obstruction for $\omega_{m,\gamma}'$ to be of contact type. Thus, as we also discuss in the next subsection, one could conjecture that under this hypothesis $\omega_{m,\gamma}'$ is of contact type. On the other hand, we claim that having $K_m>0$ is not a necessary condition for the contact property to hold. For this purpose it is enough to exhibit a non-convex surface for which $m_\gamma<2$. This can be achieved as a consequence of the fact that the curvature depends on the second derivative of $\gamma$, whereas $m_\gamma$ depends only on the first derivative.

We can start from $S^2_{\gamma_0}$, the round sphere of radius $1$, and find a non-convex surface of revolution $S^2_\gamma$, which is $C^1$-close to the sphere and coincides with it around the poles. This implies that $m_\gamma=\sup_{t\in[0,\ell_\gamma]}\left|\frac{\Gamma(t)+\dot{\gamma}(t)}{\gamma(t)}\right|$ can be taken smaller than $2$ since it is close as we like to $m_{\gamma_0}=0$. Hence, every energy level of $(S^2_\gamma,g_\gamma,\mu_\gamma)$ is of contact type and the following proposition is proved.
\begin{prp}\label{nonnec}
The condition $K_m>0$ is not necessary for $\Sigma_m$ to be of contact type.
\end{prp}

On the other hand, we now show that is not true, in general, that the contact property holds on every energy level.
\begin{prp}\label{noncon}
There exists a symplectic magnetic system $(S^2,g,\sigma)$ that has an energy level not of contact type.
\end{prp}
\begin{proof}
We will achieve this goal by finding $m$ and $\gamma$ such that $X^m_\gamma$ has a closed orbit projecting to a latitude with negative action. Then, the proof is complete applying Proposition \ref{mcdcri}. 

Fix some $\varepsilon\in(0,1)$. We claim that, for every $\delta\in(0,\pi/2)$, there exists a normalised profile function $\gamma_{\delta,\varepsilon}$ such that
\begin{equation}\label{strein}
\dot{\gamma}_{\delta,\varepsilon}(\delta)<-\varepsilon.
\end{equation}
Such profile function can be obtained as in Lemma \ref{bigmlem}. Take an $a\in\left(\frac{\delta}{\pi},\frac{\delta}{\pi/2}\right)$ (this is equivalent to $\delta\in\left(\frac{\pi a}{2},\pi a\right)$) such that $-\sin\left(\frac{\delta}{a}-\frac{\pi}{2}\right)<-\varepsilon$. Consider the profile function $\gamma_a:[-\pi a/2,\pi a/2]\rightarrow\mathbb R$ of a round sphere of radius $a$. Thanks to the last inequality, $\dot{\gamma}_a(\delta-\frac{\pi a}{2})<-\varepsilon$ and, therefore, $\gamma_a$ satisfies \eqref{strein} (up to a shift of the domain). Now we stretch an interval compactly supported in $(\delta-\frac{\pi a}{2}, \pi a/2)$ by a family of diffeomorphisms $F_C$ as we did in Lemma \ref{bigmlem} (though here the condition on the second derivative of $F_C$ is not necessary). In this way we obtain a family of profile functions $\gamma^C_a$ satisfying \eqref{strein}. Since the area diverges with $C$, we find $C_2>0$ such that $\gamma^{C_2}_a$ is normalised. This finishes the proof of the claim.

Since $\big|\dot{\gamma}_{\delta,\varepsilon}(\delta)\big|\leq1$, we have that $\gamma_{\delta,\varepsilon}(\delta)\leq\delta$ and $\Gamma_{\delta,\varepsilon}(\delta)\leq-1+\delta^2$. The latitude at height $t=\delta$ of such surface is a closed orbit for $X^{m_{\delta,\varepsilon}}_{\gamma_{\delta,\varepsilon}}$, where $m_{\delta,\varepsilon}:=\frac{\gamma_{\delta,\varepsilon}(\delta)}{|\dot{\gamma}_{\delta,\varepsilon}(\delta)|}$. Using formula \eqref{acpar} we see that the action of the corresponding invariant measure $\zeta_{\delta,\varepsilon}$ is negative for $\delta$ small enough:
\begin{equation*}
\mathcal A(\zeta_{\delta,\varepsilon})\leq\frac{\delta^2-(-\varepsilon)(-1+\delta^2)}{\varepsilon^2}=-\frac{1}{\varepsilon}+o(\delta)<0.\qedhere
\end{equation*}
\end{proof}

\section{Action of ergodic measures}\label{sub_act}
When $K_m>0$, we also have a way to compute numerically the action of ergodic invariant measures. We consider only ergodic measures since they are the extremal points of the set of probability invariant measures by Choquet's Theorem and, therefore, it is enough to check the positivity of the action of these measures, in order to apply Proposition \ref{mcdcri}. Every ergodic measure $\zeta$ is concentrated on a unique level set $\{I_{m,\gamma}=I(\zeta)\}$, for some $I(\zeta)\in\R$. Moreover, if $I(\zeta)=I(\zeta')$ there exists a rotation $\Phi^{\widetilde{\partial}_\theta}_{\theta'}$ such that $\big(\Phi^{\widetilde{\partial}_\theta}_{\theta'}\big)_*\zeta=\zeta'$. Since the action is $\widetilde{\partial}_\theta$-invariant, we deduce that it is a function of $I(\zeta)$ only and we can define $\mathcal A:[\min I_{m,\gamma},\max I_{m,\gamma}]\rightarrow\R$. We already have an expression for the action at the minimum and maximum of $I_{m,\gamma}$. We now give a formula for the action when $I\in(\min I_{m,\gamma}
,\max I_{m,\gamma})$.

Every integral line $z:\R\rightarrow SS^2$ of $X^m_\gamma$, such that $I_{m,\gamma}(z)=I$ oscillates between the latitudes at height $t^-(I)$ and $t^+(I)$. Their numerical values can be easily read off from the graphs of $\widehat{I}^\pm_{m,\gamma}$ drawn in Figure \ref{pic}. If we take $z$ with $t(z(0))=t^-(I)$,  there exists a smallest $u(I)>0$ such that $t\big(z(u(I))\big)=t^+(I)$. By Birkhoff's ergodic theorem
\begin{equation}
\mathcal A(I)=\frac{1}{u(I)}\int_0^{u(I)}\left(m^2-m\frac{\beta_\theta(t)\sin\varphi}{\gamma(t)}(z(u))+1\right)du.
\end{equation}
Using this identity, we computed with \textit{Mathematica} the action when $S^2_\gamma$ is an ellipsoid. We found that is positive on every energy level in the interval $[m_{-,\gamma},m_{+,\gamma}]$. By Proposition \ref{mcdcri}, this shows numerically that $\op{Con}_{\gamma}=[0,+\infty)$, hence corroborating the conjecture that $\Sigma_{m}$ is of contact type if $K_m>0$. On the other hand, we know that, when the ellipsoid is very thin, its curvature is concentrated on its poles and hence, by Proposition \ref{bigm}, the set $[m_{-,\gamma},m_{+,\gamma}]$ is not empty. Therefore, these data would also show numerically that the inclusion $[0,m_{-,\gamma})\cup(m_{+,\gamma},+\infty)\subset\op{Con}_\gamma$ can be strict.
\chapter{Dynamically convex Hamiltonian structures}\label{cha_dc}
In this chapter we introduce dynamically convex Hamiltonian structures, for which we have sharper results on the set of periodic orbits. They were introduced by Hofer, Wysocki and Zehnder in \cite{hwz1} as a way of generalising the contact structures arising on the boundary of convex domains in $\C^2$ (see Example \ref{exa_conhyp} below) with a notion which makes sense in the contact category. Later on, this notion has been extended by Hryniewicz, Licata and Salom\~ao \cite{hls} to lens spaces and, at the moment this thesis is being written, Abreu and Macarini \cite{am} are developing a general theory of dynamical convexity, which will eventually embrace also non-exact magnetic systems on surfaces of higher genus (see Section \ref{gengen}).

The abstract results we present below will play a crucial role in Chapter \ref{cha_le}, when we will analyse symplectic magnetic systems on $S^2$.

\begin{dfn}\label{dfn_dc}
Let $\omega$ be a HS of contact type on $L(p,q)$. We say that $\omega$ is \textit{dynamically convex} if, for every contractible periodic orbit $\gamma$, $\mu^l_{\op{CZ}}(\gamma)\geq3$. By abuse of terminology, we will call a contact primitive of $\omega$ dynamically convex, as well.
\end{dfn}
\begin{rmk}\label{rmk_ind}
Thanks to \eqref{rmkdyn} the condition on the index is equivalent to $I(\Psi^{C,\Upsilon}_\gamma)\subset(1,+\infty)$ (see \eqref{psigamma}).
\end{rmk}

\begin{rmk}\label{rmk_cov}
If $\pi:S^3\rightarrow L(p,q)$ is the quotient map, then $\omega$ is dynamically convex if and only if $\pi^*\omega$ is dynamically convex. 
\end{rmk}
\begin{exa}\label{exa_conhyp}
If $\Sigma\hookrightarrow \C^2$ is a closed hypersurface bounding a convex domain, then $\lambda_{\op{st}}\big|_\Sigma$ is dynamically convex. 
\end{exa}
\begin{exa}\label{exa_geodc}
We remarked in Example \ref{exa_consm} that, if $(M,g)$ is a Riemannian surface, $\lambda\big|_{SM}$ is a contact form. Let $M=S^2$ and notice that $SM\simeq L(2,1)$. Harris and G. Paternain showed in \cite{pathar} that, if $\frac{1}{4}\leq K<1$, then $\lambda\big|_{SS^2}$ is dynamically convex.
\end{exa}
For further examples of dynamical convexity, in the context of the circular planar restricted three-body problem and of the rotating Kepler problem, we refer the reader to \cite{affhk} and \cite{affk}, respectively.

As we show in the next two sections, dynamical convexity has two main consequences for the associated Reeb flow: the existence of a symplectic Poincar\'e section and the existence of an elliptic periodic orbit.

\section{Poincar\'e sections}

The main result on the dynamics of dynamically convex Hamiltonian Structures reads as follows.

\begin{thm}[\cite{hwz1},\cite{hls}]\label{thm_dcg}
Let $\omega$ be a HS on $L(p,q)$ and suppose, furthermore, that, if $p>1$, $\omega$ is non-degenerate. If $\omega$ is dynamically convex, then there exists a Poincar\'e section of disc-type $D^2\rightarrow L(p,q)$ (which is a $p$-sheeted cover on $\partial D^2$) for the characteristic distribution of $\omega$. 
\end{thm}
\begin{rmk}
The case of $S^3$, namely $p=1$, is due to Hofer, Wysocki and Zehnder, while the case $p>1$ is due to Hryniewicz, Licata and Salom\~ao. It is likely that the non-degeneracy condition for $p>1$ can be dropped by repeating the lengthy argument contained in \cite[Section 6--8]{hwz1}. However, for the applications to periodic orbits of a degenerate system we can simply apply Theorem \ref{thm_dcg} after we have lifted the problem from $L(p,q)$ to $S^3$. The details of such argument will be explained below.
\end{rmk}
For the convenience of the reader we now recall the notion of Poincar\'e section.

\begin{dfn}\label{dfn_ret}
Let $N$ be a closed $3$-manifold. A \textit{Poincar\'e section} for $Z\in\Gamma(N)$ is a compact surface $i:S\rightarrow N$ such that
\begin{itemize}
 \item it is an embedding on the interior $\mathring{S}:=S\setminus \partial S$;
 \item $i(\partial S)$ is the disjoint union of a finite collection of embedded loops $\{\gamma_k\}$ and  $i\big|_{i^{-1}(\gamma_k)}: i^{-1}(\gamma_k)\rightarrow \gamma_k$ is a finite cover;
 \item the vector field $Z$ is transverse to $i(\mathring{S})$ and every $\gamma_k$ is the support of a periodic orbit for $Z$;
 \item every flow line of $Z$ hits the surface in forward and backward time.
\end{itemize}
If $z\in\mathring{S}$, let $t(z)$ be the smallest positive number such that $F_{\mathring{S}}(z):=\Phi^Z_{t(z)}(z)\in\mathring{S}$. The map $F_{\mathring{S}}:\mathring{S}\rightarrow\mathring{S}$ defines a diffeomorphism called the \textit{Poincar\'e first return map}. 

Finally, if $Z$ is a positive section for $\ker\omega$, where $\omega$ is a HS on $N$, a map $i$ satisfying the requirements above is called a Poincar\'e section for $\omega$.
\end{dfn}

If $Z$ and $N$ are as in Definition \ref{dfn_ret}, the discrete dynamical system $(\mathring{S},F_{\mathring{S}})$ carries important information about the qualitative dynamics on $N$, since periodic points of $F_{\mathring{S}}$ correspond to the periodic orbits of $Z$ different from the $\gamma_k$'s.

When $Z$ is a positive section of $\ker\omega$, with $\omega\in\Omega^2(N)$ a Hamiltonian Structure, $(\mathring{S},i^*\omega)$ is a symplectic manifold with finite area and $F_{\mathring{S}}$ is a symplectomorphism. 

If we suppose, in addition, that $S$ is a disc, then $N\simeq L(p,q)$ (see \cite{hls}) and Proposition 5.4 in \cite{hwz1} implies that $F_{\mathring{S}}$ is $C^0$-conjugated to a homeomorphism of the disc preserving the standard Lebesgue measure. The work of Brouwer (\cite{bro}) and Franks (\cite{fra2} and \cite{fra3}) on area-preserving homeomorphisms of the disc imply that either $F_{\mathring{S}}$ has only a single periodic point, which is a fixed point for $F_{\mathring{S}}$, or there are infinitely many periodic points.
\begin{cor}\label{hwzcor}
Suppose $\omega\in\Omega^2(L(p,q))$ is a HS having a Poincar\'e section of disc-type. Then, $\ker\omega$ has $2$ periodic orbits $\gamma_1$ and $\gamma_2$ whose lifts to $S^3$, $\widehat{\gamma}_1$ and $\widehat{\gamma}_2$, form a Hopf link (namely they are unknotted and $|\op{lk}(\widehat{\gamma},\widehat{\gamma}')|=1$). Either these are the only periodic orbits of $\ker\omega$ or there are infinitely many of them. The second alternative holds if there exists a periodic orbit whose lift to $S^3$ is knotted or if there are two periodic orbits $\gamma$ and $\gamma'$ such that $|\op{lk}(\widehat{\gamma},\widehat{\gamma}')|\neq 1$.
\end{cor}
\begin{exa}
Let us verify the above general results in a concrete example. For any pair $(p,q)$ of positive real numbers, consider the ellipsoid
\begin{equation*}
\Sigma_{p,q}:=\left\{(z_1,z_2)\in\C^2\ \Big|\ \frac{|z_1|^2}{2p}+\frac{|z_2|^2}{2q}=1\right\}\subset \C^2. 
\end{equation*}
Since $\Sigma_{p,q}$ is convex, we know by Example \ref{exa_conhyp} that it is also dynamically convex and, hence, by Theorem \ref{thm_dcg}, its Reeb flow has a Poincar\'e section of disc-type. We can construct explicitly one such disc as follows. First, using the Hamilton equation for the Hamiltonian $(z_1,z_2)\mapsto\frac{|z_1|^2}{p}+\frac{|z_2|^2}{q}$, we get $R^{\lambda}=\frac{\partial_{\theta_1}}{p}+\frac{\partial_{\theta_2}}{q}$, where $\theta_i$ is the angular coordinate in the $z_i$-plane. Thus, $\Phi^{R^\lambda}_t(z_1,z_2)=(e^{\frac{it}{p}}z_1,e^{\frac{it}{q}}z_2)$ and we see that, for any $\theta\in \T_{2\pi}$, $S_\theta:=\{(z_1,z_2)\in\Sigma_{p,q}\,\big|\, \theta_1=\theta\}$ is a Poincar\'e section of disc-type parametrised by $z_2$. The return map is $F_{\mathring{S}_\theta}(z_2)=e^{2\pi i\frac{p}{q}}z_2$, namely a rotation by the angle $2\pi\frac{p}{q}$. Such rotation has the origin as the unique fixed point if and only if $\frac{p}{q}$ is irrational. When $\frac{p}{q}$ is rational, every orbit is periodic. 
\end{exa}

If $\omega$ is dynamically convex and, when $p>1$, it is also assumed to be non-degenerate, Theorem \ref{thm_dcg} directly implies that the hypotheses of Corollary \ref{hwzcor} are satisfied and, thus, we have a dichotomy between two and infinitely many periodic orbits for $\ker\omega$. We now aim to prove that the dichotomy still holds for a dynamically convex HS, without any non-degeneracy assumption.

We restrict to $p=2$, since $SS^2\simeq L(2,1)$ and we will need only this case in the applications.
 
Regard $L(2,1)$ as the quotient of $S^3$ by the antipodal map $A:S^3\rightarrow S^3$. The quotient map $\pi:S^3\rightarrow L(2,1)$ is a double cover with $A$ as the only non-trivial deck transformation. There is a bijection $Z\mapsto \widehat{Z}$ between $\Gamma(L(2,1))$ and $\Gamma_A(S^3)\subset\Gamma(S^3)$ the subset of $A$-invariant vector fields. The antipodal map permutes the flow lines of $\widehat{Z}$. Moreover, a lift of a trajectory for $Z$ is a trajectory for $\widehat{Z}$ and the projection of a trajectory for $\widehat{Z}$ is a trajectory for $Z$. In the next lemma we restrict this correspondence to prime contractible periodic orbits of $Z$. We are grateful to Marco Golla for communicating to us the elegant proof about the parity of the linking number presented below.
\begin{lem}\label{lemlin}
There is a bijection between contractible prime orbits $z$ of $Z$ and pairs of antipodal prime orbits $\{\widehat{z},A(\widehat{z})\}$ of $\widehat{Z}$ such that $\widehat{z}$ and $A(\widehat{z})$ are disjoint. Furthermore, the linking number $lk(\widehat{z},A(\widehat{z}))$ between them is even.
\end{lem}
\begin{proof} 
Associate to a contractible periodic orbit $z$ its two distinct lifts $\widehat{z}_1$ and $\widehat{z}_2=A(\widehat{z}_1)$. Since $z$ is contractible both lifts are closed. They are also prime since a lift of an embedded path is still embedded. Suppose that the two lifts intersect. This implies that there exist points $t_1$ and $t_2$ such that $\widehat{z}_1(t_1)=\widehat{z}_2(t_2)$. Applying $\pi$ to this equality, we find $z(t_1)=z(t_2)$ and so $t_1=t_2$ modulo the period of $z$. Hence, $\widehat{z}_1=\widehat{z}_2$ contradicting the fact that the two lifts are distinct.

For the inverse correspondence, associate to two antipodal disjoint prime periodic orbits $\{\widehat{z},A(\widehat{z})\}$ their common projection $\pi(\widehat{z})$. The projected curve is contractible since its lifts are closed. Furthermore, it is also prime since, if $\pi(\widehat{z})(t_1)=\pi(\widehat{z})(t_2)$, either $\widehat{z}(t_1)=A(\widehat{z})(t_2)$ and $\widehat{z}\cap A(\widehat{z})\neq \emptyset$ or $\widehat{z}(t_1)=\widehat{z}(t_2)$ and $t_1=t_2$ modulo the period of $\widehat{z}$.

We now compute the linking number between the two knots. Consider $S^3$ as the boundary of $B^4$ the unit ball inside $\mathbb R^4$ and denote still by $A$ the antipodal map on $B^4$, which extends the antipodal map on $S^3$. Take an embedded surface $S_1\subset B^4$ such that $\partial S_1=\widehat{z}_1$ and transverse to the boundary of $B^4$. By a small perturbation we can also assume that $0\in B^4$ does not belong to the surface. The antipodal surface $S_2:= A(S_1)$ has the curve $\widehat{z}_2$ as boundary and $lk(\widehat{z}_1,\widehat{z}_2)$ is equal to the intersection number between $S_1$ and $S_2$. By perturbing again $S_1$ we can suppose that all the intersections are transverse. Indeed, if we change $S_1$ close to a point $z$ of intersection, this will affect $S_2=A(S_1)$ only near the antipodal point $A(z)=-z$, which is different from $z$ since the origin does not belong to $S_1$. Now that transversality is achieved, we claim that the number of intersections is even. This stems from the fact that, if $z\in S_1\cap S_2$, then $A(z)\in A(S_1)\cap A(S_2)=S_2\cap S_1$ and $z$ and $A(z)$ are different since $z\neq0$. As a consequence, the algebraic intersection number between the two surfaces is even as well and the lemma follows.\end{proof}
\begin{rmk}
In the proof of the lemma, the sign of the intersection between $S_1$ and $S_2$ at $z$ is the same as the sign at $A(z)$, since $A$ preserves the orientation. Thus, we cannot conclude that the total intersection number is zero. Indeed, for any $k\in\mathbb Z$ one can find a pair of antipodal knots, whose linking number is $2k$.

\end{rmk}

\begin{prp}\label{covthe}
If $\omega$ is a dynamically convex HS on $L(2,1)$, then $\ker \omega$ has either two or infinitely many periodic orbits. In the first case, the two orbits are non-contractible.
\end{prp}

\begin{proof}
By Corollary \ref{hwzcor} there exist two prime closed orbits $\widehat{z}_1$ and $\widehat{z}_2$ of $\ker \pi^*\omega$ forming a Hopf link and if there is any other periodic orbit geometrically distinct from these two, $\ker \pi^*\omega$ has infinitely many periodic orbits.

We claim that $z_1:= \pi(\widehat{z}_1)$ and $z_2:= \pi(\widehat{z}_2)$ are geometrically distinct closed orbits for $\ker \omega$ on $L(2,1)$. If, by contradiction, $z_1$ coincides with $z_2$, by Lemma \ref{lemlin}, $\widehat{z}_1$ and $\widehat{z}_2$ are antipodal and their linking number is even. This is a contradiction since $|lk(\widehat{z}_1,\widehat{z}_2)|=1$. Therefore, we conclude that $z_1$ and $z_2$ are distinct. On the other hand, if $\ker \pi^*\omega$ has infinitely many periodic orbits the same is true for $\ker\omega$. Hence, also $\ker\omega$ has either $2$ or infinitely many distinct periodic orbits.

If $\ker\omega$ has a prime contractible periodic orbit $w$, its lifts $\widehat{w}_1$ and $\widehat{w}_2$ are disjoint, antipodal and prime periodic orbits for $\ker \pi^*\omega$ by Lemma \ref{lemlin}. Since $lk(\widehat{w}_1,\widehat{w}_2)$ is even, $\{\widehat{w}_1,\widehat{w}_2\}\neq \{\widehat{z}_1,\widehat{z}_2\}$
and, therefore, there are at least three distinct periodic orbits for $\ker \pi^*\omega$. So there are infinitely many periodic orbits for $\ker \pi^*\omega$ and, hence, also for $\ker\omega$.
\end{proof}

\section{Elliptic periodic orbits}

Generally speaking, once we have established the existence of a periodic orbit $(\gamma,T)$ for a flow $\Phi^Z$, we can use it to understand the behaviour of the dynamical system in a neighbourhood of $\gamma$ by looking at the spectral properties of $t\mapsto d_{\gamma(0)}\Phi^Z_t$, the linearisation of the flow at the periodic orbit.

We have seen in Definition \ref{def_non} that when $\gamma$ is a non-degenerate periodic orbit of a Hamiltonian Structure $\omega$, we have only two possibilities: either the transverse spectrum lies on the real line ($\gamma$ hyperbolic) or it lies on the unit circle of the complex plane ($\gamma$ elliptic). 

As a sample of the applications one might get in the first case we mention the work of Hofer, Wysocki and Zehnder in \cite{hwz2}. The authors proved in Theorem 1.9 that a \textit{generic} Reeb flow $R^\tau$ on $S^3$ without Poincar\'e sections of disc-type, has a hyperbolic orbit whose stable and unstable manifolds intersect in a homoclinic orbit. Following \cite[Chapter III]{mos}, this yields the existence of a \textit{Bernoulli shift} embedded in the flow of $R^\tau$ through the construction of local Poincar\'e sections. Since a Bernoulli shift has infinitely many periodic points, $\Phi^{R^\tau}$ has infinitely many periodic orbits as well.

In contrast to the chaotic behaviour caused by the presence of the Bernoulli-shift, when $\gamma$ is elliptic the flow is expected to be stable and quasi-periodic close to the periodic orbit. As explained in \cite[Section 2.4.d]{mos}, if the transverse spectrum does not contain any root of unity, the KAM theorem implies the existence of a fundamental system of open neighbourhoods of the orbit, whose boundaries are invariant under the flow. In particular, each of these neighbourhoods is $R^\tau$-invariant. Hence, $\gamma$ is a stable periodic orbit and, as a result, the Reeb flow is non-ergodic.

Existence of elliptic periodic orbits for the Reeb flow on the boundary of convex domains in $\C^2$ has been proved in particular cases and it is an open problem raised by Ekeland \cite[page 198]{eke2} to determine whether an elliptic orbit is to be found on every system of this kind. Its existence has been showed if
\begin{itemize}
 \item the curvature satisfies a suitable pinching condition \cite{eke1};
 \item the hypersurface is symmetric with respect to the origin \cite{dadoe}.
\end{itemize}
The latter case is the most interesting to us since, in view of the double cover $S^3\rightarrow SS^2$, lifted contact forms will automatically be invariant with respect to the antipodal map.

Recently, Abreu and Macarini have announced an extension of the results contained in \cite{dadoe} to dynamically convex symmetric systems.
\begin{thm}[\cite{am}]
If $\omega$ is a dynamically convex Hamiltonian Structure on $L(p,q)$, with $p>1$, then it has an elliptic periodic orbit.
\end{thm}
In the same work, the authors will also define a notion of dynamical convexity for contact forms not necessarily on lens spaces. This has potential applications to symplectic magnetic fields on surfaces of genus at least two on low energy levels. We will comment on such generalisation in Section \ref{gengen}.
\chapter{\texorpdfstring{Low energy levels of symplectic magnetic flows on $S^2$}{Low energy levels of symplectic magnetic flows on S2}}\label{cha_le}

In this chapter we study in more detail \textit{low} energy levels of magnetic flows $(S^2,g,\sigma)$ when $\sigma$ is a symplectic form normalised in such a way that $[\sigma]=4\pi$. Throughout this chapter we are going to use the parameter $m$ and the associated notation introduced in Section \ref{sec_mag} just before Remark \ref{rmk_tancot}.

\begin{rmk}
Consider more generally a mechanical system $(S^2,g,\sigma,U)$ with $\sigma$ symplectic. We saw in Section \ref{sec_mau} that when $m$ is small and $\Vert U\Vert$ is small compared to $m^2$, the dynamics on $\Sigma_m$ is the dynamics of a corresponding magnetic system on low energy level. Thus, we can apply to this case the results presented below. 

Moreover, we saw in Section \ref{sub_phy} a concrete example of mechanical system $(S^2,\hat{g},\frac{k}{m}\sigma_{\hat{g}},\hat{U}_k)$ obtained by reduction from $(SO(3),g,0,U)$. The associated magnetic system is $(S^2,\hat{g}_{m,\hat{U}_k},\frac{k}{m}\sigma_{\hat{g}})$, with $\Vert \hat{g}_{m,\hat{U}_k}-\hat{g}\Vert\leq \frac{2\Vert \hat{U}_k\Vert}{m^2}$ and
\begin{equation*}
\frac{2\Vert \hat{U}_k\Vert}{m^2}=\frac{2\Vert U+\frac{k^2}{2\nu^2}\Vert}{m^2}\leq \frac{2\Vert U\Vert}{m^2}+\frac{k^2}{m^2}\left\Vert \frac{1}{\nu^2}\right\Vert.
\end{equation*}
Thus, we are allowed to approximate the reduced dynamics of $(SO(3),g,0,U)$ with momentum $k$ and energy $\frac{m^2}{2}$, with a symplectic magnetic system on a low energy level provided $\Vert U\Vert$ and $k^2$ are small with respect to $m^2$.
\end{rmk}

\section{Contact forms on low energy levels}
Given $\beta\in\pri^{\sigma-\sigma_g}$, fix for the rest of the chapter a primitive $\lambda^g_{m,\beta}=m\lambda-\pi^*\beta+\tau$ of $\omega_m'$. We know that there exists $m_\beta>0$ such that, for $m\in[0,m_\beta)$, the function
\begin{equation}\label{hmbeta}
h_{m,\beta}(x,v):=\lambda^g_{m,\beta}(X^m)_{(x,v)}=m^2-\beta_x(v)m+f(x) 
\end{equation}
is positive and, as a consequence $\lambda^g_{m,\beta}$ is a positive contact form. Denote by $R^{m,\beta}:=R^{\lambda^g_{m,\beta}}$, the Reeb vector field, by $\Phi^{m,\beta}$ its flow and by $\xi^{m,\beta}:=\ker\lambda^g_{m,\beta}$ the contact distribution. We readily compute
\begin{equation}\label{rmbeta}
R^{m,\beta}=\frac{X^m}{h_{m,\beta}}=\frac{m}{h_{m,\beta}}X+\frac{f}{h_{m,\beta}}V.
\end{equation}
We proceed to relate the time parameter of a flow line for $\Phi^{m,\beta}$ and the length of its projection on $S^2$.

\begin{lem}\label{pinchelt}
If $\gamma:[0,T]\rightarrow SS^2$ is a flow line of $\Phi^{m,\beta}$, then
\begin{equation}
\frac{m}{\max h_{m,\beta}}T\leq \ell(\pi(\gamma))\leq\frac{m}{\min h_{m,\beta}}T,
\end{equation}
where $\ell$ denotes the length of a curve in $S^2$ with respect to the metric $g$.
\end{lem}
\begin{proof}
We compute the norm of the tangent vector of $\pi(\gamma)$:
\begin{equation}
\left|\frac{d}{dt}\pi(\gamma)\right|=\left|d_\gamma\pi\left(\frac{d}{dt}\gamma\right)\right|=\left|\frac{m}{h_{m,\beta}}\gamma\right|=\frac{m}{h_{m,\beta}}.
\end{equation}
Plugging this identity into the formula for the length of $\pi(\gamma)$, we get
\begin{equation}
\ell(\pi(\gamma))=\int_0^T\left|\frac{d}{dt}\pi(\gamma)\right|dt=\int_0^T\frac{m}{h_{m,\beta}}dt
\end{equation}
and the lemma follows bounding the integrand from below and from above.
\end{proof}

We define a global $\omega_m'$-symplectic trivialisation of $\xi^{m,\beta}$ in the next lemma. 

\begin{lem}\label{lem_tri}
The contact structure $\xi^{m,\beta}$ admits a global $\omega_m'$-symplectic frame:
\begin{equation*}
\left( \begin{aligned}
\check{H}^{m,\beta}&:= \frac{H+\beta_x(\jmath_xv)V}{\sqrt{h^{m,\beta}}}\\
\check{X}^{m,\beta}&:= \frac{X+\big(\beta_x(v)-m\big)V}{\sqrt{h^{m,\beta}}}	
        \end{aligned}
\right).
\end{equation*}
Call $\Upsilon^{m,\beta}:\xi^{m,\beta}\rightarrow(\epsilon^2_{SS^2},\omega_{\op{st}})$ the symplectic trivialisation associated to this frame. It is given by
$\Upsilon^{m,\beta}(Z)=\sqrt{h^{m,\beta}}(\eta(Z),\lambda(Z))\in\R^2$.
\end{lem}
\begin{proof}
To find a basis for $\xi^{m,\beta}$, we set $\widetilde{H}^{m,\beta}:= H+a_HV$ and $\widetilde{X}^{m,\beta}:= X+a_XV$, for some $a_H,a_X\in\mathbb R$. Imposing $\lambda^g_{m,\beta}(\widetilde{H}^{m,\beta})=0$, we get 
\begin{equation*}
0=m\lambda(H+a_HV)-\pi^*\beta(H+a_HV)+\tau(H+a_HV)=0-\beta_x(\jmath_xv)+a_H.
\end{equation*}
Hence, we have $a_H=\beta_x(\jmath_xv)$. In the same way we find $a_X=\beta_x(v)-m$. In order to turn this basis into a symplectic one, we compute
\begin{align*}
\omega_m'(\widetilde{H}^{m,\beta},\widetilde{X}^{m,\beta})&=\omega_m'(H+a_HV,X+a_XV)\\
&=\omega_m'(H,X)+a_X\omega_m'(H,V)+a_H\omega_m'(V,X)\\
&=-(-f)+a_X\cdot(-m)+a_H\cdot0\\
&=h_{m,\beta}.
\end{align*}
Thus $\left(\check{H}^{m,\beta},\check{X}^{m,\beta}\right)$, as defined in the statement of this lemma, is a symplectic basis. We now find the coordinates of $Z=a^1\check{H}^{m,\beta}+a^2\check{X}^{m,\beta}$ with respect to this basis:
\begin{equation*}
\eta(Z)=\eta(a^1\check{H}^{m,\beta}+a^2X^{m,\beta})=a^1\eta(\check{H}^{m,\beta})+a^2\eta(X^{m,\beta})=\frac{a^1}{\sqrt{h_{m,\beta}}}+a^2\cdot0.
\end{equation*}
In the same way, $\lambda(Z)=\frac{a^2}{\sqrt{h_{m,\beta}}}$, so that $(a^1,a^2)=\sqrt{h_{m,\beta}}(\eta(Z),\lambda(Z))$.
\end{proof}

When $m=0$, the objects defined above reduce to
\begin{align*}
\lambda^g_{0,\beta}=-\pi^*\beta+\tau\quad\quad&\quad\quad R^{0,\beta}=V,\\
\check{H}^{0,\beta}= \frac{1}{\sqrt{f}}\Big(H+\beta_x(\jmath_xv)V\Big)\quad\quad&\quad\quad X^{0,\beta}= \frac{1}{\sqrt{f}}\Big(X+\beta_x(v)V\Big).\\
\end{align*}
In particular, $\lambda^g_{0,\beta}$ is an $S^1$-connection form on $SS^2$ with curvature $\sigma$. Therefore, we can think of $m\mapsto \lambda^g_{m,\beta}$ as a deformation of an $S^1$-connection with positive curvature through contact forms. This observation leads to the following result. 

\begin{lem}\label{lemconb}
There exists a diffeomorphism $F_{m,\beta}:SS^2\rightarrow SS^2$ and a real function $q_{m,\beta}:SS^2\rightarrow\mathbb R$ such that 
\begin{equation}\label{gra}
F_{m,\beta}^*\lambda^g_{m,\beta}=e^{q_{m,\beta}}\lambda^g_{0,\beta}.
\end{equation}
The family of diffeomorphisms is generated by the vector field
\begin{equation}
Z^{m,\beta}:=\frac{\check{H}^{m,\beta}}{\sqrt{h_{m,\beta}}}=\frac{1}{h_{m,\beta}}\Big(H+\beta_x(\jmath_xv)V\Big)\in\ker\lambda^g_{m,\beta}. 
\end{equation}

The map
\begin{align*}
[0,m_\beta)&\longrightarrow\ C^{\infty}(SS^2,\R)\\
m&\longmapsto\ q_{m,\beta}
\end{align*}
is smooth and admits the Taylor expansion at $m=0$
\begin{equation}\label{exp_q}
q_{m,\beta}=\frac{1}{2f}m^2+o(m^2).
\end{equation}
\end{lem}
\begin{proof}
We apply Gray Stability Theorem (see for example \cite[Theorem 2.2.2]{gei}) to the family $m\mapsto\lambda^g_{m,\beta}$ and get the equation
\begin{equation}
\imath_{Z^{m,\beta}}\omega_m'=-\frac{d}{dm}\lambda^g_{m,\beta}=-\lambda,\quad\quad \mbox{with }\ Z^{m,\beta}\in\ker\lambda^g_{m,\beta},
\end{equation}
for $Z^{m,\beta}$ and the equation
\begin{equation}\label{equ_q}
\left\{\begin{aligned}
       \frac{\partial}{\partial m} q_{m,\beta}&=\left(\frac{d}{dm}\lambda^g_{m,\beta}\right)(R^{m,\beta})_{F_{m,\beta}}=\frac{m}{h_{m,\beta}\circ F_{m,\beta}}\\
       q_{0,\beta}&=0
\end{aligned}\right. 
\end{equation}
for the function $q_{m,\beta}$. There exists a unique pair $(Z^{m,\beta},q_{m,\beta})$ satisfying such relations. By Lemma \ref{lem_tri}, we know that $Z^{m,\beta}=a_H \check{H}^{m,\beta}+a_X \check{X}^{m,\beta}$. Using the fact that $(\check{H}^{m,\beta},\check{X}^{m,\beta})$ is an $\omega'_m$-symplectic basis, we get
\begin{align*}
a_H&=\imath_{Z^{m,\beta}}\omega_m'(X^{m,\beta})=-\lambda(X^{m,\beta})=-\frac{1}{\sqrt{h_{m,\beta}}}\\
a_X&=-\imath_{Z^{m,\beta}}\omega_m'(\check{H}^{m,\beta})=\lambda(\check{H}^{m,\beta})=0.
\end{align*}

Since the function $(m,z)\mapsto\frac{d}{dm}q_{m,\beta}(z)$ is smooth on $[0,m_\beta)\times SS^2$, the same is true for $(m,z)\mapsto q_{m,\beta}(z)$. Therefore, the map $m\mapsto q_{m,\beta}$ is smooth in the $C^\infty$-topology and we can expand it at $m=0$. From \eqref{equ_q} we see that $q_{0,\beta}=0$, $\frac{\partial }{\partial m}q_{0,\beta}=0$ and $\frac{\partial^2}{\partial m^2}q_{0,\beta}=\frac{1}{h_{0,\beta}\circ F_{0,\beta}}=\frac{1}{f}$. Thus, \eqref{exp_q} follows. 
\end{proof}

After this preliminary discussion we are ready to prove the results on periodic orbits we outlined in the introduction. In the next section we start with the expansion in the parameter $m$ of the action function $S_m:SS^2\rightarrow \R$ defined by Ginzburg in \cite{gin1}. We point out that this results holds true also for magnetic systems on surfaces of higher genus.
\begin{prp}\label{exp_fun}
There exists a smooth family of smooth functions $m\mapsto S_m$, where $S_m:SS^2\rightarrow \R$, such that
\begin{itemize}
 \item the critical points of $S_m$ are the support of those periodic orbits of $X^m$ which are close to a vertical fibre;
 \item the following expansion holds
\begin{equation}
S_m=2\pi+\frac{\pi}{f}m^2+o(m^2).
\end{equation}
\end{itemize}
\end{prp}
\begin{cor}\label{cor_nondeg}
If $x\in S^2$ is a non-degenerate critical point of $f:S^2\rightarrow\R$, then there exists a smooth family of curves $m\mapsto\gamma_m$, such that
\begin{itemize}
 \item $\gamma_0$ winds uniformly once around $S_xS^2$ in the positive sense;
 \item the support of $\gamma_m$ is a periodic orbit for $X^m$. 
\end{itemize}
\end{cor}
\begin{rmk}
Corollary \ref{cor_nondeg} was already known by the experts, even if there is no explicit proof in the literature. Ginzburg mentions it in its survey paper \cite{gin2} in the paragraph after Remark 3.5 on page 136. Moreover, Castilho, using Arnold's approach of the guiding centre approximation \cite{arn}, constructs a local normal form for a magnetic system on low energy levels which is similar to the one we found in Lemma \ref{lemconb} (see \cite[Theorem 3.1]{cas}). It could be used to prove a local version of Proposition \ref{exp_fun} and, hence, the corollary.
\end{rmk}

In Section \ref{conve} and Section \ref{sec_ies}, we are going to give two independent proofs of the following theorem.
\begin{thm}\label{dyncon_thm}
The $1$-form $\lambda^g_{m,\beta}\in\Omega^1(SS^2)$ is dynamically convex for small $m$.
\end{thm}
Combining this theorem with the abstract results contained in Chapter \ref{cha_dc} we get the result below.
\begin{cor}\label{dyncon_cor}
Let $(S^2,g,\sigma)$ be a symplectic magnetic system. If $m$ is small enough, then the magnetic flow on $\Sigma_m$ has either two or infinitely many periodic orbits. In the first case, the two orbits are non-contractible. Therefore, the second alternative holds if there exists a prime contractible periodic orbit.

If we suppose in addition that $\Sigma_m$ is a non-degenerate level, then
\begin{enumerate}[\itshape a)]
 \item there is an elliptic periodic orbit (hence the system is not ergodic);
 \item there is a Poincar\'e section of disc-type.
\end{enumerate}
\end{cor}
\begin{rmk}\label{rmk_gin}
The existence of two periodic orbits on low energy levels of symplectic magnetic systems is not new. In \cite[Assertion 3]{gin1}, Ginzburg proves the existence of $2$ periodic orbits on $S^2$ and of $3$ periodic orbits on surfaces of genus at least $1$. Such orbits are close to the fibres of $\Sigma_m\rightarrow M$. Furthermore, when the energy level is non-degenerate, he improves these lower bounds by finding at least $4-e_M$ periodic orbits, $2$ of which are elliptic (see the characterisation of the return map as the Hessian of the action at the end of page 104 in \cite{gin1}).
\end{rmk}
\begin{rmk}\label{rmk_two}
Example \ref{exa_mod} exhibits energy levels where all orbits are periodic and no prime periodic orbit is contractible. Theorem 1.3 in \cite{sch1} (see also Section \ref{sec_twi} below) shows that there exists an energy level with only two periodic orbits close to the fibres. It is still an open problem to find an energy level on a non-exact magnetic system on $S^2$ with exactly two periodic orbits. We point out that in the class of exact magnetic systems such example has been found. Indeed, Katok showed in \cite{kat} that there are Randers metrics on $S^2$ with only two closed geodesics (see also \cite{zil}) and G. Paternain in \cite[Section 2]{patnot} showed that every geodesic flow of a Randers metric arises, up to time reparametrisation, on a supercritical energy level of some exact magnetic field on $S^2$.
\end{rmk}

Combining Corollary \ref{cor_nondeg} with Corollary \ref{dyncon_cor}, we can formulate a sufficient condition for the existence of infinitely many periodic orbits. 

\begin{cor}\label{cor_infmany}
Let $(S^2,g,\sigma=f\mu)$ be a symplectic magnetic system. If the function $f:S^2\rightarrow \R$ has three distinct critical points $x_{\op{min}}$, $x_{\op{Max}}$ and $x_{\op{nondeg}}$ such that $x_{\op{min}}$ is an absolute minimiser, $x_{\op{Max}}$ is an absolute maximiser and $x_{\op{nondeg}}$ is non-degenerate, then there exist infinitely many periodic orbits of the magnetic flow on every sufficiently small energy level.  
\end{cor}

In Section \ref{sec_sl} we prove Theorem \ref{alabanc}, establishing a dichotomy between short and long periodic orbits. We rephrase it here for the convenience of the reader.

\begin{thm}\label{alaban}
Suppose $(g,\sigma)\in \op{Mag}(S^2)$ and $\sigma$ is symplectic. Given $\varepsilon>0$ and a positive integer $n$, there exists $m_{\varepsilon,n}>0$ such that for every $m<m_{\varepsilon,n}$ the projection $\pi(\gamma)$ of a periodic prime solution $\gamma:\T_T$ of $X^m$ either is a simple curve on $S^2$ with length in $(\frac{2\pi-\varepsilon}{\max f}m,\frac{2\pi+\varepsilon}{\min f}m)$ or has at least $n$ self-intersections and length larger than $\frac{m}{\varepsilon}$. 
\end{thm}

This result was inspired to us by \cite{hrysal1}, where a completely analogous statement for Reeb flows on convex hypersurface in $\C^2$ close to the round $S^3$ is proved. In that paper the authors show that short orbits are unknotted, the linking number between two short orbits is $1$, while the linking number between shorts and long orbits goes to infinity as the hypersurface approaches the round $S^3$. In view of Corollary \ref{corconb}, it is likely that Theorem \ref{alaban} can also be proven as a corollary of Theorem 1.6 in \cite{hrysal1} by making explicit the relation between linking numbers on $S^3$ and self-intersections of the projected curve on $S^2$.
\medskip

Finally, the aim of Section \ref{sec_twi} is to prove Proposition \ref{pro_twi} regarding rotationally symmetric magnetic flows. In such proposition we formulate a sufficient condition for having infinitely many periodic orbits and a sufficient condition for having exactly two short periodic orbits, on every low energy level.

\section{The expansion of the Ginzburg action function}\label{sec_gin}
We recall the definition of $S_m$. Set $\check{\lambda}_0:=\lambda^g_{0,\beta}$, $\check{\lambda}_m:=e^{q_{m,\beta}}\check{\lambda}_0$. Let $\check{R}^m=R^{\check{\lambda}_m}$ be the associated Reeb vector field. For each $(x,v)\in SS^2$ we construct a local Poincar\'e section as follows. If $y$ is a curve on $S^2$, denote by $P^y_t:S_{y(0)}S^2\rightarrow S_{y(t)}S^2$ the parallel transport with respect to the $S^1$-connection $\check{\lambda}_0$. Let $B^2_\delta\subset S^2$ be a small geodesic ball centered at $x$. For every $x'\in B^2_\delta$, let $y^x_{x'}:[0,1]\rightarrow S^2$ be the geodesic connecting $x$ to $x'$ and define 
\begin{align*}
i_{(x,v)}:B^2_\delta&\longrightarrow\ SS^2\\
x'&\longmapsto\ \left(x',P^{y^x_{x'}}_1(v)\right).
\end{align*}
The map $i_{(x,v)}$ is a local section of the $S^1$-bundle $SS^2$ such that $i_{(x,v)}(x)=(x,v)$. Therefore, it is transverse to $\check{R}^m$ for every $m$ sufficiently small and there exists a first positive return time $t_m(x,v)$ such that $\Phi^{\check{R}^m}_{t_m(x,v)}=i_{(x,v)}(x'_m(x,v))$, for some $x'_m(x,v)\in B^2_\delta$. Define the path $\check{\gamma}_m(x,v):[0,1]\rightarrow SS^2$ by 
\begin{equation*}
\check{\gamma}_m(x,v)(r):= \left(y^x_{x'_m(x,v)}(r),P^{y^x_{x'_m(x,v)}}_{r}(v)\right) 
\end{equation*}
and define $\overline{\check{\gamma}_m}(x,v)$ to be the inverse path. Notice that $\frac{d}{dr}\check{\gamma}_m(x,v)\in\ker\check{\lambda}_0=\ker\check{\lambda}_m$. The loop $\gamma_m(x,v):[0,2]\rightarrow SS^2$ is obtained by concatenation
\begin{equation}
\gamma_m(x,v)(r)=
\begin{cases}
\Phi^{R^m}_{t_m(x,v)r}(x,v)&\mbox{ for }r\in[0,1],\\
\overline{\check{\gamma}_m}(x,v)(r-1)&\mbox{ for }r\in[1,2].
\end{cases}
\end{equation}
Denote by $\check{C}^\infty(\T_2,SS^2)$ the space of piecewise smooth loops of period $2$ in $SS^2$ and notice that the family of maps $m\mapsto \Big(\gamma_m:SS^2\rightarrow \check{C}^\infty(\T_2,SS^2)\Big)$ is smooth. Finally, define $S_m$ as the composition of the action functional associated with $\check{\lambda}_m$ and the above embedding of $SS^2$ inside the space of loops:
\begin{align*}
S_m:=\mathcal A^{\check{\lambda}_m}\circ \gamma_m:SS^2&\longrightarrow\ \R\\
(x,v)&\longmapsto\ \begin{aligned}[t]
                    \int_{\T_2}\gamma_m(x,v)^*\check{\lambda}_m&=\ t_m(x,v)+\int_{\T_1}\overline{\check{\gamma}_m}(x,v)^*\check{\lambda}_m\\
                    &=\ t_m(x,v). 
                   \end{aligned}
\end{align*}
It is proven in \cite[Lemma 3, page 104]{gin2} that, if $(x,v)$ is a critical point of $S_m$, then $\Phi^{\check{R}^m}_{t_m(x,v)}=(x,v)$ and, therefore, $(x,v)$ is in the support of a periodic orbit.

Notice that since the family $m\mapsto S_m$ is smooth, it admits an approximating Taylor expansion truncated at any order. To find such expansion, we observe that $\check{\lambda}_m=\frac{\check{\lambda}_0}{\pi^*\rho_m}+o(m^2)$ by Lemma \ref{lemconb}, where $\rho_m:S^2\rightarrow \R$ is given by $\rho_m(x):=1-\frac{m^2}{2f(x)}$. Thus, we prove Proposition \ref{exp_fun} in two steps:
\begin{enumerate}[{\itshape i)}]
 \item we find an expansion for $S'_m:SS^2\rightarrow\R$, the action function for the $1$-form $\check{\lambda}'_m:=\frac{\check{\lambda}_0}{\pi^*\rho_m}$ and we see that is equal to the desired expansion of $S_m$;
 \item we show that $S_m-S'_m=o(m^2)$. 
\end{enumerate}
For the first step we need the lemma below.
\begin{lem}\label{lem_intexp}
Let $\pi:E\rightarrow M$ be an $S^1$-bundle over a closed orientable surface. Let $\tau$ be an $S^1$-connection on $E$ with positive curvature form $\sigma$. Fix $\rho:M\rightarrow(0,+\infty)$ a positive function and define the contact form $\tau_\rho:=\frac{\tau}{\pi^*\rho}\in\Omega^1(E)$. Then, the Reeb vector field of $\tau_\rho$ splits as
\begin{equation}
R^{\tau_\rho}_{(x,v)}=-L^{\mathcal H}_{(x,v)}\Big((X_\rho)_x\Big)+\rho(x) V_{(x,v)},
\end{equation}
where $L^{\mathcal H}$ is the horizontal lift with respect to $\ker\tau$ and $X_\rho\in\Gamma(M)$ is the $\sigma$-Hamiltonian vector field associated to $\rho$ (namely, $\imath_{X_\rho}\sigma=-d\rho$).
\end{lem}
\begin{proof}
Since $\sigma$ is positive, $\tau$ is a contact form. Hence, also $\tau_\rho$, which is obtained multiplying $\tau$ by a positive function, is a contact form as well. Without loss of generality we write $R^{\tau_\rho}_{(x,v)}=L^{\mathcal H}_{(x,v)}(Z_{(x,v)})+a(x,v) V_{(x,v)}$, where $Z(x,v)\in T_xM$ and $a(x,v)\in \R$ are to be determined. Imposing that $1=\tau_\rho(R^{\tau_\rho})$, we obtain that $a(x,v)=\rho(x)$. Imposing that $0=\imath_{R^{\tau_\rho}}d\tau_\rho$, we get
\begin{align*}
0&=\imath_{L^{\mathcal H}(Z)+\pi^*\rho V}\left(\frac{-\pi^*\sigma}{\pi^*\rho}-\pi^*\left(\frac{d\rho}{\rho^2}\right)\wedge\tau\right)\\
&=\pi^*\left(\frac{-\imath_Z\sigma+d\rho}{\rho}\right)-\pi^*\left(\frac{d\rho(Z)}{\rho^2}\right)\tau.
\end{align*}
This implies that $-\imath_Z\sigma+d\rho=0$ and $\frac{d\rho(Z)}{\rho^2}=0$. The first condition implies that $-Z$ is the $\sigma$-Hamiltonian vector field with Hamiltonian $\rho$. Since the Hamiltonian function is a constant of motion, the second condition is also satisfied and the lemma is proven.
\end{proof}
Let us proceed to the proof of the proposition following steps \textit{i)} and \textit{ii)}.
\begin{proof}[Proof of Proposition \ref{exp_fun}]
Let us find the expansion for $m\rightarrow S'_m$. Since we already know that this expansion will be uniform, we can fix the point $(x,v)\in SS^2$ and expand the function $m\mapsto S'_m(x,v)$. We denote with a prime all the objects associated with $\check{\lambda}'_m$. Denote by $(\widetilde{x}(t),\widetilde{v}(t))$ the flow line of $R'^m$ going through $(x,v)$ at time $0$. Consider $P^{\widetilde{x}}_{t'_m(x,v)}(v)\in S_{x'_m(x,v)}S^2$, the parallel transport of $v$ along $\widetilde{x}$. The angle between $P^{\widetilde{x}}_{t'_m(x,v)}(v)$ and $\check{\gamma}'_m(x,v)(1)$ is equal to the integral of the curvature $\sigma$ on a disc bounding the concatenated loop $\widetilde{x}\ast\pi\left(\overline{\check{\gamma}'_m}(x,v)\right)$. Since this loop is contained inside a ball of radius $d(x,x'_m(x,v))$, we can bound the area of the disc by some constant times $d(x,x'_m(x,v))^2$. However, $\frac{d}{dt}\widetilde{x}=X_{\rho_m}=-m^2X_{2/f}$ implies that $d(x,x'_m(x,v))\leq t_m(x,v)\cdot m^2\left\Vert X_{2/f}\right\Vert$. Hence, the area of the disc goes to zero faster than $m^3$. As a consequence, the angle between $P^{\widetilde{x}}_{t'_m(x,v)}(v)$ and $\check{\gamma}'_m(x,v)(1)$ is of order $o(m^3)$. Consider the continuous loop $\mathring{\gamma}'_m(x,v)$ obtained by the concatenation of three paths:
\begin{itemize}
 \item the path $(\widetilde{x}(t),P^{\widetilde{x}}_{t}(v))$;
 \item a path contained in the fibre over $x'_m(x,v)$ and connecting $P^{\widetilde{x}}_{t'_m(x,v)}(v)$ with $\check{\gamma}'_m(x,v)(1)$ following the shortest angle;
 \item the path $\overline{\check{\gamma}'_m}(x,v)$.
\end{itemize}
Since $\pi(\mathring{\gamma}'_m(x,v))=\pi(\gamma'_m(x,v))$, we can consider the increment in the angle between $\mathring{\gamma}'_m(x,v)$ and $\gamma'_m(x,v)$ along this loop. As $\mathring{\gamma}'_m(x,v)(0)=\gamma'_m(x,v)(0)$, this angle will be a multiple of $2\pi$. For $m=0$ we readily see that the angle is $2\pi$, hence it remains $2\pi$ by continuity. However, the derivative of the increment of this angle is
\begin{itemize}
 \item $t'_m(x,v)\rho_m(x)$, over $\widetilde{x}$;
 \item of order $o(m^3)$, over the path contained in the fibre of $x'_m(x,v)$;
 \item zero, over $\pi(\overline{\check{\gamma}'_m}(x,v))$. 
\end{itemize}
Putting things together, we find $2\pi=t'_m(x,v)\rho_m(x)+o(m^3)$, which implies Step \textit{i)}:
\begin{equation*}
S'_m(x,v)=t'_m(x,v)=2\pi\left(1-\frac{m^2}{2f(x)}\right)+o(m^3).
\end{equation*}

For Step \textit{ii)} we observe that there exists a smooth path $m\mapsto \check{\lambda}''_m\in\Omega^1(SS^2)$ such that $\check{\lambda}_m=\check{\lambda}'_m+m^3\check{\lambda}''_m$. As a consequence, $R^{\check{\lambda}_m}=R^{\check{\lambda}'_m}+m^3R''_m$, for some smooth family $m\mapsto R''_m\in\Gamma(SS^2)$. This implies that $t_m(x,v)-t_m'(x,v)=o(m^2)$ and, hence, $S_m=t_m=t'_m+o(m^2)=S'_m+o(m^2)$. This establish Step \textit{ii)} and the entire proposition.
\end{proof}
We now have all the tools to prove the corollary.
\begin{proof}[Proof of Corollary \ref{cor_nondeg}]
Applying the proposition we just proved, we find that $S_m=2\pi+m^2\check{S}_m$, where $\check{S}_0=\frac{\pi}{f}$. Therefore, the critical points of $S_m$ are the same as the critical points of $\check{S}_m$. Notice that the critical points of $\check{S}_0$ are the vertical fibres $S_xS^2$, where $x\in S^2$ is a critical point of $f$. Let $x$ be a non-degenerate critical point of $f$ and consider a tubular neighbourhood $\T_{2\pi}\times B_\delta$ of $S_xS^2$ with coordinates $(\psi,z)$. For every $\psi\in\T_{2\pi}$ consider the restriction of $\check{S}_m$ to $\{\psi\}\times B_\delta$ and call it $\check{S}_m^\psi$. Since $x$ is a non-degenerate critical point, the functions $\check{S}_0^\psi$ have $0\in B_\delta$ as non-degenerate critical point. Therefore, by the inverse function theorem, for every sufficiently small $m$ there is a path $z_m:\T_{2\pi}\rightarrow B_\delta$ such that $z_m(\psi)$ is the unique critical point of $S_m^\psi$. We claim that $\gamma_m(\psi):=(\psi,z_m(\psi))\in\T_{2\pi}\times B_\delta$ is a critical point for $\check{S}_m$. Indeed, the function $S_m\circ \gamma_m:\T_{2\pi}\rightarrow \R$ has at least one critical point $\psi_*$, so that $d_{\gamma_m(\psi_*)}S_m\big(\frac{d}{d\psi}\gamma_m(\psi_*)\big)=0$. Since $\frac{d}{d\psi}\gamma_m(\psi_*)$ is transverse to $\{0\}\times T_{z_m(\psi_*)} B_\delta\subset T_{\gamma_m(\psi_*)}\T_{2\pi}\times B_\delta$ and \begin{equation*}
d_{\gamma_m(\psi_*)}S_m\big|_{\{0\}\times T_{z_m(\psi_*)}B_\delta}=d_{z_m(\psi_*)}S^{\psi_*}_m=0,
\end{equation*}
we have that $\gamma_m(\psi_*)$ is a critical point for $S_m$. However, the critical points for $S_m$ comes in $S^1$-families since they correspond to periodic orbits of $\check{R}^m$. This implies the claim and finishes the proof of the corollary.  
\end{proof}
\begin{rmk}
The approximation of the magnetic flow with the Reeb flow of $R^{\check{\lambda}_m'}$ is called the \textit{guiding centre approximation}. A precise formulation can be found in \cite{arn2}. It was used by Castilho in \cite{cas} to prove a theorem on region of stabilities for the magnetic flow via Moser's invariant curve theorem \cite{mos3}. Recently, Raymond and V{\~u} Ng{\d{o}}c \cite{rasa} employed this approach to study the semiclassical limit of a magnetic flow with low energy close to a non-degenerate minimum of $f$.
\end{rmk}

\section{Contactomorphism with a convex hypersurface}\label{conve}
The main goal of this section is to construct a convex hypersurface in $\C^2$ which is a contact double cover of $(SS^2,\lambda^g_{m,\beta})$. Example \ref{exa_conhyp} will then imply that $\lambda^g_{m,\beta}$ is dynamically convex and give a first proof of Theorem \ref{dyncon_thm}.

\begin{prp}\label{convi}
If $m\in[0,m_\beta)$, there exists a double cover $p_{m,\beta}:S^3\rightarrow SS^2$ and an embedding $\upsilon_{m,\beta}:S^3\rightarrow \mathbb C^2$ bounding a region starshaped around the origin and such that $p_{m,\beta}^*\lambda^g_{m,\beta}=-\upsilon_{m,\beta}^*\lambda_{\op{st}}$. Furthermore, as $m$ goes to zero, $\upsilon_{m,\beta}$ tends in the $C^2$-topology to the embedding of $S^3$ as the Euclidean sphere of radius $2$. In particular, $\upsilon_{m,\beta}$ is a convex embedding for $m$ sufficiently small.
\end{prp}

We construct the double cover $p_{m,\beta}$ in three steps. 
Let $(S^2,g_0,\mu_0)$ be the magnetic system on the round sphere of radius $1$ given by the area form. We denote by $SS^2_0$ the unit sphere bundle, by $\jmath_0$ the rotation by $\pi/2$ and by $\tau_0$ the vertical form associated with the metric $g_0$. We have already seen in Lemma \ref{lemconb} that there exists a contactomorphism $F_{m,\beta}$ between $e^{q_{m,\beta}}\lambda^g_{0,\beta}$ and $\lambda^g_{m,\beta}$. Our next task is to relate $\lambda^g_{0,\beta}$ with $\tau_0$. For this purpose, we need the following proposition due to Weinstein \cite{wei3}. For a proof we refer to \cite[Appendix B]{gui}.
\begin{prp}
Suppose $E_i\rightarrow S^2$, with $i=0,1$, are two $S^1$-bundles endowed with $S^1$-connection forms $\tau_i\in\Omega^1(E_i)$. Call $\sigma_i\in \Omega^2(S^2)$ their curvature forms and suppose they are both symplectic and such that 
\begin{equation}\label{intsi}
\left|\int_{S^2}\sigma_0\right|=\left|\int_{S^2}\sigma_1\right|,
\end{equation}
Then, there is an $S^1$-equivariant diffeomorphism $B:E_0\rightarrow E_1$ such that $B^*\tau_1=\tau_0$.
\end{prp}
Thanks to the normalisation $[\sigma]=4\pi$ and the fact that $\lambda^g_{0,\beta}$ is an $S^1$-connection form on $SS^2$, we get the following corollary.
\begin{cor}\label{corconb}
There is an $S^1$-equivariant diffeomorphism $B_\beta:SS^2_0\rightarrow SS^2$ such that $B_\beta^*\lambda^g_{0,\beta}=\tau_0$.
\end{cor}
What we have found so far tells us that we only need to study the pull-back of $\tau_0$ to $S^3$. This will be our next task. The ideas we use are taken from \cite{conoli,pathar}.

We identify $\mathbb C^2$ with the space of quaternions by setting $\mathbf1:= (1,0)$, $\mathbf i:= (i,0)$, $\mathbf j:= (0,1)$ and $\mathbf k:= (0,i)$. With this choice left multiplication by $\mathbf i$ corresponds to the action of $J_{\op{st}}$. Let $\upsilon:S^3\rightarrow \mathbb C^2$ be the inclusion of the unit Euclidean sphere. Identify the Euclidean space $\mathbb R^3$ with the vector space spanned by $\mathbf i,\mathbf j,\mathbf k$ endowed with the restricted inner product. We think the round sphere $(S^2,g_0)$ as embedded in this version of the Euclidean space. Thus, the unit sphere bundle $SS^2_0$ is embedded in $\mathbb R^3\times\mathbb R^3$ as the pair of vectors $(u_1,u_2)$ such that $u_1,u_2\in S^2$ and $g_{\op{st}}(u_1,u_2)=0$.

If $z=(u_1,u_2)\in SS^2_0\subset\mathbb R^3\times\mathbb R^3$ and $Z=(v_1,v_2)\in T_zSS^2_0\subset\mathbb R^3\times\mathbb R^3$, then
\begin{equation}\label{formucon}
(\tau_0)_{z}(Z)=g_0\Big(v_2-g_{\op{st}}\big(v_2,u_1\big)u_1,{\jmath_0}_{u_1}(u_2)\Big)=g_{\op{st}}\big(v_2,{\jmath_0}_{u_1}(u_2)\big)
\end{equation}
as a consequence of the relation between the Levi-Civita connections on $S^2$ and $\mathbb R^3$.

For any $U\in S^3$, we define a map $C_U:\mathbb R^3\rightarrow\mathbb R^3$ using quaternionic multiplication and inverse by $C_U(U')=U^{-1}U'U$. The quaternionic commutation relations and the compatibility between the metric and the multiplication tell us that $C_U$ restricts to an isometry of $S^2$. Hence, $dC_U$ yields a diffeomorphism of the unit sphere bundle onto itself given by $(u_1,u_2)\mapsto d_{u_1}C_U(u_2)=(C_U(u_1),C_U(u_2))$. Moreover, since $C_U$ is an isometry, $(dC_U)^*\tau_0=\tau_0$. 

We are now ready to define the covering map $p_0:S^3\rightarrow SS^2_0$. It is given by $p_0(U):= d_{\mathbf i}C_U(\mathbf j)$. Let us compute the pull-back of $\tau_0$ by $p_0$.
\begin{prp}\label{stalif}
The covering map $p_0$ relates $\tau_0$ and $\lambda_{\op{st}}$ in the following way:
\begin{equation}\label{equicon}
p_0^*\tau_0=-4\upsilon^*\lambda_{\op{st}}. 
\end{equation}
\end{prp}
\begin{proof}
First of all we prove that both sides of \eqref{equicon} are invariant under right multiplication. For every $U\in S^3$, we define $R_U:S^3\rightarrow S^3$ as $R_U(U'):= U'U$. Thus, the identity $p_0\circ R_U=dC_U\circ p_0$ holds. Let us show that $p_0^*\tau_0$ is right invariant:
\begin{equation*}
R_U^*\big(p_0^*\tau_0\big)=(p_0\circ R_U)^*\tau_0=(dC_U\circ p_0)^*\tau_0=p_0^*\big((dC_U)^*(\tau_0)\big)=p_0^*\tau_0.
\end{equation*}
On the other hand, $\upsilon^*\lambda_{\op{st}}$ is also right invariant:
\begin{align*}
\big(R_U^*(\upsilon^*\lambda_{\op{st}})\big)_{U'}(W)=\big(\upsilon^*\lambda_{\op{st}}\big)_{R_U(U')}(dR_UW)&=\frac{1}{2}g_{\op{st}}\big((\mathbf iU')U,WU\big)\\
&=\frac{1}{2}g_{\op{st}}\big(\mathbf iU',W\big)\\
&=\big(\upsilon^*\lambda_{\op{st}}\big)_{U'}(W),
\end{align*}
where we used that $R_U:\mathbb C^2\rightarrow\mathbb C^2$ is an isometry.

Therefore, it is enough to check equality \eqref{equicon} only at the point $\mathbf 1$. A generic element $W$ of $T_{\mathbf 1}S^3$ can be written as $s\mathbf i+w\mathbf j=s\mathbf i+\mathbf j\overline{w}$, where $w:= w_1\mathbf 1+w_2\mathbf i$ and $\overline{w}:= w_1\mathbf 1-w_2\mathbf i$ with $s,w_1 $ and $w_2$ real numbers. On the one hand,
\begin{equation}\label{ess1}
(\upsilon^*\lambda_{\op{st}})_{\mathbf 1}(W)=\frac{1}{2}g_{\op{st}}\big(\mathbf i\mathbf 1,W\big)=\frac{1}{2}g_{\op{st}}\big(\mathbf i,s\mathbf i+w\mathbf j\big)=\frac{s}{2}.
\end{equation}
On the other hand, we have that $d_{\mathbf 1}p_0(W)=(\mathbf iW-W\mathbf i,\mathbf jW-W\mathbf j)$. From the definition of $\tau_0$ we see that we are only interested in the second component:
\begin{equation*}
\mathbf jW-W\mathbf j=\mathbf j(s\mathbf i+\mathbf j\overline{w})-(s\mathbf i+w\mathbf j)\mathbf j=-2s\mathbf k+(w-\overline{w})=-2s\mathbf k+2w_2\mathbf i.
\end{equation*}
Now we apply formula \eqref{formucon} with $(u_1,u_2)=(\mathbf i,\mathbf j)$ and $v_2=-2s\mathbf k+2w_2\mathbf i$. In this case ${\jmath_0}_{u_1}$ is left multiplication by $\mathbf i$, so that ${\jmath_0}_{u_1}(u_2)=\mathbf i\mathbf j=\mathbf k$ and we find that
\begin{equation}\label{ess2}
g_{\op{st}}\big(-2s\mathbf k+2w_2\mathbf i,\mathbf k\big)=-2s.
\end{equation}
Comparing \eqref{ess1} with \eqref{ess2} we finally get $(p_0^*\tau_0)_{\mathbf 1}=-4(\upsilon^*\lambda_{\op{st}})_{\mathbf 1}$.
\end{proof}
Putting things together, we arrive at the following intermediate step. 
\begin{prp}\label{intermi}
There exists a covering map $p_{m,\beta}:S^3\rightarrow SS^2$ and a real function $\widehat{q}_{m,\beta}:S^3\rightarrow \mathbb R$ such that
\begin{equation*}
p_{m,\beta}^*\lambda^g_{m,\beta}=-4e^{\widehat{q}_{m,\beta}}\upsilon^*\lambda_{\op{st}}.
\end{equation*}
Moreover the function $\widehat{q}_{m,\beta}$ tends to $0$ in the $C^\infty$-topology as $m$ goes to zero.
\end{prp}
\begin{proof}
Lemma \ref{lemconb} gives us $F_{m,\beta}:SS^2\rightarrow SS^2$ and $q_{m,\beta}:SS^2\rightarrow \mathbb R$. \ref{corconb} gives us $B_\beta:SS^2_0\rightarrow SS^2$. If we set $p_{m,\beta}:= F_{m,\beta}\circ B_\beta\circ p_0:S^3\rightarrow SS^2$, then
\begin{align*}
p_{m,\beta}^*\lambda^g_{m,\beta}=p_0^*\big(B_\beta^*(F_{m,\beta}^*\lambda^g_{m,\beta})\big)&=p_0^*\big(B_\beta^*(e^{q_{m,\beta}}\lambda^g_{0,\beta})\big)\\
&=p_0^*(e^{q_{m,\beta}\circ B_\beta}\tau_0)\\
&=-4e^{q_{m,\beta}\circ B_\beta\circ p_0}\upsilon^*\lambda_{\op{st}}.
\end{align*}
Defining $\widehat{q}_{m,\beta}:= q_{m,\beta}\circ B_\beta\circ p_0$ we only need to show that $\widehat{q}_{m,\beta}$ goes to $0$ in the $C^\infty$ topology. This is true since, by Lemma \ref{lemconb}, the same holds for $q_{m,\beta}$. 
\end{proof}

\begin{proof}[Proof of Proposition \ref{convi}]
The final step in the proof is to notice that contact forms of the type $\rho\upsilon^*\lambda_{\op{st}}\in\Omega^1(S^3)$ with $\rho:S^3\rightarrow(0,+\infty)$ arise from embeddings of $S^3$ in $\mathbb C^2$ as the boundary of a star-shaped domain. To see this, define $\upsilon_{\sqrt{\rho}}:S^3\hookrightarrow \C^2$ as $\upsilon_{\sqrt{\rho}}(z):= \sqrt{\rho(z)}\upsilon(z)$. A computation shows that  $\upsilon_{\sqrt{\rho}}^*\lambda_{\op{st}}=\rho\upsilon^*\lambda_{\op{st}}$.

Using this observation we see that $\widehat{p}_{m,\beta}^*\lambda^g_{m,\beta}=-\upsilon_{m,\beta}^*\lambda_{\op{st}}$ with $\upsilon_{m,\beta}:= \upsilon_{\sqrt{\rho_{m,\beta}}}$ and $\rho_{m,\beta}:= 4e^{\widehat{q}_{m,\beta}}$. For small $m$, $\rho_{m,\beta}$ is $C^2$-close to the constant $4$ and, therefore, the embedding $\upsilon_{m,\beta}$ is $C^2$-close to the sphere of radius $2$. This shows that $\upsilon_{m,\beta}(S^3)$ for small $m$ and the proof is complete.
\end{proof}
\begin{rmk}
In general the convexity of the embedding $\upsilon_\rho$ can be verified defining the function $Q_{\rho}:\mathbb C^2\rightarrow[0,+\infty)$, where $Q_{\rho}(z):= |z|^2/\rho(\frac{z}{|z|})$. Then, $\upsilon_{\sqrt{\rho}}(S^3)=\{Q_{\rho}=1\}$ and $\upsilon_{\sqrt{\rho}}(S^3)$ is convex if and only if the Hessian of $Q_{\rho}$ is positive.
\end{rmk}

\section{A direct estimate of the index}\label{sec_ies}
In this subsection we present an alternative proof of Theorem \ref{dyncon_thm} showing that $\lambda^g_{m,\beta}$ is dynamically convex via a direct estimate of the index, as prescribed by Definition \ref{dfn_dc}. The advantage of this method is that it generalises to systems on surfaces of genus at least two as we discuss in Subsection \ref{gengen}. We can think of this proof as the magnetic analogue of what Harris and Paternain did for the geodesic flow, where the pinching condition on the curvature plays the same role as assuming $m$ small and $\sigma$ symplectic. After writing this proof, we discovered a similar argument in \cite[Section 3.2]{hrysal1}. 

To compute the index, we consider for each $z=(x,v)\in SS^2$ the path
\begin{equation*}
\Psi^{m,\beta}_z(t):= \Upsilon^{m,\beta}_{\Phi^{m,\beta}_t(z)}\circ d_z\Phi^{m,\beta}_t\circ \left(\Upsilon^{m,\beta}_z\right)^{-1}\in \operatorname{Sp}(1).
\end{equation*}
We define the auxiliary path $B^{m,\beta}_z(t):= \dot{\Psi}^{m,\beta}_z(\Psi^{m,\beta}_z)^{-1}\in\mathfrak{gl}(2,\R)$. The bracket relations \eqref{relsm} for $(X,V,H)$ allow us to give the following estimate for this path.
\begin{lem}\label{linea}
We can write $B^{m,\beta}_z=J_{\op{st}}+\rho^{m,\beta}_z$, where $\rho^{m,\beta}_z:\mathbb R\rightarrow \mathfrak{gl}(2,\R)$ is a path of matrices, whose supremum norm is of order $O(m)$ uniformly in $z$ as $m$ goes to zero. In other words, there exist $\overline{m}>0$ and $C>0$ not depending on $z$, but only on $\sup f$, $\inf f$, $\Vert df\Vert$ and $\Vert\beta\Vert$, such that
\begin{equation}\label{odim}
\forall m\leq \overline{m},\quad\Vert\rho^{m,\beta}_z\Vert\leq Cm.
\end{equation}
\end{lem}
\begin{proof}
For any point $z\in SS^2$ and for any vector $\overrightarrow{a_0}=(a^1_0,a^2_0)\in\mathbb R^2$ we have a path $\overrightarrow{a}=(a^1,a^2):\mathbb R\rightarrow\mathbb R^2$ defined by the relation
\begin{equation}\label{lintri}
\overrightarrow{a}=\Psi^{m,\beta}_z\overrightarrow{a}_0.
\end{equation}
Using the definition of $\Psi^{m,\beta}_z$, we see that $\overrightarrow{a}$ satisfies
\begin{equation*}
Z^{\overrightarrow{a_0}}_z(t):= d_z\Phi^{m,\beta}_t\left(a^1_0\check{H}^{m,\beta}_z+a^2_0\check{X}^{m,\beta}_z\right)=a^1(t)\check{H}^{m,\beta}_{\Phi^{m,\beta}_t(z)}+a^2(t)X^{m,\beta}_{\Phi^{m,\beta}_t(z)}.
\end{equation*} 
If we differentiate with respect to $t$ the identity
\begin{equation}\label{equadif}
a^1_0\check{H}^{m,\beta}_z+a^2_0\check{X}^{m,\beta}_z=d_{\Phi^{m,\beta}_t(z)}\Phi^{m,\beta}_{-t}Z^{\overrightarrow{a_0}}_z(t),
\end{equation}
we get the following differential equation for $\overrightarrow{a}$:
\begin{equation}\label{linsym}
\begin{aligned}
0&=\dot{a}^1(t)\check{H}^{m,\beta}_{\Phi^{m,\beta}_t(z)}+\dot{a}^2(t)\check{X}^{m,\beta}_{\Phi^{m,\beta}_t(z)}\\
&\quad+a^1(t)\left[R^{m,\beta},\check{H}^{m,\beta}\right]_{\Phi^{m,\beta}_t(z)}+a^2(t)\left[R^{m,\beta},\check{X}^{m,\beta}\right]_{\Phi^{m,\beta}_t(z)}.
\end{aligned}
\end{equation}

To estimate the first Lie bracket, we observe that $m\mapsto [R^{m,\beta},\check{H}^{m,\beta}]$ is a $\Gamma(SS^2)$-valued map which is continuous in the $C^0$-topology since the maps $m\mapsto R^{m,\beta}$ and $m\mapsto \check{H}^{m,\beta}$ are continuous in the $C^1$-topology. As a result, we have that $\left[R^{m,\beta},\check{H}^{m,\beta}\right]=\left[R^{0,\beta},\check{H}^{0,\beta}\right]+\rho$, where $\rho$ is an $O(m)$ in the $C^0$-topology, as $m$ goes to zero. Furthermore,
\begin{align*}
\left[R^{0,\beta},\check{H}^{0,\beta}\right]=\left[V,\frac{1}{\sqrt{f}}\big(H+(\beta\circ\jmath) V\big)\right]&=\frac{1}{\sqrt{f}}\Big([V,H]+V(\beta\circ\jmath)V\Big)\\
&=\frac{1}{\sqrt{f}}\Big(-X-\beta V\Big)\\
&=-X^{0,\beta}\\
&=-X^{m,\beta}+\rho'.
\end{align*}
Putting things together, $\left[R^{m,\beta},\check{H}^{m,\beta}\right]=-X^{m,\beta}+\rho_1$, where $\rho_1$ is an $O(m)$ in the $C^0$-topology. In a similar way we find that $\left[R^{m,\beta},X^{m,\beta}\right]=\check{H}^{m,\beta}+\rho_2$. Substituting these expressions for the Lie brackets inside \eqref{linsym}, we find
\begin{equation*}
\dot{a}^1\check{H}^{m,\beta}+\dot{a}^2\check{X}^{m,\beta}=a^1\check{X}^{m,\beta}-a^2\check{H}^{m,\beta}-a^1\rho_1-a^2\rho_2.
\end{equation*}
Applying the trivialisation $\Upsilon^{m,\beta}$, we get
\begin{equation}\label{equadifa}
\dot{\overrightarrow{a}}=(J_{\op{st}}+\rho^{m,\beta}_z)\overrightarrow{a},
\end{equation}
where $\rho^{m,\beta}_z\in\mathfrak{gl}(2,\R)$ is of order $O(m)$.

On the other hand, differentiating Equation \eqref{lintri}, we have
\begin{equation}\label{equadifb}
\dot{\overrightarrow{a}}=\dot{\Psi}^{m,\beta}_z\overrightarrow{a}_0=\dot{\Psi}^{m,\beta}_z\left(\Psi^{m,\beta}_z\right)^{-1}\overrightarrow{a}=B^{m,\beta}_z\overrightarrow{a}.
\end{equation}
Thus, comparing \eqref{equadifa} and \eqref{equadifb}, we finally arrive at $B^{m,\beta}_z=J_{\op{st}}+\rho^{m,\beta}_z$.
\end{proof}

The previous lemma together with the following proposition reduces dynamical convexity to a condition on the period of Reeb orbits. First, we need the following notation. For each $Z\in\Gamma(N)$ we call $T_0(Z)$ the minimal period of a closed contractible orbit of $\Phi^Z$. We set $T_0(Z)=0$, if $\Phi^Z$ has a rest point. We remark that the map $Z\mapsto T_0(Z)$ is lower semicontinuous with respect to the $C^0$-topology.

\begin{prp}\label{finpro}
Let $C$ be the constant contained in \eqref{odim}. If the inequality 
\begin{equation}\label{dincopera}
 \frac{2\pi}{T_0(R^{m,\beta})}<1-Cm,
\end{equation}
is satisfied, then $\lambda^g_{m,\beta}$ is dynamically convex.
\end{prp}
\begin{proof}
Let $\gamma$ be a contractible periodic orbit for $R^{m,\beta}$ with period $T$ and such that $\gamma(0)=z$. Consider $\Psi^{m,\beta}_z\big|_{[0,T]}\in\op{Sp}_T(1)$ defined as before and fix a $u\in\R^2\setminus\{0\}$. If $\theta_u^{\Psi^{m,\beta}_z}$ is defined by Equation \eqref{teta}, we can bound its first derivative by means of Lemma \ref{linea}, as follows:
\begin{align*}
\dot{\theta}_u^{\Psi^{m,\beta}_z}=\frac{g_{\op{st}}\big(\dot{\Psi}^{m,\beta}_zu,J_{\op{st}}\Psi^{m,\beta}_zu\big)}{|\Psi^{m,\beta}_zu|^2}&=\frac{g_{\op{st}}\big(B^{m,\beta}_z\Psi^{m,\beta}_zu,J_{\op{st}}\Psi^{m,\beta}_zu\big)}{|\Psi^{m,\beta}_zu|^2}\\
&=\frac{g_{\op{st}}\big((J_{\op{st}}+\rho^{m,\beta}_z)\Psi^{m,\beta}_zu,J_{\op{st}}\Psi^{m,\beta}_zu\big)}{|\Psi^{m,\beta}_zu|^2}\\
&=\frac{|\Psi^{m,\beta}_zu|^2+g_{\op{st}}\big(\rho^{m,\beta}_z\Psi^{m,\beta}_zu,J_{\op{st}}\Psi^{m,\beta}_zu\big)}{|\Psi^{m,\beta}_zu|^2}\\
&\geq1-\Vert \rho^{m,\beta}_z\Vert.
\end{align*}
Hence, we can estimate the normalised increment in the interval $[0,T]$ by
\begin{equation*}
\Delta(\Psi^{m,\beta}_z\big|_{[0,T]},u)=\frac{1}{2\pi}\int_0^T\dot{\theta}_u^{\Psi^{m,\beta}_z}(t)dt\geq (1-Cm)\frac{T}{2\pi}.
\end{equation*}
Therefore, by Remark \eqref{rmk_ind}, $\mu^l_{\op{CZ}}(\gamma)\geq 3$ provided $(1-Cm)\frac{T}{2\pi}>1$. Asking this condition for every contractible periodic orbit is the same as asking Inequality \eqref{dincopera} to hold. The proposition is thus proved.
\end{proof}
We are now ready to reprove Theorem \ref{dyncon_thm}.
\begin{proof}[Second proof of \ref{dyncon_thm}]

Thanks to Proposition \ref{finpro}, it is enough to show that Inequality \eqref{dincopera} holds for small $m$. First, we compute the periods of contractible orbits for $R^{0,\beta}=V$. A loop going around the vertical fibre $k$ times with unit angular speed has period $2\pi k$ and is contractible if and only if $k$ is even. Hence, $T_0(R^{0,\beta})=4\pi$.

Using the lower semicontinuity of the minimal period, we find that
\begin{equation*}
 \limsup_{m\rightarrow 0}\left(\frac{2\pi}{T_0(R^{m,\beta})}+Cm\right)\leq \frac{2\pi}{4\pi}+0=\frac{1}{2}<1
\end{equation*}
and, therefore, the inequality is still true for $m$ small enough.
\end{proof}
\subsection{Generalised dynamical convexity}\label{gengen}

Abreu and Macarini \cite{am} have recently defined a generalisation of dynamical convexity to arbitrary contact manifolds. We give a brief sketch of it in this subsection since it applies to low energy values of symplectic magnetic systems on surfaces of genus at least two.

Let $(N,\xi)$ be a closed contact manifold and let $\nu$ be a free homotopy class of loops in $N$. Suppose that the contact homology $HC^\nu(N,\xi)$ of $(N,\xi)$ in the class $\nu$ is well defined and that $c_1(\xi)$ is $\nu$-atoroidal, so that the homology is also $\Z$-graded.
\begin{dfn}\label{dfn_gdc}
Define $k_-$ and $k_+$ in $\Z\cup\{-\infty,+\infty\}$ by
\begin{equation}
\bullet\ k_-:=\inf_{k\in\Z}\Big\{HC^\nu_k(N,\xi)\neq0\Big\},\quad\quad\bullet\ k_+:=\sup_{k\in\Z}\Big\{HC^\nu_k(N,\xi)\neq0\Big\}.
\end{equation}
A $\nu$-nondegenerate contact form $\alpha$ supporting $\xi$ is called \textit{positively $\nu$-dynamically convex} if $k_-\in\Z$ and $\mu_{\op{CZ}}(\gamma)\geq k_-$ for every Reeb orbit of $\alpha$ in the class $\nu$. Similarly, $\alpha$ is called \textit{negatively $\nu$-dynamically convex} if $k_+\in\Z$ and $\mu_{\op{CZ}}(\gamma)\leq k_+$ for every Reeb orbit of $\alpha$ in the class $\nu$.
\end{dfn}
Notice that a contact form $\alpha$ supporting $(S^3,\xi_{\op{st}})$, that is dynamically convex according to Definition \ref{dfn_dc}, is indeed positively dynamically convex according to Definition \ref{dfn_gdc}, with $k_-=3$. 
The new notion of convexity allows Abreu and Macarini to prove the existence of an elliptic orbit for Reeb flows on Boothby-Wang manifolds.
\begin{dfn}
A contact manifold $(N,\xi)$ is called \textit{Boothby-Wang} if it admits a supporting contact form $\beta$ whose Reeb flow gives a free $S^1$-action.
\end{dfn}
\begin{thm}[\cite{am}]
Let $(N,\xi=\ker\beta)$ be a Boothby-Wang contact manifold and let $\nu$ be the free homotopy class of the simple closed orbits of $R^\beta$. Assume that one of the following hypotheses holds:
\begin{itemize}
 \item $M/S^1$ admits a Morse function whose critical points have all even index;
 \item all the iterates of $\nu$ are non-contractible.
\end{itemize}
If $G< S^1$ is a non-trivial finite subgroup, then the Reeb flow of every $G$-invariant positively (respectively, negatively) $\nu$-dynamically convex contact form $\alpha$ supporting $\xi$ up to isotopy has an elliptic closed orbit $\gamma$ representing $\nu$ such that $\mu_{\op{CZ}}(\gamma)=k_-$ (respectively, $\mu_{\op{CZ}}(\gamma)=k_+$).
\end{thm}

Let $M$ be a surface with genus at least two and let $\sigma$ be a symplectic magnetic form such that $[\sigma]=2\pi e_M$. In Subsection \ref{sub_lg} we proved that when $m$ is low the primitives of $\omega'_m$ of the type $\lambda^g_{m,\beta}$ are contact forms. Consider the cover $p:N\rightarrow SM$ which restricts to the standard $|e_M|$-sheeted cover above every fibre $S_xM$. Let $\nu_M$ be the class on $N$ corresponding to the lift of a curve in $SM$ winding $|e_M|$ times around $S_xM$. Then, $(N,\ker\lambda^g_{0,\beta})$ and the class $\nu_M$ satisfy the hypotheses of the theorem (a curve going around $S_xM$ once is a free element in $\pi_1(SM)$ as can be seen from the long exact sequence of homotopy groups for the fibration $SM\rightarrow M$). Moreover, Abreu and Macarini computed $HC^\nu_*(N,\ker p^*\lambda^g_{m,\beta})$ and proved that $k_+=2 e_M+1$.

Therefore, in order to prove that $p^*\lambda^g_{m,\beta}$ is negatively $\nu$-dynamically convex, and be able to apply the theorem, it is enough to show that, for every periodic orbit $(\gamma,T)$ of $R^{\lambda^g_{m,\beta}}$ homotopic to a curve winding $|e_M|$ times around the fibres of $SM\rightarrow M$, the following inequality holds:
\begin{equation}\label{gendyn}
\mu_{\op{CZ}}(\gamma)\leq 2e_M+1.
\end{equation}

Now we briefly explain how to adapt the argument used on $S^2$ to get Inequality \eqref{gendyn}.
First of all, $R^{m,\beta}$ converges to $-V$ (and not to $V$), for $m$ tending to $0$. This difference of sign leads to the estimate $\Vert B^{m,\beta}_z-(-J_{\op{st}})\Vert=Cm$ (compare with Lemma \ref{linea}). We can use this inequality to obtain the upper bound
\begin{equation*}
\Delta(\Psi^{m,\beta}_z\big|_{[0,T]},u)\leq (-1+Cm)\frac{T}{2\pi}, \quad\quad \forall u\in\xi_{m,\beta}\big|_{\gamma(0)}.
\end{equation*}
On the other hand, by the lower semicontinuity of the minimal period in the class $\nu_M$, we have $T> 2\pi(|e_M|-\varepsilon)$, for an arbitrary $\varepsilon$ and $m<m_\varepsilon$. Putting together these two inequalities we get
\begin{equation*}
I\big(\Psi^{m,\beta}_z\big|_{[0,T]}\big)\subset(-\infty,e_M+Cm(|e_M|-\varepsilon)+\varepsilon) \subset (-\infty,e_M+1) 
\end{equation*}
for $m<m_\varepsilon$ small enough. This inclusion implies \eqref{gendyn}, thanks to \eqref{rmkdyn}.

\subsection{A geometric estimate of the minimal period}\label{sub_minper}

We end this section by giving a geometric proof of Inequality \eqref{dincopera} for small values of $m$ giving an \textit{ad hoc} proof of the lower semicontinuity of the minimal period in this case. The construction we present here will turn useful again in Section \ref{sec_sl}.

We consider a finite collection of closed discs $\mathbf D:= \{D_i\ |\ D_i\subset S^2\}$ such that the open discs $\dot{\mathbf D}:= \{\dot{D}_i\}$ cover $S^2$. We also fix a collection of vector fields of unit norm $\mathbf Z=\{Z_i\ |\ Z_i\in\Gamma(D_i),\ |Z_i|=1\}$. Let $\delta$ be the Lebesgue number of the cover $\mathbf D$ with respect to the Riemannian distance. Finally, let $\varphi_i:SD_i\rightarrow\T_{2\pi}$ be the angular function associated to $Z_i$. We set
\begin{equation*}
C^\beta_{(\mathbf D,\mathbf Z)}(m):= \sup_i\left(\sup_{D_i}\frac{|d\varphi_i(X)|}{h_{m,\beta}}\right)
\end{equation*}
and we observe that $m\mapsto C^\beta_{(\mathbf D,\mathbf Z)}(m)$ is locally bounded around $m=0$. We now write down the condition of lower semicontinuity for the minimal period explicitly. Let $\varepsilon>0$ be arbitrary. We claim that there exists $m_\varepsilon>0$ such that, for $m<m_\varepsilon$, the period $T$ of a contractible periodic orbit $\gamma$ of $R^{m,\beta}$ is bigger than $4\pi-\varepsilon$. The claim will follow from the next two lemmas.
\begin{lem}\label{lem_long}
Suppose that $(\gamma,T)$ is a periodic orbit for $R^{m,\beta}$ and that $\pi(\gamma)$ is not contained in any $D_i$. If $T_*$ is any positive real number and $m<\frac{2\delta\inf h_{m,\beta}}{T_*}$, then $T>T_*$. 
\end{lem}
\begin{proof}
By assumption $\pi(\gamma)$ is not contained in any ball of radius $\delta$. Thus,
\begin{equation*}
2\delta\leq\ell(\pi(\gamma))\leq \frac{m}{\inf h_{m,\beta}}T<\frac{2\delta}{T_*}T, 
\end{equation*}
where the second inequality is obtained using Lemma \ref{lem_long}. Multiplying both sides by $\frac{T_*}{2\delta}$ yields the desired conclusion.
\end{proof}

\begin{lem}\label{lem_short}
Suppose that $(\gamma,T)$ is a periodic orbit for $R^{m,\beta}$ and that $\pi(\gamma)$ is contained in some $D_{i_0}$. Let $\widetilde{\varphi}_{i_0}:[0,T]\rightarrow\R$ be a lift of $\varphi_{i_0}\circ\gamma\big|_{[0,T]}:[0,T]\rightarrow\T_{2\pi}$ and let $2\pi N:=\widetilde{\varphi}_{i_0}(T)-\widetilde{\varphi}_{i_0}(0)$, with $N\in\Z$. For every $\varepsilon'>0$, there exists $m'_{\varepsilon'}>0$ small enough and independent of $(\gamma,T)$ such that $N\geq 1$ and 
\begin{equation}
T\geq 2\pi N(1-\varepsilon').
\end{equation}
\end{lem}
\begin{proof}
Using the fundamental theorem of calculus we get
\begin{equation*}
\widetilde{\varphi}_{i_0}(T)-\widetilde{\varphi}_{i_0}(0)=\int_0^T\frac{d\varphi_{i_0}}{dt}dt=\int_0^Td\varphi_{i_0}(R^{m,\beta})dt.
\end{equation*}
Since $d\varphi_{i_0}(R^{m,\beta})=md\varphi_{i_0}\left(\frac{X}{h_{m,\beta}}\right)+\frac{f}{h_{m,\beta}}$ and $\frac{f}{h_{0,\beta}}=1$, we have, for $m$ small enough,
\begin{equation}\label{phipos}
d\varphi_{i_0}(R^{m,\beta})\geq \inf_{SS^2} \frac{f}{h_{m,\beta}}-mC^\beta_{(\mathbf D,\mathbf Z)}(m)>0.
\end{equation}
This implies that $N\geq1$ and, moreover, that 
\begin{equation*}
2\pi N\leq \int_0^Td\varphi_{i_0}(R^{m,\beta})dt\leq\left(\sup_{SS^2}\frac{f}{h_{m,\beta}}+mC^\beta_{(\mathbf D,\mathbf Z)}(m)\right)T.
\end{equation*}
Exploiting $\frac{f}{h_{0,\beta}}=1$ again, we see that there exists $m'_{\varepsilon'}$ such that, if $m<m'_\varepsilon$,
\begin{equation*}
2\pi N(1-\varepsilon')<\frac{2\pi N}{\displaystyle\sup_{SS^2}\frac{f}{h_{m,\beta}}+mC^\beta_{(\mathbf D,\mathbf Z)}(m)}\leq T.\qedhere
\end{equation*}
\end{proof}
We now prove the claim taking $m_\varepsilon:=\min\left\{\frac{2\delta\inf h_{m,\beta}}{4\pi-\varepsilon},m'_{\varepsilon'}\right\}$, where $m'_{\varepsilon'}$ is obtained from Lemma \ref{lem_short} with a value of $\varepsilon'$ given by the equation $4\pi(1-\varepsilon')=4\pi-\varepsilon$. If $\gamma$ is not contained in a $D_i$, the claim follows from Lemma \ref{lem_long} with $T_*=4\pi-\varepsilon$. If $\gamma$ is contained in some $D_{i_0}$, then $N$ is even since $\gamma$ is contractible. Therefore, $T>4\pi(1-\varepsilon')=4\pi-\varepsilon$ and the claim is proved also in this case.

\section{A dichotomy between short and long orbits}\label{sec_sl}

In this section we prove Theorem \ref{alaban} (and, hence, Theorem \ref{alabanc}). The statement will readily follow from the general Corollary \ref{cor_ban} after we show that
\begin{itemize}
 \item periodic orbits, whose projection on the base $S^2$ has a fixed number of self-intersections, have bounded length (Corollary \ref{corinte});
 \item the projection on $S^2$ of periodic orbits with period close to $2\pi$ is a simple curve (Lemma \ref{nosimp}).
\end{itemize}

The foundational result for our proof is Proposition 1 of \cite{ban}.
\begin{prp}[Bangert]\label{prp_ban}
Let $\overline{\Phi}\!:\!\mathbb R\times \!M\rightarrow M$ be a $C^1$-flow and $\overline{p}\in M$ a periodic point of prime period $\overline{T}>0$. Then, for every $\varepsilon>0$ there exist a neighbourhood $\mathfrak U$ of $\overline{\Phi}$ in the weak $C^1$-topology on $C^1(M\times\mathbb R,M)$ and a neighbourhood $U$ of $\overline{p}$ in $M$ such that the following is true: If a flow $\Phi\in\mathfrak U$ has a periodic
point $p\in U$ with prime period $T$ then $T>\varepsilon^{-1}$ or there exists a positive integer $k$ such that $|kT-\overline{T}|<\varepsilon$ and the linear map $d_{\overline{p}}\Phi_{\overline{T}}:T_{\overline{p}}M\rightarrow T_{\overline{p}}M$ has an eigenvalue which generates the group of $k$-th roots of unity. 
\end{prp}
We can apply this proposition to $\overline{T}$-periodic flows. A flow is $\overline{T}$-periodic if and only if through every point there is a prime periodic orbit whose period is $\overline{T}$.

\begin{cor}\label{cor_ban}
Let $\overline{\Phi}$ be a $\overline{T}$-periodic flow on a compact manifold $M$. Then, for every $\varepsilon>0$ there exists a $C^1$-neighbourhood
$\mathfrak U$ of $\overline{\Phi}$, such that if $\gamma$ is a periodic orbit of $\Phi\in\mathfrak U$ with prime period $T$, either $T>\varepsilon^{-1}$
or $|T-\overline{T}|<\varepsilon$. 
\end{cor}
Let us go back to symplectic magnetic systems on low energy levels on $S^2$. We observed before that for $m=0$ the Reeb vector field is equal to the vector field $V$. Hence its flow is $2\pi$-periodic and we can apply Bangert result for $m$ small.
\begin{cor}\label{pincht}
For every $\varepsilon>0$, there exists $m_\varepsilon>0$, such that, if $m<m_\varepsilon$, every prime periodic orbit $\gamma$ of $R^{m,\beta}$ on $SS^2$ either has period
bigger than $\varepsilon^{-1}$ or in the interval $(2\pi-\varepsilon,2\pi+\varepsilon)$. 
\end{cor}
In the next lemma we prove that knowing the number of self intersections of $\pi(\gamma)$ gives a uniform bound on $\ell(\pi(\gamma))$. Thanks to Lemma \ref{pinchelt}, this will imply a bound on the period, as well. 
\begin{lem}\label{interel}
Let $x:I\rightarrow S^2$ be a closed curve which is a smooth immersion except possibly at $n_0$ points and has no more than
$n_1$ self-intersections. Suppose moreover that, when defined, the geodesic curvature of $x$ is bounded away from zero. Namely,
$\min k_x>0$. Then,
\begin{equation}\label{inel}
\ell(x)\leq \frac{\pi(4n_1+2+n_0)+(n_1+1)|\min K|\op{vol}_g(S^2)}{\min k_x}
\end{equation}
\end{lem}
\begin{proof}
We argue by induction on $n_1$. Suppose $n_1=0$, then $x$ bounds a disc $D$ and we can apply the Gauss-Bonnet formula:
\begin{equation*}
\int_x k_x(t)dt=2\pi-\int_DKdA-\sum_{i=1}^{n_0}\vartheta_i.
\end{equation*}
Since $\vartheta_i\geq-\pi$, the right-hand side is smaller than $2\pi+|\min K|\op{vol}_g(S^2)+n_0\pi$. The left-hand side is bigger than $\ell(x)\min k_x$ and dividing both sides by $\min k_x$ we get \eqref{inel} in this case.

Take now a curve $x$ with $n_1\geq1$ self-intersections. Then, there exists a smooth loop $x_1$ inside $x$. Up to a
change of basepoint, $x$ is the concatenation of $x_1$ and $x_2$, where $x_2$ has at most $n_1-1$ self-intersections and
$n_0+1$ corners.  By induction we can apply inequality \eqref{inel} to the two pieces:
\begin{align*}
\ell(x_1)&\leq\ \frac{3\pi+|\min K|\op{vol}_g(S^2)}{\min k_{x_1}},\\
\ell(x_2)&\leq\ \frac{\pi(4n_1-2+n_0+1)+n_1|\min K|\op{vol}_g(S^2)}{\min k_{x_2}}.
\end{align*}
Since $\min k_x=\min\{\min k_{x_1},\min k_{x_2}\}$ we can substitute it in the denominators above without affecting the inequality. 
Adding up the resulting relations we get
\begin{equation*}
\ell(x)=\ell(x_1)+\ell(x_2)\leq \frac{\pi(3+4n_1-2+n_0+1)+(n_1+1)|\min K|\op{vol}_g(S^2)}{\min k_x},
\end{equation*}
which is the desired result.
\end{proof}
We get immediately the following corollary.
\begin{cor}\label{corinte}
Let $\gamma$ be a periodic orbit of $R^{m,\beta}$ such that $\pi(\gamma)$ has at most $n$ self-intersections. Then,
\begin{equation*}
\frac{\ell(\pi(\gamma))}{m}\leq\frac{4n+2+(n+1)|\min K|\op{vol}_g(S^2)}{\min f}
\end{equation*}
and, therefore,
\begin{equation}\label{concn}
T\leq \max h_{m,\beta}\frac{4n+2+(n+1)|\min K|\op{vol}_g(S^2)}{\min f}=:C_n. 
\end{equation}
\end{cor}
\begin{proof}
To get the first inequality, we apply Lemma \ref{interel} to $\pi(\gamma)$, which has at most $n$ self-intersections and no corners, and observe that $k_{\pi(\gamma)}=\frac{f(\pi(\gamma))}{m}$. The second inequality follows from Lemma \ref{pinchelt}.
\end{proof}
If we fix $n\in\mathbb N$, Corollaries \ref{corinte} and \ref{pincht} show that, for $m$ small enough, a solution with $n$ self-intersections can exist if and only if its period is close to $2\pi$. The next step will be to prove that this happens only if $\pi(\gamma)$ is a simple curve.
\begin{lem}\label{nosimp}
For every $\varepsilon'>0$, there exists $\overline{m}_{\varepsilon'}>0$ such that, if $m\leq \overline{m}_{\varepsilon'}$ and $\gamma$ is a prime periodic orbit whose projection $\pi(\gamma)$ has at least $1$ self-intersection, the period of $\gamma$ is bigger than $4\pi-\varepsilon'$.
\end{lem}
\begin{proof}
We use the same strategy of Section \ref{sub_minper} and we consider a collection of discs and unit vector field $(\mathbf D,\mathbf Z)$. The lemma is proven once we show that, if $\pi(\gamma)\subset D_i$, for some $D_i$, then $N\geq2$. We achieve this goal by making a particular choice of $(\mathbf D,\mathbf Z)$. Namely, we take $D_i$ contained in some orthogonal chart $(x^1_i,x^2_i)$ and we take $Z_i:=\partial_{x^1_i}/|\partial_{x^1_i}|$. 

Suppose without loss of generality that the self-intersection is happening at $t=0$. Thus, we know that there
exists $t_*\in(0,T)$ such that $\pi(\gamma(0))=\pi(\gamma(t_*))$. By precomposing the coordinate chart with an orthogonal linear map we can assume that the tangent vector at zero is parallel to $\partial_{x^1_i}$, namely that $\varphi_i(0)=0$. We consider the function $x^2_i\circ \gamma:\T_T\rightarrow \R$ and compute its derivative. Since the parametrisation is orthogonal we know that
\begin{equation*}
\frac{d}{dt}\pi(\gamma)=\ \left|\frac{d}{dt}\pi(\gamma)\right|\left(\cos\varphi_i\frac{\partial_{x^1_i}}{|\partial_{x^1_i}|}+\sin\varphi_i\frac{\partial_{x^2_i}}{|\partial_{x^2_i}|}\right).
\end{equation*}
However, in coordinates we always have
\begin{equation*}
\frac{d}{dt}\pi(\gamma)=\ \frac{d}{dt}(x^1_i\circ \gamma)\partial_{x^1_i}+\frac{d}{dt}(x^2_i\circ \gamma)\partial_{x^2_i}.
\end{equation*}

Comparing the two expressions we find
\begin{equation}\label{compari}
\frac{d}{dt}(x^2_i\circ \gamma)=\frac{\left|\frac{d}{dt}\pi(\gamma)\right|\sin\varphi_i}{|\partial_{x^2_i}|}.
\end{equation}
A point $t_0$ is critical for $x^2_i\circ \gamma$ if and only if $\varphi_i(t_0)=0$ or $\varphi_i(t_0)=\pi$. Computing the second derivative at a critical point we get
\begin{equation}
\frac{d^2}{dt^2}\Big|_{t=t_0}(x^2_i\circ \gamma)=\frac{\left|\frac{d}{dt}\pi(\gamma)\right|}{|\partial_{x^2_i}|}\cos\varphi_i\dot{\varphi}_i\Big|_{t=t_0}.
\end{equation}
Since $\dot{\varphi}_i>0$ by \eqref{phipos}, we know that $t_0$ is a strict local minimum if $\varphi_i(t_0)=0$ and a strict local maximum if $\varphi_i(t_0)=\pi$.
 From the previous discussion we know that $t=0$ is a strict minimum and that $x^2_i\circ \gamma(0)=x^2_i\circ \gamma(t_*)$. Hence, there must exist two strict maxima $t_1\in(0,t_*)$ and $t_2\in(t_*,T)$. At these points $\varphi_i=\pi$ and, thus, $N\geq2$.
\end{proof}
We are now ready to prove Theorem \ref{alaban}.
\begin{proof}[Proof of Theorem \ref{alaban}]
Given $\varepsilon$ and $n$, take any $\varepsilon'_{\varepsilon,n}$ such that the following inequalities hold for small $m$:
\begin{equation*}
\bullet\ \ \varepsilon'_{\varepsilon,n}<\min\{1,\varepsilon\},\quad\ \bullet\ \ C_{n-1}<\frac{1}{\varepsilon'_{\varepsilon,n}},\quad\ \bullet\ \ \varepsilon'_{\varepsilon,n}\max h_{m,\beta}<\varepsilon,
\end{equation*}
where $C_{n-1}$ is the constant defined in \eqref{concn}. Define $\widetilde{m}_{\varepsilon,n}=\min\{m_{\varepsilon'_{\varepsilon,n}},\overline{m}_{\varepsilon'_{\varepsilon,n}}\}$, where $m_{\varepsilon'_{\varepsilon,n}}$ and $\overline{m}_{\varepsilon'_{\varepsilon,n}}$ are given by Corollary \ref{pincht} and Lemma \ref{nosimp}, respectively.

If $m<\widetilde{m}_{\varepsilon,n}$ and $\gamma$ is a periodic orbit with period smaller than $\frac{1}{\varepsilon'_{\varepsilon,n}}$, using Corollary \ref{pincht}, we find that $T\in(2\pi-\varepsilon'_{\varepsilon,n},2\pi+\varepsilon'_{\varepsilon,n})$. Thus $T<4\pi-\varepsilon'_{\varepsilon,n}$ and, by Lemma \ref{nosimp}, $\pi(\gamma)$ is a simple curve in $S^2$.
We use Lemma \ref{pinchelt} to estimate the length of $\pi(\gamma)$ in this case:
\begin{equation}\label{sand}
\frac{2\pi-\varepsilon'_{\varepsilon,n}}{\max h_{m,\beta}}m\leq\ell(\pi(\gamma))\leq\frac{2\pi+\varepsilon'_{\varepsilon,n}}{\min h_{m,\beta}}m.
\end{equation}
Shrinking the interval $[0,\widetilde{m}_{\varepsilon,n}]$ if necessary, we find some $m_{\varepsilon,n}\leq\widetilde{m}_{\varepsilon,n}$, such that, if $m\in[0,m_{\varepsilon,n}]$, inequalities \eqref{sand} imply
\begin{equation*}
\frac{2\pi-\varepsilon}{\max f}m\leq\ell(\pi(\gamma))\leq\frac{2\pi+\varepsilon}{\min f}m,
\end{equation*}
as required. If the period $T$ is bigger than $\frac{1}{\varepsilon'_{\varepsilon,n}}>\frac{1}{\varepsilon}$, then $\pi(\gamma)$ has at least $n$ self-intersections by the condition $C_{n-1}<\frac{1}{\varepsilon'_{\varepsilon,n}}$ and Corollary \ref{corinte}. In this case we also have
\begin{equation*}
\ell(\pi(\gamma))\geq \frac{m}{\varepsilon'_{\varepsilon,n}\max h_{m,\beta}}\geq\frac{m}{\varepsilon}.
\end{equation*}
This finishes the proof of the theorem. 
\end{proof}

\section{A twist theorem for surfaces of revolution}\label{sec_twi}

As observed in Remark \ref{rmk_two}, we do not have any example of a contact-type energy level for a non-exact magnetic system on $S^2$ with exactly two periodic orbits. This is due to the fact that the only case where we can compute all the trajectories of the magnetic flow is Example \ref{exa_mod}, where all the orbits are shown to be periodic.

Observe, indeed, that knowing that the flow of $X^m\in\Gamma(SS^2)$ is a perturbation of a flow with exactly two periodic orbits, whose iterations are assumed to be non-degenerate, is not enough to deduce that the magnetic flow has only two periodic orbits. The only thing that we can deduce is that $X^m$ has exactly two \textit{short} periodic orbits \cite{sch1}. However, \textit{long} periodic orbits may appear in the perturbed system.

On the other hand, if the perturbation satisfies a suitable \textit{twist condition}, we can prove the existence of infinitely many long periodic orbits for the new flow.

In this section we are going to use such perturbative approach on surfaces of revolution. Unlike Chapter \ref{cha_rev}, we consider more general systems of the kind $(S^2_\gamma,g_\gamma,f\mu_\gamma)$, where $f:S^2_\gamma\rightarrow(0,+\infty)$ is any positive function invariant under the rotations around the axis of symmetry. This means that $f$ depends on the $t$-variable only and we can write $f:[0,\ell_\gamma]\rightarrow(0,+\infty)$.

We know that for $m$ small, $\lambda^g_{m,\beta^\gamma}$ is a dynamically convex contact form and, hence, at least in the non degenerate case, $\Phi^{m,\beta^\gamma}$ has a Poincar\'e section of disc-type $i_{m,\beta^\gamma}:D\rightarrow SS^2_\gamma$. Moreover, since $\Phi^{0,\beta^\gamma}$ is $2\pi$-periodic, we know that the return map $F_{m,\beta^\gamma}:\mathring{D}\rightarrow\mathring{D}$ for $m=0$ is the identity. Then, if we could find explicitly a \textit{smooth} isotopy $m\mapsto i_{m,\beta^\gamma}$ of Poincar\'e sections, we could try to expand the maps $F_{m,\beta^\gamma}$ in the parameter $m$ around zero to get some information on the dynamics. However, it is in general a difficult task to construct by hand a Poincar\'e section and the rotational symmetry does not seem to help finding a disc in this case. On the other hand, we claim that such symmetry allows to find a Poincar\'e section of \textit{annulus} type, as we will show in the next subsection.
\medskip

Historically, annuli were the first kind of Poincar\'e sections to be studied. The discovery of such section in the restricted $3$-body problem \cite{poi1} led Poincar\'e to formulate his Last Geometric Theorem \cite{poi2}. It asserts the existence of infinitely many periodic points of the return map $F$, if such map satisfies the so-called \textit{twist condition}: namely, $F$ extends continuously to the boundary of the annulus and it rotates the inner and outer circle by different angles. One year later Birkhoff proved Poincar\'e's Theorem \cite{bir1} and in subsequent research found a section of annulus type for the geodesic flow of a convex two-sphere \cite[Chapter VI.10]{bir4}.
\medskip

In the last subsection we are going to give a condition on $\gamma$ and $f$ that ensures that the Poincar\'e maps for small values of $m$ are twist. By contrast, the complementary condition will single out a class of magnetic systems whose long periodic orbits have period or order $O(m^{-2})$ or higher. It would be interesting to compare this estimate with the abstract divergence rate coming from Bangert's proof of Proposition \ref{prp_ban}. A direction of future research would be to look for systems with exactly $2$ periodic orbits within this class.

From now on, we consider a fixed profile function $\gamma$ and the rotationally invariant primitive $\beta^\gamma=\beta^\gamma_\theta d\theta$, so that we can safely suppress the symbols $\gamma$ and $\beta$ from the subscripts and superscripts.
\subsection{Definition of the Poincar\'e map}

Take the loop $c:[-\ell,\ell]/_\sim\rightarrow SS^2$, where $\sim$ is the relation that identifies the boundary of the interval. It is defined in the coordinates $(t,\varphi,\theta)$, as
\begin{equation*}
c(u):=
\begin{cases}
(-u,-\pi/2,0), &\mbox{if }u<0,\\
(u,\pi/2,\pi), &\mbox{if }u>0.
\end{cases}
\end{equation*}
We extend it smoothly for $u\in\{-\ell,0,\ell\}$ and we observe that $t(c(u))=|u|$. Using the $\T_{2\pi}$-action given by the rotational symmetry we can move transversally $c$ in $SS^2$ and form an embedded $2$-torus $C:[-\ell,\ell]/_{\sim}\times \T_{2\pi}\rightarrow SS^2$:
\begin{equation*}
C(u,\psi):=
\begin{cases}
(-u,-\pi/2,\psi),&\mbox{if }u<0,\\
(u,\pi/2,\psi+\pi), &\mbox{if }u>0.
\end{cases}
\end{equation*}
In particular, notice that $\psi=\theta$, for $u<0$ and $\psi=\theta+\pi$, for $u>0$.
Since $m$ is sufficiently small, Proposition \ref{reddyn} implies that $X^m_\gamma:=mX+fV$ has only two periodic orbits $\zeta^-_m$ and $\zeta^+_m$ supported on latitudes $t^-_m$ and $t^+_m$.

The sign tells us whether the projection of $\zeta^\pm_m$ rotates in the same direction as $\partial_\theta$ or not. The numbers $t^\pm_m$ satisfy the equation 
\begin{equation*}
\pm m\dot{\gamma}(t)=f(t)\gamma(t).  
\end{equation*}
With our choices we have $t^+_m<t^-_m$. These two orbits lie inside the image of the torus $C$. If we define $u^-_m:=-t^-_m$ and $u^+_m:=t^+_m$, then $\psi\mapsto C(u^\pm_m,\psi)$ is a reparametrisation of $\zeta^\pm_m$. Hence, the $2$-torus is divided into the union of two closed cylinders $C^-_m$ and $C^+_m$ with the two orbits as common boundary. Each of the cylinders is a Poincar\'e section for $X^m$ and we would like to compute its first return map $F^\pm_m:\mathring{C}^\pm_m\rightarrow\mathring{C}^\pm_m$. If we look at the first return map $F_m:[-\ell,\ell]/_{\sim}\setminus\{u^-_m,u^+_m\}\times \T_{2\pi}\rightarrow [-\ell,\ell]/_{\sim}\times \T_{2\pi}$, we see that this map swaps the cylinders and $F^\pm_m=F_m^2\big|_{C^\pm_m}$. We have $F_0(u,\psi)=(-u,\psi+\pi)$ and, as we expected, $F^\pm_0(u,\psi)=(u,\psi)$. We know that $F_m$ is a \textit{smooth} family of \textit{smooth} maps and we claim that it extends on the whole torus to a \textit{smooth} family of \textit{continuous} maps.

\begin{prp}\label{sm_prp}
The family $m\mapsto F_m$ admits an extension to a smooth family of continuous maps $m\mapsto\widehat{F}_m:[-\ell,\ell]/_{\sim}\times \T_{2\pi}\rightarrow[-\ell,\ell]/_{\sim}\times \T_{2\pi}$.
\end{prp}
For the proof of the proposition, we need first to compute the projection of the differential of the Reeb vector field $R^{m}$ at $\zeta^\pm_m$ on the contact distribution $\xi_m$.
\begin{lem}
At a point $(t^\pm_m,\pm\frac{\pi}{2},m\frac{\theta}{\gamma(t^\pm_m)})$ we have the identity
\begin{equation}
\left(\begin{array}{c}
\mbox{}[\check{H}^{m},R^{m}]\\
\mbox{}[\check{X}^{m},R^{m}]
\end{array}\right)=\left(\begin{array}{cc}
0&\frac{f}{h_{m}}-m \frac{H(h_{m})}{h^2_{m}}\\
-1&0
\end{array}\right)\left(\begin{array}{c}
\check{H}^{m}\\
\check{X}^{m}
\end{array}
\right).
\end{equation}
\end{lem}
\begin{proof}
Thanks to Lemma \ref{lem_tri}, the matrix above is given by
\begin{equation}
\sqrt{h_{m}}\left(\begin{array}{cc}
                        \eta\big([\check{H}^{m},R^{m}]\big)&\lambda\big([\check{H}^{m},R^{m}]\big)\vspace{5pt}\\ 
			\eta\big([\check{X}^{m},R^{m}]\big)&\lambda\big([\check{X}^{m},R^{m}]\big)
                        \end{array}
\right).
\end{equation}
Thus, we have to compute the Lie brackets only up to multiples of $V$. Below we use the symbol $\doteq$ between two vectors that are equal up to a multiple of $V$. Observe that on the support of $\zeta^\pm_m$ we have the identities
\begin{align*}
\bullet\ \beta=\pm\frac{\beta_\theta}{\gamma},\quad\ \bullet \ \beta\circ\jmath=0,\quad\ \bullet\ V(\beta)=\beta\circ\jmath=0,\quad\ \bullet\ H=-\pm\widetilde{\partial}_t,\quad\ \bullet \ X=\pm\frac{\widetilde{\partial}_\theta}{\gamma}.
\end{align*}
We compute
\begin{align*}
\mbox{}[\check{H}^{m},R^{m}]&=\left[\check{H}^{m},\frac{m}{h_{m}}X+\frac{f}{h_{m}}V\right]\\
&\doteq m \check{H}^{m}\left(\frac{1}{h_{m}}\right)X+\frac{m}{h_{m}}[\check{H}^{m},X]+\frac{f}{h_{m}}[\check{H}^{m},V]\\
&\doteq m \frac{1}{\sqrt{h_{m}}}H\left(\frac{1}{h_{m}}\right)X+\frac{m}{h_{m}\sqrt{h_{m}}}[H,X]+\frac{f}{h_{m}\sqrt{h_{m}}}[H,V]\\
&\doteq m \frac{1}{\sqrt{h_{m}}}H\left(\frac{1}{h_{m}}\right)X+\frac{f}{h_{m}\sqrt{h_{m}}}X.
\end{align*}
Similarly,
\begin{align*}
[\check{X}^{m},R^{m}]&=\left[\check{X}^{m},\frac{m}{h_{m}}X+\frac{f}{h_{m}}V\right]\\
&\doteq m \check{X}^{m}\left(\frac{1}{h_{m}}\right)X+\frac{m}{h_{m}}[\check{X}^{m},X]+\frac{f}{h_{m}}[\check{X}^{m},V]\\
&\doteq 0+\frac{m}{h_{m}\sqrt{h_{m}}}(\beta-m)[V,X]+\frac{f}{h_{m}\sqrt{h_{m}}}[X,V]\\
&=-\frac{m^2-m\beta+f}{h_{m}\sqrt{h_{m}}}H\\
&=-\frac{1}{\sqrt{h_{m}}}H.
\end{align*}
Applying the trivialisation $\Upsilon^{m}$, we get the desired formula.
\end{proof}
We now prove a lemma which yields a local version of the proposition, under a further assumption on the differential of the vector field.
\begin{lem}\label{lem_extloc}
Consider coordinates $(z=(x,y),\varphi)$ on $\R^2\times\T_{2\pi}$ and let $(r,\theta)$ be polar coordinates on the first factor. Let $B_\delta:=\{r<\delta\}\subset \R^2$ be the open disc of radius $\delta$ and $\T^3_\delta:=B_\delta\times\T_{2\pi}$ be the solid torus with section $B_\delta$. Suppose that $m\mapsto Z_m\in\Gamma(\T^3_\delta)$ is a smooth family of smooth vector fields and decompose $Z_m=Z_m^{\R^2}+d\varphi(Z_m)\partial_\varphi$. Assume that 
\begin{enumerate}[\itshape a)]
 \item $\forall\,(m,\varphi)$, $(Z_m^{\R^2})_{(0,\varphi)}=0$ and the endomorphism $d^{\R^2}_{(0,\varphi)}Z_m^{\R^2}$ is antisymmetric;
 \item the smooth family of functions $a_m:\T_{2\pi}\rightarrow\R$ defined by $d^{\R^2}_{(0,\varphi)}Z_m^{\R^2}=a_m J_{\op{st}}$ is uniformly bounded from below by a positive constant.
\end{enumerate}
Under these hypotheses, there exists some $\delta'<\delta$ and a smooth family of first return maps on the set $\big\{(0,y,\varphi)\ \big| \ 0<|y|<\delta'\big\}$. Such family extends to a smooth family of continuous maps $P_m:\big\{(0,y,\varphi)\ \big| \ |y|<\delta'\big\}\rightarrow\big\{(0,y,\varphi)\ \big| \ |y|<\delta'\big\}$.
\end{lem}
\begin{proof}
We claim that the family of functions $m\mapsto d\theta(Z_m^{\R^2})$ extends to a smooth family of continuous functions $m\mapsto \widehat{a}_m$ on the whole $\T^3_\delta$. Indeed, observe that $d_{(z,\varphi)}\theta(\cdot)=\omega^{\R^2}_{\op{st}}(z,\cdot)/r^2$. As a consequence,
\begin{align*}
d_{(z,\varphi)}\theta(Z_m^{\R^2})=\frac{\omega^{\R^2}_{\op{st}}(z,Z_m^{\R^2})}{r^2}&=\frac{\omega_{\op{st}}^{\R^2}\big(z,d^{\R^2}_{(0,\varphi)}Z_m^{\R^2}z+o(r)\big)}{r^2}\\
&=\frac{g_{\op{st}}\big(z,-J_{\op{st}}d^{\R^2}_{(0,\varphi)}Z_m^{\R^2}z\big)}{r^2}+\frac{\omega^{\R^2}_{\op{st}}\big(z,o(r)\big)}{r^2}\\
&=a_m+o(1).
\end{align*}
Since the term $o(1)$ is uniform in $m$ and $\varphi$, the claim follows.

From the fact that $a_m$ is bounded away from zero, we deduce that on some small $\T^3_{\delta'}$, $d_{(z,\varphi)}\theta(Z_m)=d_{(z,\varphi)}\theta(Z_m^{\R^2})$ is also bounded away from zero and, hence, there is a well defined first return time $t_m(y,\varphi)$ for the flow of $Z_m$ on $\{(0,y,\varphi))\ \big| \ 0<|y|<\delta'\}$. It is uniquely determined by the equation
\begin{equation*}
\pi=\int_0^{t_m(y,\varphi)}d_{\Phi^{Z_m}_t(0,y,\varphi)}\theta(Z_m^{\R^2})\, dt=\int_0^{t_m(y,\varphi)}\widehat{a}_m(\Phi^{Z_m}_t(0,y,\varphi))\, dt.
\end{equation*}
Hence, $t_m(y,\varphi)$ is defined also for $z=0$ and $m\mapsto t_m$ is a smooth family of continuous functions because the same is true for $m\mapsto \widehat{a}_m$. The required extension of the Poincar\'e map is given by $P_m(y,\varphi)=\Phi^{Z_m}_{t_m(y,\varphi)}(0,y,\varphi)$.
\end{proof}
We are now ready to prove the proposition.
\begin{proof}[Proof of Proposition \ref{sm_prp}]
We need only to consider the system close to the latitudes $\zeta^\pm_m$. Since $m$ is small, $\zeta^-_m$ and $\zeta^+_m$ will be near the north and south pole, respectively. We only analyse the case of the south pole, since for the north pole we can use a similar argument. Consider the vector field $\mathring{X}^{m}:=b_m \check{X}^{m}$, where $b_m:SS^2\rightarrow(0,+\infty)$ is a rotationally invariant function to be determined later. Notice that such vector field is transverse to $C$ at the south pole, namely when $u=0$. Therefore, the map $(x_m,u_m,\psi)\mapsto \Phi^{\mathring{X}^{m}}_{x_{m}}(C(u_m+u^+_m,\psi))$ yields a well-defined smooth family of diffeomorphisms between a neighbourhood of the fibre over the south pole and $B_\delta\times\T_{2\pi}\subset\R^2\times\T_{2\pi}$. We use these maps to pull back the Reeb vector field to $B_\delta\times\T_{2\pi}$. We still denote by $R^{m}$ the pull-back.

We show that $R^m$ satisfies the hypotheses of Lemma \ref{lem_extloc}. On $x_m=0$, we have \begin{align*}
\bullet\ \partial_{x_m}=\mathring{X}^{m}=b_m \check{X}^{m} \quad\quad\quad\bullet\ \partial_{u_m}=\widetilde{\partial}_t=-H=-\sqrt{h_m}\check{H}^m, 
\end{align*}
so that $(\partial_{x_m},\partial_{u_m})$ is a basis of $\xi_m\big|_{\{x_m=0\}}$. Denote by
\begin{equation*}
A:=\left(\begin{array}{cc}
b_m&0\\
0&\sqrt{h_m}
\end{array}\right) 
\end{equation*}
the change of basis matrix. We also have dual relations on $\xi_m^*$: 
\begin{align*}
\bullet\ dx_m=\frac{\sqrt{h_m}}{b_m}\lambda \quad\quad\quad\bullet\ du_m=-\frac{\sqrt{h_m}}{\sqrt{h_m}}\eta. 
\end{align*}
Notice that, since $b_m$ and $\sqrt{h_m}$ are rotationally invariant, at $\{x_m=0\}$, we have $[\partial_{x_m},R^m]=b_m[\check{X}^m,R^m]$ and $[\partial_{u_m},R^m]=-\sqrt{h_m}[\check{H}^m,R^m]$. Then, we compute the matrix of $d_{(0,\psi)}(R^m)^{\R^2}$ in the $(x_m,u_m)$ coordinates
\begin{align*}
d_{(0,\psi)}(R^m)^{\R^2}&=\left(\begin{array}{cc}
dx_m\big([\partial_{x_m},R^m]\big)&dx_m\big([\partial_{u_m},R^m]\big)\vspace{2pt}\\
du_m\big([\partial_{x_m},R^m]\big)&du_m\big([\partial_{u_m},R^m]\big)
\end{array}\right)\\
&=A^{-1}\left(\begin{array}{cc}
        \sqrt{h_m}\lambda\big([\check{X}^{m},R^m]\big)&\sqrt{h_m}\lambda\big([\check{H}^{m},R^m]\big)\vspace{2pt}\\
	\sqrt{h_m}\eta\big([\check{X}^{m},R^m]\big)&\sqrt{h_m}\eta\big([\check{H}^{m},R^m]\big)
        \end{array}\right)A\\
&=A^{-1}\left(\begin{array}{cc}
        0&\frac{f}{h_{m}}-m \frac{H(h_{m})}{h^2_{m}}\\
	-1&0
        \end{array}\right)A\\
&=\left(\begin{array}{cc}
        0&-\frac{\sqrt{h_m}}{b_m}\left(\frac{f}{h_{m}}-m \frac{H(h_{m})}{h^2_{m}}\right)\\
	\frac{b_m}{\sqrt{h_m}}&0
        \end{array}\right)
\end{align*}
Therefore, we need $\frac{b_m}{\sqrt{h_m}}=\frac{\sqrt{h_m}}{b_m}\left(\frac{f}{h_{m}}-m \frac{H(h_{m})}{h^2_{m}}\right)$. Namely,
\begin{equation*}
b_m=+\sqrt{f-m\frac{H(h_{m})}{h_{m}}}.
\end{equation*}
In this case we have
\begin{equation*}
d_{(0,\psi)}(R^m)^{\R^2}=\sqrt{\frac{f}{h_m}-m\frac{H(h_{m})}{h^2_{m}}}J_{\op{st}}
\end{equation*}
which gives at once that both hypothesis \textit{a)} and \textit{b)} are satisfied. Thanks to Lemma \ref{lem_extloc}, the Reeb flow admits a smooth family of continuous first return maps on $\{(0,u_m,\psi)\ |\ |u_m|<\delta'\}\subset C$. Since this set is an open neighbourhood of $\zeta^+_m$ in $C$, the proposition follows.
\end{proof}
\begin{rmk}
We do not know if the extended return maps can be taken to be smooth.
\end{rmk}

The rotational symmetry of the system implies that $F_m$ can be written in the form $F_m(u,\psi)=(u_m(u),\psi+\psi_m(u))$. When $u$ and $u_m(u)$ are not in $\{-\ell,0,\ell\}$, then we can also define the longitude $\theta$ of the starting point and the longitude $\theta+\theta_m(u)$ of the ending point. If $u$ and $u_m(u)$ have the same sign, then $\theta_m(u)=\psi_m(u)$, while if $u$ and $u_m(u)$ have opposite sign $\theta_m(u)=\psi_m(u)+\pi$.

The system has also another type of symmetries given by the reflections along planes in $\R^3$ containing the axis of rotation. In coordinates these maps can be written as $(t,\theta)\mapsto (t,2\theta_0-\theta)$, for some $\theta_0\in\T_{2\pi}$. Trajectories of the flow are sent to trajectories of the flow traveled in the opposite direction. Therefore, if $u$ and $u_m(u)$ are not in $\{-\ell,0,\ell\}$, then $u_m(u_m(u))=u$ and $\theta_m(u_m(u))=\theta_m(u)$. This implies that $\psi_m(u_m(u))=\psi_m(u)$ for \textit{every} $u\in[-\ell,+\ell]/_\sim$. Thus, $F_m^2(u,\psi)=(u,\psi+2\psi_m(u))$.

When $|u|$ is different from $0$ and $\ell$, and $m$ is suitably small ($m$ will be smaller and smaller as $|u|$ is close to $0$ or $\ell$), $u$ and $u_m(u)$ have different signs and we can write an integral formula for $\psi_m(u)=\theta_m(u)+\pi$. Without loss of generality, we fix some $u<0$ for the forthcoming computation. Let $z^u_m:[0,R]\rightarrow SS^2$ be the only solution of the magnetic flow, passing through $C(u,0)$ at time $0$ and through $C(u_m(u),\psi_m(u))$ at time $R$. Then, $\varphi\circ z^u_m:[0,R]\rightarrow [-\pi/2,\pi/2]$ is a monotone increasing function, such that $\varphi(z^u_m(0))=-\pi/2$ and $\varphi(z^u_m(R))=\pi/2$. Indeed, thanks to Equation \eqref{relaz1}, its derivative satisfies the equation
\begin{equation*}
\frac{d}{dr}\varphi=f-\frac{\dot{\gamma}}{\gamma}m\sin\varphi
\end{equation*}
(where we dropped the composition with $z^u_m$ in the notation). Hence, we can reparametrise $z^u_m$ using the variable $\varphi$ and get the following expression for $\theta_m(u)$:
\begin{equation}\label{eq_th}
\theta_m(u)=\int_{-\frac{\pi}{2}}^{\frac{\pi}{2}}\frac{d\theta}{d\varphi}\,d \varphi=\int_{-\frac{\pi}{2}}^{\frac{\pi}{2}}\frac{d\theta}{dr}\frac{dr}{d\varphi}\,d \varphi=\int_{-\frac{\pi}{2}}^{\frac{\pi}{2}}\frac{m\sin\varphi}{\gamma f-\dot{\gamma}m\sin\varphi}\,d \varphi,
\end{equation}
where we substituted $\frac{d\theta}{dr}=\frac{m\sin\varphi}{\gamma}$, using Equation \eqref{relaz1} again.

In the next section we expand $\theta_m(u)$ with respect to the variable $m$ around $m=0$. We aim at computing the first non-zero term.
\subsection{The Taylor expansion of the return map}
For small $m$ we have that $t(z^u_m)=|u|+O_u(m)$ (the size of $O_u(m)$ depending on $u$), hence $\gamma f$ is big compared to $\dot{\gamma}m\sin\varphi$. Thus, we expand the denominator in the integrand in \eqref{eq_th}, up to a term of order $m^2$, using the formula for the geometric series:
\begin{align}
\frac{m\sin\varphi}{\gamma f-\dot{\gamma}m\sin\varphi}&=\frac{m\sin\varphi}{\gamma f}\left(1+\frac{\dot{\gamma}m\sin\varphi}{\gamma f}\right)+o_u(m^2)\nonumber\\
&=m\frac{1}{\gamma f}\sin\varphi+m^2\frac{\dot{\gamma}}{(\gamma f)^2}\sin^2\varphi+o_u(m^2).\label{inte}
\end{align}
Discarding the $o_u(m^2)$ remainder, we plug \eqref{inte} into \eqref{eq_th} and compute the two resulting integrals separately. 

We use integration by parts for the first one
\begin{align*}
\int_{-\frac{\pi}{2}}^{\frac{\pi}{2}}m\frac{1}{\gamma f}\sin\varphi\, d \varphi&=m\int_{-\frac{\pi}{2}}^{\frac{\pi}{2}}\left(\frac{d}{dt}\frac{1}{\gamma f}\right)\frac{dt}{d\varphi}\cos\varphi\, d \varphi\\
&=m\int_{-\frac{\pi}{2}}^{\frac{\pi}{2}}\left(\frac{d}{dt}\frac{1}{\gamma f}\right)\frac{m\cos\varphi}{f}\cos\varphi\, d \varphi+o_u(m^2)\\
&=m^2\left(\frac{d}{dt}\frac{1}{\gamma f}\right)\frac{1}{f}\big(|u|\big)\int_{-\frac{\pi}{2}}^{\frac{\pi}{2}}\cos^2\varphi\, d \varphi+o_u(m^2)\\
&=\frac{m^2\pi}{2}\left(\frac{d}{dt}\frac{1}{\gamma f}\right)\frac{1}{f}\big(|u|\big)+o_u(m^2).
\end{align*}
For the second integral we readily find
\begin{align*}
\int_{-\frac{\pi}{2}}^{\frac{\pi}{2}}m^2\frac{\dot{\gamma}}{(\gamma f)^2}\sin^2\varphi\, d\varphi&=m^2\frac{\dot{\gamma}}{(\gamma f)^2}\big(|u|\big)\int_{-\frac{\pi}{2}}^{\frac{\pi}{2}}\sin^2\varphi\,  d\varphi+o_u(m^2)\\
&=\frac{m^2\pi}{2}\frac{\dot{\gamma}}{(\gamma f)^2}\big(|u|\big)+o_u(m^2).
\end{align*}
Putting things together we get
\begin{align*}
\theta_m(u)&=\frac{m^2\pi}{2}\left(\left(\frac{d}{dt}\frac{1}{\gamma f}\right)\frac{1}{f}+\frac{\dot{\gamma}}{(\gamma f)^2}\right)\big(|u|\big)+o_u(m^2)\\
&=\frac{m^2\pi}{2}\left(-\left(\frac{\dot{\gamma}}{\gamma^2f}+\frac{\dot{f}}{\gamma f^2}\right)\frac{1}{f}+\frac{\dot{\gamma}}{(\gamma f)^2}\right)\big(|u|\big)+o_u(m^2)\\
&=\frac{m^2\pi}{2}\left(-\frac{\dot{f}}{\gamma f^3}\right)\big(|u|\big)+o_u(m^2).
\end{align*}
By Proposition \ref{sm_prp}, we know that $\psi_m(u)=\theta_m(u)+\pi$ is a smooth family of functions. Hence, the expansion above translates in an expansion for $\psi_m$ that holds on the whole torus and such that the remainder is of order $o(m^2)$ uniformly in $u$. Indeed, observe that $\dot{f}(0)=\dot{f}(\ell)=0$ and, therefore, the function $-\dot{f}/(\gamma f^3)$ extends smoothly at the poles, taking the value $-\ddot{f}/f^3$ at the south pole and $\ddot{f}/f^3$ at the north pole. Call this extension $\Omega_f:S^2\rightarrow \R$ and set $\Omega_{f}^-:=\inf|\Omega_f|$. We have arrived at the final result for this section.
\begin{prp}
The family $F^2_m:[-\ell,+\ell]/_\sim\times\T_{2\pi}\rightarrow [-\ell,+\ell]/_\sim\times\T_{2\pi}$ admits the expansion
\begin{equation}\label{eq_ret}
F^2_m(u,\psi)=\left(u,\psi+\pi\Omega_fm^2+o(m^2)\right).
\end{equation}
\end{prp}
\begin{cor}
If $-\frac{\dot{f}}{\gamma f^3}$ is not constant, the magnetic flow has infinitely many periodic orbits on every low energy level. Such condition is satisfied if, for example, $\frac{\ddot{f}}{f^3}(0)\neq-\frac{\ddot{f}}{f^3}(\ell)$ and in this case both $F^+_m$ and $F^-_m$ are twist maps for small $m$.

If $\Omega_f^->0$ (namely $\dot{f}=0$ only at the poles and $\ddot{f}\neq0$ there), the period $T$ of an orbit different from the latitudes $\zeta^\pm_m$ satisfies
\begin{equation}
T\geq \frac{2}{\Omega_f^-m^2}+O\left(\frac{1}{m^3}\right).
\end{equation}
In particular, $\zeta^+_m$ and $\zeta^-_m$ are the only two short orbits.
\end{cor}

From the expansion \eqref{eq_ret}, it follows that a necessary condition for having $\psi_m$ constant for some small $m$ is to require that
\begin{equation*}
-\frac{\dot{f}}{\gamma f^3}=2k,
\end{equation*}
for some $k\in\R$. We are interested in this condition, because the only way of having exactly two closed periodic orbits for $X^m$ is to have $\psi_m=a_m\in\R/\Q$. The equation above can be rewritten as
\begin{equation*}
\frac{d}{dt}\frac{1}{f^2}=k\gamma
\end{equation*}
and integrating in the variable $t$ we find the solution
\begin{equation}\label{mag_nec}
f=\frac{1}{\sqrt{k\Gamma+h}},
\end{equation}
where $\Gamma:[0,\ell]\rightarrow \R$ is a primitive of $\gamma$ such that $\Gamma(0)=0$, $\Gamma(\ell)=\frac{\op{vol}}{2\pi}$ and $h\in\R$ is any constant such that 
\begin{equation*}
\begin{cases}
h>-\frac{k\op{vol}}{2\pi},&\mbox{ if }k<0,\\
h>0,&\mbox{ if }k\geq0,
\end{cases} 
\end{equation*}
so that the quantity under square root is strictly positive.

A direction of future research would be to study magnetic systems with $f$ given by \eqref{mag_nec} for suitable choices of $\gamma$, $k$ and $h$, to see if the flow of any of them can be written down explicitly. In this way one could check if $\psi_m= a_m\in\R/\Q$, for some value of $m$.

\backmatter

\bibliographystyle{amsalpha}
\bibliography{thesisdraft}

\newcommand{\etalchar}[1]{$^{#1}$}
\def\cprime{$'$}
\providecommand{\bysame}{\leavevmode\hbox to3em{\hrulefill}\thinspace}
\providecommand{\MR}{\relax\ifhmode\unskip\space\fi MR }
\providecommand{\MRhref}[2]{%
  \href{http://www.ams.org/mathscinet-getitem?mr=#1}{#2}
}
\providecommand{\href}[2]{#2}
\begin{thebibliography}{AFvKP12}

\bibitem[AB14]{ab}
L.~Asselle and G.~Benedetti, \emph{Periodic orbits of non-exact magnetic flows
  on surfaces of genus at least two}, preprint, arXiv:1405.0415, 2014.

\bibitem[Ada97]{ada1}
T.~Adachi, \emph{A comparison theorem on magnetic {J}acobi fields}, Proc.
  Edinburgh Math. Soc. (2) \textbf{40} (1997), no.~2, 293--308.

\bibitem[AFF{\etalchar{+}}12]{affhk}
P.~Albers, J.~W. Fish, U.~Frauenfelder, H.~Hofer, and O.~van Koert,
  \emph{Global surfaces of section in the planar restricted 3-body problem},
  Arch. Ration. Mech. Anal. \textbf{204} (2012), no.~1, 273--284.

\bibitem[AFFvK13]{affk}
P.~Albers, J.~W. Fish, U.~Frauenfelder, and O.~van Koert, \emph{The
  {C}onley-{Z}ehnder indices of the rotating {K}epler problem}, Math. Proc.
  Cambridge Philos. Soc. \textbf{154} (2013), no.~2, 243--260.

\bibitem[AFvKP12]{affkp}
P.~Albers, U.~Frauenfelder, O.~van Koert, and G.~P. Paternain, \emph{Contact
  geometry of the restricted three-body problem}, Comm. Pure Appl. Math.
  \textbf{65} (2012), no.~2, 229--263.

\bibitem[AK98]{arkh}
V.~I. Arnold and B.~A. Khesin, \emph{Topological methods in hydrodynamics},
  Applied Mathematical Sciences, vol. 125, Springer-Verlag, New York, 1998.

\bibitem[AM]{am}
M.~Abreu and L.~Macarini, \emph{Dynamical convexity and elliptic orbits for
  {R}eeb flows}, presented at the {W}orkshop on {C}onservative {D}ynamics and
  {S}ymplectic {G}eometry {\itshape({S}eptember 2013, {IMPA}, {R}io de
  {J}aneiro, {B}razil)} and at the {VIII} {W}orkshop on {S}ymplectic
  {G}eometry, {C}ontact {G}eometry and {I}nteractions {\itshape({J}anuary 2014,
  {IST}, {L}isbon, {P}ortugal}).

\bibitem[AM78]{abm}
R.~Abraham and J.~E. Marsden, \emph{Foundations of mechanics},
  Benjamin/Cummings Publishing Co. Inc. Advanced Book Program, Reading, Mass.,
  1978, Second edition, revised and enlarged with the assistance of Tudor
  Ra{\c{t}}iu and Richard Cushman.

\bibitem[AMMP14]{ammp}
A.~Abbondandolo, L.~Macarini, M.~Mazzucchelli, and G.~P. Paternain,
  \emph{Infinitely many periodic orbits of exact magnetic flows on surfaces for
  almost every subcritical energy level}, preprint, arXiv:1404.7641, 2014.

\bibitem[AMP13]{amp1}
A.~Abbondandolo, L.~Macarini, and G.~P. Paternain, \emph{On the existence of
  three closed magnetic geodesics for subcritical energies}, Comm. Math. Helv.
  (2013), (to appear), arXiv:1305.1871.

\bibitem[Arn86]{arn}
V.~I. Arnol$'$d, \emph{The first steps of symplectic topology}, Russian Math.
  Surveys \textbf{41} (1986), no.~6, 1--21.

\bibitem[Arn97]{arn2}
V.~I. Arnol{\cprime}d, \emph{Remarks on the {M}orse theory of a divergence-free
  vector field, the averaging method, and the motion of a charged particle in a
  magnetic field}, Tr. Mat. Inst. Steklova \textbf{216} (1997), no.~Din. Sist.
  i Smezhnye Vopr., 9--19.

\bibitem[AS06]{absc}
A.~Abbondandolo and M.~Schwarz, \emph{On the {F}loer homology of cotangent
  bundles}, Comm. Pure Appl. Math. \textbf{59} (2006), no.~2, 254--316.

\bibitem[Bal83]{bal}
W.~Ballmann, \emph{On the lengths of closed geodesics on convex surfaces},
  Invent. Math. \textbf{71} (1983), no.~3, 593--597.

\bibitem[Ban86]{ban}
V.~Bangert, \emph{On the lengths of closed geodesics on almost round spheres},
  Math. Z. \textbf{191} (1986), no.~4, 549--558.

\bibitem[BF11]{bf}
Y.~Bae and U.~Frauenfelder, \emph{Continuation homomorphism in {R}abinowitz
  {F}loer homology for symplectic deformations}, Math. Proc. Cambridge Philos.
  Soc. \textbf{151} (2011), no.~3, 471--502.

\bibitem[Bir13]{bir1}
G.~D. Birkhoff, \emph{Proof of {P}oincar\'e's geometric theorem}, Trans. Amer.
  Math. Soc. \textbf{14} (1913), no.~1, 14--22.

\bibitem[Bir66]{bir4}
\bysame, \emph{Dynamical systems}, With an addendum by J\"urgen Moser. American
  Mathematical Society Colloquium Publications, Vol. IX, American Mathematical
  Society, Providence, R.I., 1966.

\bibitem[BO09a]{bo2}
F.~Bourgeois and A.~Oancea, \emph{An exact sequence for contact- and symplectic
  homology}, Invent. Math. \textbf{175} (2009), no.~3, 611--680.

\bibitem[BO09b]{bo}
\bysame, \emph{Symplectic homology, autonomous {H}amiltonians, and
  {M}orse-{B}ott moduli spaces}, Duke Math. J. \textbf{146} (2009), no.~1,
  71--174.

\bibitem[Boo86]{boo}
W.~M. Boothby, \emph{An introduction to differentiable manifolds and
  {R}iemannian geometry}, second ed., Pure and Applied Mathematics, vol. 120,
  Academic Press Inc., Orlando, FL, 1986.

\bibitem[Bou09]{bou}
F.~Bourgeois, \emph{A survey of contact homology}, New perspectives and
  challenges in symplectic field theory, CRM Proc. Lecture Notes, vol.~49,
  Amer. Math. Soc., Providence, RI, 2009, pp.~45--71.

\bibitem[BR]{alga}
G.~Benedetti and A.~F. Ritter, \emph{Symplectic {C}ohomology and deformations
  of non-exact symplectic manifolds}, to appear.

\bibitem[Bro12]{bro}
L.~E.~J. Brouwer, \emph{Beweis des ebenen {T}ranslationssatzes}, Math. Ann.
  \textbf{72} (1912), no.~1, 37--54.

\bibitem[BT98]{bahtai}
A.~Bahri and I.~A. Ta{\u\i}manov, \emph{Periodic orbits in magnetic fields and
  {R}icci curvature of {L}agrangian systems}, Trans. Amer. Math. Soc.
  \textbf{350} (1998), no.~7, 2697--2717.

\bibitem[Cas01]{cas}
C.~Castilho, \emph{The motion of a charged particle on a {R}iemannian surface
  under a non-zero magnetic field}, J. Differential Equations \textbf{171}
  (2001), no.~1, 110--131.

\bibitem[CF09]{cf}
K.~Cieliebak and U.~Frauenfelder, \emph{A {F}loer homology for exact contact
  embeddings}, Pacific J. Math. \textbf{239} (2009), no.~2, 251--316.

\bibitem[CFO10]{cfo}
K.~Cieliebak, U.~Frauenfelder, and A.~Oancea, \emph{Rabinowitz {F}loer homology
  and symplectic homology}, Ann. Sci. \'Ec. Norm. Sup\'er. (4) \textbf{43}
  (2010), no.~6, 957--1015.

\bibitem[CFP10]{cfp}
K.~Cieliebak, U.~Frauenfelder, and G.~P. Paternain, \emph{Symplectic topology
  of {M}a\~n\'e's critical values}, Geom. Topol. \textbf{14} (2010), no.~3,
  1765--1870.

\bibitem[CGH12]{crihut}
D.~Cristofaro-Gardiner and M.~Hutchings, \emph{From one {R}eeb orbit to two},
  preprint, arXiv:1202.4839, 2012.

\bibitem[CGK04]{cgk}
K.~Cieliebak, V.~L. Ginzburg, and E.~Kerman, \emph{Symplectic homology and
  periodic orbits near symplectic submanifolds}, Comment. Math. Helv.
  \textbf{79} (2004), no.~3, 554--581.

\bibitem[CIPP98]{cipp}
G.~Contreras, R.~Iturriaga, G.~P. Paternain, and M.~Paternain, \emph{Lagrangian
  graphs, minimizing measures and {M}a\~n\'e's critical values}, Geom. Funct.
  Anal. \textbf{8} (1998), no.~5, 788--809.

\bibitem[CMP04]{cmp}
G.~Contreras, G.~Macarini, and G.~P. Paternain, \emph{Periodic orbits for exact
  magnetic flows on surfaces}, Int. Math. Res. Not. (2004), no.~8, 361--387.

\bibitem[CO04]{conoli}
G.~Contreras and F.~Oliveira, \emph{{$C^2$} densely the 2-sphere has an
  elliptic closed geodesic}, Ergodic Theory Dynam. Systems \textbf{24} (2004),
  no.~5, 1395--1423.

\bibitem[Con06]{con}
G.~Contreras, \emph{The {P}alais-{S}male condition on contact type energy
  levels for convex {L}agrangian systems}, Calc. Var. Partial Differential
  Equations \textbf{27} (2006), no.~3, 321--395.

\bibitem[DDE95]{dadoe}
G.~Dell'Antonio, B.~D'Onofrio, and I.~Ekeland, \emph{Periodic solutions of
  elliptic type for strongly nonlinear {H}amiltonian systems}, The {F}loer
  memorial volume (H.~Hofer, C.~H. Taubes, A.~Weinstein, and E.~Zehnder, eds.),
  IAS/Park City Math. Ser., vol. 133, Birkh\"auser Verlag, Basel, 1995,
  pp.~327--333.

\bibitem[Eke86]{eke1}
I.~Ekeland, \emph{An index theory for periodic solutions of convex
  {H}amiltonian systems}, Nonlinear functional analysis and its applications,
  {P}art 1 ({B}erkeley, {C}alif., 1983), Proc. Sympos. Pure Math., vol.~45,
  Amer. Math. Soc., Providence, RI, 1986, pp.~395--423.

\bibitem[Eke90]{eke2}
\bysame, \emph{Convexity methods in {H}amiltonian mechanics}, Ergebnisse der
  Mathematik und ihrer Grenzgebiete (3) [Results in Mathematics and Related
  Areas (3)], vol.~19, Springer-Verlag, Berlin, 1990.

\bibitem[Flo89]{flo}
A.~Floer, \emph{Symplectic fixed points and holomorphic spheres}, Comm. Math.
  Phys. \textbf{120} (1989), no.~4, 575--611.

\bibitem[FMP12]{fmp1}
U.~Frauenfelder, W.~J. Merry, and G.~P. Paternain, \emph{Floer homology for
  magnetic fields with at most linear growth on the universal cover}, J. Funct.
  Anal. \textbf{262} (2012), no.~7, 3062--3090.

\bibitem[FMP13]{fmp2}
\bysame, \emph{Floer homology for non-resonant magnetic fields on flat tori},
  preprint, arXiv:1305.3141v2 [math.SG], 2013.

\bibitem[Fra92]{fra2}
J.~Franks, \emph{Geodesics on {$S^2$} and periodic points of annulus
  homeomorphisms}, Invent. Math. \textbf{108} (1992), no.~2, 403--418.

\bibitem[Fra96]{fra3}
\bysame, \emph{Area preserving homeomorphisms of open surfaces of genus zero},
  New York J. Math. \textbf{2} (1996), 1--19, electronic.

\bibitem[FS07]{frasch2}
U.~Frauenfelder and F.~Schlenk, \emph{Hamiltonian dynamics on convex symplectic
  manifolds}, Israel J. Math. \textbf{159} (2007), 1--56.

\bibitem[Gei08]{gei}
H.~Geiges, \emph{An introduction to contact topology}, Cambridge Studies in
  Advanced Mathematics, vol. 109, Cambridge University Press, Cambridge, 2008.

\bibitem[GG04]{gg1}
V.~L. Ginzburg and B.~Z. G{\"u}rel, \emph{Relative {H}ofer-{Z}ehnder capacity
  and periodic orbits in twisted cotangent bundles}, Duke Math. J. \textbf{123}
  (2004), no.~1, 1--47.

\bibitem[GG09]{gg2}
\bysame, \emph{Periodic orbits of twisted geodesic flows and the
  {W}einstein-{M}oser theorem}, Comment. Math. Helv. \textbf{84} (2009), no.~4,
  865--907.

\bibitem[Gin87]{gin1}
V.~L. Ginzburg, \emph{New generalizations of {P}oincar\'e's geometric theorem},
  Functional Anal. Appl. \textbf{21} (1987), no.~2, 100--106.

\bibitem[Gin96]{gin2}
\bysame, \emph{On closed trajectories of a charge in a magnetic field. {A}n
  application of symplectic geometry}, Contact and symplectic geometry
  ({C}ambridge, 1994), Publ. Newton Inst., vol.~8, Cambridge Univ. Press,
  Cambridge, 1996, pp.~131--148.

\bibitem[GK99]{ginker1}
V.~L. Ginzburg and E.~Kerman, \emph{Periodic orbits in magnetic fields in
  dimensions greater than two}, Geometry and topology in dynamics
  ({W}inston-{S}alem, {NC}, 1998/{S}an {A}ntonio, {TX}, 1999), Contemp. Math.,
  vol. 246, Amer. Math. Soc., Providence, RI, 1999, pp.~113--121.

\bibitem[GK02a]{ginker2}
\bysame, \emph{Periodic orbits of {H}amiltonian flows near symplectic extrema},
  Pacific J. Math. \textbf{206} (2002), no.~1, 69--91.

\bibitem[GK02b]{guka}
S.~Gudmundsson and E.~Kappos, \emph{On the geometry of tangent bundles}, Expo.
  Math. \textbf{20} (2002), no.~1, 1--41.

\bibitem[Gui76]{gui}
V.~Guillemin, \emph{The {R}adon transform on {Z}oll surfaces}, Advances in
  Math. \textbf{22} (1976), no.~1, 85--119.

\bibitem[HK99]{hofkri}
H.~Hofer and M.~Kriener, \emph{Holomorphic curves in contact dynamics},
  Differential equations: {L}a {P}ietra 1996 ({F}lorence), Proc. Sympos. Pure
  Math., vol.~65, Amer. Math. Soc., Providence, RI, 1999, pp.~77--131.

\bibitem[HLS13]{hls}
U.~L. Hryniewicz, J.~E. Licata, and P.~A.~S. Salom{\~a}o, \emph{A dynamical
  characterization of universally tight lens spaces}, preprint,
  arXiv:1306.6617, 2013.

\bibitem[HP08]{pathar}
A.~Harris and G.~P. Paternain, \emph{Dynamically convex {F}insler metrics and
  {$J$}-holomorphic embedding of asymptotic cylinders}, Ann. Global Anal. Geom.
  \textbf{34} (2008), no.~2, 115--134.

\bibitem[HS95]{hofsal}
H.~Hofer and D.~A. Salamon, \emph{Floer homology and {N}ovikov rings}, The
  {F}loer memorial volume, Progr. Math., vol. 133, Birkh\"auser, Basel, 1995,
  pp.~483--524.

\bibitem[HS12]{hrysal1}
U.~L. Hryniewicz and P.~A.~S. Salom{\~a}o, \emph{Global properties of tight
  {R}eeb flows with applications to {F}insler geodesic flows on ${S}^2$}, Math.
  Proc. Camb. Phil. Soc. \textbf{154} (2012), no.~1, 1--27.

\bibitem[HT09]{huttau1}
M.~Hutchings and C.~Taubes, \emph{The {W}einstein conjecture for stable
  {H}amiltonian structures}, Geom. Topol. \textbf{13} (2009), no.~2, 901--941.

\bibitem[Hut10]{hut}
M.~Hutchings, \emph{Taubes's proof of the {W}einstein {C}onjecture in dimension
  three}, Bull. Amer. Math. Soc. (N.S.) \textbf{47} (2010), no.~1, 73--125.

\bibitem[HWZ98]{hwz1}
H.~Hofer, K.~Wysocki, and E.~Zehnder, \emph{The dynamics on three-dimensional
  strictly convex energy surfaces}, Ann. of Math. (2) \textbf{148} (1998),
  no.~1, 197--289.

\bibitem[HWZ03]{hwz2}
\bysame, \emph{Finite energy foliations of tight three-spheres and
  {H}amiltonian dynamics}, Ann. of Math. (2) \textbf{157} (2003), no.~1,
  125--255.

\bibitem[Kat73]{kat}
A.~Katok, \emph{Ergodic perturbations of degenerate integrable {H}amiltonian
  systems}, Math. USSR Izv. \textbf{7} (1973), 535--572.

\bibitem[Ker99]{ker}
E.~Kerman, \emph{Periodic orbits of {H}amiltonian flows near symplectic
  critical submanifolds}, Internat. Math. Res. Notices (1999), no.~17,
  953--969.

\bibitem[Ker05]{ker3}
\bysame, \emph{Squeezing in {F}loer theory and refined {H}ofer-{Z}ehnder
  capacities of sets near symplectic submanifolds}, Geom. Topol. \textbf{9}
  (2005), 1775--1834.

\bibitem[Kha79]{kha}
M.~P. Kharlamov, \emph{Some applications of differential geometry in the theory
  of mechanical systems}, Mekh. Tverd. Tela (1979), no.~11, 37--49, 118.

\bibitem[Kli95]{kli}
W.~P.~A. Klingenberg, \emph{Riemannian geometry}, second ed., de Gruyter
  Studies in Mathematics, vol.~1, Walter de Gruyter \& Co., Berlin, 1995.

\bibitem[Kob56]{kob}
S.~Kobayashi, \emph{Principal fibre bundles with the {$1$}-dimensional toroidal
  group}, T\^ohoku Math. J. (2) \textbf{8} (1956), 29--45.

\bibitem[Koh09]{koh}
D.~Koh, \emph{On the evolution equation for magnetic geodesics}, Calc. Var.
  Partial Differential Equations \textbf{36} (2009), no.~3, 453--480.

\bibitem[Lon02]{lon}
Y.~Long, \emph{Index theory for symplectic paths with applications}, Progress
  in Mathematics, vol. 207, Birkh\"auser Verlag, Basel, 2002.

\bibitem[Lu06]{lu}
G.~Lu, \emph{Finiteness of the {H}ofer-{Z}ehnder capacity of neighborhoods of
  symplectic submanifolds}, Int. Math. Res. Not. (2006), Art. ID 76520, 33.

\bibitem[Mac03]{mac2}
L.~Macarini, \emph{Hofer-{Z}ehnder semicapacity of cotangent bundles and
  symplectic submanifolds}, preprint, arXiv:math/0303230v3, 2003.

\bibitem[Mac04]{mac1}
\bysame, \emph{Hofer-{Z}ehnder capacity and {H}amiltonian circle actions},
  Commun. Contemp. Math. \textbf{6} (2004), no.~6, 913--945.

\bibitem[McD87]{mcd}
D.~McDuff, \emph{Applications of convex integration to symplectic and contact
  geometry}, Ann. Inst. Fourier (Grenoble) \textbf{37} (1987), no.~1, 107--133.

\bibitem[Mer10]{mer1}
W.~J. Merry, \emph{Closed orbits of a charge in a weakly exact magnetic field},
  Pacific J. Math. \textbf{247} (2010), no.~1, 189--212.

\bibitem[Mer11]{mer2}
\bysame, \emph{On the {R}abinowitz {F}loer homology of twisted cotangent
  bundles}, Calc. Var. Partial Differential Equations \textbf{42} (2011),
  no.~3-4, 355--404.

\bibitem[Mof69]{mof}
H.~K. Moffatt, \emph{The degree of knottedness of tangled vortex lines}, J.
  Fluid Mech. \textbf{106} (1969), 117--129.

\bibitem[Mos65]{mos}
J.~Moser, \emph{On the volume elements on a manifold}, Trans. Amer. Math. Soc.
  \textbf{120} (1965), 286--294.

\bibitem[Mos77]{mos3}
\bysame, \emph{Proof of a generalized form of a fixed point theorem due to {G}.
  {D}. {B}irkhoff}, Geometry and topology ({P}roc. {III} {L}atin {A}mer.
  {S}chool of {M}ath., {I}nst. {M}at. {P}ura {A}plicada {CNP}q, {R}io de
  {J}aneiro, 1976), Springer, Berlin, 1977, pp.~464--494. Lecture Notes in
  Math., Vol. 597.

\bibitem[MS74]{milsta}
J.~W. Milnor and J.~D. Stasheff, \emph{Characteristic classes}, Princeton
  University Press, Princeton, N. J.; University of Tokyo Press, Tokyo, 1974,
  Annals of Mathematics Studies, No. 76.

\bibitem[Nov81a]{nov2}
S.~P. Novikov, \emph{Multivalued functions and functionals. {A}n analogue of
  the {M}orse theory}, Dokl. Akad. Nauk SSSR \textbf{260} (1981), no.~1,
  31--35.

\bibitem[Nov81b]{nov3}
\bysame, \emph{Variational methods and periodic solutions of equations of
  {K}irchhoff type. {II}}, Funktsional. Anal. i Prilozhen. \textbf{15} (1981),
  no.~4, 37--52, 96.

\bibitem[Nov82]{nov4}
\bysame, \emph{The {H}amiltonian formalism and a multivalued analogue of
  {M}orse theory}, Uspekhi Mat. Nauk \textbf{37} (1982), no.~5(227), 3--49,
  248.

\bibitem[NS81]{nov1}
S.~P. Novikov and I.~Shmel'tser, \emph{Periodic solutions of {K}irchhoff
  equations for the free motion of a rigid body in a fluid and the extended
  {L}yusternik-{S}hnirel'man-{M}orse theory. {I}}, Funktsional. Anal. i
  Prilozhen. \textbf{15} (1981), no.~3, 54--66.

\bibitem[Oan04]{oan}
A.~Oancea, \emph{A survey of {F}loer homology for manifolds with contact type
  boundary or symplectic homology}, Symplectic geometry and {F}loer homology. A
  survey of {F}loer homology for manifolds with contact type boundary or
  symplectic homology, Ensaios Mat., vol.~7, Soc. Brasil. Mat., Rio de Janeiro,
  2004, pp.~51--91.

\bibitem[Osu05]{osu}
O.~Osuna, \emph{Periodic orbits of weakly exact magnetic flows}, preprint,
  2005.

\bibitem[Pat99a]{pat}
G.~P. Paternain, \emph{Geodesic flows}, Progress in Mathematics, vol. 180,
  Birkh\"auser Boston, Inc., Boston, MA, 1999.

\bibitem[Pat99b]{patnot}
\bysame, \emph{On two noteworthy deformations of negatively curved {R}iemannian
  metrics}, Discrete Contin. Dynam. Systems \textbf{5} (1999), no.~3, 639--650.

\bibitem[Pat06]{pat2}
\bysame, \emph{Magnetic rigidity of horocycle flows}, Pacific J. Math.
  \textbf{225} (2006), no.~2, 301--323.

\bibitem[Pat09]{pat1}
\bysame, \emph{Helicity and the {M}a\~n\'e critical value}, Algebr. Geom.
  Topol. \textbf{9} (2009), no.~3, 1413--1422.

\bibitem[Poi12]{poi2}
H.~Poincar{\'e}, \emph{Sur un th\'eor\'eme de g\'eom\'etrie}, Rend. Circ. Mat.
  Palermo \textbf{33} (1912), 375--407.

\bibitem[Poi87]{poi1}
\bysame, \emph{Les m\'ethodes nouvelles de la m\'ecanique c\'eleste. {T}ome
  {III}}, Les Grands Classiques Gauthier-Villars. [Gauthier-Villars Great
  Classics], vol. III, ch.~33, Librairie Scientifique et Technique Albert
  Blanchard, 1987, Reprint of the 1899 original, Biblioth{\`e}que Scientifique
  Albert Blanchard.

\bibitem[Pol98]{pol1}
L.~Polterovich, \emph{Geometry on the group of {H}amiltonian diffeomorphisms},
  Proceedings of the {I}nternational {C}ongress of {M}athematicians, {V}ol.
  {II} ({B}erlin, 1998), no. Extra Vol. II, 1998, pp.~401--410 (electronic).

\bibitem[PP97]{pats}
G.~P. Paternain and M.~Paternain, \emph{Critical values of autonomous
  {L}agrangian systems}, Comment. Math. Helv. \textbf{72} (1997), no.~3,
  481--499.

\bibitem[Rab78]{rab1}
P.~H. Rabinowitz, \emph{Periodic solutions of {H}amiltonian systems}, Comm.
  Pure Appl. Math. \textbf{31} (1978), no.~2, 157--184.

\bibitem[Rab79]{rab2}
\bysame, \emph{Periodic solutions of a {H}amiltonian system on a prescribed
  energy surface}, J. Differential Equations \textbf{33} (1979), no.~3,
  336--352.

\bibitem[Rit10]{ar}
A.~F. Ritter, \emph{Deformations of symplectic cohomology and exact
  {L}agrangians in {ALE} spaces}, Geom. Funct. Anal. \textbf{20} (2010), no.~3,
  779--816.

\bibitem[Rit13]{ritqft}
\bysame, \emph{Topological quantum field theory structure on symplectic
  cohomology}, J. Topol. \textbf{6} (2013), no.~2, 391--489.

\bibitem[RVN13]{rasa}
N.~Raymond and S.~V{\~u}~Ng{\d{o}}c, \emph{Geometry and spectrum in magnetic
  2{D} wells}, preprint, arXiv:1306.5054, 2013.

\bibitem[Sal99]{sal}
D.~Salamon, \emph{Lectures on {F}loer homology}, Symplectic geometry and
  topology ({P}ark {C}ity, {UT}, 1997), IAS/Park City Math. Ser., vol.~7, Amer.
  Math. Soc., Providence, RI, 1999, pp.~143--229.

\bibitem[Sch06]{schl}
F.~Schlenk, \emph{Applications of {H}ofer's geometry to {H}amiltonian
  dynamics}, Comment. Math. Helv. \textbf{81} (2006), no.~1, 105--121.

\bibitem[Sch11]{sch1}
M.~Schneider, \emph{Closed magnetic geodesics on {$S^2$}}, J. Differential
  Geom. \textbf{87} (2011), no.~2, 343--388.

\bibitem[Sch12a]{sch2}
\bysame, \emph{Alexandrov embedded closed magnetic geodesics on {$S^2$}},
  Ergodic Theory Dynam. Systems \textbf{32} (2012), no.~4, 1471--1480.

\bibitem[Sch12b]{sch3}
\bysame, \emph{Closed magnetic geodesics on closed hyperbolic {R}iemann
  surfaces}, Proc. Lond. Math. Soc. (3) \textbf{105} (2012), no.~2, 424--446.

\bibitem[Sei08]{sei}
P.~Seidel, \emph{A biased view of symplectic cohomology}, Current developments
  in mathematics, 2006, Int. Press, Somerville, MA, 2008, pp.~211--253.

\bibitem[SW06]{sawe}
D.~A. Salamon and J.~Weber, \emph{Floer homology and the heat flow}, Geom.
  Funct. Anal. \textbf{16} (2006), no.~5, 1050--1138.

\bibitem[Ta{\u\i}91]{tai2}
I.~A. Ta{\u\i}manov, \emph{Non self-intersecting closed extremals of
  multivalued or not-everywhere-positive functionals}, Izv. Akad. Nauk SSSR
  Ser. Mat. \textbf{55} (1991), no.~2, 367--383.

\bibitem[Ta{\u\i}92]{tai4}
\bysame, \emph{Closed extremals on two-dimensional manifolds}, Uspekhi Mat.
  Nauk \textbf{47} (1992), no.~2(284), 143--185, 223.

\bibitem[Ta{\u\i}10]{tai5}
\bysame, \emph{Periodic magnetic geodesics on almost every energy level via
  variational methods}, Regul. Chaotic Dyn. \textbf{15} (2010), no.~4-5,
  598--605.

\bibitem[Tau07]{tau1}
C.~H. Taubes, \emph{The {S}eiberg-{W}itten equations and the {W}einstein
  conjecture}, Geom. Topol. \textbf{11} (2007), 2117--2202.

\bibitem[Ush09]{ush}
M.~Usher, \emph{Floer homology in disk bundles and symplectically twisted
  geodesic flows}, J. Mod. Dyn. \textbf{3} (2009), no.~1, 61--101.

\bibitem[Vit96]{vit2}
C.~Viterbo, \emph{Functors and {C}omputations in {F}loer homology with
  {A}pplications {P}art {II}}, preprint, revised 2003,
  \url{http://www.math.ens.fr/~viterbo/FCFH.II.2003.pdf}, 1996.

\bibitem[Vit99]{vit}
\bysame, \emph{Functors and computations in {F}loer homology with applications.
  {I}}, Geom. Funct. Anal. \textbf{9} (1999), no.~5, 985--1033.

\bibitem[Wei75]{wei3}
A.~Weinstein, \emph{Fourier integral operators, quantization, and the spectra
  of {R}iemannian manifolds}, G\'eom\'etrie symplectique et physique
  math\'ematique ({C}olloq. {I}nternat. {CNRS}, {N}o. 237, {A}ix-en-{P}rovence,
  1974), \'Editions Centre Nat. Recherche Sci., Paris, 1975, With questions by
  W. Klingenberg and K. Bleuler and replies by the author, pp.~289--298.

\bibitem[Wei78]{wei1}
\bysame, \emph{Periodic orbits for convex {H}amiltonian systems}, Ann. of Math.
  (2) \textbf{108} (1978), no.~3, 507--518.

\bibitem[Wei79]{wei2}
\bysame, \emph{On the hypotheses of {R}abinowitz' periodic orbit theorems}, J.
  Differential Equations \textbf{33} (1979), no.~3, 353--358.

\bibitem[Zil77]{zil2}
W.~Ziller, \emph{The free loop space of globally symmetric spaces}, Invent.
  Math. \textbf{41} (1977), no.~1, 1--22.

\bibitem[Zil83]{zil}
\bysame, \emph{Geometry of the {K}atok examples}, Ergodic Theory Dynam. Systems
  \textbf{3} (1983), no.~1, 135--157.

\end{thebibliography}
\end{document}